\documentclass[a4paper]{amsart}

\usepackage{hyperref}
\usepackage{relsize}
\usepackage{cleveref}
\usepackage[dvipsnames]{xcolor}
\usepackage{amssymb}
\usepackage{enumitem}
\usepackage[margin=1.0in]{geometry}
\usepackage{mathtools}
\usepackage{array}
\usepackage{bbm} %
\usepackage{bm} %
\usepackage{wasysym} %
\usepackage{changepage} %
\usepackage{tikz}
    \usetikzlibrary{positioning,calc}
\usepackage{tikz-cd}
\usepackage{tabstackengine}
\setstackgap{L}{.75\baselineskip}

\usepackage[shortcuts]{extdash} %

\usepackage[dvipsnames]{xcolor}

\usepackage{subcaption}
\usepackage[mathscr]{euscript}

\usepackage{pdflscape}
\usepackage{afterpage}
\usepackage{capt-of}
\usepackage{stmaryrd}

\DeclareSymbolFont{symbolsC}{U}{txsyc}{m}{n}
\SetSymbolFont{symbolsC}{bold}{U}{txsyc}{bx}{n}
\DeclareFontSubstitution{U}{txsyc}{m}{n}
\DeclareMathSymbol{\multimapboth}{\mathrel}{symbolsC}{"13}

\numberwithin{equation}{section}

\newtheorem{theorem}{Theorem}[section]
\newtheorem{lemma}[theorem]{Lemma}
\newtheorem{proposition}[theorem]{Proposition}
\newtheorem{corollary}[theorem]{Corollary}

\theoremstyle{definition}
\newtheorem{definition}[theorem]{Definition}

\newtheorem{construction}[theorem]{Construction}

\newtheorem{convention}[theorem]{Convention}

\theoremstyle{remark}
\newtheorem{remark}[theorem]{Remark}

\definecolor{niceblue}{rgb}{0.0,0.4,0.9}

\newcommand{\closedtraj}{OliveGreen!50}
\newcommand{\saddletraj}{black}
\newcommand{\generictraj}{gray!50}
\newcommand{\septraj}{black}
\newcommand{\conn}{niceblue}
\newcommand{\otherconn}{red}
\newcommand{\arrcol}{black}
\newcommand{\edgecol}{gray}

\newcommand{\surf}[1]{\mathbf{#1}}
\newcommand{\bs}{\surf{S}}
\newcommand{\sfm}{\surf{M}}
\newcommand{\bp}{\surf{P}}
\newcommand{\bo}{\surf{O}}
\newcommand{\mbso}{(\bs, \sfm, \bo)}
\newcommand{\triangnum}[2]{\Delta^{\circlearrowleft}(#1, #2)} %

\newcommand{\ctb}{\omega_{X}}
\newcommand{\ctbs}{\ctb^{\otimes 2}}
\DeclareMathOperator{\crit}{Crit}
\newcommand{\infcrit}{\crit_{\infty}}
\newcommand{\nocrit}[1]{{#1}^{\circ}}
\DeclareMathOperator{\im}{Im}
\newcommand{\uhp}{\mathbb{H}}
\newcommand{\hh}[1]{\hat{H}(#1)} %
\newcommand{\hhspan}[1]{\hat{H}^{\mathrm{s}}(#1)}
\newcommand{\hhc}[1]{\bm{#1}} %
\newcommand{\spc}[1]{\overline{\Sigma}} %
\newcommand{\spct}[1]{\widetilde{\Sigma}} %
\newcommand{\infpreim}{D_{\infty}}

\newcommand{\st}{:} %

\DeclareMathOperator{\Ext}{Ext}

\DeclareMathOperator{\End}{End}
\DeclareMathOperator{\Hom}{Hom}
\DeclareMathOperator{\additive}{add}
\DeclareMathOperator{\modules}{mod}
\DeclareMathOperator{\Sim}{Sim}
\newcommand{\lorth}[1]{\prescript{\perp}{}{#1}}
\newcommand{\rorth}[1]{{#1}^{\perp}}

\newcommand{\simp}[1]{\Sim #1}
\newcommand{\simpinf}[1]{\Sim(#1)}
\newcommand{\extclos}[1]{\langle #1 \rangle}

\DeclareMathOperator{\dimu}{\underline{\dim}}
\newcommand{\pot}[1]{\mathrm{Pot}(#1)} %
\newcommand{\jac}[1]{\widehat{J}(#1)} %
\newcommand{\jacu}[1]{J(#1)} %
\newcommand{\wmin}{W_{\!m}}
\newcommand{\akd}[2]{\mathcal{A}(#1, #2)} %
\newcommand{\akddb}{\mathcal{A}(\qwdb)} %
\newcommand{\akdqw}{\akd{Q}{W}}

\newcommand{\vm}{S} %
\newcommand{\am}{E} %
\newcommand{\cyc}{\mathrm{cyc}} %

\newcommand{\glnc}{\mathrm{GL}_n(\mathbb{C})}
\newcommand{\gld}{\mathrm{GL}_{\ud}(\mathbb{C})}
\newcommand{\gla}[1]{\mathrm{GL}_{#1}(\mathbb{C})}

\newcommand{\jord}[2]{\mathrm{Jord}(#1, #2)} %

\newcommand{\obj}[1]{\mathrm{obj}(#1)} %
\newcommand{\tw}[1]{\mathsf{tw}(#1)} %
\newcommand{\twz}[1]{\mathsf{tw}_{0}(#1)} %
\newcommand{\two}{\mathrm{tw}} %
\newcommand{\id}{\mathrm{id}} %
\newcommand{\cy}[2]{( #1, #2 )^{\mathrm{CY}}} %
\newcommand{\ef}[2]{\langle #1, #2 \rangle_{Q}} %
\newcommand{\naef}[2]{\chi(#1, #2)_{Q}} %
\newcommand{\intp}[2]{\langle #1, #2 \rangle_{\cap}} %
\newcommand{\lef}[2]{\langle #1, #2 \rangle_{\mathrm{Lef}}} %

\newcommand{\ideal}[1]{\langle #1 \rangle} %
\DeclarePairedDelimiter{\ceil}{\lceil}{\rceil}

\newcommand{\db}{{\circ\!-\!\circ}}
\newcommand{\kb}{{\circ\!\circ}}
\newcommand{\qwdb}{Q^{\db}, W^{\db}}
\newcommand{\qwkb}{Q^{\kb}, W^{\kb}}
\newcommand{\dinf}[1]{\prescript{\infty}{}{#1}^{\infty}} %
\newcommand{\kbcat}{\mathcal{C}_1} %
\newcommand{\olcat}{\mathcal{C}_2} %

\DeclareMathOperator{\Stab}{Stab}
\newcommand{\cl}{\mathrm{cl}} %
\newcommand{\fgg}{\Gamma} %
\newcommand{\zss}{\mathrm{ss}}
\newcommand{\ssc}[2]{\mathcal{P}^{#1}(#2)} %
\newcommand{\ssce}[2]{\mathcal{P}_{\infty}^{#1}(#2)} %
\DeclareMathOperator{\phase}{ph}
\newcommand{\scph}[1]{\phase(#1)} %
\newcommand{\qdph}[1]{\phase(#1)} %

\newcommand{\rou}{\hat{\mu}} %
\newcommand{\roua}[1]{\mu_{#1}} %

\newcommand{\kzstaffa}[1]{\mathscr{K}_{0,#1}}
\newcommand{\kzstaffm}{\mathscr{K}_{0,\mathfrak{M}}}

\newcommand{\kzgstaffa}[1]{\mathscr{K}_{0,#1}^{G}}
\newcommand{\kzgstaffm}{\mathscr{K}_{0,\mathfrak{M}}^{G}}

\newcommand{\kgstaffa}[1]{\mathscr{K}^{G}_{#1}}
\newcommand{\kgstaffm}{\mathscr{K}^{G}_\mathfrak{M}}
\newcommand{\kstaffc}{\mathscr{K}_{\mathrm{pt}}}

\newcommand{\kgvarm}{\mathsf{K}^{G}_\mathfrak{M}}

\newcommand{\kastaffm}[1]{\mathscr{K}^{#1}_\mathfrak{M}}

\newcommand{\kmstaffm}{\mathscr{K}^{\rou}_\mathfrak{M}}
\newcommand{\kmstaffc}{\mathscr{K}^{\rou}_{\mathrm{pt}}}
\newcommand{\kmstaffa}[1]{\mathscr{K}^{\rou}_{#1}}

\newcommand{\kmvarc}{\mathsf{K}^{\rou}_{\mathrm{pt}}}
\newcommand{\kvarc}{\mathsf{K}_{\mathrm{pt}}}
\newcommand{\kmvara}[1]{\mathsf{K}^{\rou}_{#1}}

\newcommand{\kstaffqw}{\mathscr{K}_{\qwst}}
\newcommand{\kstaffqwn}{\mathscr{K}_{\qwstn}}
\newcommand{\kstaffq}{\mathscr{K}_{\qst}}
\newcommand{\kstaffqn}{\mathscr{K}_{\qstn}}
\newcommand{\kvarxdn}{\mathsf{K}_{\qwdstn}}
\newcommand{\kvarxd}{\mathsf{K}_{\qwdst}}
\newcommand{\kstaffqwdb}{\mathscr{K}_{\dbst}}

\newcommand{\kbgstaffm}{\overline{\mathscr{K}}^{G}_\mathfrak{M}}
\newcommand{\kbmvara}[1]{\overline{\mathsf{K}}^{\rou}_#1}
\newcommand{\kbmstaffc}{\overline{\mathscr{K}}^{\rou}_{\mathrm{pt}}}

\newcommand{\tsp}{\omega} %

\newcommand{\qring}{{\mathscr{R}}_{Q}} %
\newcommand{\qringdb}{{\mathscr{R}}_{Q^{\db}}} %

\newcommand{\vcf}{\phi_f} %
\newcommand{\vca}[1]{\phi_{#1}}
\newcommand{\vcfrel}{\phi_{f}^{\mathrm{rel}}}
\newcommand{\vcrela}[1]{\phi_{#1}^{\mathrm{rel}}}
\newcommand{\crloc}{\mathsf{CL}}

\newcommand{\exprod}{\astrosun}

\newcommand{\intks}{\Phi^{\mathrm{KS}}}
\newcommand{\intbbs}{\Phi^{\mathrm{BBS}}}

\newcommand{\sdet}[1]{\mathsf{sDet}(#1)}
\newcommand{\cop}{\mathsf{P}} %
\newcommand{\odchoice}{\delta_{Q, W}} %
\newcommand{\odcha}[1]{\delta_{#1}}
\newcommand{\odqf}{F} %

\newcommand{\ag}[1]{\mathsf{#1}} %
\newcommand{\qrlet}{M}

\newcommand{\qdsch}{\ag{\qrlet}_{Q, \ud}}
\newcommand{\qdschn}{\ag{\qrlet}_{Q, \ud}^{\mathrm{nil}}}
\newcommand{\qwdsch}{\ag{\qrlet}_{Q, W, \ud}}
\newcommand{\qwdschn}{\ag{\qrlet}_{Q, W, \ud}^{\mathrm{nil}}}
\newcommand{\qaasch}[2]{\ag{\qrlet}_{#1, #2}}
\newcommand{\qaaschn}[2]{\ag{\qrlet}_{#1, #2}^{\mathrm{nil}}}

\newcommand{\qdst}{\mathfrak{\qrlet}_{Q, \ud}}
\newcommand{\qdstn}{\mathfrak{\qrlet}_{Q, \ud}^{\mathrm{nil}}}
\newcommand{\qwdst}{\mathfrak{\qrlet}_{Q, W, \ud}}
\newcommand{\qwdstn}{\mathfrak{\qrlet}_{Q, W, \ud}^{\mathrm{nil}}}
\newcommand{\qaast}[2]{\mathfrak{\qrlet}_{#1, #2}}
\newcommand{\qaastn}[2]{\mathfrak{\qrlet}_{#1, #2}^{\mathrm{nil}}}
\newcommand{\qast}[1]{\mathfrak{\qrlet}_{#1}}
\newcommand{\qastn}[1]{\mathfrak{\qrlet}_{#1}^{\mathrm{nil}}}

\newcommand{\qst}{\mathfrak{\qrlet}_{Q}}
\newcommand{\qstn}{\mathfrak{\qrlet}_{Q}^{\mathrm{nil}}}
\newcommand{\qwst}{\mathfrak{\qrlet}_{Q, W}}
\newcommand{\qwstn}{\mathfrak{\qrlet}_{Q, W}^{\mathrm{nil}}}

\newcommand{\qwastn}[1]{\mathfrak{\qrlet}_{Q, W, #1}^{\mathrm{nil}}}

\newcommand{\sqdsch}{\ag{\qrlet}_{Q, \ud}^{\zss}}
\newcommand{\sqdschn}{\ag{\qrlet}_{Q, \ud}^{\mathrm{nil}, \zss}}
\newcommand{\sqwdsch}{\ag{\qrlet}_{Q, W, \ud}^{\zss}}
\newcommand{\sqwdschn}{\ag{\qrlet}_{Q, W, \ud}^{\mathrm{nil}, \zss}}

\newcommand{\sqdst}{\mathfrak{\qrlet}_{Q, \ud}^{\zss}}
\newcommand{\sqdstn}{\mathfrak{\qrlet}_{Q, \ud}^{\mathrm{nil}, \zss}}
\newcommand{\sqwdst}{\mathfrak{\qrlet}_{Q, W, \ud}^{\zss}}
\newcommand{\sqwdstn}{\mathfrak{\qrlet}_{Q, W, \ud}^{\mathrm{nil}, \zss}}

\newcommand{\sqpst}{\mathfrak{\qrlet}_{Q, \vartheta}^{\zss}}
\newcommand{\sqpstn}{\mathfrak{\qrlet}_{Q, \vartheta}^{\mathrm{nil}, \zss}}
\newcommand{\sqwpst}{\mathfrak{\qrlet}_{Q, W, \vartheta}^{\zss}}
\newcommand{\sqwpstn}{\mathfrak{\qrlet}_{Q, W, \vartheta}^{\mathrm{nil}, \zss}}

\newcommand{\sdbpst}{\mathfrak{\qrlet}_{\qwdb, \vartheta}^{\zss}}
\newcommand{\dbst}{\mathfrak{\qrlet}_{\qwdb}}
\newcommand{\dbdst}{\mathfrak{\qrlet}_{Q^{\db}, (2d, 2d)}}
\newcommand{\dbstn}{\mathfrak{\qrlet}_{\qwdb}^{\mathrm{nil}}}
\newcommand{\dbcatast}{\mathfrak{\qrlet}_{\qwdb}^{\kbcat}}
\newcommand{\dbcatbst}{\mathfrak{\qrlet}_{\qwdb}^{\olcat}}
\newcommand{\dbcatadsch}{\mathsf{\qrlet}_{\qwdb, (d, d)}^{\kbcat}}
\newcommand{\dbcatbdsch}{\mathsf{\qrlet}_{\qwdb, (2d, 2d)}^{\olcat}}
\newcommand{\dbcatadst}{\mathfrak{\qrlet}_{\qwdb, (d, d)}^{\kbcat}}
\newcommand{\dbcatbdst}{\mathfrak{\qrlet}_{\qwdb, (2d, 2d)}^{\olcat}}

\newcommand{\tr}[1]{\mathrm{Tr}(#1)} %
\newcommand{\trst}[1]{\mathfrak{Tr}(#1)} %

\newcommand{\spec}[1]{\mathrm{Spec}(#1)}

\newcommand{\dt}[3]{\Omega_{#1}^{#2}(#3)} %
\newcommand{\qdl}[1]{\mathbf{E}(#1)} %
\newcommand{\dtnum}[1]{\Omega_{\mathrm{num}}(#1)}
\newcommand{\dtref}[1]{\Omega(#1)}

\newcommand{\genser}[1]{\mathrm{G}(#1)}
\newcommand{\genserp}[2]{\mathrm{G}_{#1}(#2)}
\newcommand{\gensern}[1]{\mathrm{G}^{\mathrm{nil}}(#1)}
\newcommand{\gensernp}[2]{\mathrm{G}_{#1}^{\mathrm{nil}}(#2)}

\newcommand{\ud}{\underline{d}}

\makeatletter
\def\@tocline#1#2#3#4#5#6#7{\relax
  \ifnum #1>\c@tocdepth %
  \else
    \par \addpenalty\@secpenalty\addvspace{#2}%
    \begingroup \hyphenpenalty\@M
    \@ifempty{#4}{%
      \@tempdima\csname r@tocindent\number#1\endcsname\relax
    }{%
      \@tempdima#4\relax
    }%
    \parindent\z@ \leftskip#3\relax \advance\leftskip\@tempdima\relax
    \rightskip\@pnumwidth plus4em \parfillskip-\@pnumwidth
    #5\leavevmode\hskip-\@tempdima
      \ifcase #1
       \or\or \hskip 1em \or \hskip 2em \else \hskip 3em \fi%
      #6\nobreak\relax
    \dotfill\hbox to\@pnumwidth{\@tocpagenum{#7}}\par
    \nobreak
    \endgroup
  \fi}
\makeatother

\usepackage[backend=biber,
	style=numeric,
        giveninits=true,
	url=false,
	doi=false,
	isbn=false,
	maxbibnames=99]{biblatex}
\DeclareFieldFormat{title}{#1} %
\renewbibmacro{in:}{} %
\DeclareFieldFormat{postnote}{#1} %

\title[DT invariants for the Bridgeland--Smith correspondence]{Donaldson--Thomas invariants for \\ the Bridgeland--Smith correspondence}
\author{Omar Kidwai}
\address{School of Mathematics, Watson Building, University of Birmingham, Edgbaston, Birmingham, B15 2TT, United Kingdom}
\address{Theoretical Sciences Visiting Program, Okinawa Institute of Science and Technology Graduate University, Onna, 904-0495, Japan}
\urladdr{http://omarkidwai.ca/}
\email{o.kidwai@bham.ac.uk}
\author{Nicholas J. Williams}
\address{Department of Mathematics and Statistics, Fylde College, Lancaster University, Lancaster, LA1 4YF, United Kingdom}
\address{Theoretical Sciences Visiting Program, Okinawa Institute of Science and Technology Graduate University, Onna, 904-0495, Japan}
\urladdr{https://nchlswllms.github.io/}
\email{nicholas.williams@lancaster.ac.uk}


\begin{document}

\begin{abstract}
Famous work of Bridgeland and Smith shows that certain moduli spaces of quadratic differentials are isomorphic to spaces of stability conditions on particular 3-Calabi--Yau triangulated categories.
This result has subsequently been generalised and extended by several authors.
One facet of this correspondence is that finite-length trajectories of the quadratic differential are related to categories of semistable objects of the corresponding stability condition, which have associated Donaldson--Thomas invariants.
On the other hand, computations in the physics literature suggest certain values of these invariants according to the type of trajectory.
In this paper, we show that the category recently constructed by Christ, Haiden, and Qiu gives Donaldson--Thomas invariants which agree with the predictions from physics; in particular, degenerate ring domains of the quadratic differential give rise to non-zero Donaldson--Thomas invariants.
In calculating all of the invariants, we obtain a novel application of string and band techniques from representation theory.
\end{abstract}

\maketitle

\tableofcontents

\section{Introduction}
Inspired on the one hand by the work of Gaiotto, Moore, and Neitzke describing BPS states in certain $d = 4$, $\mathcal{N} = 2$ gauge theories \cite{gmn_cmp,gmn_adv}, and on the other hand by the observation of Kontsevich and Seidel on the similarity between spaces of quadratic differentials and spaces of stability conditions \cite{bs}, celebrated work of Bridgeland and Smith shows how certain moduli spaces of quadratic differentials on compact Riemann surfaces are equivalent to spaces of stability conditions on particular $3$-Calabi--Yau categories \cite{bs}.
This work has subsequently been generalised to wider classes of differentials and different categories \cite{bmqs,bqs,chq,haiden,hkk,ikeda,iq,kq2,lqw,wu}.

An important part of the work of Bridgeland and Smith is that finite-length trajectories of the quadratic differential correspond to stable objects of the stability condition.
Indeed, such finite-length trajectories correspond to BPS states in the work of Gaiotto, Moore, and Neitzke \cite{gmn_adv}.
Pioneering work of Kontsevich and Soibelman \cite{ks_stability} gives a formalism for counting semistable objects in $3$-Calabi--Yau triangulated $A_\infty$-categories, in which the so-called Donaldson--Thomas (DT) invariants may be interpreted as counting these BPS states.
Their framework generalises the original DT invariants of \cite{thomas} counting stable sheaves on CY 3-folds, whilst another generalisation was given by Joyce and Song \cite{joyce_ci,joyce_cii,joyce_ciii,joyce_civ,js}.

DT theory gives rise to certain infinite-dimensional Riemann--Hilbert problems, in which the DT invariants appear as exponents defining automorphisms of a certain torus \cite{bbs_rh,barbieri_stoppa,bridgeland_rh,fgs}. Related Riemann--Hilbert problems are used in the work of Gaiotto, Moore, and Neitzke \cite{gmn_cmp,gmn_adv} to construct hyperk\"ahler metrics on moduli spaces of Higgs bundles, which appear as target spaces in sigma models which are dimensional reductions of the field theories they consider; see also \cite{bs_hk}.
In cases arising via quadratic differentials, these Riemann--Hilbert problems are known to be related to exact WKB analysis \cite{allegretti2021stability,iwaki2023topological,ini}.
Work of Bridgeland views the DT invariants of a 3-Calabi--Yau triangulated category as encoded by a recently introduced structure on its manifold of stability conditions known as a `Joyce structure' \cite{bridgeland_dt_geometry}, see also \cite{joyce_gt}.
Recovering the Joyce structure on the stability manifold from the DT theory corresponds to solving the Riemann--Hilbert problem.

In apparently unrelated work, Iwaki, Koike, and Takei study the topological recursion for spectral curves given by a class of simple quadratic differentials \cite{ikt1,ikt2}.
Topological recursion is a construction introduced by Eynard and Orantin \cite{eo} abstracting and generalising solutions to the loop equations appearing in the theory of matrix models.
Given various input data, it can produce a remarkable array of different outputs, including Gromov--Witten invariants, Hurwitz numbers, and Mirzakhani's Weil--Petersson volume \cite{bems,dmnps,eynard,ems,eo_gw,eo_wp}.
The result of \cite{ikt1,ikt2} is to find an explicit formula for the topological recursion `free energy', as well as WKB-theoretic invariants called {Voros coefficients}, corresponding to initial data specified by certain simple quadratic differentials.

Work of the first author and Iwaki \cite{ik} then gives the formulas of \cite{ikt1,ikt2} an interpretation in terms of so-called BPS structures associated to the quadratic differentials, and \cite{iwaki2023topological} solves the corresponding Riemann-Hilbert problem. The BPS structure associated to a quadratic differential roughly consists of the set of finite-length trajectories, along with associated numbers according to different types of finite-length trajectories, known as BPS invariants.
A BPS structure describes the output of DT theory given a fixed stability condition.

It was initially unclear whether the BPS structures found in \cite{ik} did indeed give the output of DT theory applied to some particular 3-Calabi--Yau triangulated category.
Work of Haiden obtains some, but not all, of the invariants in a modified setting, in which the quadratic differentials do not have poles of order higher than one \cite{haiden}.
For certain types of finite-length trajectories --- type~I saddle trajectories and non-degenerate ring domains --- it follows from the results of Bridgeland and Smith that the BPS invariant associated to the trajectories coincides with the (numerical) DT invariant associated to corresponding category of semistable objects \cite[Theorem~11.6]{bs}.
However, these do not encompass all types of finite-length trajectories; moreover, using the 3-Calabi--Yau triangulated category from \cite{bs}, the DT invariant associated to families of finite-length trajectories which form so-called `degenerate ring domains' is~$0$, which does \emph{not} agree with the formula from~\cite{ik}.
Another physical computation from the A-brane perspective \cite{blr} also suggests this value should not vanish, but instead give the value $-1$ as in~\cite{ik}.

Building on earlier work in \cite{christ_ps}, in \cite{chq} a new 3-Calabi--Yau triangulated category is associated to triangulated surfaces coming from quadratic differentials by Christ, Haiden, and Qiu.
Using an approach that unifies the results of \cite{bs,hkk}, the authors of \cite{chq} show that stability conditions on this category are also equivalent to certain quadratic differentials.
In their correspondence, stability conditions correspond to quadratic differentials with poles of fixed orders, whereas in \cite{bs} double poles may degenerate to simple poles after collision with a zero.
As observed in \cite{chq}, this is a consequence of the fact that the category from \cite{bs} has no semistable objects corresponding to degenerate ring domains, which is what produces the associated DT invariant of~$0$.

The goal of this paper is to compute the (refined) DT invariants of the 3-Calabi--Yau case of the category from \cite{chq}, thus obtaining full agreement with the invariants suggested in~\cite{ik}.
In particular, rather than using the marked bordered surfaces of \cite{bs}, we consider 3-Calabi--Yau triangulated categories associated to marked bordered surfaces with orbifold points of order three, where the latter represent simple poles of the quadratic differential.

\textbf{Results.}
The main result of this paper is that the 3-Calabi--Yau $A_\infty$-category of \cite{chq} produces the DT invariants predicted for different types of finite-length trajectories by~\cite{ik}.
Indeed, we compute the refined DT invariants in the following theorem.
Here, we have that $\mathcal{D}$ is the 3-Calabi--Yau triangulated category with $\mathcal{D}_{\infty}$ its $A_\infty$-enhancement (Section~\ref{sect:3cy_stab:bs}), $\Gamma$ is its Grothendieck group, which is isomorphic to a certain homology group of a double cover of the Riemann surface $X$ on which the quadratic differential $\varphi$ lives (Section~\ref{sect:qd:hh}).
See Definition~\ref{def:gmn} for the notion of an infinite GMN differential, and Section~\ref{sect:prelim} for the precise definition of generic.

\begin{theorem}[{Section~\ref{sect:main}}]\label{thm:intro}
Let $\varphi$ be a saddle-free generic infinite GMN differential.
Let $\sigma_{\varphi}$ be the corresponding stability condition on~$\mathcal{D}$.
Given a class $\hhc{\gamma} \in \Gamma \setminus \{0\}$, under $\sigma_{\varphi}$ the refined DT invariant of class $\hhc{\gamma}$ is \[
\dtref{\hhc{\gamma}} = N_{\mathrm{I}}(\hhc{\gamma}) + 2N_{\mathrm{II}}(\hhc{\gamma}) + 4N_{\mathrm{III}}(\hhc{\gamma}) + q^{-1/2}N_{\mathrm{DRD}}(\hhc{\gamma}) + (q^{1/2} + q^{-1/2})N_{\mathrm{NRD}}(\hhc{\gamma}).
\]
\end{theorem}

Here $N_{\mathrm{I}}(\hhc{\gamma})$ (resp., $N_{\mathrm{II}}(\hhc{\gamma})$ and $N_{\mathrm{III}}(\hhc{\gamma})$) is the number of type~I saddle connections of $\varphi$ of class~$\pm \hhc{\gamma}$ (resp., type~II and type~III saddle connections), whilst $N_{\mathrm{DRD}}(\hhc{\gamma})$ (resp., $N_{\mathrm{NRD}}(\hhc{\gamma})$) is the number of degenerate ring domains of class $\pm \hhc{\gamma}$ (resp., non-degenerate ring domains) appearing in rotations of~$\varphi$.
See Sections~\ref{sect:qd:traj_type} and~\ref{sect:qd:domains} for the different types of finite-length connections, and Section~\ref{sect:qd:ssc} for the relation between finite-length connections and homology classes.
Provided that the polar type of $\varphi$ is not $\{-2\}$, our work also shows that Theorem~\ref{thm:intro} holds in the case where $\varphi$ is not saddle-free.
By setting $q^{1/2} = -1$, we obtain the numerical DT invariant \[
\dtnum{\hhc{\gamma}} = N_{\mathrm{I}}(\hhc{\gamma}) + 2N_{\mathrm{II}}(\hhc{\gamma}) + 4N_{\mathrm{III}}(\hhc{\gamma}) - N_{\mathrm{DRD}}(\hhc{\gamma}) - 2N_{\mathrm{NRD}}(\hhc{\gamma}),
\]
exactly as predicted by \cite{ik}.
The invariants of Theorem~\ref{thm:intro} also agree with those coming from the string theory literature \cite{blr}.
These invariants are consistent with those computed by Haiden in \cite{haiden} for a different class of quadratic differentials whose only critical points are simple zeros and simple poles.
The quadratic differentials considered in \cite{bs} usually have poles of higher order.
Since degenerate ring domains contain a double pole at their centre, they do not appear for the differentials from \cite{haiden}, and so the associated invariant of $-1$ does not appear there.

We prove our main result by first classifying the different configurations of finite-length trajectories that can occur for generic quadratic differentials in the class we consider, infinite GMN differentials.
We then describe the associated category of semistable objects in each of the different types in terms of quivers with potential.
The problem then reduces to computing the DT invariants for these quivers with potential.
We present Table~\ref{table} as a summary of the paper, with the details of the concepts and notation used in this table explained later.

\afterpage{%
    \clearpage%
    \thispagestyle{empty}%
    \newgeometry{margin=3cm,top=2cm,bottom=2cm}
    \begin{landscape}%
        \centering %
        \renewcommand{\arraystretch}{1.4}%
        \begin{tabular}{|c|c|c|c|c|c|c|}
        \hline
		Trajectory type & Illustration & Quiver & Potential & Stability & DT generating series \\
		\hline
		Type~I saddle & \begin{tabular}{c}
                \begin{tikzpicture}
		    \draw[\saddletraj] (-1,0) -- (1,0);
                \node at (-1,0) {$\bm{\times}$};
                \node at (1,0) {$\bm{\times}$};
		      \end{tikzpicture}
                \end{tabular} & $\bullet$ & 0 & N/A & $\qdl{t}$ \\
		\hline
		Type~II saddle & \begin{tabular}{c}
                \begin{tikzpicture}
		    \draw[\saddletraj] (-1,0) -- (1,0);
                \node at (-1,0) {$\bm{\times}$};
                \node at (1,0) {$\bm{\otimes}$};
		    \end{tikzpicture}
                \end{tabular}   & $\begin{tikzcd}\bullet \ar[loop left,"a"]\end{tikzcd}$ & $a^3$ & N/A & $\qdl{t}^2$ \\
            \hline
             \parbox{2.6cm}{\centering Standard degenerate ring domain} & \begin{tabular}{c}
                \begin{tikzpicture}
		    \node at (0,0) {$\bullet$};
                \draw[\closedtraj] (0,0) circle (4pt);
                \draw[\closedtraj] (0,0) circle (7pt);
                \draw[\closedtraj] (0,0) circle (10pt);
                \draw[\closedtraj] (0,0) circle (13pt);
                \draw[\saddletraj] (0,0) circle (16pt);
                \node at (0,0.56) {$\bm{\times}$};
		    \end{tikzpicture}
                \end{tabular}
                 & $\begin{tikzcd}\bullet \ar[loop left]\end{tikzcd}$ & 0 & N/A & $\qdl{-q^{-1/2}t}^{-1}$ \\
		\hline
		\parbox{2.9cm}{\centering Toral degenerate ring domain} & \begin{tabular}{c}
                \begin{tikzpicture}
                \begin{scope}[yscale=cos(62),scale=1.1]
                \node at (0,1.75) {};
                \draw[double distance=7mm] (0:0.9) arc (0:181:0.9);
                \draw[double distance=7mm] (180:0.9) arc (180:361:0.9);    
                \end{scope}

                \begin{scope}[scale=0.6,shift={(-1.65,0)},font=\tiny]
                \filldraw[white] (0,0) circle (16pt);
                \node at (0,0) {$\bullet$};
                \draw[\closedtraj] (0,0) circle (4pt);
                \draw[\closedtraj] (0,0) circle (8pt);
                \draw[\closedtraj] (0,0) circle (12pt);
                \draw[\saddletraj] (0,0) circle (16pt);
                \node at (0,0.56) {$\bm{\times}$};
                \node at (0,-0.56) {$\bm{\times}$};
                \end{scope}

                \begin{scope}[scale=0.6,shift={(1.65,0)},font=\tiny]
                \filldraw[white] (0,0) circle (16pt);
                \end{scope}
		    \end{tikzpicture}
                \end{tabular} & $\begin{tikzcd}[ampersand replacement=\&] \overset{1}{\bullet} \ar[r,shift left] \& \overset{2}{\bullet} \ar[l,shift left] \end{tikzcd}$ & 0 & $\vartheta = \scph{S_1} = \scph{S_2}$ & $\qdl{t}^{2}\qdl{-q^{-1/2}t^{2}}^{-1}$ \\
            \hline
		\parbox{2.6cm}{\centering Type III degenerate ring domain} &
            \begin{tabular}{c}
                 \begin{tikzpicture}[scale=1.2]
                    \node at (0,0.6) {};
                    \coordinate(sp1) at (130:0.4);
                    \node at (sp1) {\scriptsize $\bm{\otimes}$};
                    \coordinate(sp2) at (50:0.4);
                    \node at (sp2) {\scriptsize $\bm{\otimes}$};
                    \coordinate(dp) at (270:0.435);
                    \node at (dp) {$\bullet$};
                    \draw[\saddletraj] (sp1) -- (sp2);
                    \draw[\closedtraj] (120:0.5) to [out=10,in=170] (60:0.5);
                    \draw[\closedtraj] (150:0.5) to [out=-20,in=200] (30:0.5);
                    \draw[\closedtraj] (160:0.5) to [out=-10,in=190] (20:0.5);
                    \draw[\closedtraj] (170:0.5) to [out=-5,in=185] (10:0.5);
                    \draw[\closedtraj] (180:0.5) -- (0:0.5);
                    \draw[\closedtraj] (190:0.5) to [out=5,in=175] (-10:0.5);
                    \draw[\closedtraj] (200:0.5) to [out=10,in=170] (-20:0.5);
                    \draw[\closedtraj] (210:0.5) to [out=15,in=165] (-30:0.5);
                    \draw[\closedtraj] (225:0.5) to [out=35,in=145] (-45:0.5);
                    \draw[\closedtraj] (240:0.5) to [out=90,in=90] (-60:0.5);
                    \draw[\closedtraj] ($(-90:0.445)+(-20:0.125)$) arc (-20:200:0.125);
                    \draw (0,0) circle (0.5cm);
                 \end{tikzpicture}
            \end{tabular}& $\begin{tikzcd} \bullet \ar[loop left,"a"] \ar[loop right,"b"]\end{tikzcd}$ & $a^3 + b^3$ & N/A & $\qdl{t}^4\qdl{-q^{-1/2}t^2}^{-1}$ \\
		\hline
		\parbox{2.8cm}{\centering Standard non-degenerate ring domain} & \begin{tabular}{c}
                \begin{tikzpicture}
                \draw[\saddletraj] (0,0) circle (7pt);
                \draw[\closedtraj] (0,0) circle (10pt);
                \draw[\closedtraj] (0,0) circle (13pt);
                \draw[\saddletraj] (0,0) circle (16pt);
                \node at (0,0.56) {$\bm{\times}$};
                \node at (0,-0.236) {$\bm{\times}$};
		    \end{tikzpicture}
                \end{tabular} & $\begin{tikzcd}[ampersand replacement=\&] \overset{1}{\bullet} \ar[r,shift left] \ar[r,shift right] \& \overset{2}{\bullet} \end{tikzcd}$ & 0 & \begin{tabular}{c}$\vartheta = \scph{S_1 \oplus S_2}$\\ $\scph{S_2} < \scph{S_1}$\end{tabular} & $\qdl{-q^{1/2}t}^{-1}\qdl{-q^{-1/2}t}^{-1}$ \\
            \hline
		\parbox{2.6cm}{\centering Toral non-degenerate ring domain} & 
            \begin{tabular}{c}
                \begin{tikzpicture}

                \begin{scope}[yscale=cos(63),scale=1.1]
                \node at (0,1.75) {};
                \draw[double distance=7mm] (0:0.9) arc (0:181:0.9);
                \draw[double distance=7mm] (180:0.9) arc (180:361:0.9);    
                \end{scope}

                \begin{scope}[scale=0.6,shift={(-1.65,0)},font=\tiny]
                \filldraw[white] (0,0) circle (16pt);
                \draw[\saddletraj] (0,0) circle (7pt);
                \draw[\closedtraj] (0,0) circle (10pt);
                \draw[\closedtraj] (0,0) circle (13pt);
                \draw[\saddletraj] (0,0) circle (16pt);
                \node at (0,0.56) {$\bm{\times}$};
                \node at (0,-0.56) {$\bm{\times}$};
                \node at (-0.236,0) {$\bm{\times}$};
                \end{scope}

                \begin{scope}[scale=0.6,shift={(1.65,0)},font=\tiny]
                \filldraw[white] (0,0) circle (16pt);
                \end{scope}
                
		    \end{tikzpicture}
                \end{tabular} & $\begin{tikzcd}[ampersand replacement=\&] \overset{1}{\bullet} \ar[rr] \ar[dr] \&\& \overset{2}{\bullet} \\ \& \overset{3}{\bullet} \ar[ur] \& \end{tikzcd}$ & 0 & \begin{tabular}{c}$\vartheta = \scph{S_1 \oplus S_2} = \scph{S_3}$,\\ $\scph{S_2} < \scph{S_1}$\end{tabular} & $\qdl{t}^2\qdl{-q^{1/2}t^2}^{-1}\qdl{-q^{-1/2}t^2}^{-1}$ \\
            \hline
		\parbox{2.7cm}{\centering Type III non-degenerate ring domain} &
            \begin{tabular}{c}
            \begin{tikzpicture}
                \draw[\saddletraj] (-0.35,0) -- (0.35,0);
                \node at (-0.35,0) {\scriptsize $\bm{\otimes}$};
                \node at (0.35,0) {\scriptsize $\bm{\otimes}$};
                \draw[\closedtraj] (0,0) ellipse (0.65cm and 0.15cm);
                \draw[\closedtraj] (0,0) ellipse (0.75cm and 0.25cm);
                \draw[\closedtraj] (0,0) ellipse (0.85cm and 0.35cm);
                \draw[\closedtraj] (0,0) ellipse (0.95cm and 0.45cm);
                \draw[\saddletraj] (0,0) ellipse (1.05cm and 0.55cm);
                \node at (0,-0.55) {$\bm{\times}$};
                \node at (0,0.75) {};
            \end{tikzpicture}
            \end{tabular}
            & $\begin{tikzcd}[ampersand replacement=\&] \overset{1}{\bullet} \ar[loop left,"a"] \ar[r] \& \overset{2}{\bullet} \ar[loop right,"b"] \end{tikzcd}$ & $a^3 + b^3$ & \begin{tabular}{c}$\vartheta = \scph{S_1 \oplus S_2}$ \\ $\scph{S_2} < \scph{S_1}$\end{tabular} & $\qdl{t}^4\qdl{-q^{1/2}t^{2}}^{-1}\qdl{-q^{-1/2}t^{2}}^{-1}$ \\
		\hline
		\end{tabular}
		\bigskip
		\captionsetup[table]{hypcap=false}
        \captionof{table}{The different configurations of finite-length trajectories and their associated DT generating series.}\label{table}%
        \renewcommand{\arraystretch}{1}%
    \end{landscape}
    \restoregeometry
    \clearpage%
}

The DT invariants for all but one of the quivers with potential that arise have already been computed elsewhere in the literature \cite{dm_loop,haiden,mr}.
The one that has not is the quiver \[ 
\begin{tikzcd}
    \overset{1}{\bullet} \ar[loop left,"a"] \ar[r] & \overset{2}{\bullet} \ar[loop right,"b"]
\end{tikzcd}
\] with potential $a^3 + b^3$, under a stability condition such that the simple at $1$ has a greater phase than the simple at~$2$.
We refer to this as the barbell quiver with potential.
This case has both a non-trivial potential and a non-trivial stability condition, whereas for the other cases at least one of either the potential or stability condition is trivial.
We are not aware of other computations in the literature of DT invariants for a quiver with both non-trivial potential and non-trivial stability for a non-symmetric quiver, such as the one above.
The fact that the quiver is asymmetric also means that we obtain a quantum dilogarithm identity.
We tackle this difficult case using techniques from the representation theory of gentle algebras to understand the semistable representations.
It does not seem that such techniques have previously been used for calculating DT invariants.

We also prove some technical results the 3-Calabi--Yau triangulated category~$\mathcal{D}$.
We prove that, under simple tilting of hearts, the $A_\infty$-enhancement $\mathcal{D}_\infty$ of this category is replaced by a quasi-equivalent $A_\infty$-enhancement.
We also prove that simple tilting of hearts preserves `orientation data', part of the input data for DT theory in the approach of Kontsevich and Soibelman \cite{ks_stability}.
In his thesis \cite{davison_thesis}, Davison constructs a canonical choice of orientation data for quivers with potential and shows that it is invariant under simple tilting of the heart in the case where the simple objects form a cluster collection.
Unfortunately, our collections of simple objects do not always form cluster collections, but we show in Section~\ref{sect:inv} that the orientation data from \cite{davison_thesis} is nevertheless invariant under simple tilting.

\textbf{Organisation.}
This paper is structured as follows.
We attempt to be relatively self-contained by providing substantial background.
We cover quadratic differentials in Section~\ref{sect:qd}, followed by 3-Calabi--Yau categories and stability conditions in Section~\ref{sect:3cy_stab}, with background on DT theory in Section~\ref{sect:dt}.
In Section~\ref{sect:inv} we prove invariance of the $A_\infty$-enhancement and orientation data under simple tilting.
In Section~\ref{sect:prelim}, we describe the different possible configurations of finite-length trajectories.
The main section of the paper is Section~\ref{sect:main}, where we describe the categories of semistable objects and DT invariants corresponding to the finite-length trajectories.
Two of these cases require more extensive calculations, which are postponed to Section~\ref{sect:spiral}.
Finally, Section~\ref{sect:db} is concerned with computing the DT invariants for the barbell quiver with potential.

{\bf Acknowledgements.}
We thank Matt Booth, Tom Bridgeland, Merlin Christ, Ben Davison, Fabian Haiden, Kohei Iwaki, Kento Osuga, David Pauksztello, Yu Qiu, and Markus Reineke for helpful exchanges.
OK was supported by the Leverhulme Research Project Grant `Extended Riemann--Hilbert Correspondence, Quantum Curves, and Mirror Symmetry'.
NJW was supported by EPSRC grant EP/V050524/1.
Part of this research was conducted while the authors were visiting the Okinawa Institute of Science and Technology (OIST) through the Theoretical Sciences Visiting Program (TSVP).
Finally, we thank the University of Tokyo and JSPS, where this project began whilst the authors were JSPS fellows under the respective mentorship of Kohei Iwaki and Osamu Iyama, whom we also thank.

\section{Quadratic differentials}\label{sect:qd}

A (\emph{meromorphic}) \emph{quadratic differential} $\varphi$ on a Riemann surface $X$ is a meromorphic section of $\ctbs$, the square of the canonical bundle~$\ctb$.
Quadratic differentials appear in a wide range of applications, from Teichm\"uller theory \cite{gardiner} and flat surfaces \cite{mz,zorich} to WKB analysis \cite{ini,inii} and orthogonal polynomials \cite{ammt,mmo}.
The standard reference for the geometry of quadratic differentials is \cite{strebel}.

\subsection{Metric, foliation, and domain decomposition}

A quadratic differential $\varphi$ induces a number of structures on the Riemann surface $X$, as we now examine.

\subsubsection{Critical points}\label{sect:qd:crit_points}

\emph{Critical points} of $\varphi$ consist of zeros and poles. The set of critical points of $\varphi$ is denoted $\crit(\varphi)$.
The \emph{infinite critical points} of $\varphi$ consist of the poles of order $\geqslant 2$ and are denoted $\infcrit(\varphi)$.
The \emph{finite critical points} $\crit(\varphi) \setminus \infcrit(\varphi)$ consist of the zeros and simple poles.
Points in $X\setminus \crit(\varphi)$ are called \emph{regular points}.
One can also refer to zeros of order~$k$ as \emph{critical points of order~$k$} and poles of order~$k$ as \emph{critical points of order~$-k$}.
It is sometimes convenient to think of regular points as criticial points \emph{of order~$0$}, though of course they are not critical points at all.

\begin{convention}
In figures, we illustrate simple zeros of a quadratic differential by `$\times$', simple poles by `$\otimes$', and poles of order $\geqslant 2$ by `$\bullet$'.
\end{convention}

We will be concerned with the following class of quadratic differentials in this paper.

\begin{definition}\label{def:gmn}
An \emph{infinite GMN differential} is a quadratic differential $\varphi$ on a compact connected Riemann surface $X$ such that all zeros of $\varphi$ are simple, $\varphi$ has at least one infinite critical point, and $\varphi$ has at least one finite critical point.
\end{definition}

We restrict to this class of quadratic differentials because they are the ones which induce suitable triangulations of the Riemann surface $X$, producing categories which are 3-Calabi--Yau, see \cite[Proposition~7.6]{chq}.

\begin{remark}
Infinite GMN differentials are a special case of the GMN differentials from \cite[Definition~2.1]{bs}, the difference being that GMN differentials are only required to have at least one (possibly simple) pole, rather than at least one infinite critical point.
\end{remark}

The Riemann--Roch theorem provides restrictions on the poles and zeros a quadratic differential $\varphi$ may have, according to the genus of $X$.
Indeed, suppose that $\varphi$ is an infinite GMN differential on a Riemann surface $X$ of genus $g$ with $d$ poles of respective order $m_1, m_2, \dots, m_d$ and $l$ simple zeros.
We refer to the multiset $m = \{m_1, m_2, \dots, m_d\}$ as the \emph{polar type} of $\varphi$.
By the Riemann--Roch theorem, we must have (see~\cite[Section~5.2]{bs})
\begin{equation}\label{eq:rr}
    l = 4g - 4 + \sum_{i = 1}^d m_i.
\end{equation}

\subsubsection{Metric and foliation}\label{sect:qd:foliation+metric}

The quadratic differential $\varphi$ determines two structures on $X \setminus \crit(\varphi)$, namely a flat metric known as the $\varphi$-metric and a foliation known as the horizontal foliation.
At a point $x_0 \in X \setminus \crit(\varphi)$, there is a distinguished local coordinate~$w$ \cite[Section~5.1]{strebel} defined (up to sign) by 
\begin{equation} \label{eq:distinguished-coordinate}
w(x) := \int^{x}_{x_0} \sqrt{\varphi}.
\end{equation}
In the coordinate~$w$, $\varphi$ is expressed locally as $\varphi(w) = dw \otimes dw$.
The \emph{$\varphi$-metric} is defined locally on $X \setminus \crit(\varphi)$ by pulling back the Euclidean metric on $\mathbb{C}$ using the distinguished local coordinate~$w$~\eqref{eq:distinguished-coordinate}.
By taking the metric completion, the $\varphi$-metric extends to $X \setminus \infcrit(\varphi)$, see \cite[Section~2.3]{chq}.

We define the \emph{horizontal foliation} on $X \setminus \crit(\varphi)$ given by the lines $\im w = \lambda$ for $\lambda \in \mathbb{R}$ any constant \cite[Section~2.1]{bs}.
An \emph{arc} in $X$ is a smooth path $\gamma \colon I \to X \setminus \crit(\varphi)$ defined on an interval $I \subset \mathbb{R}$.
A \emph{straight arc of phase $\vartheta$} is an arc which makes a constant angle $\pi\vartheta$ with the horizontal foliation \cite[Section~3.1]{bs}.
That is, straight arcs are locally pullbacks of straight lines in the $w$-plane by \eqref{eq:distinguished-coordinate}.
Note that the phase of a straight arc is a well-defined element of $\mathbb{R}/\mathbb{Z}$; our convention will be to choose representatives lying in~$(0,1]$. 
We denote the phase of a straight arc $\gamma$ by $\qdph{\gamma} \in (0, 1]$.
A \emph{horizontal} straight arc is one which has phase~$1$ with respect to the foliation.
A straight arc is \emph{maximal} if it is not the restriction of a straight arc on a larger interval.
A \emph{connection} is a maximal straight arc; a \emph{trajectory} is a horizontal connection.
Given $0 \leqslant \vartheta < 1$, one can define the \emph{rotation} $\varphi^{(\vartheta)}:=e^{-2i\pi\vartheta}\varphi$ of $\varphi$.
Connections of phase $\vartheta$ for $\varphi$ are trajectories for $\varphi^{(\vartheta)}$.
This gives two possible different perspectives: either one can consider connections of phase $\vartheta$ in~$\varphi$, or one can consider trajectories in the rotation $\varphi^{(\vartheta)}$.
We will find both perspectives useful in this paper.

\subsubsection{Different types of trajectories}\label{sect:qd:traj_type}

Every trajectory of $\varphi$ falls into exactly one of the following types \cite{strebel}, \cite[Section~3.4]{bs}:
\begin{enumerate}
    \item \emph{saddle trajectories}, which approach finite critical points at both ends;
    \item \emph{separating trajectories}, which approach a finite critical point at one end and an infinite critical point at the other;
    \item \emph{generic trajectories}, which approach infinite critical points at both ends;
    \item \emph{closed trajectories}, which are simple closed curves in $X \setminus \crit(\varphi)$;
    \item \emph{recurrent trajectories}, which are \emph{recurrent}, meaning that the set of limit points along the trajectory in at least one direction is equal to the closure of the trajectory in $X$ \cite[Section~10.1]{strebel}.
\end{enumerate}

We are primarily interested in trajectories with at least one endpoint on a finite critical point; we call these \emph{critical trajectories}.
These consist of the saddle trajectories and separating trajectories. Since only finitely many horizontal arcs emerge from each finite critical point, the number of critical trajectories is finite.
The set of critical trajectories (with a certain orientation) is often called the \emph{spectral network $\mathcal{W}_0({\varphi})$ of $\varphi$}.
In the literature on WKB analysis, the critical trajectories are called {\em Stokes curves}.
Amongst the different types of trajectories, the ones which have finite length in the $\varphi$-metric are precisely the saddle trajectories and closed trajectories, which we therefore refer to as \emph{finite-length trajectories}.

\begin{definition}\label{def:traj_types}
In this paper, it will be important to distinguish the following diferent types of saddle connections of an infinite GMN differential~$\varphi$.
\begin{itemize}
\item A \emph{type~I} saddle connection connects two distinct zeros of $\varphi$. 
\item A \emph{type~II} saddle connection connects a zero and a simple pole of $\varphi$.
\item A \emph{type~III} saddle connection connects two simple poles of $\varphi$, which are then necessarily distinct.
\item A \emph{closed} saddle connection begins and ends at the same zero of $\varphi$.
\end{itemize}
We call $\varphi$ \emph{saddle-free} if it contains no saddle trajectories.
\end{definition}

It is well-known that there is a simple ``normal form'' \cite[Theorems~6.1, 6.2, 6.3, and~6.4]{strebel} for $\varphi$ around the different types of critical points and around regular points, from which one can deduce the local behaviour of trajectories.
We refer the reader to \cite{bs,ik,strebel} for numerous illustrations and additional details of the foliations and spectral networks.

\subsubsection{Domains}\label{sect:qd:domains}

It is well known that each connected component of the complement
$X \setminus {\mathcal W}_{0}(\varphi)$
is one of the following \cite[Section~3.4]{bs}.
Here, quadratic differentials $\varphi_{1}$ and $\varphi_{2}$ on respective Riemann surfaces $X_{1}$ and $X_{2}$ are \emph{equivalent} if there is a biholomorphism $f \colon X_{1} \to X_{2}$ such that $f^{\ast}(\varphi_{2}) = \varphi_{1}$.
\begin{enumerate}
\item 
A {\em half plane} is equivalent to a domain
$\{ w \in \mathbb{C} \mid \im w > \lambda \}$
for some $\lambda \in {\mathbb R}$, equipped with the quadratic differential $dw^2$,  via the local coordinate \eqref{eq:distinguished-coordinate}.
Around a pole of $\varphi$ of order $\geqslant 3$ there are always $k - 2$ half planes, and this is the only situation in which half planes appear.

\item
A {\em horizontal strip} is equivalent to a domain $\{ w \in \mathbb{C} \mid \lambda_{1} < \im w < \lambda_{2} \}$ for some $\lambda_{1}, \lambda_{2} \in {\mathbb R}$  with $\lambda_{1} < \lambda_{2}$, equipped with the quadratic differential $dw^2$, 
via the local coordinate~\eqref{eq:distinguished-coordinate}. 

\item 
A {\em ring domain} is a connected domain consisting of points 
$x \in X \setminus \crit(\varphi)$ such that the trajectory 
passing through $x$ is a closed trajectory. 
It is equivalent to $\{ z \in \mathbb{C} \mid \lambda_{1} < |z| < \lambda_{2} \}$ for some $\lambda_{1}, \lambda_{2} \in {\mathbb R}$ with $\lambda_{1} < \lambda_{2}$, equipped with the quadratic differential $\nu dz^2 / z^2$ 
for some $\nu \in \mathbb{C}^\ast$. 
We call a ring domain {\em degenerate} if $\lambda_{1} = 0$, and 
{\em non-degenerate} otherwise.

\item 
A {\em spiral domain} is defined to be the interior of 
the closure of a recurrent trajectory.
\end{enumerate}
The canonical example of a spiral domain is a torus with an irrational foliation.
By certain surgeries, this example can be attached to other surfaces; see \cite[end of Section~6.4]{bs} for more details.
Spiral domains will be important in Section~\ref{sect:spiral}.

The boundaries of half planes and horizontal strips consist of unions of saddle trajectories and separating trajectories, whereas the boundaries of non-degenerate ring domains and spiral domains only consist of saddle trajectories.
In the case of a degenerate ring domain, one boundary is a union of saddle trajectories, and the other boundary is a double pole.

\subsubsection{Geodesics}\label{sect:qd:geodesics}

Recall from Section~\ref{sect:qd:foliation+metric} that $\varphi$ induces a metric on $\nocrit{X} := X \setminus \infcrit(\varphi)$.
A \emph{$\varphi$-geodesic} is a locally rectifiable path $\gamma \colon [0, 1] \to \nocrit{X}$ which is locally length-minimising \cite[Definition~5.5.1]{strebel}.
It is not necessarily assumed that $\gamma$ is the shortest path between its endpoints.

It is immediate that a straight arc is a geodesic, and that away from $\crit(\varphi)$, geodesics are straight arcs.
In particular, which arcs are geodesics is invariant under rotation of the differential, and saddle connections are geodesics.

The following lemma describes the behaviour of geodesics around a simple zero. It is essentially the statement of \cite[Theorem~8.1]{strebel}, but we briefly indicate the proof since our formulation is different.

\begin{lemma}[{\cite[Theorem~8.1]{strebel}, \cite[Section~3.7]{bs}}]\label{lem:geodesics}
Let $x_1$ and $x_2$ be points which lie in or on the boundary of two respective domains $D_1$ and $D_2$ incident to a zero $z$ of~$\varphi$, such that $D_2$ is immediately anti-clockwise around the zero $z$ from~$D_1$.
Denote by $\gamma_{1}$ and $\gamma_{2}$ the straight arcs from $x_1$ and $x_2$ respectively to~$z$.
There is then a unique geodesic from $x_1$ to $x_2$, which is described as follows.
\begin{enumerate}[label=\textup{(}\arabic*\textup{)}]
    \item If $\qdph{\gamma_{1}} > \qdph{\gamma_{2}}$, the geodesic from $x_1$ to $x_2$ is composed of a single straight arc.
    \item If $\qdph{\gamma_{1}} < \qdph{\gamma_{2}}$, the geodesic from $x_1$ to $x_2$ is composed of the two straight arcs $\gamma_{1}$ and $\gamma_{2}$.
    \item If $\qdph{\gamma_{1}} = \qdph{\gamma_{2}}$, the geodesic from $x_1$ to $x_2$ is the single straight arc which is the union of $\gamma_{1}$ and $\gamma_{2}$.
\end{enumerate}
\end{lemma}

\begin{proof}
Suppose that we gradually rotate $\varphi$, so that the separating trajectories emanating from $z$ gradually move clockwise and the phases of $\gamma_1$ and $\gamma_2$ gradually increase, until one of these phases becomes~$1$.
We illustrate the different cases in Figure~\ref{fig:stability_via_geodesics}, where the top row illustrates $\varphi$ before rotation and the bottom row illustrates $\varphi^{(\vartheta)}$, where $\vartheta$ is the lesser of $1 - \qdph{\gamma_1}$ and $1 - \qdph{\gamma_2}$.

If $\qdph{\gamma_{1}} > \qdph{\gamma_{2}}$, then, as one rotates, the phase of $\gamma_1$ will hit~$1$ before the phase of $\gamma_2$ does.
When the phase of $\gamma_1$ hits~$1$, $x_1$ will lie on a separating trajectory, so that $x_1$ and $x_2$ then lie in the closure of the same domain, as in the left-hand case of Figure~\ref{fig:stability_via_geodesics}.
There is then a straight arc directly between $x_1$ and $x_2$ lying in this domain.
On the other hand, if $\qdph{\gamma_{1}} < \qdph{\gamma_{2}}$, then $x_2$ will lie on a separating trajectory before $x_1$ does, as one rotates, as in the right-hand case of Figure~\ref{fig:stability_via_geodesics}.
We are then in the second case of \cite[Theorem~8.1]{strebel}, and the geodesic is given by the union of $\gamma_{1}$ and~$\gamma_{2}$.
In the final case, $x_1$ and $x_2$ will lie on separating trajectories simultaneously as one rotates, and $\gamma_{1} \cup \gamma_{2}$ forms a straight arc along the boundary of a domain, as in the middle case of Figure~\ref{fig:stability_via_geodesics}.
\end{proof}

\begin{figure}
    \[
    \begin{tikzpicture}[scale=0.9]

        \begin{scope}[shift={(4.5,0)}]
        \draw (0:0) -- (210:2);
        \draw (0:0) -- (90:2);
        \draw (0:0) -- (330:2);

        \draw[\conn, line width=0.55pt] (0:0) -- (20:1.5);
        \draw[\conn, line width=0.55pt] (0:0) -- (150:1.5);

        \draw[\generictraj] (87.5:2) to [out=270,in=150] (332.5:2);
        \draw[\generictraj] (82.5:1.9) to [out=280,in=140] (337.5:1.9);
        \draw[\generictraj] (77.5:1.8) to [out=290,in=130] (342.5:1.8);
        \draw[\generictraj] (70:1.7) to [out=300,in=120] (350:1.7);

        \draw[\generictraj] (207.5:2) to [out=30,in=270] (92.5:2);
        \draw[\generictraj] (202.5:1.9) to [out=40,in=260] (97.5:1.9);
        \draw[\generictraj] (197.5:1.8) to [out=50,in=250] (102.5:1.8);
        \draw[\generictraj] (190:1.7) to [out=60,in=240] (110:1.7);

        \draw[\generictraj] (327.5:2) to [out=150,in=30] (212.5:2);
        \draw[\generictraj] (322.5:1.9) to [out=160,in=20] (217.5:1.9);
        \draw[\generictraj] (317.5:1.8) to [out=170,in=10] (222.5:1.8);
        \draw[\generictraj] (310:1.7) to [out=180,in=0] (230:1.7);

        \node at (20:1.5) {$\bm{\circ}$};
        \node at (150:1.5) {$\bm{\circ}$};
        \node at (0,0) {$\bm{\times}$};

        \node at (5:1) {$\gamma_1$};
        \node at (165:1) {$\gamma_2$};

        \node at (20:1.5) [above] {$x_1$};
        \node at (150:1.5) [above] {$x_2$};
        
        \end{scope}

        \begin{scope}[shift={(4.5,-3.5)}]
        \draw (0:0) -- (150:2);
        \draw (0:0) -- (30:2);
        \draw (0:0) -- (270:2);

        \draw[\generictraj] (-87.5:2) to [out=-270,in=-150] (-332.5:2);
        \draw[\generictraj] (-82.5:1.9) to [out=-280,in=-140] (-337.5:1.9);
        \draw[\generictraj] (-77.5:1.8) to [out=-290,in=-130] (-342.5:1.8);
        \draw[\generictraj] (-70:1.7) to [out=-300,in=-120] (-350:1.7);

        \draw[\generictraj] (-207.5:2) to [out=-30,in=-270] (-92.5:2);
        \draw[\generictraj] (-202.5:1.9) to [out=-40,in=-260] (-97.5:1.9);
        \draw[\generictraj] (-197.5:1.8) to [out=-50,in=-250] (-102.5:1.8);
        \draw[\generictraj] (-190:1.7) to [out=-60,in=-240] (-110:1.7);

        \draw[\generictraj] (-327.5:2) to [out=-150,in=-30] (-212.5:2);
        \draw[\generictraj] (-322.5:1.9) to [out=-160,in=-20] (-217.5:1.9);
        \draw[\generictraj] (-317.5:1.8) to [out=-170,in=-10] (-222.5:1.8);
        \draw[\generictraj] (-310:1.7) to [out=-180,in=0] (-230:1.7);

        \draw[\conn, line width=0.55pt ] (0:0) -- (20:1.5);
        \draw[\conn, line width=0.55pt] (0:0) -- (150:1.5);

        \node at (20:1.5) {$\bm{\circ}$};
        \node at (150:1.5) {$\bm{\circ}$};
        \node at (0,0) {$\bm{\times}$};

        \node at (5:1) {$\gamma_1$};
        \node at (165:1) {$\gamma_2$};

        \node at (20:1.5) [right] {$x_1$};
        \node at (150:1.5) [above] {$x_2$};
        \end{scope}

        \node at (4.5,2.75) {$\qdph{\gamma_1} < \qdph{\gamma_2}$};
        \draw[dotted] (-2.25,-5.75) -- (-2.25,3.25);

        \begin{scope}[shift={(0,0)}]
        \draw (0:0) -- (210:2);
        \draw (0:0) -- (90:2);
        \draw (0:0) -- (330:2);

        \draw[\generictraj] (87.5:2) to [out=270,in=150] (332.5:2);
        \draw[\generictraj] (82.5:1.9) to [out=280,in=140] (337.5:1.9);
        \draw[\generictraj] (77.5:1.8) to [out=290,in=130] (342.5:1.8);
        \draw[\generictraj] (70:1.7) to [out=300,in=120] (350:1.7);

        \draw[\generictraj] (207.5:2) to [out=30,in=270] (92.5:2);
        \draw[\generictraj] (202.5:1.9) to [out=40,in=260] (97.5:1.9);
        \draw[\generictraj] (197.5:1.8) to [out=50,in=250] (102.5:1.8);
        \draw[\generictraj] (190:1.7) to [out=60,in=240] (110:1.7);

        \draw[\generictraj] (327.5:2) to [out=150,in=30] (212.5:2);
        \draw[\generictraj] (322.5:1.9) to [out=160,in=20] (217.5:1.9);
        \draw[\generictraj] (317.5:1.8) to [out=170,in=10] (222.5:1.8);
        \draw[\generictraj] (310:1.7) to [out=180,in=0] (230:1.7);

        \draw[\conn, line width=0.55pt] (0:0) -- (30:1.5);
        \draw[\conn, line width=0.55pt] (0:0) -- (150:1.5);

        \node at (30:1.5) {$\bm{\circ}$};
        \node at (150:1.5) {$\bm{\circ}$};
        \node at (0,0) {$\bm{\times}$};

        \node at (15:1) {$\gamma_1$};
        \node at (165:1) {$\gamma_2$};

        \node at (30:1.5) [above] {$x_1$};
        \node at (150:1.5) [above] {$x_2$};
        \end{scope}
        
        \begin{scope}[shift={(0,-3.5)}]
        \draw (0:0) -- (270:2);
        \draw (0:0) -- (150:2);
        \draw (0:0) -- (30:2);

        \draw[\generictraj] (-87.5:2) to [out=-270,in=-150] (-332.5:2);
        \draw[\generictraj] (-82.5:1.9) to [out=-280,in=-140] (-337.5:1.9);
        \draw[\generictraj] (-77.5:1.8) to [out=-290,in=-130] (-342.5:1.8);
        \draw[\generictraj] (-70:1.7) to [out=-300,in=-120] (-350:1.7);

        \draw[\generictraj] (-207.5:2) to [out=-30,in=-270] (-92.5:2);
        \draw[\generictraj] (-202.5:1.9) to [out=-40,in=-260] (-97.5:1.9);
        \draw[\generictraj] (-197.5:1.8) to [out=-50,in=-250] (-102.5:1.8);
        \draw[\generictraj] (-190:1.7) to [out=-60,in=-240] (-110:1.7);

        \draw[\generictraj] (-327.5:2) to [out=-150,in=-30] (-212.5:2);
        \draw[\generictraj] (-322.5:1.9) to [out=-160,in=-20] (-217.5:1.9);
        \draw[\generictraj] (-317.5:1.8) to [out=-170,in=-10] (-222.5:1.8);
        \draw[\generictraj] (-310:1.7) to [out=-180,in=0] (-230:1.7);

        \draw[\conn, line width=0.55pt] (0:0) -- (30:1.5);
        \draw[\conn, line width=0.55pt] (0:0) -- (150:1.5);

        \node at (30:1.5) {$\bm{\circ}$};
        \node at (150:1.5) {$\bm{\circ}$};
        \node at (0,0) {$\bm{\times}$};

        \node at (15:1) {$\gamma_1$};
        \node at (165:1) {$\gamma_2$};

        \node at (30:1.5) [above] {$x_1$};
        \node at (150:1.5) [above] {$x_2$};
        \end{scope}

        \node at (0,2.75) {$\qdph{\gamma_1} = \qdph{\gamma_2}$};
        \draw[dotted] (2.25,-5.75) -- (2.25,3.25);

        \begin{scope}[shift={(-4.5,0)}]
        \draw (0:0) -- (210:2);
        \draw (0:0) -- (90:2);
        \draw (0:0) -- (330:2);

        \draw[\generictraj] (87.5:2) to [out=270,in=150] (332.5:2);
        \draw[\generictraj] (82.5:1.9) to [out=280,in=140] (337.5:1.9);
        \draw[\generictraj] (77.5:1.8) to [out=290,in=130] (342.5:1.8);
        \draw[\generictraj] (70:1.7) to [out=300,in=120] (350:1.7);

        \draw[\generictraj] (207.5:2) to [out=30,in=270] (92.5:2);
        \draw[\generictraj] (202.5:1.9) to [out=40,in=260] (97.5:1.9);
        \draw[\generictraj] (197.5:1.8) to [out=50,in=250] (102.5:1.8);
        \draw[\generictraj] (190:1.7) to [out=60,in=240] (110:1.7);

        \draw[\generictraj] (327.5:2) to [out=150,in=30] (212.5:2);
        \draw[\generictraj] (322.5:1.9) to [out=160,in=20] (217.5:1.9);
        \draw[\generictraj] (317.5:1.8) to [out=170,in=10] (222.5:1.8);
        \draw[\generictraj] (310:1.7) to [out=180,in=0] (230:1.7);

        \draw[\conn, line width=0.55pt] (140:1.5) to [out=-20,in=200] (30:1.5);

        \node at (30:1.5) {$\bm{\circ}$};
        \node at (140:1.5) {$\bm{\circ}$};
        \node at (0,0) {$\bm{\times}$};

        \node at (30:1.5) [above] {$x_1$};
        \node at (140:1.5) [above] {$x_2$};
        \end{scope}

        \begin{scope}[shift={(-4.5,-3.5)}]
        \draw (0:0) -- (270:2);
        \draw (0:0) -- (150:2);
        \draw (0:0) -- (30:2);

        \draw[\generictraj] (-87.5:2) to [out=-270,in=-150] (-332.5:2);
        \draw[\generictraj] (-82.5:1.9) to [out=-280,in=-140] (-337.5:1.9);
        \draw[\generictraj] (-77.5:1.8) to [out=-290,in=-130] (-342.5:1.8);
        \draw[\generictraj] (-70:1.7) to [out=-300,in=-120] (-350:1.7);

        \draw[\generictraj] (-207.5:2) to [out=-30,in=-270] (-92.5:2);
        \draw[\generictraj] (-202.5:1.9) to [out=-40,in=-260] (-97.5:1.9);
        \draw[\generictraj] (-197.5:1.8) to [out=-50,in=-250] (-102.5:1.8);
        \draw[\generictraj] (-190:1.7) to [out=-60,in=-240] (-110:1.7);

        \draw[\generictraj] (-327.5:2) to [out=-150,in=-30] (-212.5:2);
        \draw[\generictraj] (-322.5:1.9) to [out=-160,in=-20] (-217.5:1.9);
        \draw[\generictraj] (-317.5:1.8) to [out=-170,in=-10] (-222.5:1.8);
        \draw[\generictraj] (-310:1.7) to [out=-180,in=0] (-230:1.7);

        \draw[\conn, line width=0.55pt] (140:1.5) to [out=-20,in=200] (30:1.5);

        \node at (30:1.5) {$\bm{\circ}$};
        \node at (140:1.5) {$\bm{\circ}$};
        \node at (0,0) {$\bm{\times}$}; 

        \node at (30:1.5) [above] {$x_1$};
        \node at (140:1.5) [above] {$x_2$};
        \end{scope}

        \node at (-4.5,2.75) {$\qdph{\gamma_1} > \qdph{\gamma_2}$};

        \node at (-7.5,0.25) {\parbox{2cm}{\centering Before \\ rotation}};
        \node at (-7.5,-3.75) {\parbox{2cm}{\centering After \\ rotation}};
    
    \end{tikzpicture}
    \]
        \caption{Geodesics in the neighbourhood of a simple zero}    \label{fig:stability_via_geodesics}
\end{figure}
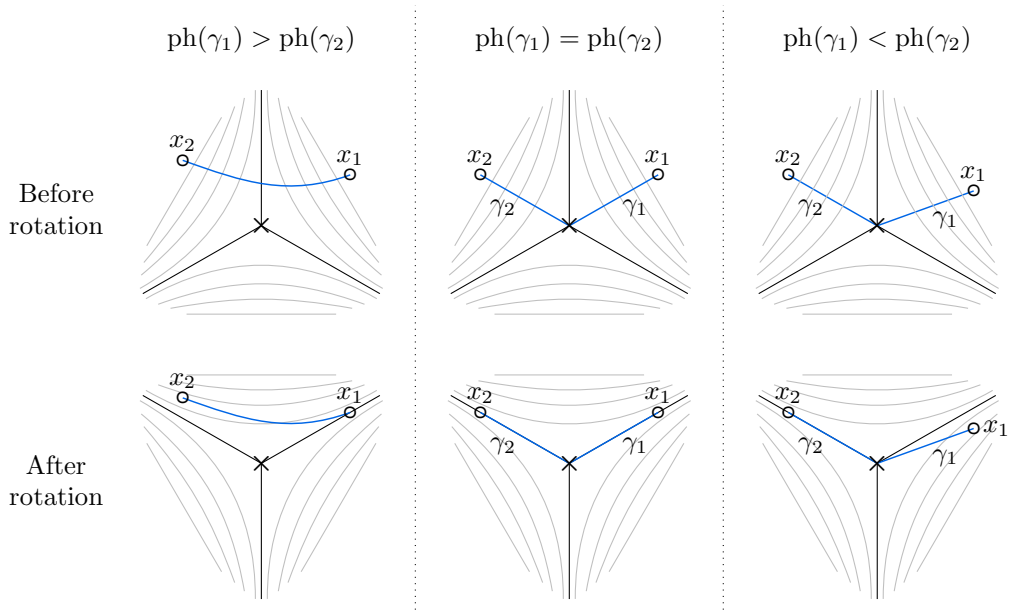

\subsection{Spectral cover and hat-homology}\label{sect:qd:hh}

Finite-length trajectories of $\varphi$ give homology classes for a certain curve known as the spectral cover of $X$.
We follow the exposition from \cite[Section~2.3]{bs}.
Suppose that $\varphi$ is an infinite GMN differential on a compact Riemann surface $X$ with poles of order $m_i$ at points $p_i \in X$.
Letting $E$ denote the divisor $E = \sum_i \ceil*{\frac{m_i}{2}}\cdot p_i$, one can view $\varphi$ as a holomorphic section $\overline{\varphi} \in H^0(X, \ctb(E)^{\otimes 2})$, which has simple zeros at both the zeros and odd order poles of $\varphi$.
Here $\ctb(E)$ is the tensor product of $\ctb$ with the line bundle associated to the divisor~$E$.
The \emph{spectral cover} of $X$ defined by $\varphi$ is the compact Riemann surface
\[\spc{X} := \{(x, l(x)) \st x \in X,\, l(x) \in L_x \text{ such that } l(x) \otimes l(x) = \overline{\varphi}(x)\} \subset L,\]
where $L$ is the total space of the line bundle $\ctb(E)$ and $L_x$ is its fibre over~$x$. 
The fact that the zeros of $\overline{\varphi}$ are simple means that the spectral cover is non-singular.
The spectral cover is the Riemann surface for the square root of~$\varphi$.

The projection $\pi \colon \spc{X} \to X$, $(x, l(x)) \mapsto x$ is a double cover which is branched precisely over the zeros and odd order poles of $\varphi$. We therefore have that $\spc{X}$ is connected, since finite critical points must be branch points, and infinite GMN differentials must have at least one finite critical point.

There is also an involution $\sigma\colon \spc{X} \to \spc{X}$, $(x, l(x)) \mapsto (x, -l(x))$, which gives an induced action on homology.
Letting $\spct{X} := \pi^{-1}(\nocrit{X})$, the \emph{hat-homology group} of $\varphi$ is then \[\hh{\varphi} := \{\gamma \in H_1(\spct{X}; \mathbb{Z}) \st \sigma_\ast (\gamma) = - \gamma\}.\]

\subsubsection{Standard saddle classes}\label{sect:qd:ssc}

Finite-length connections give us particular hat-homology classes, which will be of interest to us.
Given an infinite GMN differential $\varphi$ with a finite-length connection~$\gamma$, we can associate a homology class $\hhc{\gamma} \in \hh{\varphi}$ represented by the closed cycle on $\spct{X}$ obtained as the preimage of the connection $\gamma$ under $\pi \colon \spct{X} \to X$.

This homology class $\hhc{\gamma}$ is then oriented as follows.
As in \cite[Section~2.3]{bs}, there is a certain meromorphic $1$-form $\psi$ on $\spc{X}$ such that the inverse image of the horizontal foliation under $\pi$ is given by the lines $\im \widetilde{w} = \lambda$ for constant $\lambda \in \mathbb{R}$, where $d\widetilde{w} = \psi$.
We orient $\hhc{\gamma}$ such that the evaluation of $\psi$ on the tangent vector to $\hhc{\gamma}$ should have a positive imaginary part if $\gamma$ is not horizontal, and a negative real part if $\gamma$ is horizontal.
In the case where $\gamma$ is a closed trajectory in a ring domain $R$, then we say that the ring domain $R$ has \emph{class~$\hhc{\gamma}$}.

Suppose that $\varphi$ is saddle-free, so that the boundary of every horizontal strip in the foliation consists of a union of separating trajectories.
Then every horizontal strip contains exactly one finite critical point in each of the connected components of its boundary.
There is a saddle connection $\gamma$ which crosses the strip from the finite critical point on one side to the finite critical point on the other; we call this a \emph{standard saddle connection} and call $\hhc{\gamma}$ the associated \emph{standard saddle class}.
We write $\hhspan{\varphi}$ for the $\mathbb{Z}$-span of the standard saddle classes in $\hh{\varphi}$.
It follows from \cite[Lemma~3.2]{bs} that the standard saddle classes are linearly independent over $\mathbb{Z}$, so this is a free abelian group of rank~$n$.
It is expected that $\hhspan{\varphi} \cong \hh{\varphi}$.

The \emph{central charge} of $\varphi$ is the group homomorphism $Z_\varphi \colon \hhspan{\varphi} \to \mathbb{C}$ given by \[Z_\varphi(\hhc{\gamma}) := \oint_{\hhc{\gamma}} \sqrt{\varphi}\] for $\hhc{\gamma} \in \hhspan{\varphi}$.
This is also called the \emph{period integral}, with $Z_{\varphi}(\hhc{\gamma})$ called \emph{periods}.
A consequence of the orientation of the standard saddle classes is that $\im Z_{\varphi}(\hhc{\gamma}) > 0$ for all standard saddle classes~$\hhc{\gamma}$.
For any finite-length connection $\gamma$, one can see that $\frac{1}{\pi}\arg Z_{\varphi}(\hhc{\gamma}) = \qdph{\gamma}$, since $\gamma$ has a constant angle of $\pi\qdph{\gamma}$.
In particular, $Z_\varphi(\hhc{\gamma}) \in \mathbb{R}$ if and only if $\gamma$ is horizontal.

\subsubsection{Pairings}\label{sect:qd:pairings}

There are two important pairings that we will need to consider to prove some results, following \cite[Section~2.4]{bs}.
To this end, we write \[\infpreim = \pi^{-1}(\infcrit(\varphi)),\] so that $\spct{X} = \spc{X} \setminus \infpreim$ recalling that $\pi \colon \spc{X} \to X$ is the projection from the spectral cover.
We have canonical maps of homology groups \[
H_{1}(\spct{X}; \mathbb{Z}) = H_{1}(\spc{X} \setminus \infpreim; \mathbb{Z}) \to H_{1}(\spc{X}; \mathbb{Z}) \to H_{1}(\spc{X}, \infpreim ; \mathbb{Z}).
\]
The intersection pairing on $H_1(\spc{X}; \mathbb{Z})$ is a non-degenerate, skew-symmetric pairing, and it induces a degenerate skew-symmetric pairing \[
\intp{-}{-}\colon H_1(\spct{X}; \mathbb{Z}) \times H_1(\spct{X}; \mathbb{Z}) \to \mathbb{Z},
\] which we also call the \emph{intersection pairing}.
The second relevant pairing for us is the non-degenerate pairing \[
\lef{-}{-} \colon H_1(\spc{X} \setminus \infpreim ; \mathbb{Z}) \times H_1(\spc{X}, \infpreim ; \mathbb{Z}) \to \mathbb{Z},
\] given by Lefschetz duality, which we refer to as the \emph{Lefschetz pairing}.

\subsection{Triangulations from foliations}\label{sect:qd:triangs}

Infinite GMN differentials give triangulations of marked bordered surfaces with orbifold points of order three, as we shall now explain.
This triangulation will then lead to the construction of a 3-Calabi--Yau triangulated category from the infinite GMN differential $\varphi$.

A \emph{marked bordered surface with orbifold points of order three} is a triple $\mbso$ where $\bs$ is a compact connected oriented surface with boundary, $\sfm \subset \bs$ is a finite set of \emph{marked points} of $\bs$ such that every boundary component contains at least one marked point, and $\bo$ is a set of points in the interior of $\bs$ which we call the set of \emph{orbifold points of order three}, or simply \emph{orbifold points}.
Marked points in the interior of $\bs$ are called \emph{punctures} and are denoted $\bp \subseteq \sfm$.

\begin{remark}
In \cite[Definition~2.3]{chq}, the notion of a \emph{weighted marked surface} is used instead.
Marked bordered surfaces with orbifold points of order three can be obtained from weighted marked surfaces in the case where all weights are $1$ or $-1$ by forgetting the singular points of weight $1$ and the grading structure.
The singular points of weight $-1$ then become the orbifold points.
\end{remark}

An \emph{arc} in $\mbso$ is a smooth curve connecting marked points in $\sfm$ whose interior lies in $\bs \setminus (\partial\bs \cup \sfm \cup \bo)$, which has no self-intersections, and which is not homotopic to an arc in $\partial \bs$ containing no marked points in its interior.
Two arcs in $\mbso$ are considered to be equivalent if they are related by a homotopy through a set of arcs.

A \emph{triangulation} $T$ of a marked bordered surface with orbifold points of order three $\mbso$ is a finite collection of arcs in $\bs \setminus \bo$ such that arcs only intersect in endpoints and the arcs of $T$ cut $\bs$ into triangles and monogons such that monogons contain exactly one point from~$\bo$.
Two distinct triangulations $T$ and $T'$ are related by a \emph{flip} if there are arcs $\alpha \in T$ and $\alpha' \in T'$ such that $T \setminus \{\alpha\} = T' \setminus \{\alpha'\}$.
For examples of flips, see Figures~\ref{fig:quad_mut}, ~\ref{fig:simp_pole_mut}, and~\ref{fig:2cycle}.

We refer to the closures of the connected components of the complement of the arcs of a triangulation as \emph{triangles}.
Triangles come in three different types, which are shown in Figure~\ref{fig:three_triangles}.
Note that we also consider monogons to be a type of triangle.
The edge of the self-folded triangle with two distinct endpoints is called the \emph{self-folded edge}, whilst the edge with only one distinct endpoint is called the \emph{encircling edge}.

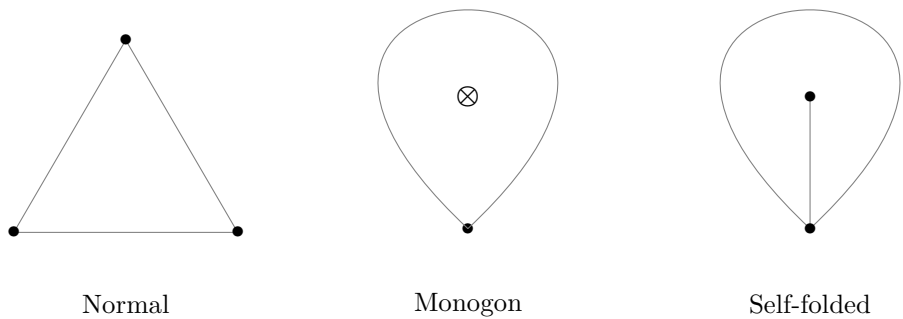
\begin{figure}
    \[
    \begin{tikzpicture}

        \begin{scope}[shift={(-4.5,-0.2)},scale=0.85]
        
       \path[use as bounding box] (-1.5,2.5) rectangle (1.5,-2.5);

        \draw[\edgecol] (90:2) -- (210:2) -- (330:2) -- (90:2);

        \node at (90:2) {$\bullet$};
        \node at (210:2) {$\bullet$};
        \node at (330:2) {$\bullet$};
        
        \end{scope}

        \begin{scope}[shift={(0,0)}]
        
       \path[use as bounding box] (-1.5,2.5) rectangle (1.5,-2.5);

        \node at (270:1) {$\bullet$};
        \node at (0,0.75) {$\bm{\otimes}$};

        \draw[\edgecol] (270:1) .. controls (35:5) and (145:5) .. (270:1);

        \node(bl) at (180:0.8) {};
        \node(br) at (360:0.8) {};
            
        \end{scope}

        \begin{scope}[shift={(4.5,0)}]
        
       	\path[use as bounding box] (-1.5,2.5) rectangle (1.5,-2.5);

        \draw[\edgecol] (270:1) .. controls (35:5) and (145:5) .. (270:1);
        \draw[\edgecol] (270:1) -- (0,0.75);

        \node(bl) at (180:0.8) {};
        \node(br) at (360:0.8) {};

        \node at (270:1) {$\bullet$};
        \node at (0,0.75) {$\bullet$};

        \end{scope}

    \node at (-4.5,-2) {Normal};
    \node at (0,-2) {Monogon};
    \node at (4.5,-2) {Self-folded};
        
    \end{tikzpicture}
    \]
    \caption{The three different types of triangle}
    \label{fig:three_triangles}
\end{figure}

A saddle-free infinite GMN differential $\varphi$ on $X$ determines a triangulation $T_\varphi$ of a marked bordered surface with orbifold points of order three $\mbso$ via the following procedure, which is a special case of \cite[Section~2.3]{chq} and which extends \cite[Section~6.1]{bs}.
The surface $\bs$ is given by taking the underlying smooth surface of $X$ and performing oriented real blow-ups on the poles of order $m \geqslant 3$.
Taking the oriented real blow-up of a pole of order $m \geqslant 3$ gives a new boundary component containing $m - 2$ marked points on it, corresponding to the $m - 2$ asymptotic directions of $\varphi$ around the pole; see for instance \cite[Figure~10]{bs}.
The other marked points in $\sfm$ are given by the poles of $\varphi$ of order $2$, which give the punctures~$\bp$.
The orbifold points $\bo$ are given by the simple poles of~$\varphi$.
The triangulation $T_\varphi$ of $\mbso$ from $\varphi$ is the collection of arcs given by taking one generic trajectory from each horizontal strip of $\varphi$.

\section{Stability conditions on 3-Calabi--Yau categories}\label{sect:3cy_stab}

Having outlined the background on quadratic differentials in Section~\ref{sect:qd}, we now give background on the other side of the Bridgeland--Smith correspondence, namely stability conditions on 3-Calabi--Yau categories constructed from quivers with potential.

\subsection{Quivers with potential from triangulations}

In this section, we explain how to obtain a quiver with potential from the triangulation $T_\varphi$ given by the quadratic differential~$\varphi$.
Our construction here deviates from the original work of Bridgeland and Smith in \cite{bs}, which follows the construction of a quiver with potential from a triangulated surface from \cite{lf}, and follows instead the category considered in \cite{chq}.

\subsubsection{Quivers with potential and Jacobian algebras}\label{sect:3cy_stab:qwp}

Quivers with potential were introduced in \cite{dwz}.
Unlike in much of the literature, the quivers with potential we consider in this paper will generally have loops and two-cycles.

A \emph{quiver} $Q = (Q_0, Q_1)$ is a directed graph with $Q_0$ the set of vertices and $Q_1$ the set of arrows.
Given an arrow $a \in Q_1$, we write $s(a) \in Q_0$ for the source of $a$ and $t(a) \in Q_0$ for the target of $a$.
If $s(a) = t(a)$, then we call $a$ a \emph{loop}.
We write $\vm = \mathbb{C}^{Q_0}$ and $\am = \mathbb{C}^{Q_1}$ for the vector spaces of $\mathbb{C}$-valued functions on $Q_0$ and $Q_1$ respectively.
We have that $\vm$ is a commutative $\mathbb{C}$-algebra under pointwise multiplication of functions, and that $\am$ is an $\vm$-bimodule with the action of $\vm$ defined such that if $e \in \vm$ and $f \in \am$, then for all $a \in Q_1$, $(e \cdot f)(a) = e(s(a))f(a)$ and $(f \cdot e)(a) = f(a)e(t(a))$.
We denote $\am^0 := \vm$ and otherwise $\am^{\ell + 1} := \am \otimes_{\vm} \am^{\ell}$.
The \emph{uncompleted path algebra} $\mathbb{C}Q$ is then defined to be the tensor algebra $\mathbb{C}Q := \bigoplus_{\ell = 0}^{\infty} E^{\ell}$, whilst the \emph{completed path algebra} $\widehat{\mathbb{C}Q}$ is defined to be the tensor algebra $\widehat{\mathbb{C}Q} := \prod_{\ell = 0}^{\infty} E^{\ell}$.
One can also view $\widehat{\mathbb{C}Q}$ as the completion of $\mathbb{C}Q$ with respect to the $I$-adic topology, where $I$ is the ideal of $\mathbb{C}Q$ generated by~$\am$.

Given $i \in Q_0$, we write $e_i$ for the function in $\vm$ such that $e_i(j) = \delta_{ij}$ for $j \in Q_0$, where $\delta_{ij}$ is the Kronecker delta.
Similarly, for $a \in Q_1$, we also write $a$ for the function in $\am$ such that $a(b) = \delta_{ab}$ for $b \in Q_1$.
Given a right $\mathbb{C}Q$-module $M$, we write $\dimu M$ for the \emph{dimension vector} $\ud = (d_i)_{i \in Q_0} \in \mathbb{N}^{Q_0}$, where $d_i = \dim_{\mathbb{C}} Me_{i}$ and $\mathbb{N} := \mathbb{Z}_{\geqslant 0}$.
Viewing $S$ as an $S$-bimodule in the usual way, we write $S_i$ for the $\vm$-sub-bimodule of $\vm$ generated by~$e_i$.
We then have that $S_i$ becomes a simple (right) module over $\mathbb{C}Q$ and $\widehat{\mathbb{C}Q}$ under the zero action of~$E$.
We refer to $S_i$ as the \emph{simple module at vertex~$i$}. 

We call the products $a_1 a_2 \dots a_{\ell} \in \mathbb{C}Q$ of arrows $a_{i} \in Q_{1}$ such that $t(a_i) = s(a_{i + 1})$ for $1 \leqslant i < \ell$ \emph{paths} of length $\ell$; these form a basis of $\am^\ell$.
We define \emph{paths of length zero} to consist of the elements $e_{i} \in \vm = E^0$.
We call such a path a \emph{cycle} if $t(a_{\ell}) = s(a_1)$.
We write $\am^{\ell}_{\cyc}$ for the $\vm$-sub-bimodule of $\am$ generated by cycles of length~$\ell$.

Let $\pot{\mathbb{C}Q} := \bigoplus_{\ell = 1}^\infty \am^{\ell}_\cyc$ and $\pot{\widehat{\mathbb{C}Q}} := \prod_{\ell = 1}^\infty \am^{\ell}_\cyc$. 
A \emph{potential} is an element of $\pot{\widehat{\mathbb{C}Q}}$, and a potential is \emph{polynomial} if it lies in the subspace $\pot{\mathbb{C}Q}$.
We say that two potentials $W$ and $W'$ are \emph{cyclically equivalent} if $W - W'$ lies in the closure of the vector subspace of $\widehat{\mathbb{C}Q}$ generated by differences $a_1 a_2 \dots a_{\ell} - a_2 \dots a_{\ell} a_1$ where $a_1 a_2 \dots a_{\ell}$ is a cycle.
For a path $p$ in $Q$, we define the map \[\partial_{p}\colon \pot{\widehat{\mathbb{C}Q}} \to \widehat{\mathbb{C}Q}\] as the linear map which takes a cycle $c$ to the sum \[\sum_{c = upv} vu\] taken over all decompositions of the cycle $c$ of the form $c = upv$, with $u$ and $v$ allowed to have length zero.
If $p = a \in Q_1$, then we call $\partial_a$ the \emph{cyclic derivative} with respect to~$a$.

Let $W \in \pot{\widehat{\mathbb{C}Q}}$ be a potential such that $W \in \bigoplus_{\ell = 2}^{\infty} \am^{\ell}$ and such that no two cyclically equivalent cycles appear in the expression of $W$ as a sum of cycles.
Then the pair $(Q, W)$ is called a \emph{quiver with potential}.
The \emph{completed Jacobian ideal} $\widehat{J}(W)$ is the closure of the two-sided ideal of $\widehat{\mathbb{C}Q}$ generated by $\{\partial_{a}W \st a \in Q_1\}$.
The \emph{completed Jacobian algebra} $\jac{Q, W}$ of a quiver with potential is the quotient of the path algebra $\widehat{\mathbb{C}Q}$ by $\widehat{J}(W)$.
If we have that, additionally, $W \in \pot{\mathbb{C}Q} \subset \pot{\widehat{\mathbb{C}Q}}$, then the \emph{uncompleted Jacobian ideal} $J(W)$ is the two-sided ideal of $\mathbb{C}Q$ generated by $\{\partial_{a}W \st a \in Q_1\}$, and the \emph{uncompleted Jacobian algebra} $\jacu{Q, W}$ is the quotient of $\mathbb{C}Q$ by $J(W)$.
We let $\modules \jac{Q, W}$ and $\modules \jacu{Q, W}$ be the respective categories of right modules which are finite-dimensional over~$\mathbb{C}$.

In this paper, the algebras that we will associate to quadratic differentials will always be completed.
But it will be necessary also to consider uncompleted algebras in order to compute some DT invariants.

\subsubsection{Quiver representations}\label{sect:3cy_stab:qwp:quiver_reps}

Quiver representations provide a concrete way of understanding modules over $\jac{Q, W}$ and $\jacu{Q, W}$.
Moreover, it will be useful to view these modules in terms of quiver representations, in order to construct stacks of modules over $\jac{Q, W}$ and $\jacu{Q, W}$ in Section~\ref{sect:dt:quiv_rep_stacks}.

Here, a \emph{representation} of $Q$ consists of a finite dimensional $\mathbb{C}$-vector space $V_{i}$ for every $i \in Q_{0}$ and a linear map $f_{a}\colon V_{i} \to V_{j}$ for every arrow $a \colon i \to j$ in $Q_{1}$.
Given a path $p = a_{1} \dots a_{\ell}$ in $Q$, we write $f_{p} = f_{a_{\ell}} \dots f_{a_{1}}$.
Similarly, given a linear combination of paths $\rho = \sum_{i = 1}^{r} \lambda_{i}p_{i}$ where $\lambda_{i} \in \mathbb{C}$, we write $f_{\rho} = \sum_{i = 1}^{r}\lambda_{i}f_{p_{i}}$.
We say that a representation $(V_{i}, f_{a})_{i \in Q_{0}, a \in Q_{1}}$ is \emph{nilpotent} if for any cycle $c$ of $Q$ there is some $r$ such that $f_c^r = 0$.
A \emph{quiver with relations} is a pair $(Q, I)$, where $I$ is a two-sided ideal of $\mathbb{C}Q$, known as the ideal of \emph{relations}.
A representation $(V_{i}, f_{a})_{i \in Q_{0}, a \in Q_{1}}$ is a \emph{representation of the quiver with relations} $(Q, I)$ if $f_{\rho} = 0$ for all $\rho \in I$.

There is a natural notion of a homomorphism of representations of a quiver with relations $(Q, I)$ \cite[Section~III.1]{ass}.
It is well-known that the category of $\mathbb{C}Q/I$-modules is equivalent to the category of representations of the quiver with relations $(Q, I)$ \cite[Chapter~III, Theorem~1.6]{ass}.
It is also well-known that the category of nilpotent representations of the quiver with relations $(Q, I)$ is equivalent to the category of $\widehat{\mathbb{C}Q}/\widehat{I}$-modules, where $\widehat{I}$ denotes the closure of the ideal $I$ inside~$\widehat{\mathbb{C}Q}$.
We thus use the terms `module' and `representation' interchangeably.

\subsubsection{The quiver with potential of a triangulation}\label{sect:3cy_stab:qwp_from_triang}

Given a triangulation $T$ of a marked bordered surface with orbifold points of order three $\mbso$, the associated quiver with potential is defined as follows.
This quiver with potential appears implicitly in \cite{chq}.
It also appears in \cite{christ_ps,ppp} in the absence of orbifold points and in \cite{lfm_i} in the absence of punctures.

For distinct edges $\alpha, \beta \in T$, we define $\triangnum{\alpha}{\beta}$ to be the number of triangles in $T$ in which $\alpha$ and $\beta$ appear as adjacent edges in an anti-clockwise order.
In particular, the edge of a monogon is considered to be adjacent to itself inside the monogon, as is the self-folded edge inside a self-folded triangle.
The quiver $Q(T)$ is then defined to have vertices given by the edges of~$T$, with $\triangnum{\alpha}{\beta}$ arrows from vertex $\alpha$ to vertex~$\beta$.

The potential associated $W(T)$ associated to the quiver $Q(T)$ is defined as follows.
We have that \[W(T) = \sum_{F}C(F),\] where $C(F)$ is the anti-clockwise three-cycle contained in each triangle $F$; if $F$ is a monogon, then $C(F) = a^3$, where $a$ is the loop at the vertex given by the arc which cuts out~$F$.
The three different types of triangles, the associated subquivers, and the terms of the potential are shown in Figure~\ref{fig:three_triangles}.
Here `$\bullet$' denotes a marked point and `$\otimes$' denotes an orbifold point.

\begin{figure}
    \[
    \begin{tikzpicture}

        \begin{scope}[shift={(-4.5,-0.2)},scale=0.85]
        
       	\path[use as bounding box] (-1.5,2.5) rectangle (1.5,-3.5);

        \draw[\edgecol] (90:2) -- (210:2) -- (330:2) -- (90:2);

        \node (lt) at ($(90:2)!0.45!(210:2)$) {};
        \node (lb) at ($(90:2)!0.55!(210:2)$) {};
        \node (br) at ($(210:2)!0.55!(330:2)$) {};
        \node (bl) at ($(210:2)!0.45!(330:2)$) {};
        \node (rt) at ($(90:2)!0.45!(330:2)$) {};
        \node (rb) at ($(90:2)!0.55!(330:2)$) {};

        \draw[\arrcol,->] (lb) -- (bl);
        \draw[\arrcol,->] (br) -- (rb);
        \draw[\arrcol,->] (rt) -- (lt);

        \node at (210:0.4) {$a$};
        \node at (330:0.4) {$b$};
        \node at (90:0.4) {$c$};

        \node at (90:2) {$\bullet$};
        \node at (210:2) {$\bullet$};
        \node at (330:2) {$\bullet$};
        
        \end{scope}

        \begin{scope}[shift={(0,0)}]
        
       	\path[use as bounding box] (-1.5,2.5) rectangle (1.5,-3.5);

        \node at (270:1) {$\bullet$};
        \node at (0,0.75) {$\bm{\otimes}$};

        \draw[\edgecol] (270:1) .. controls (35:5) and (145:5) .. (270:1);

        \node(bl) at (180:0.8) {};
        \node(br) at (360:0.8) {};

        \draw[->,\arrcol] (bl) to (br);

        \node at (0,0.25) {$a$};
            
        \end{scope}

        \begin{scope}[shift={(4.5,0)}]
        
       	\path[use as bounding box] (-1.5,2.5) rectangle (1.5,-3.5);

        \draw[\edgecol] (270:1) .. controls (35:5) and (145:5) .. (270:1);
        \draw[\edgecol] (270:1) -- (0,0.75);

        \node(bl) at (180:0.8) {};
        \node(br) at (360:0.8) {};

        \draw[->,\arrcol] (0.1,-0.5) -- (1,0.75);
        \draw[->,\arrcol] (-1,0.75) -- (-0.1,-0.5);
        \draw[<-,\arrcol] ($(280:0.25)+(0,0.75)$) arc (-80:260:0.25);

        \node at (270:1) {$\bullet$};
        \node at (0,0.75) {$\bullet$};

        \node at (-0.3,0.25) {$a$};
        \node at (0.3,0.25) {$c$};
        \node at (0,1.25) {$b$};

        \end{scope}

    \node at (-4.5,-1.7) {$abc$};
    \node at (-4.5,-2.5) {Normal};

    \node at (0,-1.7) {$a^3$};
    \node at (0,-2.5) {Monogon};

    \node at (4.5,-1.7) {$abc$};
    \node at (4.5,-2.5) {Self-folded};
        
    \end{tikzpicture}
    \]
    \caption{Subquivers and terms of potential for the different types of triangles}
    \label{fig:three_triangles_arrows}
\end{figure}
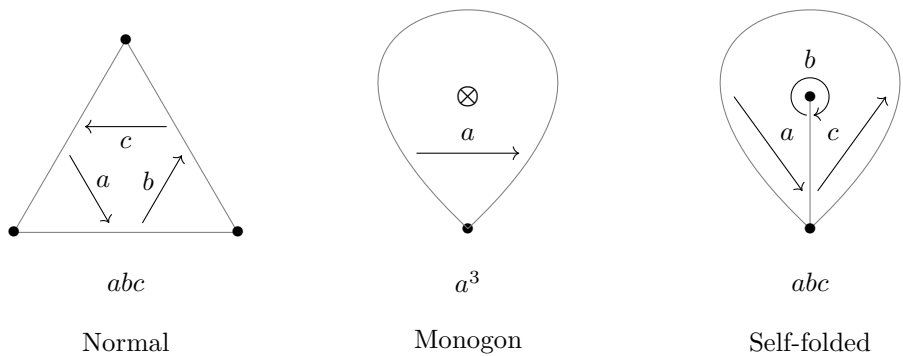

\begin{remark}
The difference between this potential and the usual one from \cite{bs,lf} is that cycles around punctures are not included.
This means that the resulting Jacobian algebras are generally infinite-dimensional.
The philosophical reason why cycles around punctures ought not to be included in the potential is as follows.
If a quadratic differential has a double pole, then there is always some rotation where a degenerate ring domain appears.
There ought to be non-trivial semistable representations corresponding to the finite-length trajectories in the degenerate ring domain.
If the cycles are included in the potential, then such representations do not exist.
As observed in \cite[Section~1.2.1]{chq}, a side effect of the non-existence of such representations is that the moduli spaces of quadratic differentials corresponding to stability conditions in \cite{bs} must contain quadratic differentials where zeros have collided with double poles to form simple poles.
These would otherwise be excluded by the support property of stability conditions \cite[Definition~1]{ks_stability}.
\end{remark}

\begin{remark}
The constructions in \cite{chq} are very general, so we briefly indicate how they specialise to the above.
The relevant construction is \cite[Construction~6.18]{chq} with the reduction from \cite[Remark~6.19]{chq} applied.
In terms of the notation in \textit{op.\ cit.}, we have $n = 3$, since our only finite critical points are simple zeros and simple poles.
At each finite critical point $r$, the quiver $Q_{r}^3$ from \cite[Construction~6.14]{chq} following \cite[Definition~6.6, Definition~6.8]{chq} has degree-zero part as described above, depending upon whether $r$ is a zero or a simple pole.
The parts in non-zero degree define the associated Ginzburg differential graded algebra.
Note that, as per \cite[Definition~6.8]{chq}, the quiver for monogons is a quotient of the three-cycle coming from zeros by the action of the cyclic group of order three.
One can also consider the Koszul dual construction from \cite[Definition~7.2]{chq}.
\end{remark}

\subsection{3-Calabi--Yau categories}

First we must recall some of the general theory on 3-Calabi--Yau $A_\infty$-categories.

\subsubsection{$A_\infty$-categories}\label{sect:3cy_stab:a_inf}

For reasons of space, we do not give full background on $A_\infty$-categories, which can be found in \cite{kajiura_jppa,keller_a_inf,keller_a_inf_intro,klh_thesis}.
Instead, recall that a (small) $A_\infty$-category $\mathcal{A}$ over $\mathbb{C}$ is given by
\begin{itemize}
    \item a set of objects $\obj{\mathcal{A}}$,
    \item for all $A, B \in \obj{\mathcal{A}}$, a $\mathbb{Z}$-graded $\mathbb{C}$-vector space $\Hom_{\mathcal{A}}^{\bullet}(A, B)$.
    \item for all $n \geqslant 1$ and all $A_0, A_1, \dots, A_n \in \obj{\mathcal{A}}$, a graded $\mathbb{C}$-linear map \[
    m_n \colon \Hom_{\mathcal{A}}^{\bullet}(A_{n - 1}, A_n) \otimes \Hom_{\mathcal{A}}^{\bullet}(A_{n - 2}, A_{n - 1}) \otimes \dots \otimes \Hom_{\mathcal{A}}^{\bullet}(A_{0}, A_1) \to \Hom_{\mathcal{A}}^{\bullet}(A_{0}, A_n)
    \] of degree~$2-n$; these are required to satisfy some compatibility conditions, see \cite[Section~7.2]{keller_a_inf_intro}.
\end{itemize}
A morphism $\mathrm{id}_A \in \Hom_{\mathcal{A}}^{0}(A, A)$ is a \emph{strict identity} for $A \in \obj{\mathcal{A}}$ if $m_2(\mathrm{id}_A, f) = f$ and $m_2(g, \mathrm{id}_{A}) = g$ whenever these make sense, and $m_n(\dots, \mathrm{id}_{A}, \dots) = 0$ for all $n \neq 2$.
We say that $\mathcal{A}$ is \emph{strictly unital} if every object has a strict identity.

Here, $m_1$ is akin to a differential on $\Hom_{\mathcal{A}}^{\bullet}(A, B)$ and $m_2$ is akin to composition of maps.
Let us then write $\Ext_{\mathcal{A}}^{n}(A, B) := H^n(\Hom_{\mathcal{A}}^{\bullet}(A, B))$, where cohomology is taken with respect to~$m_1$.
For $n \geqslant 3$, $m_n$ are higher operations which resolve the failure of $m_2$ to be associative.
An $A_{\infty}$-category $\mathcal{A}$ is thus not a category in the usual sense, due to this failure.
Define the \emph{homotopy category} $H^{0}(\mathcal{A})$ to have the same objects as $\mathcal{A}$, and with morphism spaces $\Hom_{H^{0}(\mathcal{A})}(A, B) := \Ext^{0}_{\mathcal{A}}(A, B)$. The homotopy category $H^{0}(\mathcal{A})$ is then a category if $\mathcal{A}$ is strictly unital \cite[Remark~5.1.2.3]{klh_thesis}.
Given a category $\mathcal{C}$ which is equivalent to $H^{0}(\mathcal{A})$ for an $A_\infty$-category $\mathcal{A}$, we say that $\mathcal{A}$ is a \emph{$A_\infty$-enhancement of~$\mathcal{C}$}.

Given two $A_\infty$-categories $\mathcal{A}$ and $\mathcal{B}$, an \emph{$A_\infty$-functor} $F \colon \mathcal{A} \to \mathcal{B}$ consists of a map of sets $F \colon \obj{\mathcal{A}} \to \obj{\mathcal{B}}$ along with, for all $A_0, A_1, \dots, A_n \in \obj{\mathcal{A}}$ a graded $\mathbb{C}$-linear map \[
    F_n \colon \Hom_{\mathcal{A}}^{\bullet}(A_{n - 1}, A_n) \otimes \Hom_{\mathcal{A}}^{\bullet}(A_{n - 2}, A_{n - 1}) \otimes \dots \otimes \Hom_{\mathcal{A}}^{\bullet}(A_{0}, A_1) \to \Hom_{\mathcal{B}}^{\bullet}(F(A_{0}), F(A_n))
    \] of degree~$1 - n$; these are also required to satisfy some compatibility conditions \cite[Section~7.3]{keller_a_inf_intro}.
If $\mathcal{A}$ and $\mathcal{B}$ are additionally strictly unital, then a \emph{unital} functor $F$ must also preserve identities, see \cite[Definition~3.3.4]{davison_thesis}.
We will always require our functors to be strictly unital.
An $A_\infty$-functor $F \colon \mathcal{A} \to \mathcal{B}$ is called a \emph{quasi-equivalence} if it induces an equivalence of categories $H^{0}(F) \colon H^{0}(\mathcal{A}) \xrightarrow{\sim} H^{0}(\mathcal{B})$ and if, for all $A, B \in \obj{\mathcal{A}}$, the map $F_1 \colon \Hom_{\mathcal{A}}^{\bullet}(A, B) \to \Hom_{\mathcal{B}}^{\bullet}(F(A), F(B))$ is a quasi-isomorphism of chain complexes.
See also \cite[Th\'eor\`eme~9.2.0.4]{klh_thesis}.
An $A_\infty$-category is \emph{minimal} if $m_1 = 0$.
A \emph{minimal model} of an $A_\infty$-category $\mathcal{A}$ is a minimal $A_\infty$-category $\mathcal{B}$ which is quasi-equivalent to~$\mathcal{A}$.
Minimal models are unique up to isomorphism; see for instance  \cite{kadeishvili}, \cite[Corollary~2.27]{kajiura_jppa}, \cite[Theorem~3.3]{keller_a_inf_intro}.

\subsubsection{3-Calabi--Yau $A_\infty$-categories}\label{sect:3cy_stab:3cy_a_inf}

A $\mathbb{C}$-linear $A_{\infty}$-category $\mathcal{A}$ is \emph{$3$-Calabi--Yau} if there is a functorial non-degenerate pairing \[
\cy{-}{-} \colon \Hom_{\mathcal{A}}^{\bullet}(A, B) \otimes_\mathbb{C} \Hom_{\mathcal{A}}^{\bullet}(B, A) \to \mathbb{C}[-3]
\]
for all $A, B \in \mathcal{A}$ such that $\frac{1}{n}\cy{m_{n - 1}(-, \dots, -)}{-}$ is invariant under cyclic permutation of its arguments; see, for example, \cite[Definition~4.1.6]{davison_thesis}.
These are also known as \emph{cyclic $A_\infty$-categories}.

A \emph{functor between $3$-Calabi--Yau $A_\infty$-categories} $\mathcal{A}$ and $\mathcal{B}$ is given by an $A_\infty$-functor $F \colon \mathcal{A} \to \mathcal{B}$ such that for any $A, B \in \obj{\mathcal{A}}$ and $f \in \Hom_{\mathcal{A}}^{\bullet}(A, B)$, $g \in \Hom_{\mathcal{A}}^{\bullet}(B, A)$, we have that \[\cy{F_1(f)}{F_1(g)}_{\mathcal{B}} = \cy{f}{g}_{\mathcal{A}},\] and such that $\sum_{k + l = n} \cy{-}{-}_{\mathcal{B}} \circ (F_k \otimes F_l) = 0$ for all $n \geqslant 3$, see \cite[Definition~4.1.12]{davison_thesis} and \cite[Definition~2.13]{kajiura}.
We have that $F$ is moreover a \emph{quasi-equivalence} if it is a quasi-equivalence of $A_\infty$-categories.
A \emph{cyclic minimal model} of a 3-Calabi--Yau $A_\infty$-category $\mathcal{A}$ is a 3-Calabi--Yau minimal $A_\infty$-category $\mathcal{B}$ which is quasi-equivalent to~$\mathcal{A}$ as a 3-Calabi--Yau $A_\infty$-category; see Proposition~\ref{prop:min_model}.

Later we will need the \emph{potential} $W_{\mathcal{A}}$ on $\Hom_{\mathcal{A}}^{1}(A, A)$ for an object $A \in \mathcal{A}$, which is defined as the formal function
\begin{equation}\label{eq:cat_potential}
    W_{\mathcal{A}}(x) = \sum_{n \geqslant 2} \frac{1}{n} \cy{m_{n - 1}(x, x, \dots, x)}{x},
\end{equation}
where $x \in \Hom_{\mathcal{A}}^{1}(A, A)$.
Here a \emph{formal function} is a function on the scheme-theoretic formal neighbourhood of $0 \in \Hom_{\mathcal{A}}^{1}(A, A)$.
To put it more intuitively, this is \textit{a priori} a power series since infinitely many of the terms may be non-zero.

\subsubsection{Twisted complexes}

We briefly recall the notion of twisted complexes over an $A_\infty$-category~$\mathcal{A}$ \cite[Section~7.6]{keller_a_inf_intro}.
Let $\mathbb{Z}\mathcal{A}$ be the category of objects $(A, n)$ with $A \in \mathcal{A}$ and $n \in \mathbb{Z}$, with morphisms from $(A, n)$ to $(A', n')$ given by the graded vector space $\Hom_{\mathcal{A}}^{\bullet + n' - n}(A, A')$.
The $A_\infty$-operations of $\mathbb{Z}\mathcal{A}$ are given by shifting those of $\mathcal{A}$.
We write $[n]$ for the functor on $\mathbb{Z}\mathcal{A}$ defined $(A, n')[n] := (A, n + n')$.
The category $\tw{\mathcal{A}}$ of (\emph{one-sided}) \emph{twisted complexes} over $\mathcal{A}$ has as objects pairs $(B, a)$, where $B = (A_1, A_2, \dots, A_d)$ is a sequence of objects of $\mathbb{Z}\mathcal{A}$ and $a = (a_{ij})$ a matrix of degree~$1$ morphisms \[
a_{ij} \in \Hom_{\mathcal{A}}^{1}(A_i, A_j)
\] with $a_{ij} = 0$ for $i \leqslant j$, which satisfies the compatibility condition
\begin{equation}\label{eq:mc}
    \sum_{n = 1}^{\infty} (-1)^{\frac{n(n - 1)}{2}} m_n(a, a, \dots, a) = 0
\end{equation}
with respect to the $A_\infty$-operations, known as the \emph{Maurer--Cartan equation}.
Here we have extended the operations $m_n$ to matrices of morphisms in the natural way.
Note that the strict lower triangularity of $a$ ensures that this sum is finite.

The morphisms of $\tw{\mathcal{A}}$ are defined by $\Hom^{\bullet}_{\tw{\mathcal{A}}}((B, a), (B', a')) := \bigoplus_{i, j}\Hom^{\bullet}_{\mathbb{Z}\mathcal{A}}(A_i, A'_j)$, for $B = (A_1, A_2, \dots, A_d)$ and $B' = (A'_1, A'_2, \dots, A'_{d'})$.
The $A_\infty$-operations of $\tw{\mathcal{A}}$ are then defined as follows.
Let $C_0 = (B_0, a_0)$, $C_1 = (B_1, a_1)$, and $C_n = (B_n, a_n)$ be objects of $\tw{\mathcal{A}}$.
The operation \[
    m_n^{\two} \colon \Hom_{\tw{\mathcal{A}}}^{\bullet}(C_{n - 1}, C_n) \otimes \Hom_{\tw{\mathcal{A}}}^{\bullet}(C_{n - 2}, C_{n - 1}) \otimes \dots \otimes \Hom_{\tw{\mathcal{A}}}^{\bullet}(C_0, C_1) \to \Hom_{\tw{\mathcal{A}}}^{\bullet}(C_0, C_n)
\]
is defined
\begin{equation}\label{eq:tw_op}
m_n^{\two} := \sum_{r = 0}^{\infty} \sum_{\substack{i_1 + \dots + i_k = n \\
j_1 + \dots + j_k = r}} \pm m_{n + r} \circ (\id^{\otimes i_1} \otimes a^{\otimes j_1} \otimes \id^{\otimes i_2} \otimes a^{\otimes j_2} \otimes \dots \otimes \id^{\otimes i_k} \otimes a^{\otimes j_k}),    
\end{equation}
where we have written $a$ for all of the $a_i$.
The sign is given by the identity \[
(sx)^{i_1}(sy)^{j_1} \dots (sx)^{i_r}(sy)^{j_r} = \pm s^{n + 1}x^{i_1}y^{j_1}\dots x^{i_r}y^{j_r}
\]
in the algebra $\mathbb{Z}\langle x, y, s\rangle/\ideal{sx - xs, sy + ys}$.
The functor $[n]$ on $\mathbb{Z}\mathcal{A}$ also gives a functor on $\tw{\mathcal{A}}$.
If $\mathcal{A}$ is a 3-Calabi--Yau $A_\infty$-category, then $\tw{\mathcal{A}}$ is also 3-Calabi--Yau \cite[Theorem~4.1.9]{davison_thesis}.

We have that $H^{0}(\tw{\mathcal{A}})$ is a triangulated category if $\mathcal{A}$ is strictly unital \cite[Th\'eor\`eme~7.1.0.4]{klh_thesis}.
A $\mathbb{C}$-linear triangulated category $\mathcal{D}$ is \emph{$3$-Calabi--Yau} if for all $i \in \mathbb{Z}$ there is a functorial non-degenerate pairing $\Hom_{\mathcal{D}}(M, N[i]) \otimes_{\mathbb{C}} \Hom_{\mathcal{D}}(N, M[3 - i]) \to \mathbb{C}$ for all $M, N \in \mathcal{D}$.
In particular, triangulated categories with 3-Calabi--Yau $A_\infty$-enhancements will be 3-Calabi--Yau.

Given twisted complexes $C = (B, a)$ and $C' = (B', a')$ with $B = (A_1, A_2, \dots, A_n)$ and $B' = (A'_1, A'_2, \dots, A'_{n'})$ with a degree 0 morphism $f \colon C \to C'$ such that $m_1^{\two}(f) = 0$, the \emph{mapping cone} $\mathrm{Cone}(f)$ is defined to be the twisted complex (see \cite[Proposition~3.8.6]{davison_thesis}) \[
\mathrm{Cone}(f) := \left((A_1[1], A_2[1], \dots, A_n[1], A'_1, A'_2, \dots, A'_{n'}), \left(\begin{smallmatrix}a & 0 \\ f & a' \end{smallmatrix} \right) \right).
\]

We write $\twz{\mathcal{A}}$ for the full subcategory of $\tw{\mathcal{A}}$ consisting of twisted complexes $(B, a)$ for $B = (A_1, A_2, \dots, A_d)$ a sequence of objects of~$\mathcal{A}$, rather than~$\mathbb{Z}\mathcal{A}$.
We refer to these as \emph{degree~$0$ twisted complexes}.

\subsubsection{3-Calabi--Yau $A_\infty$-categories from quivers with potential}

One can define the following 3-Calabi--Yau $A_\infty$-category from a quiver with potential, see \cite[p.91]{davison_thesis} and \cite[pp.1261--2]{dm_curves}. 

\begin{definition}\label{def:akd}
Recalling the notation $\am$ and $e_i$ from Section~\ref{sect:3cy_stab:qwp}, we define the 3-Calabi--Yau $A_\infty$-category $\akdqw$ as follows.
The set of objects of $\akdqw$ is $\{S_i \st i \in Q_0\}$ and the morphisms are as follows.
\begin{enumerate}
    \item The degree 0 morphisms of $\akdqw$ consist of the one-dimensional $\mathbb{C}$-vector spaces spanned by the identity morphisms $\mathrm{id}_i$ at $S_i$.
    \item The degree 1 morphisms from $S_i$ to $S_j$ are given by the vector space $e_i \am e_j$.
    \item The degree 2 morphisms from $S_i$ to $S_j$ are given by the dual vector space $(e_i \am e_j)^{\ast}$.
    \item For each $S_i$, there is a one-dimensional $\mathbb{C}$-vector space of degree~$3$ endomorphisms, spanned by a morphism $\mathrm{id}^{\ast}_i$.
\end{enumerate}
These comprise \emph{all} the morphisms of $\akdqw$.

The $A_\infty$-operations $m_n$ are then given as follows.
\begin{enumerate}
    \item The degree 0 morphisms $\mathrm{id}_i$ are strict identities for all $i \in Q_0$.
    \item If $a \in e_i \am e_j$ is a degree 1 morphism from $S_i$ to $S_j$ and $b^{\ast} \in (e_i \am e_j)^{\ast}$ is a degree 2 morphism from $S_i$ to $S_j$, then we have that $m_2(a, b^\ast) = b^\ast(a)\mathrm{id}^{\ast}_j$ and $m_2(b^\ast, a) = b^\ast(a)\mathrm{id}^{\ast}_i$.
    \item Let $a_1, a_2, \dots, a_n$ be a composable sequence of degree 1 morphisms in $\akdqw$ beginning at $S_i$ and ending at~$S_j$.
    Furthermore, let $W_{n + 1} = \sum_{k = 1}^r \lambda_k b_{1,k}b_{2,k} \dots b_{n + 1, k}$ be the sum of the terms in the potential $W$ containing cyclic words of length $n + 1$, where $\lambda_k \in \mathbb{C}$.
    Then
    \begin{align*}
        m_n(a_n, a_{n - 1}, \dots, a_1) &:= \sum_{k = 1}^r \frac{\lambda_k}{n + 1}\bigl(b_{1,k}^\ast(a_1) b_{2,k}^\ast(a_2) \dots b_{n, k}^\ast(a_n)b^{\ast}_{n + 1, k}  \\
        &\qquad\qquad\qquad + b_{2,k}^\ast(a_1) b_{3,k}^\ast(a_2) \dots b_{n + 1, k}^\ast(a_n)b^{\ast}_{1, k} + \dots \\
        &\qquad\qquad\qquad + b_{n + 1,k}^\ast(a_n) b_{1,k}^\ast(a_2) \dots b_{n - 1, k}^\ast(a_1)b^{\ast}_{n, k}\bigr).
    \end{align*}
    To explain this more informally, $m_n(a_n, a_{n - 1}, \dots, a_1)$ will be zero unless $a_1 a_2 \dots a_n a_{n + 1}$ is a summand of the potential (up to cyclic rotation) for some $a_{n + 1}$, in which case the result will be the degree 2 morphism given by $a_{n + 1}^{\ast}$, up to scalar.
\end{enumerate}
All other $A_\infty$-operations give zero.
The pairing $\cy{-}{-}$ which makes this into a 3-Calabi--Yau $A_\infty$-category is given by the vector-space duality $\cy{a^{\ast}}{b} = \cy{b}{a^\ast} = a^\ast(b)$ for $a^\ast \in \Hom^1_{\akd{Q}{W}}(S_i, S_j)$ and $b \in \Hom^2_{\akd{Q}{W}}(S_j, S_i)$, and by $\cy{\mathrm{id}_i}{\mathrm{id}^{\ast}_i} = \cy{\mathrm{id}^{\ast}_i}{\mathrm{id}_i} = 1$ for degree~$0$ and degree~$3$ morphisms.
\end{definition}

It is a good exercise to check that the potential $W_{\akdqw}$ from Section~\ref{sect:3cy_stab:3cy_a_inf} indeed coincides with~$W$ on $\End^{1}_{\akdqw}(\bigoplus_{i \in Q_0} S_i)$.
The $A_\infty$-category $\akdqw$ is the Koszul dual of the Ginzburg differential graded algebra of $(Q, W)$, see \cite[Proposition~4.2.2]{davison_thesis} or \cite[Section~4A]{dm_curves}, \cite{ginzburg}, and \cite[Section~4 and~5]{lpwz}.
The relationship between $\jac{Q, W}$ and $\akdqw$ is as follows.
We abbreviate $\mathcal{H}_{\infty}(Q, W) := \twz{\akdqw}$

\begin{theorem}[{\cite[Section~4A]{dm_curves}, \cite[Section~7.7]{keller_a_inf_intro}}]\label{thm:koszul_duality_heart}
We have that $H^{0}(\mathcal{H}_{\infty}(Q, W)) \simeq \modules \jac{Q, W}$, so that $\mathcal{H}_{\infty}(Q, W)$ is a 3-Calabi--Yau $A_\infty$-enhancement of $\modules \jac{Q, W}$.
\end{theorem}

\begin{remark}
In the case where the terms of the potential $W$ are all cubic, as they will be for $W(T)$, where $T$ is a triangulation, we have that $\End_{\akdqw}^{\bullet}(\bigoplus_{i \in Q_0} S_i)$ is a graded algebra.
Moreover, it coincides with the graded algebra from \cite[Definition~7.2]{chq}. 
\end{remark}

\subsection{Stability conditions and Bridgeland--Smith correspondence}

Having obtained the relevant category, we now give the definition of Bridgeland stability conditions, and explain how they are related to the quadratic differentials from Section~\ref{sect:qd}.

\subsubsection{Hearts and t-structures}\label{sect:app:heart_tstr}

Given a triangulated category $\mathcal{D}$ with full subcategories $\mathcal{T}$ and $\mathcal{F}$, we write \[
\mathcal{T} \ast \mathcal{F} = \{D \in \mathcal{D} \st \text{there is a triangle } T \to D \to F \to T[1], \text{ with } T \in \mathcal{T}, F \in \mathcal{F}\}.
\]
We say that $(\mathcal{T}, \mathcal{F})$ is a \emph{torsion pair} of $\mathcal{D}$ if $\Hom_{\mathcal{D}}(\mathcal{T},\mathcal{F}) = 0$ and $\mathcal{T} \ast \mathcal{F} = \mathcal{D}$.
A torsion pair $(\mathcal{T}, \mathcal{F})$ is a \emph{t-structure} if $\mathcal{T}[1] \subset \mathcal{T}$, where $[1]$ denotes the shift functor in $\mathcal{D}$.
A t-structure $(\mathcal{T}, \mathcal{F})$ is \emph{bounded} if for any $D \in \mathcal{D}$, we have that $D[\gg 0] \in \mathcal{T}$ and $D[\ll 0] \in \mathcal{F}$.
The \emph{heart} of the t-structure is $\mathcal{H} = \mathcal{T}[-1] \cap \mathcal{F}$, and is always an abelian category.
The heart $\mathcal{H}$ is \emph{length} if it is a \emph{length category}, meaning that every object has a finite-length composition series.

\subsubsection{Stability conditions}\label{sect:3cy_stab:stab}

We use standard modifications to the notion of Bridgeland stability condition introduced in \cite{ks_stability}; see also \cite{hkk}.
Fix a triangulated category $\mathcal{D}$ and a group homomorphism $\cl \colon K_0(\mathcal{D}) \to \fgg$, where $K_0(\mathcal{D})$ is the Grothendieck group of $\mathcal{D}$ and $\fgg$ is a finitely generated abelian group.

\begin{definition}[{\cite[Definition~1.1]{bridgeland}}]\label{def:stab_cond}
A \emph{stability condition} $\sigma = (Z, \mathcal{P}_{\sigma})$ on $\mathcal{D}$ consists of a group homomorphism $Z \colon \fgg \to \mathbb{C}$ called the \emph{central charge}, and a \emph{slicing} given by full additive subcategories $\ssc{\sigma}{\vartheta} \subset \mathcal{D}$ for each $\vartheta \in \mathbb{R}$, which together satisfy the following axioms.
\begin{enumerate}
    \item If $M \in \ssc{\sigma}{\vartheta}$ then $Z(\cl(M)) \in \mathbb{R}_{>0} \cdot e^{i \pi \vartheta} \subset \mathbb{C}$.
    \item For all $\vartheta \in \mathbb{R}$, $\ssc{\sigma}{\vartheta + 1} = \ssc{\sigma}{\vartheta}[1]$.
    \item If $\vartheta_1 > \vartheta_2$ and $M_i \in \ssc{\sigma}{\vartheta_i}$, then $\Hom_{\mathcal{D}}(M_1, M_2) = 0$.
    \item The Harder--Narasimhan property holds \cite[Definition~1.1(d)]{bridgeland}.
    \item The support property holds \cite[Definition~1]{ks_stability}.
\end{enumerate}
\end{definition}

Just as one can rotate quadratic differentials, one can also rotate stability conditions.
Namely, given a stability condition $\sigma = (Z, \mathcal{P}_{\sigma})$ and $\varepsilon > 0$, one can rotate to $e^{\pm i\pi\varepsilon}\sigma := (e^{\pm i\pi\varepsilon}Z, \mathcal{P}_{e^{\pm i\pi\varepsilon}\sigma})$ where $\ssc{e^{\pm i\pi\varepsilon}\sigma}{\vartheta} := \ssc{\sigma}{\vartheta \mp \varepsilon}$.

It follows from Definition~\ref{def:stab_cond} that the categories $\ssc{\sigma}{\vartheta}$ are abelian \cite[Lemma~5.2]{bridgeland}.
The non-zero objects of $\ssc{\sigma}{\vartheta}$ are said to be \emph{semistable of phase $\vartheta$}, whilst the simple objects of $\ssc{\sigma}{\vartheta}$ are said to be \emph{stable of phase~$\vartheta$}.

A stability condition $\sigma$ determines a bounded t-structure whose heart $\mathcal{H}$ is the extension closure of $\bigcup_{\vartheta \in (0,1]} \ssc{\sigma}{\vartheta}$.
Conversely, if we have a heart $\mathcal{H}$ of a bounded t-structure on $\mathcal{D}$, then there is an isomorphism $K_0(\mathcal{H}) \cong K_0(\mathcal{D})$, so that we have a map $\cl \colon K_0(\mathcal{H}) \to \Gamma$.
Defining a stability condition on $\mathcal{D}$ is equivalent to giving a central charge $Z \colon \fgg \to \mathbb{C}$ which satisfies the Harder--Narasimhan and support properties and is such that for all $0 \neq M \in \mathcal{H}$, we have that $Z(\cl(M)) \in \uhp := \{r e^{i\pi\vartheta} \st r > 0,\, 0 < \vartheta \leqslant 1\} \subset \mathbb{C}$ \cite[Proposition~5.3]{bridgeland}.
Since the heart is abelian, in this way it also makes sense to talk of stability conditions on abelian categories.

Given a non-zero object $M \in \mathcal{H}$, we write $\scph{M} := \scph{\cl(M)} = \frac{1}{\pi}\arg Z(\cl(M)) \in (0, 1]$ for the \emph{phase} of~$M$. %
For $\vartheta \in (0, 1]$ and an object $M \in \mathcal{H}$, we have that $M$ is semistable if and only if for all subobjects $L \subset M$ in $\mathcal{H}$, we have that $\scph{L} \leqslant \scph{M}$.
Similarly, $M$ is stable if and only if for all subobjects $L \subset M$, we have that $\scph{L} < \scph{M}$.

One then obtains a stability condition, in the sense of Definition~\ref{def:stab_cond}, from a heart with a central charge as follows.
Suppose that we have a heart $\mathcal{H}$ of a bounded t-structure on $\mathcal{D}$ and a central charge $Z \colon \Gamma \to \mathbb{C}$ such that for all $0 \neq M \in \mathcal{H}$ we have $Z(\cl(M)) \in \uhp$.
One obtains the slicing $\mathcal{P}_{\sigma}$ by defining $\ssc{\sigma}{\vartheta}$ to be the category of semistable objects of phase $\vartheta$ in $\mathcal{H}$ for $\vartheta \in (0, 1]$, using the criterion for semistability of the previous paragraph.
One then extends this by $\ssc{\sigma}{\vartheta \pm n} := \ssc{\sigma}{\vartheta}[\pm n]$.

We write $\Stab(\mathcal{D})$ for the set of stability conditions on $\mathcal{D}$.
By a theorem of Bridgeland \cite[Theorem~1.2]{bridgeland}, we have that $\Stab(\mathcal{D})$ has the structure of a complex manifold, known as the \emph{stability manifold}, obtained by taking $Z(\hhc{\gamma})$, $\hhc{\gamma} \in \fgg$ as holomorphic coordinates.

\subsubsection{Bridgeland--Smith correspondence}\label{sect:3cy_stab:bs}

We now explain the construction from quadratic differentials to stability conditions.

Given a triangulation $T$ of a marked bordered surface with orbifold points of order three $\mbso$, we associate a 3-Calabi--Yau $A_\infty$-category $\mathcal{D}_\infty(T) := \tw{\akd{Q(T)}{W(T)}}$ whose homotopy category is a 3-Calabi--Yau triangulated category $\mathcal{D}(T) := H^{0}(\mathcal{D}_\infty(T))$.
The category $\mathcal{D}(T)$ is equivalent to the homotopy $1$-category of the $\infty$-category considered in \cite[Theorem~6.20]{chq} by Koszul duality; see \cite[Section~7]{chq}, or \cite[Section~4A]{dm_curves}, \cite[Proposition~4.2.2]{davison_thesis}, and \cite[Section 4 and 5]{lpwz}.
It will be a consequence of Theorem~\ref{thm:enhance_equiv} that if $T'$ is a flip of $T$ then there is a quasi-equivalence $\mathcal{D}_\infty(T) \simeq \mathcal{D}_\infty(T')$.
See \cite[Corollary~4.21]{chq} for more general results using $\infty$-categories.
All triangulations of a marked bordered surface with orbifold points of order three $\mbso$ are connected by flips by \cite{cs,fst_orb}, see \cite[Theorem~3.8]{lfm_i}.
We therefore have a $3$-Calabi--Yau category $\mathcal{D}=\mathcal{D}\mbso$ with $A_\infty$-enhancement $\mathcal{D}_\infty=\mathcal{D}_\infty\mbso$ which only depends upon the surface $\mbso$, up to quasi-equivalence, and not upon a particular choice of triangulation, which we drop from the notation.
We also commonly write $\mathcal{H}_\infty = \mathcal{H}_\infty(T):= \mathcal{H}_{\infty}(Q(T),W(T))$ and $\mathcal{H} =\mathcal{H}(T) := H^{0}(\mathcal{H}_\infty(T))$, which is the heart of a t-structure on $\mathcal{D}$.

\begin{construction}\label{const:bs}
Let $\varphi$ be a saddle-free infinite GMN differential on a compact connected Riemann surface~$X$.
We construct a stability condition $\sigma_{\varphi}$ on $\mathcal{D} = H^{0}({\mathcal{D}_
\infty{(Q(T_\varphi)},{W(T_\varphi))}})$ by giving a heart with an appropriate central charge, recalling the triangulation $T_{\varphi}$ from Section~\ref{sect:qd:triangs}.

We choose the heart $\mathcal{H} = \mathcal{H}(T_{\varphi}) \simeq \modules \jac{T_\varphi}$, where $\jac{T_\varphi} := \jac{Q(T_\varphi), W(T_\varphi)}$.
Since we are working over completed path algebras, we have that $K_0(\mathcal{D}) \cong K_0(\mathcal{H})$ is a finitely generated abelian group which can be identified with $\mathbb{Z}^{Q_0(T_\varphi)}$, with the class of a $\jac{T_\varphi}$-module $M$ being given by $\dimu M$.
Hence, we set $\fgg = \mathbb{Z}^{Q_0(T_{\varphi})}$ and have $\cl \colon K_0(\mathcal{D}) \to \fgg$ as the identity.
We thus suppress the $\cl$ map.

Define $\hhspan{\varphi} \xrightarrow{\sim} K_0(\mathcal{\mathcal{H}})$ by sending the class $\hhc{\gamma}$ of a standard saddle connection $\gamma$ to $\dimu S_i$, where $S_i$ is the simple module at the vertex of the quiver $i$ given by the horizontal strip containing $\gamma$.
This is an isomorphism since the standard saddle classes form a $\mathbb{Z}$-basis of $\hhspan{\varphi}$, whilst the classes of $S_i$ are a $\mathbb{Z}$-basis of the Grothendieck group $K_0(\modules \jac{T_\varphi})$.
The central charge for the stability condition is then given by the central charge $Z_\varphi \colon K_0(\mathcal{H}) \cong \hhspan{\varphi} \to \mathbb{C}$ from Section~\ref{sect:qd:ssc}.
\end{construction}

Note that Construction~\ref{const:bs} only applies to saddle-free quadratic differentials, and that it involves a choice of heart of $\mathcal{D}$, since one could also choose a shift of $\mathcal{H}$.
In the case where the polar type of the quadratic differential is not $m = \{-2\}$, Christ, Haiden, and Qiu use Construction~\ref{const:bs} to construct a map of complex manifolds from a suitably-defined moduli space of quadratic differentials of a fixed polar type to the stability manifold $\Stab(\mathcal{D})$.
Said map is an isomorphism onto its image, which is a union of connected components of $\Stab(\mathcal{D})$ \cite[Theorem~5.4, Theorem~6.20, Lemma~6.23]{chq}.
This map works by fixing a basepoint saddle-free infinite GMN differential $\varphi_{0}$ on $X$.
Given another saddle-free infinite GMN differential $\varphi$ on~$X$, one obtains a stability condition $\sigma_{\varphi}$ on $\mathcal{D}(T_{\varphi_0})$ by applying a certain equivalence between $\mathcal{D}(T_{\varphi_0})$ and $\mathcal{D}(T_{\varphi})$ to the stability condition $\sigma_{\varphi}$ on $\mathcal{D}(T_{\varphi})$.
One can then extend this map to infinite GMN differentials which have a single saddle trajectory, and then inductively extend the map to differentials with increasingly more saddle trajectories and recurrent trajectories.
For more detail, consult \cite[Section~11.2]{bs}, \cite[Sections~8.1 and ~8.2]{bmqs}, or \cite[Section~5.2]{chq}.
This map is equivariant with respect to rotation of the quadratic differential and stability condition.
Extending the map to infinite GMN differentials with more than one saddle trajectory uses results that do not cover the case where the polar type is $m = \{-2\}$.
Note that by~\eqref{eq:rr}, the polar type is $\{-2\}$ if and only if $\varphi$ is a quadratic differential on the torus with a single double pole as its only infinite critical point.

From now on, for an infinite GMN differential $\varphi$, so long as it is not of polar type $\{-2\}$ with more than one saddle trajectory, we write $\sigma_{\varphi}$ for the corresponding stability condition on the associated 3-Calabi--Yau triangulated category~$\mathcal{D}$.

\begin{remark}
Note that the phase $\qdph{\gamma}$ of a finite-length connection $\gamma$, in the sense of Section~\ref{sect:qd:foliation+metric}, coincides with the phase $\scph{M}$ of objects $M \in \mathcal{H}$ with $\dimu M = \hhc{\gamma}$ under the isomorphism $\hhspan{\varphi} \cong K_0(\mathcal{H})$, in the sense of Section~\ref{sect:3cy_stab:stab}.
\end{remark}

\section{Donaldson--Thomas theory}\label{sect:dt}

DT invariants count objects in 3-Calabi--Yau categories.
Originally introduced by Thomas \cite{thomas} for coherent sheaves on CY $3$-folds, they were later described by Behrend in terms of certain weighted Euler characteristics \cite{behrend}.
This led to two different generalisations of DT invariants by Joyce and Song \cite{joyce_ci,joyce_cii,joyce_ciii,joyce_civ,js}, and Kontsevich and Soibelman \cite{ks_coha,ks_stability,ks_sum}.

Interest in counting objects in 3-Calabi--Yau categories comes from string theory, where the relevant category is the derived category of coherent sheaves on a CY 3-fold.
A tractable source of 3-Calabi--Yau categories is provided by the derived categories of Ginzburg DG algebras of quivers with potential.
Moreover, it has been shown that categories of coherent sheaves on 3-Calabi--Yau varieties locally look like representations of quivers with potential \cite{toda}, so understanding the case of quivers with potential is key to understanding the CY 3-fold case.

We now explain these ideas in more detail, although for reasons of space we are not able to go into full technicalities.
The principal reference we follow is \cite{dm_curves}, which is in turn based on \cite{bbs,ks_stability}.
However, we only consider refined DT invariants, rather than the more general motivic DT invariants.
This allows us to simplify our presentation in some parts.

\subsection{Grothendieck rings of na\"ive  motives}

We recall some details on Grothendieck rings of na\"ive motives from \cite[Section~3A]{dm_curves}.
This is the conceptual apparatus we need to count objects in our categories.

\subsubsection{Fundamental definitions}\label{sect:dt:motives:fund}

Given an Artin stack $\mathfrak{M}$ which is locally of finite type over $\mathbb{C}$, define $\kzstaffm$ to be the abelian group generated by isomorphism classes of morphisms \[
\mathfrak{X} \xrightarrow{f} \mathfrak{M}
\] of finite type with $\mathfrak{X}$ a separated reduced stack over $\mathbb{C}$, each of whose $\mathbb{C}$-points has an affine stabiliser, subject to the relation that \[
[\mathfrak{X} \xrightarrow{f} \mathfrak{M}] = [\mathfrak{Y} \xrightarrow{f|_{\mathfrak{Y}}} \mathfrak{M}] + [\mathfrak{X} \setminus \mathfrak{Y} \xrightarrow{f|_{\mathfrak{X} \setminus \mathfrak{Y}}} \mathfrak{M}],
\]
for closed substacks $\mathfrak{Y} \subset \mathfrak{X}$.
Elements of $\kzstaffm$ are called (\emph{na\"ive}) \emph{motives}.
In the case $\mathfrak{M} \simeq \spec{\mathbb{C}}=:{\mathrm{pt}}$, they are called \emph{absolute motives}, and are otherwise called \emph{relative motives}.
In the absolute motive case, we commonly abbreviate $[\mathfrak{X}] := [\mathfrak{X} \to \spec{\mathbb{C}}]$.

We get an algebra structure if $(\mathfrak{M}, \epsilon, 0)$ is a commutative monoid in the category of Artin stacks over~$\mathbb{C}$, where \[
\epsilon \colon \mathfrak{M} \times \mathfrak{M} \to \mathfrak{M} \quad \text{ and } \quad 0 \colon \spec{\mathbb{C}} \to \mathfrak{M}
\]
are the respective multiplication and unit maps, with $\epsilon$ of finite type.
Here we have abbreviated the fibre product $\times_{\spec{\mathbb{C}}}$ to $\times$, and we will continue to do this throughout.
This gives us that $\kzstaffm$ is a commutative $\kzstaffa{\mathrm{pt}}$-algebra with product
\begin{equation}\label{eq:motive_product}
[\mathfrak{X}_1 \xrightarrow{f_1} \mathfrak{M}] \cdot [\mathfrak{X}_2 \xrightarrow{f_2} \mathfrak{M}] = [\mathfrak{X}_1 \times \mathfrak{X}_2 \xrightarrow{f_1 \times f_2} \mathfrak{M} \times \mathfrak{M} \xrightarrow{\epsilon} \mathfrak{M}],
\end{equation}
and $\kzstaffa{\mathrm{pt}}$-action given by the inclusion \[
[\mathfrak{X} \xrightarrow{f} \spec{\mathbb{C}}] \mapsto [\mathfrak{X} \xrightarrow{0 \circ f} \mathfrak{M}].
\]

Given an algebraic group $G$, there are $G$-equivariant versions $\kzgstaffm$ of the above group and algebra for $G$-equivariant stacks $\mathfrak{M}$, whose elements are (relative or absolute) \emph{equivariant motives}.
In this case, the morphisms $\mathfrak{X} \xrightarrow{f} \mathfrak{M}$ must be $G$-equivariant, and we require that every point in $\mathfrak{X}$ must lie in a $G$-equivariant affine neighbourhood.
We again have that $\kzgstaffm$ is a $\kzstaffa{\mathrm{pt}}$-algebra if $\mathfrak{M}$ is a monoid in the category of locally finite-type $G$-equivariant Artin stacks with finite-type monoid map.
From now on we will assume that $\mathfrak{M}$ is a commutative monoid in this way.
Additionally, we define $\kgstaffm$ to be the result of making the following modifications to $\kzgstaffm$ from \cite[Section~3A]{dm_curves}.
\begin{enumerate}
    \item If $\mathfrak{X}' \xrightarrow{\pi'} \mathfrak{X}$ is a $G$-equivariant vector bundle of rank $r$ then we impose on $\kgstaffm$ the extra relation that \[
    [\mathfrak{X}' \xrightarrow{f \circ \pi} \mathfrak{M}] = \mathbb{L}^{r} \cdot [\mathfrak{X} \xrightarrow{f} \mathfrak{M}]
    \] where $\mathbb{L} = [\mathbb{A}^{1}_{\mathbb{C}}]$ is the class of the affine line $\mathbb{A}^{1}_{\mathbb{C}}$ in $\kzstaffa{\mathrm{pt}}$, recalling that our assumption that $\mathfrak{M}$ is a monoid gives that $\kgstaffm$ is a $\kzstaffa{\mathrm{pt}}$-algebra.
    \item We complete with respect to the topology having as closed neighbourhoods of $0 \in \kzgstaffm$ the subgroups of $\kzgstaffm$ \[
    \{[\mathfrak{X} \to \mathfrak{M}] \in \kzgstaffm \st [\mathfrak{X} \times_{\mathfrak{M}} \mathfrak{N} \to \mathfrak{N}] = 0 \text{ in } \kzgstaffa{\mathfrak{N}}\}
    \]
    for open substacks $\mathfrak{N} \subset \mathfrak{M}$.
    This condition is required when $\mathfrak{M}$ is not of finite type.
\end{enumerate}

One can pull back and push forward motives along certain morphisms of stacks.
Indeed, if $h \colon \mathfrak{M} \to \mathfrak{N}$ is a morphism of locally finite-type Artin stacks over $\mathbb{C}$ one defines the \emph{pull-back}
\begin{align*}
h^{\ast} \colon \kgstaffa{\mathfrak{N}} &\to \kgstaffm \\
[\mathfrak{X} \to \mathfrak{N}] &\mapsto [\mathfrak{X} \times_{\mathfrak{N}} \mathfrak{M} \to \mathfrak{M}].
\end{align*}
When $h$ is of finite type, one also defines the \emph{push-forward}
\begin{align*}
h_{\ast} \colon \kgstaffm &\to \kgstaffa{\mathfrak{N}} \\
[\mathfrak{X} \xrightarrow{f} \mathfrak{M}] &\mapsto [\mathfrak{X} \xrightarrow{h \circ f} \mathfrak{N}].
\end{align*}
Given a morphism $g \colon \mathfrak{N} \rightarrow \mathfrak{M}$ of finite type, we also write $\int_{g} := h_{\ast} \circ g^{\ast}$, where $h \colon \mathfrak{N} \to \spec{\mathbb{C}}$ is the structure morphism rendering $\mathfrak{N}$ a stack over~$\mathbb{C}$.

We let $\kgvarm$ be the subgroup of $\kgstaffm$ spanned by classes $[\ag{X} \to \mathfrak{M}]$ where $\ag{X}$ is a $G$-equivariant variety over~$\mathbb{C}$.
By a \emph{variety} over $\mathbb{C}$, we mean a reduced separated scheme of finite type.
We have by \cite[Theorem~1.2]{ekedahl} that the natural map
\begin{equation}\label{eq:stack_ring_desc}
    \kgvarm[[\glnc]^{-1} \st n \in \mathbb{Z}_{> 0}] \to \kgstaffm
\end{equation}
is an isomorphism; see also \cite{bd}, \cite[Lemma~3.8, Lemma~3.9]{bridgeland_intro}, \cite[Theorem~4.10]{joyce_mot}, and \cite[Theorem~3.10]{toen}.
Here $[\glnc]^{-1}$ is sent to the motive of the stack-theoretic quotient of a point by $\glnc$.
This is contained in $\kgstaffm$ via the monoid unit of $\mathfrak{M}$.
The motive $[\glnc] \in \kvarc$ can be calculated in terms of $\mathbb{L}$ as \[
[\glnc] = \prod_{i = 0}^{n - 1} (\mathbb{L}^n - \mathbb{L}^{i});
\] 
for instance, see \cite[Lemma~2.6]{bridgeland_intro}.

\subsubsection{Equivariant motives with respect to roots of unity}

We will be particularly interested in the case where the group $G$ is equal to $\rou$, the group of roots of unity.
The group $\rou$ is defined by taking the limit of the groups $\roua{n}$ of $n$-th roots of unity under the surjections $\roua{mn} \twoheadrightarrow \roua{n}$, $z \mapsto z^{m}$.
These surjections induce injections $\kastaffm{\roua{n}} \hookrightarrow \kastaffm{\roua{mn}}$, and we define $\kmstaffm$ to be the colimit of this system of injections.
We then define $\kmstaffm[\mathbb{L}^{1/2}]$ to be the result of adding a formal square root of~$\mathbb{L}$.
Note that we then also have an inverse square root $\mathbb{L}^{-1/2}$, since $\mathbb{L}$ is invertible in $\kstaffc$; see \cite[Lemma~3.8]{bridgeland_intro}.

The group $\kmvarc[\mathbb{L}^{1/2}]$ carries a product which differs from the usual product $\cdot$ defined by \eqref{eq:motive_product}.
This is Looijenga's `exotic' convolution product $\exprod$ defined in \cite{looijenga}, and generalised in \cite{glm}.
The exotic product $\exprod$ nevertheless coincides with the direct product $\cdot$ from \eqref{eq:motive_product} on the subring $\kvarc \subset \kmvarc$ where the action of $\rou$ is trivial.

In this paper we will need equivariant motives in order to access the machinery producing DT invariants, but we will simplify the output of this machinery so that the invariants are Laurent polynomials in $q^{1/2}$, rather than equivariant motives.
These are known as `refined' DT invariants, rather than `motivic' DT invariants.
In order to do this, we use the following fact.

\begin{proposition}[{\cite[p.69, p.101]{ks_stability} or \cite[p.1244]{dm_curves}}]
There is a ring homomorphism \[
\tsp \colon \kmstaffc[\mathbb{L}^{1/2}] \to  \mathbb{Z}[q^{\pm 1/2}, (1 - q^n)^{-1} \st n \in \mathbb{Z}_{>0}] =:\mathscr{R}.
\]
where the subring $\kmvarc[\mathbb{L}^{1/2}]$ carries the exotic product $\exprod$, which is extended to $\kmstaffc[\mathbb{L}^{1/2}]$ using the isomorphism~\eqref{eq:stack_ring_desc}.
\end{proposition}

This maps $\mathbb{L}^{1/2} \mapsto q^{1/2}$ and maps the motive of a variety $\ag{X}$ with trivial $\rou$-action to its Serre polynomial, and then applies the map $q^{1/2} \mapsto -q^{1/2}$; see \cite[Section~2.2]{bbs} for an explicit construction.
In \cite{haiden,ks_stability}, this is referred to as the \emph{twisted} Serre polynomial.
Note that the Serre polynomial is also known as the `weight polynomial' or `virtual Poincar\'e polynomial' and is obtained by specialising the Hodge polynomial.
When $\ag{X}$ is smooth projective, it coincides with the Poincar\'e polynomial of~$\ag{X}$.

The framework of Kontsevich and Soibelman \cite[Section~4.5]{ks_stability} requires us to impose an extra relation on our Grothendieck rings of motives.
To this end, we define $\kbgstaffm$ to be the quotient of $\kgstaffm$ described in \cite[p.1268]{dm_curves}.
It is known that the kernel of the quotient map $\kmstaffc \to \kbmstaffc$ also lies in the kernel of the map $\tsp \colon \kmstaffc \to \mathscr{R}$ \cite[p.1268]{dm_curves}, \cite[Section~4.5]{ks_stability}.
Hence, we may also extend the map $\tsp$ to $\tsp \colon \kbmstaffc[\mathbb{L}^{1/2}] \to \mathscr{R}$.

\subsection{Motivic Hall algebras}

We now review the definition of the motivic Hall algebras for stacks of finite-dimensional modules over Jacobian algebras.

\subsubsection{Stacks of quiver representations}\label{sect:dt:quiv_rep_stacks}

Given $\ud \in \mathbb{N}^{Q_0}$, we define the variety of representations of $Q$ of dimension vector $\ud$ as \[
\qdsch := \prod_{i \to j \in Q_1} \Hom_{\mathbb{C}}(\mathbb{C}^{d_i}, \mathbb{C}^{d_j}).
\]
We then define $\qdschn$ and $\qwdsch$ to be the closed subvarieties consisting respectively of nilpotent representations and representations satisfying the relations of the Jacobian ideal $J(W)$.
We furthermore define $\qwdschn$ to be the intersection of these two closed subvarieties, giving the closed subvariety of representations of $\jac{Q, W}$.
This is the variety we are principally interested in.

We let $\gla{\ud}$ be the group $\gla{\ud} := \prod_{i \in Q_0} \gla{d_i}$, which acts on the variety $\qdsch$ and its above subvarieties by change of basis at each $\mathbb{C}^{d_i}$.
Namely, given $g = (g_i)_{i \in Q_0} \in \gla{\ud}$ and $f = (f_a)_{a \colon i \to j}$ where $f_a \in \Hom_{\mathbb{C}}(\mathbb{C}^{d_i}, \mathbb{C}^{d_j})$, we have \[
g \cdot f := (g_j \circ f_a \circ g_i^{-1})_{a \colon i \to j}.
\]
We then define the moduli stack \[
\qwdstn := \qwdschn/\gla{\ud}
\]
using the stack-theoretic quotient.
We finally take the disjoint union \[
\qwstn := \coprod_{\ud \in \mathbb{N}^{Q_0}} \qwdstn.
\]
This stack is a commutative monoid in the category of Artin stacks over $\mathbb{C}$ with the operation given by $\oplus$ and the identity given by the zero representation; see \cite[Lemma~2.1]{mr}.
We similarly define $\qdst$, $\qdstn$, $\qwdst$, $\qst$, $\qstn$, and $\qwst$.

\begin{remark}
Note that by \cite[Lemme~9.1.0.4]{klh_thesis} and the fact that $\mathcal{H}_\infty(Q,W)$ is strictly unital, we have that objects in $\mathcal{H}_\infty(Q,W)$ are `$A_{\infty}$-isomorphic' if and only if they are isomorphic in $\mathcal{H}(Q,W)=H^{0}(\mathcal{H}_\infty(Q,W)) \simeq \modules \jac{Q, W}$.
Hence, following Theorem~\ref{thm:koszul_duality_heart}, we can also view $\qwstn$ as the moduli stack of objects in $\mathcal{H}_\infty(Q,W)$.
\end{remark}

For each of the module categories $\modules \jacu{Q, 0}$, $\modules \jac{Q, 0}$, $\modules \jacu{Q, W}$, and $\modules \jac{Q, W}$ there is a natural map $\cl = \dimu$ from the respective Grothendieck groups to $\fgg = \mathbb{Z}^{Q_0}$.
This is an isomorphism for the completed algebras, but may not be an isomorphism for the uncompleted algebras.
Given a stability condition $Z \colon \fgg \to \mathbb{C}$, we have an open subvariety $\sqdsch$ of $\qdsch$ consisting of semistable representations of dimension vector $\ud$, and analogous open subvarieties $\sqdschn$, $\sqwdsch$, and $\sqwdschn$; see \cite[Section~3.2]{meinhardt_intro}.
As before, we quotient by $\gla{\ud}$ to obtain the stacks $\sqdst$, $\sqdstn$, $\sqwdst$, and $\sqwdstn$ of semistable representations of dimension vector $\ud$.
We then take the disjoint union over dimension vectors of a constant phase $\vartheta$, namely \[
\sqwpstn := \coprod_{\substack{\scph{\ud} = \vartheta \\ \text{or } \ud = 0}} \sqwdstn
\]
to obtain the stack of semistable $\jac{Q, W}$-modules of phase $\vartheta$, and likewise define $\sqpst$, $\sqpstn$, and $\sqwpst$.

\begin{remark}
In this paper, $\modules \jac{Q, W}$ will be a particular heart of $\mathcal{D}$ given by a stability condition~$\sigma_{\varphi}$.
In some cases, namely the non-degenerate ring domain cases of Section~\ref{sect:main:nrd}, the heart given by the stability condition $\sigma_{\varphi}$ will not be given by a quiver with potential, so the theory here will not strictly apply \textit{per~se}.
In such cases, we will really be applying the theory in the nearby heart we obtain by rotation which does come from a quiver with potential.
Since the categories of semistable objects do not change under rotation, but only change their phase, this does not affect the values of the invariants.
\end{remark}

\subsubsection{Hall product}

One can define products on the abelian groups $\kstaffq$, $\kstaffqn$, $\kstaffqw$ and $\kstaffqwn$ giving the \emph{motivic Hall algebras} due to Joyce \cite{joyce_cii}, as we presently explain.
These abelian groups are moreover modules over $\kstaffc$ since the base stacks are monoids, recalling Section~\ref{sect:dt:motives:fund}.
Elements of these groups can roughly be thought of as continuous families of representations.
An introduction to motivic Hall algebras can be found in \cite{bridgeland_intro}, with \cite{bridge_scat} also being a good reference in the case of quivers with potential.

We give the definition of the motivic Hall algebra for $\kstaffqwn$; the other cases are defined similarly.
The product $\star$ on the Hall algebra will be linear, so it suffices to define it for classes $[\mathfrak{X}_{1} \xrightarrow{f_1} \qwstn]$ and $[\mathfrak{X}_2 \xrightarrow{f_2} \qwstn]$ which factor respectively as $[\mathfrak{X}_1 \to \qwastn{\ud^{1}} \hookrightarrow \qwstn]$ and $[\mathfrak{X}_2 \to \qwastn{\ud^{2}} \hookrightarrow \qwstn]$.
We write $\qwastn{\ud^{1}, \ud^{2}}$ for the stack of short exact sequences
\[0 \to L \to M \to N \to 0\]
in $\modules \jac{Q, W}$ such that $\dimu L = \ud^{1}$ and $\dimu N = \ud^{2}$; see \cite[Section~3B]{dm_curves} for a precise construction.
There are natural morphisms of stacks
\begin{align*}
    &\pi_1 \colon \qwastn{\ud^{1}, \ud^{2}} \to \qwastn{\ud^{1}}, \\
    &\pi_2 \colon \qwastn{\ud^{1}, \ud^{2}} \to \qwastn{\ud^{1} + \ud^{2}}, \\
    &\pi_3 \colon \qwastn{\ud^{1}, \ud^{2}} \to \qwastn{\ud^{2}},
\end{align*}
given by taking the first, second, and third terms of the short exact sequence, respectively.
We then define the product
\[[\mathfrak{X}_1 \xrightarrow{f_{1}} \qwstn] \star [\mathfrak{X}_2 \xrightarrow{f_{2}} \qwstn]\]
in the motivic Hall algebra $\kstaffqwn$ to be the composition given by the top row of the commutative diagram \[
\begin{tikzcd}
    \mathfrak{X}_{1,2} \ar[r] \ar[d] \arrow[dr, phantom, "\lrcorner", very near start] & \qwastn{\ud^{1}, \ud^{2}} \ar[r,"\pi_{2}"] \ar[d,"{\pi_{1} \times \pi_{3}}"] & \qwastn{\ud^{1} + \ud^{2}} \ar[r,hookrightarrow] & \qwstn \\
    \mathfrak{X}_1 \times \mathfrak{X}_2 \ar[r,"{f_{1} \times f_{2}}"] & \qwastn{\ud^{1}} \times \qwastn{\ud^{2}}, 
\end{tikzcd}
\]
where the square is Cartesian, as indicated.
We obtain that the motivic Hall algebra is a $\kstaffc$-algebra.
The intuition for this definition is that $\mathfrak{X}_{1,2}$ is the family of representations which are extensions of representations in $\mathfrak{X}_1$ by representations in $\mathfrak{X}_2$, so that multiplication in the motivic Hall algebra corresponds to taking extensions.

\subsection{Integration maps}\label{sect:dt:meat}

Integration maps are algebra homomorphisms from the motivic Hall algebras which count representations.
As in \cite[Section~4B]{dm_curves}, we consider two different integration maps, one due to \cite{ks_stability} and the other based on \cite{bbs}.
These will give the same output, modulo technicalities.

\subsubsection{Motivic vanishing cycles}\label{sect:dt:mvc}

In order to count representations of quivers with non-trivial potential correctly when we apply the integration map, we need the technology of motivic vanishing cycles.
The role of the motivic vanishing cycle in DT theory is to enable one to count representations with multiplicity, as required by the Behrend function \cite{behrend}, accounting for the fact that the partial derivatives of the potential may vanish to order higher than one.
Motivic vanishing cycles were introduced by Denef and Loeser in \cite{dl}, with \cite{looijenga} also being a standard reference.
We leave the details of their definition to \cite[Section~3C]{dm_curves}, or to either of these papers.

Let $f \colon \ag{M} \to \mathbb{A}_{\mathbb{C}}^{1}$ be a regular map on a smooth scheme $\ag{M}$ and denote $\ag{M}_0 := f^{-1}(0)$.
The \emph{motivic vanishing cycle} is then a certain $\rou$-equivariant motive $\vcf \in \kmvara{\ag{M}_0}[\mathbb{L}^{-1/2}]$.
We let $\crloc(f) \subset \ag{M}$ be the critical locus of $f$, the locus where the partial derivatives of $f$ all vanish.
We assume that $\crloc(f) \subseteq \ag{M}_{0}$, as will be the case in the examples we consider.
We have that $\vcf$ is \emph{supported} on $\crloc(f)$, meaning that it is in the image of the map $\iota_{\ast} \colon \kmvara{\crloc(f)} \to \kmvara{\ag{M}_0}$ where 
$\iota \colon \crloc(f) \hookrightarrow \ag{M}_0$ is the inclusion \cite[Corollary~5.27]{meinhardt_intro}, \cite[Proposition~2.10]{dm_loop}.

One extends the definition of motivic vanishing cycles to quotient stacks in the following way.
Suppose that we have an Artin stack $\mathfrak{M}$ that is a quotient $\ag{M}/\glnc$ of a smooth connected scheme~$\ag{M}$, with $\pi \colon \ag{M} \to \mathfrak{M}$ the quotient map.
If $f \colon \mathfrak{M} \to \mathbb{A}_{\mathbb{C}}^{1}$ is a morphism, then we define \[
\vcf := \mathbb{L}^{n^2/2} \cdot [\glnc]^{-1} \cdot (\pi \circ \zeta)_{\ast} \vca{f \circ \pi} \in \kmstaffm[\mathbb{L}^{1/2}],\] where $\zeta \colon \ag{M}_0 := (f \circ \pi)^{-1}(0) \hookrightarrow \ag{M}$ is the inclusion.

We will also need a relative version of the motivic vanishing cycle.
To explain this, let $\ag{M}$ be a finite type scheme with a \emph{constructible vector bundle} $\pi \colon \ag{V} \to \ag{M}$, meaning that there is a finite decomposition of $\ag{M}$ into locally closed subschemes $\ag{M} = \coprod \ag{M}_i$ with a vector bundle $\ag{V}_i$ on each of the $\ag{M}_i$, not necessarily of the same rank.
A function on the total space $\ag{V}$ of the constructible vector bundle is given by a function on each of the $\ag{V}_i$, possibly after further decomposition.
We then let $f \colon \ag{V} \to \mathbb{A}_{\mathbb{C}}^{1}$ be a function vanishing on the zero section.
There is then a \emph{relative motivic vanishing cycle} $\vcfrel \in \kbmvara{\ag{M}}[\mathbb{L}^{-1/2}]$, for whose definition see \cite[Section~3C]{dm_curves}.
Note, however, that we define this to be an element of $\kbmvara{\ag{M}}[\mathbb{L}^{-1/2}]$, rather than $\kmvara{\ag{M}}[\mathbb{L}^{-1/2}]$.
One can extend the notion of constructible vector bundles to stacks similarly to how one may extend the notion of a vector bundle to a stack, for instance \cite[Definition~2.50]{gomez}.
One can then extend relative motivic vanishing cycles to stacks similarly to the non-relative case.

\subsubsection{The minimal potential}

One of the ingredients of the integration map of Kontsevich and Soibelman \cite{ks_stability} is a potential function constructed from the $A_\infty$-structure on $\mathcal{H}_\infty(Q,W)$, related to that from~\eqref{eq:cat_potential}.
To explain this, we need the following notions and results, where we are again still following \cite{dm_curves}.

\begin{definition}[{\cite[Definition~4.3]{dm_curves}}]
A \emph{constructible $3$-Calabi--Yau vector bundle} on a scheme $\ag{X}$ is a constructible $\mathbb{Z}$-graded vector bundle $\ag{V}^{\bullet}$ along with morphisms $m_{n} \colon (\ag{V}^{\bullet})^{\otimes n} \to \ag{V}^{\bullet}[2 - n]$ and a morphism $\cy{-}{-} \colon \ag{V}^{\bullet} \otimes \ag{V}^{\bullet} \to \mathbbm{1}_{\ag{X}}[-3]$ satisfying the same conditions as a 3-Calabi--Yau $A_{\infty}$-category.
Here $\mathbbm{1}_{\ag{X}}$ is the trivial line bundle on~$\ag{X}$.

A \emph{morphism $f \colon \ag{V}^{\bullet} \to {\ag{V}'}^{\bullet}$ of constructible $3$-Calabi--Yau vector bundles} is a countable collection of morphisms of constructible vector bundles $f_n \colon (\ag{V}^{\bullet})^{\otimes n} \to {\ag{V}'}^{\bullet}$ satisfying the natural compatibility conditions with respect to the $A_\infty$-operations $m_{n}$ and form $\cy{-}{-}$, along with the additional condition that $\sum_{k + l = n} \cy{-}{-}_{{V'}^{\bullet}} \circ f_{k} \otimes f_{l} = 0$ for all $n \geqslant 3$, just as in Section~\ref{sect:3cy_stab:3cy_a_inf}.

This notion can then be extended to constructible 3-Calabi--Yau vector bundles on stacks.
\end{definition}

The following proposition splits the bundle of $\End^{\bullet}$ algebras over $\mathcal{H}_\infty(Q,W)$ into a minimal part, which has no differential, and a remainder term.

\begin{proposition}[{\cite[Theorem~4.1.13]{davison_thesis} based on \cite[Theorem~5.15]{kajiura}, see also \cite[Proposition~4.4]{dm_curves} and \cite[Remark~6.3.3]{davison_thesis}}]\label{prop:min_model}
There is a locally constructible formal isomorphism of constructible 3-Calabi--Yau vector bundles \[
(\End^{\bullet}_{\mathcal{D}_\infty(Q,W)}, m^{\two}) \cong (\Ext_{\mathcal{D}_\infty(Q,W)}^{\bullet} \oplus \mathfrak{V}^{\bullet}, m')
\]
on the stack $\qwstn$ such that $m'_{1}$ factors via a map $\mathfrak{V}^{\bullet} \to \mathfrak{V}^{\bullet}$, with $\mathfrak{V}^{\bullet}$ an acyclic complex, and such that, for $i \geqslant 2$, $m'_i$ factors via a map \[
(\Ext_{\mathcal{D}_\infty(Q,W)}^{\bullet})^{\otimes i} \to \Ext^{\bullet}_{\mathcal{D}_\infty(Q,W)}.
\]
This splitting is unique up to isomorphisms of constructible 3-Calabi--Yau vector bundles.
We thus have that $(\Ext_{\mathcal{D}_\infty(Q,W)}^{\bullet}, m')$ gives a cyclic minimal model for $(\End^{\bullet}_{\mathcal{D}_\infty(Q,W)}, m^{\two})$.
\end{proposition}

The \emph{minimal potential} $\wmin$ is then given as in the formula \eqref{eq:cat_potential} using the cyclic minimal model of Proposition~\ref{prop:min_model}, see \cite[p.1267]{dm_curves} for more detail.
It is shown in \cite[Proposition~4.5]{dm_curves} that $\vcrela{\wmin}$ is well-defined and does not depend upon the choice of cyclic minimal model.

\subsubsection{Orientation data}\label{sect:dt:orient_data}

Another ingredient of the integration map of \cite{ks_stability} is `orientation data', which provides a correction factor ensuring that the map is in fact an algebra homomorphism.
Orientation data gives a family of pairs $\{(\mathfrak{L}_{\ud}, \tau_{\ud})\}_{\ud \in \mathbb{N}^{Q_0}}$ consisting of a constructible super line bundle $\mathfrak{L}_{\ud}$ on $\qwdstn$, where \emph{super} means $\mathbb{Z}_2$-graded, and a trivialisation of its square \[
\tau_{\ud} \colon \mathfrak{L}_{\ud}^{\otimes 2} \xrightarrow{\sim} \mathbbm{1}_{\qwdstn},
\]
where $\mathbbm{1}_{\qwdstn}$ is the trivial super line bundle on $\qwdstn$, which has even parity everywhere. 
This family is required to satisfy a certain cocycle condition \cite[Section~5.2]{ks_stability}, \cite[Condition~6.3.1]{davison_thesis}.
An isomorphism of orientation data is an isomorphism of constructible super line bundles commuting with the chosen trivialisations of the squares.
In \cite[Section~3.1]{ks_stability}, the framework of ``ind-constructible'' vector bundles is used instead of families of vector bundles.

There is a canonical choice of orientation data for a quiver with potential given in \cite[Section~7.1]{davison_thesis}.
Indeed, given a twisted complex $(M, a)$ of $\mathcal{H}_\infty(Q,W)$, there is a non-degenerate quadratic form $\cy{m_{1}^{\two}(-)}{-}$ on $\End_{\mathcal{D}_\infty(Q,W)}^{1}((M, a))$.
Fixing $M$ and varying over all $a$ that satisfy the Maurer--Cartan equation \eqref{eq:mc}, one obtains a constructible super vector bundle in odd degree
\begin{equation}\label{eq:orient_data:vec_bund}
\End_{\mathcal{D}_\infty(Q,W)}^{1}(M)/\ker(m_{1}^{\two})
\end{equation}
on $\qaaschn{Q, W}{\ud}$ with a non-degenerate quadratic form $\odqf$ given by $\cy{m_{1}^{\two}(-)}{-}$, where $\ud$ is the dimension vector of the $\jac{Q, W}$-module corresponding to~$(M, a)$.
This then descends to the stack $\qaastn{Q, W}{\ud}$.
Note that \eqref{eq:orient_data:vec_bund} depends upon~$a$, since the $A_\infty$-operation $m_{1}^{\two}$ depends upon~$a$.
Consequently, the dimension of \eqref{eq:orient_data:vec_bund} may jump, which is why it is only a \textit{constructible} super vector bundle.

We now explain how to turn a constructible super vector bundle with quadratic form into a constructible super line bundle with a trivialisation of its square, so that it provides us with orientation data.
Given a constructible super vector bundle $\mathfrak{V}$ on a stack $\mathfrak{M}$, its \emph{superdeterminant} is the super line bundle defined by \[
\sdet{\mathfrak{V}} := \cop^{\dim(\mathfrak{V})} \biggl( \bigwedge^{\mathrm{top}} \mathfrak{V}_{\mathrm{even}} \otimes \bigwedge^{\mathrm{top}} \mathfrak{V}_{\mathrm{odd}}^{\ast} \biggr),
\]
where $\cop$ is the change-of-parity functor on super line bundles and $\bigwedge^{\mathrm{top}}$ denotes the top exterior power of $\mathfrak{V}_{\mathrm{even}}$ and $\mathfrak{V}_{\mathrm{odd}}$, which may differ between the different pieces of $\mathfrak{M}$ in the constructible decomposition.
If $\mathfrak{V}$ has a non-degenerate quadratic form $\odqf_{\mathfrak{V}}$, then this gives us an isomorphism $\mathfrak{V} \cong \mathfrak{V}^{\ast}$.
This then gives an isomorphism $\sdet{\mathfrak{V}} \cong \sdet{\mathfrak{V}^\ast} \cong \sdet{\mathfrak{V}}^{\ast}$, and so a trivialisation $\sdet{\mathfrak{V}}^{\otimes 2} \xrightarrow{\sim}  \mathbbm{1}_{\mathfrak{M}}$.
Different families of constructible super vector bundles with quadratic forms giving rise to isomorphic orientation data will give the same integration map, see \cite[p.1265]{dm_curves}.

The constructible super line bundle on $\qwdstn$ is then given by the superdeterminant of the constructible vector bundle \eqref{eq:orient_data:vec_bund} and the trivialisation of its square provided by the non-degenerate inner product $\cy{m_{1}^{\two}(-)}{-}$.
We denote this choice of orientation data by $\odchoice$ and refer to this as the \emph{canonical choice of orientation data} for the quiver with potential $(Q, W)$. We will need to make calculations concerning orientation data later, so we collect the following useful facts.

\begin{proposition}\label{prop:od_ug}
The following statements hold.
\begin{enumerate}[label=\textup{(}\arabic*\textup{)}]
    \item Orientation data $\{(\mathfrak{L}_{\ud}, \tau_{\ud})\}_{\ud \in \mathbb{N}^{Q_0}}$ is determined by the parity of $\mathfrak{L}_{\dimu S_i}$ for $i \in Q_0$.\label{op:od_ug:simples_determine}
    \item For the canonical choice of orientation data $\odchoice$, $\mathfrak{L}_{\dimu S_i}$ has even parity for all simple modules $S_i$ over $\jac{Q(T), W(T)}$, where $T$ is a triangulation of a marked bordered surface with orbifold points of order three.\label{op:od_ug:simples->even}
    \item Given a twisted complex $M \in \mathcal{H}_\infty(Q,W)$, the parity of the dimension of \[
    \End^{1}_{\mathcal{D}_\infty(Q,W)}(M)/\ker(m_{1}^{\two}),
    \]
    the vector space which gives the canonical choice of orientation data $\odchoice$, is equal to the parity of \[
    {\rm dim}\Ext_{\mathcal{D}_\infty(Q,W)}^{\leqslant 1}(M, M) + {\rm dim}\Hom_{\mathcal{D}_\infty(Q,W)}^{\leqslant 1}(M, M).
    \]\label{op:od_ug:calc}
\end{enumerate}
\end{proposition}
\begin{proof}
\ref{op:od_ug:simples_determine} is \cite[Theorem~8.1]{davison_thesis}.
\ref{op:od_ug:simples->even} follows from the fact that $\End_{\mathcal{D}_\infty(Q,W)}^{1}(S)/\ker(m_{1}^{\two}) = 0$ for the simple objects $S$ in the category $\mathcal{H}_{\infty}(Q(T), W(T))$.
This can be verified explicitly from the definition of $\akdqw$, given that the terms of $W(T)$ are all cubic, and so $m_{1}^{\two} = 0$ on $\End_{\mathcal{D}_\infty(Q,W)}^{1}(S)$.
Finally, \ref{op:od_ug:calc} follows from basic linear algebra, but can also be found in \cite[Remark~7.4.1]{davison_thesis} or \cite[Proof of Proposition~5.2]{dm_curves}.
\end{proof}

\subsubsection{Two integration maps}\label{sect:dt:integration_maps}

Before defining the integration maps themselves, we define the rings which are the codomains of the integration maps.
We have two pairings on $\fgg = \mathbb{Z}^{Q_{0}}$ defined by \[
\naef{\ud}{\ud'} := \sum_{i \in Q_{0}}d_{i}d'_{i} - \sum_{\substack{a \in Q_{1} \\ a\colon i \to j}} d_{i}d'_{j}
\]
and \[
\ef{\ud}{\ud'} := \naef{\ud}{\ud'} - \naef{\ud'}{\ud}.
\]
We let $\kbmstaffc[\mathbb{L}^{1/2}]\llbracket{\bm t}\rrbracket=\kbmstaffc[\mathbb{L}^{1/2}]\llbracket t^{\ud} \st \ud \in \mathbb{N}^{Q_0}\rrbracket$ be the formal power series ring in the variables $t^{\ud}$ with coefficients in $\kbmstaffc[\mathbb{L}^{1/2}]$ and multiplication between these generators defined $t^{\ud}t^{\ud'} = \mathbb{L}^{\frac{1}{2}\ef{\ud}{\ud'}}t^{\ud + \ud'}$; see, for example, \cite[Definition~1.11]{mmns}, although note that we use the opposite sign convention.
The map $\tsp$ then naturally extends to a map \[
\tsp_{Q} \colon \kbmstaffc[\mathbb{L}^{1/2}]\llbracket{\bm t}\rrbracket \to \qring,
\]
where we introduce the shorthand 
\begin{equation}\label{eq:tq}
\qring := \mathscr{R}\llbracket{\bm t}\rrbracket= \mathscr{R}\llbracket t^{\ud} \st \ud \in \mathbb{N}^{Q_0}\rrbracket,
\end{equation}
again with multiplication $t^{\ud}t^{\ud'} = q^{\frac{1}{2}\ef{\ud}{\ud'}}t^{\ud + \ud'}$.
One can similarly define the formal power series ring $\kmstaffc[\mathbb{L}^{1/2}]\llbracket{\bm t}\rrbracket$, which also has a map $\tsp_{Q}$ to~$\qring$.

The integration map of Kontsevich and Soibelman \cite{ks_stability} is then defined as follows.
Here, given regular functions $f \colon \ag{X} \to \mathbb{A}_{\mathbb{C}}^{1}$ and $g \colon \ag{Y} \to \mathbb{A}_{\mathbb{C}}^{1}$ on varieties $\ag{X}$ and $\ag{Y}$, the function $f \mathsmaller{\,\boxplus\,} g \colon \ag{X} \times \ag{Y} \to \mathbb{A}_{\mathbb{C}}^{1}$ is defined by $(f \mathsmaller{\,\boxplus\,} g)(x, y) = f(x) + g(y)$.

\begin{theorem}[{\cite{ks_stability}, \cite[Theorem~4.6]{dm_curves}}]\label{thm:ks_int}
Let $(Q, W)$ be a quiver with polynomial potential.
Moreover, let $\{\mathfrak{V}_{\ud}\}_{\ud \in \mathbb{N}^{Q_0}}$ be a family of constructible super vector bundles on $\qwdstn$ with quadratic forms $\odqf$ which gives rise to the canonical orientation data $\odchoice$.
Then there is a $\kstaffc$-algebra homomorphism \[
\intks \colon \kstaffqwn \to \kbmstaffc[\mathbb{L}^{1/2}]\llbracket{\bm t}\rrbracket
\]
determined by \[
[\ag{X} \xrightarrow{f} \qwdstn] \mapsto \biggl( \int_{f} \vcrela{\wmin \mathsmaller{\boxplus} \odqf} \biggr) t^{\ud}
\]
for $[\ag{X} \xrightarrow{f} \qwdstn] \in \kvarxdn$.
\end{theorem}

There is another integration map which has its roots in \cite{bbs}, whose Hodge-theoretic version occurs in~\cite{ks_coha}.
Given a potential $W = \sum_{i = 1}^r \lambda_i a_1 a_2 \dots a_{\ell_i}$, we define the function 
\begin{align*}
    \tr{W_{\ud}} \colon \qdsch &\to \mathbb{A}_{\mathbb{C}}^{1} \\
    (f_a)_{a \in Q_{1}} &\mapsto \sum_{i = 1}^{r} \lambda_{i} \tr{f_{\ell_i}f_{\ell_i - 1} \dots f_{a_{1}}}.
\end{align*}
Since this is invariant under the $\gld$ action, the function $\tr{W_{\ud}}$ descends to a function $\trst{W_{\ud}}$ on the stack $\qdst$.
It is well-known that $\qwdsch = \crloc(\tr{W_{\ud}})$ \cite[Proposition~3.8]{segal_dtp}.
Moreover, it can be seen that here we have $\qwdsch \subseteq \tr{W_{\ud}}^{-1}(0)$.
Recall then from Section~\ref{sect:dt:mvc} that $\vca{\tr{W_{\ud}}} \in  \kmvara{\tr{W_{\ud}}^{-1}(0)}[\mathbb{L}^{-1/2}]$ is supported on $\qwdsch$.
It follows that $\vca{\trst{W_{\ud}}} \in \kmstaffa{\qdst}[\mathbb{L}^{1/2}]$ is supported on $\qwdst$; denote by $\iota_{\ud} \colon \qwdst \to \qdst$ the inclusion.

\begin{theorem}[{\cite{bbs}, \cite[p.1270]{dm_curves}}]
For a quiver $(Q, W)$ with polynomial potential there is a $\kstaffc$-algebra homomorphism \[
\intbbs \colon \kstaffqw \to \kmstaffc[\mathbb{L}^{1/2}]\llbracket{\bm t}\rrbracket
\]
given by \[
[\ag{X} \xrightarrow{f} \qwdst] \mapsto \biggl(\int_{f} \iota_{\ud}^{\ast}\vca{\trst{W_{\ud}}}\biggr) t^{\ud}
\]
for $[\ag{X} \xrightarrow{f} \qwdst] \in \kvarxd$.
\end{theorem}

Restricted to the stack of nilpotent modules, these two integration maps agree when subject to the quotient map \[\overline{q} \colon
\kmstaffc[\mathbb{L}^{1/2}]\llbracket{\bm t}\rrbracket \to \kbmstaffc[\mathbb{L}^{1/2}]\llbracket{\bm t}\rrbracket.
\]
Note that $\kstaffqwn$ is a subalgebra of $\kstaffqw$ via pushing forward along the inclusion $\qwstn \hookrightarrow \qwst$.

\begin{theorem}[{\cite[Theorem~7.1.3]{davison_thesis}, \cite[Proposition~4.8]{dm_curves}}]
We have that \[
\overline{q} \circ \intbbs|_{\kstaffqwn} = \intks.
\]
\end{theorem}

Consequently, we have that $\tsp_{Q} \circ \intbbs|_{\kstaffqwn} = \tsp_{Q} \circ \intks$.
While we will principally be using the integration map $\intks$, our calculations in Section~\ref{sect:db} will require us to consider DT invariants for categories of non-nilpotent representations, which then requires the integration map~$\intbbs$.
The advantages of using the integration map $\intks$ are that firstly it is useful for calculations, as we show in Section~\ref{sect:db}, and secondly that it is suited to the triangulated setting.
Under change of heart the quiver with potential $(Q, W)$ will change, which affects the map $\intbbs$.
On the other hand, for $\intks$, since the potential $\wmin$ is intrinsic to the $A_\infty$-category, all we will need to do is check that the $A_\infty$-enhancement and choice of orientation data is preserved by change of heart, which we do in Section~\ref{sect:inv}.

\subsection{Donaldson--Thomas invariants}

We now at last explain how one extracts DT invariants from the above apparatus.
We let \[
\gensern{Q, W} = \tsp_{Q}(\intks([\qwstn \xrightarrow{\mathrm{id}} \qwstn])) \in \qring
\]
be the \emph{Donaldson--Thomas generating series} of $\modules \jac{Q, W}$, recalling the definition of $\qring$ from~\eqref{eq:tq}.
Note that the element $[\qwstn \xrightarrow{\mathrm{id}} \qwstn]$ is obtained via the completion from Section~\ref{sect:dt:motives:fund}, since $\qwstn$ is not of finite type.
We similarly obtain the DT generating series $\genser{Q, W}$ of $\modules \jacu{Q, W}$ by \[
\genser{Q, W} = \tsp_{Q}(\intbbs([\qwst \xrightarrow{\mathrm{id}} \qwst])) \in \qring.
\]
These are formal power series in the variables $t^{\ud}$ with coefficients in $\mathscr{R}$.
The best-known example is for $Q = \bullet$, the quiver with one vertex and no arrows, where we have \[
\genser{Q, 0} = \sum_{d = 0}^\infty \frac{(-1)^{d}q^{d/2}}{(1 - q)(1 - q^2) \dots (1 - q^d)} t^d.
\]
This formal power series is known as the \emph{quantum dilogarithm}, and we denote it $\qdl{t}$.
See \cite[Appendix~A]{haiden} for helpful discussion on related series in the literature and differing conventions.

The generating series $\gensern{Q, W}$ counts all modules over the completed Jacobian algebra $\jac{Q, W}$.
We will also need generating series counting semistable modules of a particular phase.
We let $\sigma$ be a stability condition and let $\ssc{\sigma}{\vartheta}$ be the category of semistable $\jac{Q, W}$-modules of phase~$\vartheta$.
To this end, we define \[
\gensernp{\vartheta}{Q, W} = \tsp_{Q}(\intks([\sqwpstn \hookrightarrow \qwstn])) \in \qring.
\]
We refer to this as \emph{the generating series of the category $\ssc{\sigma}{\vartheta}$ of semistable objects of phase~$\vartheta$}, with $\genserp{\vartheta}{Q, W}$ defined similarly.
This of course depends upon the choice of stability condition on $\modules \jac{Q, W}$.
One can similarly define $\genserp{\vartheta}{Q, W}$.

One can define DT invariants from the generating series as follows.

\begin{definition}[{See \cite[Section~5.1]{haiden}}]\label{def:dt}
If $\ef{-}{-}$ vanishes on $\ssc{\sigma}{\vartheta}$, then there is a factorisation \[
\gensernp{\vartheta}{Q, W} = \prod_{\underline{d} \in \dimu\ssc{\sigma}{\vartheta}}\prod_{k \in \mathbb{Z}} \qdl{(-q^{1/2})^k t^{\underline{d}}}^{(-1)^k \dt{k}{}{\underline{d}}}\]
where $\dt{k}{}{\underline{d}}$ are all integers with only finitely many non-zero, by \cite[Theorem~A]{dm_qea}.
Given $\ud$ with $\scph{\ud} = \vartheta$, the \emph{refined Donaldson--Thomas invariant} is then \[\dtref{\underline{d}} := \sum_{k \in \mathbb{Z}}\dt{k}{}{\underline{d}}q^{k/2}.\]
The \emph{numerical Donaldson--Thomas invariant} is then $\dtnum{\ud} := \dtref{\ud}\big|_{q^{1/2} = -1}$.
\end{definition}

One can make a similar definition for the category of semistable $\jacu{Q, W}$-modules of phase~$\vartheta$ using $\genserp{\vartheta}{Q, W}$.
Note that, in particular, $\ef{-}{-}$ vanishes on $\ssc{\sigma}{\vartheta}$ if $\dimu \ssc{\sigma}{\vartheta}$ is generated by a single dimension vector.

\begin{remark}\label{rmk:od_fine}
Given that we work over $\mathbb{C}$, which contains $\sqrt{-1}$, we have that $(1 - [\roua{2}])^2 = \mathbb{L}$, see \cite[End of Section~2.2]{dm_loop} or \cite[Remark~19]{ks_stability}.
Recall that $[\roua{2}]$ is the equivariant absolute motive of the group of second roots of unity, with respect to the natural $\rou$-action.
It should be a consequence of \cite[Example~8.1.1, Theorem~8.3.1]{davison_thesis}, that if one therefore imposes the relation $\mathbb{L}^{1/2} = (1 - [\roua{2}])$, then the choice of orientation data does not affect the motivic DT invariants.
For examples of how different choices of orientation data can affect the motivic invariants if one does not impose this relation, see \cite[Example~8.1.1]{davison_thesis} or \cite[Remark~5.8]{dm_curves}.

In our case, the map $\tsp$ sends both $\mathbb{L}^{1/2}$ and $(1 - [\roua{2}])$ to $q^{1/2}$; see \cite[Proof of Proposition~5.8]{haiden}.
Hence, the refined invariants we compute ought not to depend upon the choice of orientation data.
We nevertheless feel that it is worthwhile showing in Section~\ref{sect:inv} that the orientation data does not depend upon the choice of quiver with potential given by the choice of heart.
Furthermore, we feel that it is worthwhile showing that in the cases from Section~\ref{sect:spiral}, the orientation data on the stack of semistable objects coincides with the canonical orientation data from the quiver with potential we use to describe this category of semistable objects.
Doing so gives a stronger statement, which does not rely on working in the refined rather than the motivic setting, and which does not rely on making the identification $\mathbb{L}^{1/2} = (1 - [\roua{2}])$.
It is also interesting in its own right that it is always the canonical choice of orientation data that appears.
\end{remark}

\subsubsection{Calculations from the literature}

We finish this section by presenting the following proposition containing calculations of generating series, which we will use later.

\begin{proposition}[{\cite[Proposition~5.8, Proposition~5.9]{haiden}, \cite[Theorem~1.1]{dm_loop}}]\label{prop:haiden}
The following quivers with potential have the following generating series. \[
\begin{tabular}{c|c|c|c}
    Quiver & Potential & Generating series $\gensern{Q, W}$  & Generating series $\genser{Q,W}$\\
    \hline
     $\bullet$ & $0$ & $\qdl{t}$ & $\qdl{t}$ \\
     $\begin{tikzcd}\bullet \ar[loop left]\end{tikzcd}$ & $0$ & $\qdl{-q^{-1/2}t}^{-1}$ & $\qdl{-q^{1/2}t}^{-1}$ \\
     $\begin{tikzcd}\bullet \ar[loop left,"a"]\end{tikzcd}$ & $a^3$ & $\qdl{t}^2$ & $\qdl{t}^2$ \\
     $\begin{tikzcd}\bullet \ar[loop left,"a"] \ar[loop right,"b"]\end{tikzcd}$ & $a^3 + b^3$ & $\qdl{t}^4\qdl{-q^{-1/2}t^2}^{-1}$ & $\qdl{t}^4\qdl{-q^{1/2}t^2}^{-1}$ \\
     $\begin{tikzcd} \overset{1}{\bullet} \ar[r,shift left] & \overset{2}{\bullet} \ar[l,shift left]\end{tikzcd}$ & $0$ & $\qdl{t^{(1, 0)}}\qdl{t^{(0, 1)}}\qdl{-q^{-1/2}t^{(1, 1)}}^{-1}$ & $\qdl{t^{(1, 0)}}\qdl{t^{(0, 1)}}\qdl{-q^{1/2}t^{(1, 1)}}^{-1}$
\end{tabular}
\]
\end{proposition}

\section{Invariance of \texorpdfstring{$A_\infty$}{A-infinity}-enhancement and orientation data}\label{sect:inv}

Flipping a triangulation corresponds to an operation on the associated heart known as simple tilting.
In this section, we prove that simple tilting of hearts produces a quasi-equivalence of 3-Calabi--Yau $A_\infty$-categories between twisted complexes for $(Q,W)$ a quiver with potential and from $(Q',W')$ resulting from a flip.
We also show that the canonical choice of orientation data from Section~\ref{sect:dt:orient_data} is invariant under the process of simple tilting of hearts, extending \cite[Theorem~8.3.2]{davison_thesis}.
Due to Remark~\ref{rmk:od_fine}, the refined invariants should be unaffected by the change in orientation data, but it is interesting that there is a single canonical choice of orientation data for $\mathcal{D}$, independent of the choice of quiver with potential associated to a particular heart.

\begin{remark}
In \cite[Theorem~8.3.2]{davison_thesis}, Davison proves that the orientation data from Section~\ref{sect:dt:orient_data} is invariant under simple tilting of so-called cluster collections \cite[Definition~20]{ks_stability}, which are defined as follows.
Recall that an object $S$ of $\mathcal{D}$ is (3-)\emph{spherical} if $\Ext_{\mathcal{D}}^{\bullet}(S) \cong \mathbb{C} \oplus \mathbb{C}[-3]$.
A~\emph{cluster collection} is then a set of spherical objects $\{S_{1}, S_{2}, \dots, S_{n}\}$ such that for $i \neq j$, we have that $\oplus_{k \in \mathbb{Z}}\Ext_{\mathcal{D}}^{k}(S_{i}, S_{j})$ is zero, concentrated in degree~$1$, or concentrated in degree~$2$.
Our sets of simple objects do not always form cluster collections, since our quivers in general contain loops and two-cycles.
\end{remark}

\subsection{Simple tilting of hearts}\label{sect:simple_tilting}

Simple tilting is a process which produces a new heart from an old one, given a simple object in the old heart.
To describe this process, we first need the following notions.

\subsubsection{Approximations and functorial finiteness}

Let $\mathcal{C}$ be an additive category with a full subcategory~$\mathcal{X}$.
Given a map $f\colon X \rightarrow M$, where $X \in \mathcal{X}$ and $M \in \mathcal{C}$, we say that $f$ is a \emph{right $\mathcal{X}$-approximation} if for any $X' \in \mathcal{X}$, the sequence \[\mathrm{Hom}_{\mathcal{C}}(X',X) \xrightarrow{f \circ -} \mathrm{Hom}_{\mathcal{C}}(X',M) \rightarrow 0\] is exact.
\emph{Left $\mathcal{X}$-approximations} are defined dually, so that a map $f \colon M \to X$ is a left $\mathcal{X}$-approximation if for any $X' \in \mathcal{X}$ the sequence \[\Hom_{\mathcal{C}}(X, X') \xrightarrow{- \circ f} \Hom_{\mathcal{C}}(M, X') \to 0\] is exact.

We will be interested in the following particular kinds of approximations.
A morphism $f \colon X \to Y$ is \emph{right minimal} if whenever a morphism $g \colon X \to X$ is such that $fg = f$, then $g$ is an isomorphism.
\emph{Left minimal} morphisms are defined dually.
A right $\mathcal{X}$-approximation is a \emph{minimal right $\mathcal{X}$-approximation} if it is also right minimal, and \emph{minimal left $\mathcal{X}$-approximations} are defined analogously.

\subsubsection{Simple tilting}

Given an object $S$ of a triangulated category, we write $\extclos{S}$ for the smallest subcategory containing $S$ which is closed under extensions.
Given a full subcategory $\mathcal{B}$ of an additive category~$\mathcal{C}$, we also write 
\[\rorth{\mathcal{B}} = \{C \in \mathcal{C} \st \Hom_{\mathcal{C}}(B, C) = 0 \text{ for all } B \in \mathcal{B}\},\]
with $\lorth{\mathcal{B}}$ defined analogously.
Recall the notation $\ast$ from Section~\ref{sect:app:heart_tstr}.

\begin{definition}[{\cite{hrs}}]
Let $\mathcal{H}$ be the heart of a t-structure on a triangulated category $\mathcal{D}$ with $S$ a simple object of~$\mathcal{H}$.
The \emph{forwards tilt} of $\mathcal{H}$ at $S$ is $\mathcal{H}^{S, \sharp} := \extclos{S}[1] \ast \lorth{\extclos{S}}$, whilst the \emph{backwards tilt} is $\mathcal{H}^{S,\flat} := \rorth{\extclos{S}} \ast \extclos{S}[-1]$.
\end{definition}

As is very standard, we also refer to simple tilting of hearts and flipping triangulations as \emph{mutation}.
We have the following proposition which describes the simple objects of the new heart under forwards or backwards tilting in terms of the simple objects of the old heart.
We call a heart $\mathcal{H}$ \emph{finite} if it is length and has finitely many simple objects and write $\simp{\mathcal{H}}$ for the set of isomorphism-class representatives of its simple objects.

\begin{proposition}[{\cite [Proposition~A.2]{chq}}]\label{prop:simple_tilt_chq}
Let $\mathcal{D}$ be a triangulated category with a bounded t-structure with finite heart $\mathcal{H}$.
If both $\mathcal{H}^{S, \sharp}$ and $\mathcal{H}^{S,\flat}$ are finite hearts, then we have the following formulae for the simple objects in the new hearts $\mathcal{H}^{S,\sharp}$ and~$\mathcal{H}^{S,\flat}$.

For the forwards tilt we have that \[
\simp{\mathcal{H}^{S,\sharp}} = \{S[1]\} \cup \{\psi^{S,\sharp}(S') \st S' \in \simp{\mathcal{H}}, S' \neq S\},
\]
where $\psi^{S, \sharp}(S') = \mathrm{Cone}(f)[-1]$ for $f$ the minimal left $\extclos{S[1]}$-approximation of $S'$.

For the backwards tilt, we have that \[
\simp{\mathcal{H}^{S,\flat}} = \{S[-1]\} \cup \{\psi^{S,\flat}(S') \st S' \in \simp{\mathcal{H}}, S' \neq S\},
\]
where $\psi^{S,\flat}(S') = \mathrm{Cone}(g)$ for $g$ the minimal right $\extclos{S[-1]}$-approximation of $S'$.
\end{proposition}

In our case, we will always have that the hearts in Proposition~\ref{prop:simple_tilt_chq} will be finite by \cite[Proposition~4.18, Corollary~4.21, Proposition~4.22, Theorem~6.20, and Proposition~6.22]{chq}.

\subsection{Invariance under simple tilting}

For our purposes, there are four different types of simple tilting, corresponding to flipping the following four different types of arcs of the triangulation.
\begin{enumerate}
    \item We may flip an arc of the triangulation lying in a quadrilateral with no identical adjacent sides, as shown in Figure~\ref{fig:quad_mut}.
    We call these arcs \emph{standard arcs}.
    These vertices of the quiver are not incident to either loops or two-cycles.
    Note that opposite sides of the quadrilateral may be identified, but having adjacent sides identical leads to two-cycles, which are instead covered by case~\eqref{op:flips:two_cyc}.
    \item We may flip an arc which is the boundary of a monogon encircling an orbifold point, as shown in Figure~\ref{fig:simp_pole_mut}.    
    We call these arcs \emph{monogon arcs}.
    These vertices of the quiver are incident to loops whose third power is in the potential.
    Note also that there is the edge case of a sphere with one puncture and two orbifold points, with one arc cutting the sphere into two monogons, each containing one orbifold point; we do not need to consider this case since there is only one simple module in the heart.
    This case will appear in Section~\ref{sect:main:drd:iii}.
    \item We may flip an arc which is incident to a two-cycle of the quiver, but no loops.
    We call these arcs \emph{almost encircling arcs}.\label{op:flips:two_cyc}
    These either arise from the encircling edges of self-folded triangles, or from a puncture incident to precisely two arcs.
    In fact, these two cases are related by a flip, as shown in Figure~\ref{fig:2cycle}.
    \item Finally, we may perform a simple tilt at the simple object coming from a vertex of the quiver which comes from the self-folded edge of a self-folded triangle.
    In this case, the vertex is incident to a loop whose powers do not lie in the potential.
\end{enumerate}

\begin{remark}
Strictly speaking, when we show quasi-equivalences of 3-Calabi--Yau $A_\infty$-categories in this paper, we should show quasi-equivalences of ``ind-constructible'' $3$-Calabi--Yau $A_\infty$-categories, in the sense of \cite[Definition~8]{ks_stability}, in order to verify that the equivalence respects the geometric structure of the category used to define DT invariants.
Since we only ever work with categories generated by finitely many objects, we ignore this technicality.
\end{remark}

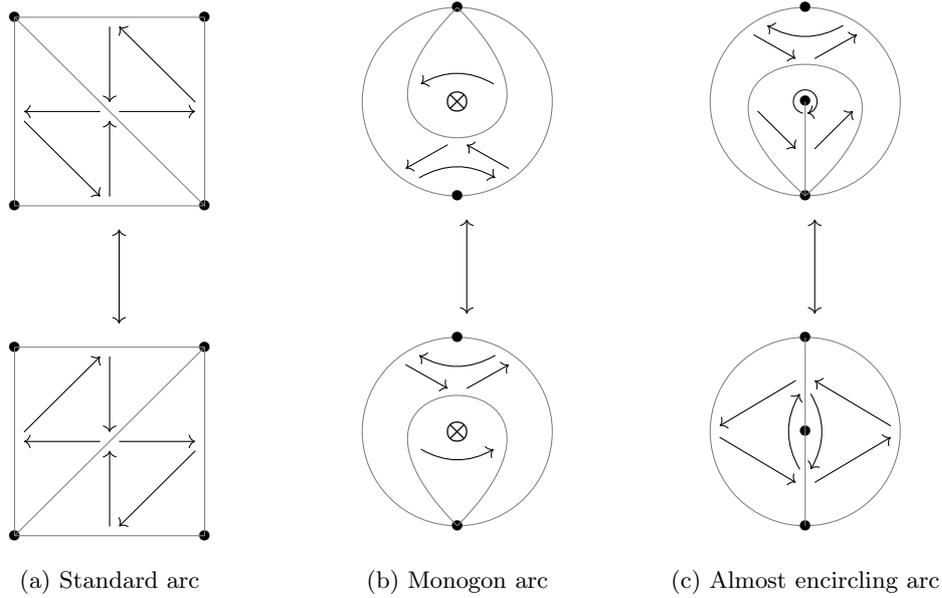
\begin{figure}
    \begin{subfigure}{0.28\textwidth}
    \[
    \begin{tikzpicture}[scale=1.25]

    \begin{scope}[shift={(0,1.75)}]

    \node at (1,1) {$\bullet$};
    \node at (-1,1) {$\bullet$};
    \node at (-1,-1) {$\bullet$};
    \node at (1,-1) {$\bullet$};

    \draw[\edgecol] (1,1) -- (-1,1) -- (-1,-1) -- (1,-1) -- (1,1);
    \draw[\edgecol] (-1,1) -- (1,-1);

    \node(c) at (0,0) {};
    \node(r) at (1,0) {};
    \node(t) at (0,1) {};
    \node(l) at (-1,0) {};
    \node(b) at (0,-1) {};

    \draw[->,\arrcol] (c) -- (r);
    \draw[->,\arrcol] (t) -- (c);
    \draw[->,\arrcol] (r) -- (t);

    \draw[->,\arrcol] (c) -- (l);
    \draw[->,\arrcol] (b) -- (c);
    \draw[->,\arrcol] (l) -- (b);

    \end{scope}

    \begin{scope}

    \draw[<->] (0.1,-0.5) -- (0.1,0.5);
        
    \end{scope}

    \begin{scope}[shift={(0,-1.75)}]

    \node at (1,1) {$\bullet$};
    \node at (-1,1) {$\bullet$};
    \node at (-1,-1) {$\bullet$};
    \node at (1,-1) {$\bullet$};

    \draw[\edgecol] (1,1) -- (-1,1) -- (-1,-1) -- (1,-1) -- (1,1);
    \draw[\edgecol] (1,1) -- (-1,-1);

    \node(c) at (0,0) {};
    \node(r) at (1,0) {};
    \node(t) at (0,1) {};
    \node(l) at (-1,0) {};
    \node(b) at (0,-1) {};

    \draw[->,\arrcol] (c) -- (r);
    \draw[->,\arrcol] (b) -- (c);
    \draw[->,\arrcol] (r) -- (b);

    \draw[->,\arrcol] (c) -- (l);
    \draw[->,\arrcol] (t) -- (c);
    \draw[->,\arrcol] (l) -- (t);
        
    \end{scope}
        
    \end{tikzpicture}
    \]
    \subcaption{Standard arc}\label{fig:quad_mut}
    \end{subfigure}
    \begin{subfigure}{0.28\textwidth}
\[
    \begin{tikzpicture}[scale=1.25]

    \begin{scope}[shift={(0,-1.75)}]
    
   	\path[use as bounding box] (-1.5,1.2) rectangle (1.5,-1.2);

    \draw[\edgecol] (1,0) arc (0:360:1);

    \node at (90:1) {$\bullet$};
    \node at (270:1) {$\bullet$};
    \node at (0,0) {$\bm{\otimes}$};
    
    \draw[\edgecol] (270:1) .. controls (25:2) and (155:2) .. (270:1);

    \node(tl) at (120:1) {};
    \node(tr) at (60:1) {};
    \node(tll) at (130:1) {};
    \node(trr) at (50:1) {};
    \node(t) at (90:0.4) {};
    \node(bl) at (195:0.5) {};
    \node(br) at (345:0.5) {};
    \node(b) at (270:0.6) {};

    \draw[->,\arrcol] (tr) to [bend left] (tl);
    \draw[->,\arrcol] (t) -- (trr);
    \draw[->,\arrcol] (tll) -- (t);
    \draw[->,\arrcol] (bl) to [bend right] (br);
    
    \end{scope}

    \begin{scope}

    \draw[<->] (0.1,-0.5) -- (0.1,0.5);
        
    \end{scope}

    \begin{scope}[shift={(0,1.75)}]
    
   	\path[use as bounding box] (-1.5,1.2) rectangle (1.5,-1.2);

    \draw[\edgecol] (1,0) arc (0:360:1);

    \node at (90:1) {$\bullet$};
    \node at (270:1) {$\bullet$};
    \node at (0,0) {$\bm{\otimes}$};

    \draw[\edgecol] (90:1) .. controls (205:2) and (335:2) .. (90:1);

    \node(tl) at (300:1) {};
    \node(tr) at (240:1) {};
    \node(tll) at (310:1) {};
    \node(trr) at (230:1) {};
    \node(t) at (270:0.4) {};
    \node(bl) at (165:0.5) {};
    \node(br) at (15:0.5) {};
    \node(b) at (90:0.6) {};

    \draw[->,\arrcol] (tr) to [bend left] (tl);
    \draw[->,\arrcol] (t) -- (trr);
    \draw[->,\arrcol] (tll) -- (t);
    \draw[->,\arrcol] (br) to [bend right] (bl);
        
    \end{scope}
        
    \end{tikzpicture}
    \]
    \caption{Monogon arc}
    \label{fig:simp_pole_mut}
    \end{subfigure}
    \begin{subfigure}{0.28\textwidth}
    \[
    \begin{tikzpicture}[scale=1.25]

    \begin{scope}[shift={(0,-1.75)}]
    
   	\path[use as bounding box] (-1.5,1.2) rectangle (1.5,-1.2);

    \draw[white] (270:1) .. controls (22.5:2.25) and (157.5:2.25) .. (270:1);

    \draw[\edgecol] (1,0) arc (0:360:1);

    \node at (90:1) {$\bullet$};
    \node at (270:1) {$\bullet$};
    \node at (0,0) {$\bullet$};

    \draw[\edgecol] (90:1) -- (0,0);
    \draw[\edgecol] (0,0) -- (270:1);

    \node(t) at (90:0.5) {};
    \node(tt) at (90:0.6) {};
    \node(b) at (270:0.5) {};
    \node(bb) at (270:0.6) {};
    \node(r) at (0:1) {};
    \node(l) at (180:1) {};

    \draw[->,\arrcol] (t) to [bend left] (b);
    \draw[->,\arrcol] (r) -- (tt);
    \draw[->,\arrcol] (bb) -- (r);

    \draw[->,\arrcol] (b) to [bend left] (t);
    \draw[->,\arrcol] (l) -- (bb);
    \draw[->,\arrcol] (tt) -- (l);
    
    \end{scope}

    \begin{scope}

    \draw[<->] (0.1,-0.5) -- (0.1,0.5);
        
    \end{scope}

    \begin{scope}[shift={(0,1.75)}]
    
   	\path[use as bounding box] (-1.5,1.2) rectangle (1.5,-1.2);

    \draw[\edgecol] (1,0) arc (0:360:1);

    \node at (90:1) {$\bullet$};
    \node at (270:1) {$\bullet$};
    \node at (0,0) {$\bullet$};

    \draw[\edgecol] (270:1) .. controls (22.5:2.25) and (157.5:2.25) .. (270:1);
    \draw[\edgecol] (0,0) -- (270:1);

    \node(tl) at (120:1) {};
    \node(tr) at (60:1) {};
    \node(tll) at (130:1) {};
    \node(trr) at (50:1) {};
    \node(t) at (90:0.4) {};
    \node(bl) at (180:0.6) {};
    \node(br) at (0:0.6) {};
    \node(b) at (270:0.6) {};

    \draw[<-,\arrcol] (280:0.125) arc (-80:260:0.125);
    \draw[->,\arrcol] (b) -- (br);
    \draw[->,\arrcol] (bl) -- (b);

    \draw[->,\arrcol] (tr) to [bend left] (tl);
    \draw[->,\arrcol] (t) -- (trr);
    \draw[->,\arrcol] (tll) -- (t);
        
    \end{scope}
        
    \end{tikzpicture}
    \]
    \subcaption{Almost encircling arc}
    \label{fig:2cycle}        
    \end{subfigure}
    \caption{Different cases of mutation}
    \label{fig:mut}
\end{figure}

Let $\mbso$ be a marked bordered surface with orbifold points and let $T$ and $T'$ be triangulations of $\mbso$ related by a flip at an arc~$\alpha_1$ with corresponding simple object~$S_{1}$.
Denote $(Q, W) = (Q(T), W(T))$ and $(Q', W') = (Q(T'), W(T'))$. 
We write as usual $\mathcal{H}_{\infty} = \mathcal{H}_{\infty}(Q, W)$ and $\mathcal{H} = H^{0}(\mathcal{H}_\infty)$.
We further write $\mathcal{H}_{\infty}^{S_{1}, \sharp}$ for the full $A_\infty$-subcategory of $\mathcal{D}_\infty(Q, W)$ corresponding to the forwards tilted heart $\mathcal{H}_{\infty}^{S_{1}, \sharp}$.
We then have orientation data on the stacks of objects in $\mathcal{H}^{S_1, \sharp}$ given by the constructible super vector bundles $\End_{\mathcal{D}_\infty(Q,W)}^{1}(-)/\ker(m_{1}^{\two})$ with quadratic form $\cy{m_{1}^{\two}(-)}{-}$,  which we also denote $\odchoice$.
Along with showing that there is a quasi-equivalence $\mathcal{H}_{\infty}^{S_1, \sharp} \simeq \mathcal{H}_{\infty}(Q', W')$, we wish to show that we also have $\odchoice \cong \odcha{Q', W'}$ on $\qastn{Q', W'}$.

\begin{theorem}\label{thm:enhance_equiv}
Let $\mbso$ be a marked bordered surface with orbifold points and let $T$ and $T'$ be triangulations of $\mbso$ related by a flip at an arc~$\alpha_1$, possibly self-folded, with corresponding simple object~$S_{1}$.
Then there are quasi-equivalences of 3-Calabi--Yau $A_\infty$-categories \[
\mathcal{H}_\infty^{S_1,\flat} \simeq \mathcal{H}_{\infty}(Q', W') \simeq \mathcal{H}_\infty^{S_1,\sharp},
\]
which therefore induce two quasi-equivalences of 3-Calabi--Yau $A_\infty$-categories $\mathcal{D}_\infty(Q,W) \simeq \mathcal{D}_\infty(Q',W')$.
Moreover, these quasi-equivalences induce isomorphisms of orientation data $\odchoice \cong \odcha{Q', W'}$ on $\qastn{Q', W'}$. 
\end{theorem}

We will only show the forwards tilting case, since the backwards tilting case is similar.
Note that by \cite{kajiura_jppa}, $A_\infty$-enhancements of triangulated categories are not generally unique up to quasi-equivalence.

To prove Theorem~\ref{thm:enhance_equiv}, we use the following lemmas, which consider the different types of flip.
We write $\simpinf{\mathcal{H}_{\infty}^{S_{1},\sharp}}$ for the full $A_\infty$-subcategory of $\mathcal{D}_\infty$ consisting of one representative of each simple object of $\mathcal{H}^{S_{1},\sharp}$.

\begin{lemma}\label{lem:enhance_equiv:standard}
If $\alpha_1$ is a standard arc, then there is a quasi-equivalence of 3-Calabi--Yau $A_\infty$-categories $\simpinf{\mathcal{H}_{\infty}^{S_{1}, \sharp}} \simeq \akd{Q'}{W'}$, with an isomorphism of orientation data $\odchoice \cong \odcha{Q', W'}$ on $\qastn{Q', W'}$.
\end{lemma}
\begin{proof}
If the arc $\alpha_1$ is a standard arc, then we are in the situation on the left of Figure~\ref{fig:enhance_equiv:quad_mut}, and we label the simple objects corresponding to the arcs and the degree~$1$ morphisms between them as shown.
Note that we are assuming here that $S_3 \neq S_5$ and $S_2 \neq S_4$.
We leave as an exercise the calculations in the cases where this fails to hold.

Note that we have the description of the simple objects of the new heart from Proposition~\ref{prop:simple_tilt_chq}. %
Hence, the new simple objects after the forwards tilt at $S_1$ are as shown on the right in Figure~\ref{fig:enhance_equiv:quad_mut}.
Outside the quadrilateral, the simple objects remain the same, since their minimal left $\extclos{S_1[1]}$-approximations are zero.
Here $S_5 \xrightarrow{e} S_1$ means the twisted complex \[
\left((S_5, S_1), \left(\begin{smallmatrix}0 & 0 \\ e & 0\end{smallmatrix}\right) \right);
\]
we will usually write twisted complexes like this.

We must show that the degree~$1$ morphisms shown indeed give non-zero cohomology classes in the minimal model.
This is clear for $a^{\ast}$, $d^{\ast}$, and $\left(\begin{smallmatrix}0 \\ \id_1\end{smallmatrix}\right)$.
For $\left(\begin{smallmatrix}0 & a\end{smallmatrix}\right) \colon (S_5 \xrightarrow{e} S_1) \to S_2$, we compute
\begin{align*}
m_1^{\two}\left( \left(\begin{smallmatrix}0 & a\end{smallmatrix}\right) \right) = m_1\left(\begin{smallmatrix}0 & a\end{smallmatrix}\right) + m_2\left(\left(\begin{smallmatrix}0 & a \end{smallmatrix}\right),\left(\begin{smallmatrix}0 & 0 \\ e & 0 \end{smallmatrix}\right)\right) = \left(\begin{smallmatrix}0 & m_1(a)\end{smallmatrix}\right) + \left(\begin{smallmatrix}m_2(a,e) & 0 \end{smallmatrix}\right) = 0
\end{align*}
Hence, $\left(\begin{smallmatrix}0 & a\end{smallmatrix}\right)$ is a cocycle; since there are no degree~$0$ morphisms from $S_5 \xrightarrow{e} S_1$ to $S_2$, it therefore gives a cohomology class.
The same can be shown for $\left(\begin{smallmatrix}0 & d\end{smallmatrix}\right)$.

We now show that there are no further degree~$1$ morphisms which give cohomology classes.
The only possibilities are  $\left(\begin{smallmatrix}0 & a\end{smallmatrix}\right) \colon (S_3 \xrightarrow{b} S_1) \to S_2$ and $\left(\begin{smallmatrix}0 & d\end{smallmatrix}\right) \colon (S_5 \xrightarrow{e} S_1) \to S_4$.
However, one can compute that the first is not a cocycle, since
\begin{align*}
    m_1^{\two}\left(\begin{smallmatrix}0 & a\end{smallmatrix}\right) = \left(\begin{smallmatrix}0 & m_1(a)\end{smallmatrix}\right) + \left(\begin{smallmatrix}m_2(a,b) & 0 \end{smallmatrix}\right) = \left(\begin{smallmatrix}c^{\ast} & 0 \end{smallmatrix}\right).
\end{align*}
A similar argument applies to the other morphism.

We now verify that the $A_\infty$-operations behave as expected, first computing the operations in the top left cycle.
\begin{align*}
    m_2^{\two}\left(a^{\ast}, \left(\begin{smallmatrix}0 & a\end{smallmatrix}\right)\right) &=  \left(\begin{smallmatrix}0 & m_2(a^{\ast}, a)\end{smallmatrix}\right) = \left(\begin{smallmatrix}0 & \id_1^{\ast}\end{smallmatrix}\right). \\
    m_2^{\two}\left(\left(\begin{smallmatrix}0 \\ \id_1\end{smallmatrix}\right), a^{\ast}\right) &=  \left(\begin{smallmatrix}0 \\ m_2(\id_1, a^{\ast})\end{smallmatrix}\right) = \left(\begin{smallmatrix}0 \\ a^{\ast}\end{smallmatrix}\right). \\
    m_2^{\two}\left(\left(\begin{smallmatrix}0 & a\end{smallmatrix}\right), \left(\begin{smallmatrix}0 \\ \id_1\end{smallmatrix}\right)\right) &=  m_2(a, \id_1) = a.
\end{align*}
Note that these degree~$2$ morphisms pair with the respective degree~$1$ morphisms $\left(\begin{smallmatrix}0 \\ \id_1\end{smallmatrix}\right)$, $\left(\begin{smallmatrix}0 & a\end{smallmatrix}\right)$, and~$a^{\ast}$, as required by the 3-Calabi--Yau property.
Similar computations apply to the other three-cycle.
We must also show that $m_2^{\two}$ gives zero on appropriate compositions of morphisms.
Indeed,
\begin{align*}
    m_2^{\two}\left(\left(\begin{smallmatrix}0 \\ \id_1\end{smallmatrix}\right)\colon S_1[1] \to (S_3 \xrightarrow{b} S_1), a^{\ast} \colon S_2 \to S_1[1]\right) = \left(\begin{smallmatrix}0 \\ a^{\ast}\end{smallmatrix}\right)\colon S_2 \to (S_3 \xrightarrow{b} S_1).
\end{align*}
However, this is now a coboundary, as can be seen from
\begin{align*}
    m_1^{\two}\left(\left(\begin{smallmatrix}c \\ 0\end{smallmatrix}\right) \colon S_2 \to (S_3 \xrightarrow{b} S_1)\right) = \left(\begin{smallmatrix}m_1(c) \\ 0\end{smallmatrix}\right) - \left(\begin{smallmatrix}0 \\ m_2(b,c)\end{smallmatrix}\right) = \left(\begin{smallmatrix}0 \\ - a^{\ast}\end{smallmatrix}\right).
\end{align*}
We also have that $m_2^{\two}$ gives zero on the pair of morphisms $d^{\ast} \colon S_4 \to S_1[1]$ and $\left(\begin{smallmatrix}0 \\ \id_1\end{smallmatrix}\right) \colon S_1[1] \to (S_5 \xrightarrow{e} S_1)$.

We conclude that we have a quasi-equivalence of 3-Calabi--Yau $A_\infty$-categories $\simpinf{\mathcal{H}_{\infty}^{S_{1}, \sharp}} \simeq \akd{Q'}{W'}$, since the degree~$1$ morphisms outside the quadrilateral are unaffected.

We now prove the statement on orientation data.
We let $S'_3 := (S_3 \xrightarrow{b} S_1)$ and $S'_5: = (S_5 \xrightarrow{e} S_1)$.
We introduce the shorthand $e^{\mathsmaller{\bullet}} = \dim \Ext_{\mathcal{D}_\infty(Q, W)}^{\mathsmaller{\bullet}}$ and $h^{\mathsmaller{\bullet}} := \dim \Hom_{\mathcal{D}_\infty(Q, W)}^{\mathsmaller{\bullet}}$.
Following Proposition~\ref{prop:od_ug}, it suffices to prove that
\begin{align*}
e^{\leqslant 1}((S'_i, S'_i) + h^{\leqslant 1}(S'_i, S'_i)
\end{align*}
is even, where $i \in \{3, 5\}$.
We have that $e^{\leqslant 1}(S'_i, S'_{i})$ is $1$ if $S'_i$ is spherical and $2$ if $S'_i$ is non-spherical.
We then have that
\begin{align*}
    h^{0}((S_i \to S_1), (S_i \to S_1)) &= 2,  \\
    h^{1}((S_i \to S_1), (S_i \to S_1)) &= h^{1}(S_i, S_i) + 1,
\end{align*}
with the second line holding since $S_1$ is spherical.
Note that we must have that $\Hom^{1}_{\mathcal{D}_{\infty}}(S_1, S_i) = 0$, since adjacent sides of the quadrilateral in Figure~\ref{fig:enhance_equiv:quad_mut} are distinct by assumption.
We then have that $h^{1}(S_i, S_i)$ is $0$ if $S_i$ is spherical and $1$ if it is non-spherical.
We then use the fact that $S_i$ is spherical if and only if $S'_{i}$ is spherical, since the mutation does not change loops at other vertices in the quiver if we are not mutating at a two-cycle.
Thus, we indeed get an even number overall, as desired.
\end{proof}

\begin{figure}[h]
    \[
    \begin{tikzpicture}[scale=1.85]

    \begin{scope}[shift={(-1.75,0)}]

    \node at (1,1) {$\bullet$};
    \node at (-1,1) {$\bullet$};
    \node at (-1,-1) {$\bullet$};
    \node at (1,-1) {$\bullet$};

    \node at (-1.2,0) {$S_2$};
    \node at (0,1.2) {$S_5$};
    \node at (1.2,0) {$S_4$};
    \node at (0,-1.2) {$S_3$};
    \node at (-0.4,0.6) {$S_{1}$};

    \draw[\edgecol] (1,1) -- (-1,1) -- (-1,-1) -- (1,-1) -- (1,1);
    \draw[\edgecol] (-1,1) -- (1,-1);

    \node(c) at (0,0) {};
    \node(r) at (1,0) {};
    \node(t) at (0,1) {};
    \node(l) at (-1,0) {};
    \node(b) at (0,-1) {};

    \draw[->,\arrcol] (c) -- (r) node [midway,below] {$d$};
    \draw[->,\arrcol] (t) -- (c) node [midway,left] {$e$};
    \draw[->,\arrcol] (r) -- (t) node [midway,above right] {$f$};

    \draw[->,\arrcol] (c) -- (l) node [midway,above] {$a$};
    \draw[->,\arrcol] (b) -- (c) node [midway,right] {$b$};
    \draw[->,\arrcol] (l) -- (b) node [midway,below left] {$c$};

    \end{scope}

    \begin{scope}

    \draw[->] (-0.3,0.1) -- (0.3,0.1);
        
    \end{scope}

    \begin{scope}[shift={(1.75,0)}]

    \node at (1,1) {$\bullet$};
    \node at (-1,1) {$\bullet$};
    \node at (-1,-1) {$\bullet$};
    \node at (1,-1) {$\bullet$};

    \node at (-1.2,0) {$S_2$};
    \node at (0,1.2) {$S_5 \xrightarrow{e} S_1$};
    \node at (1.2,0) {$S_4$};
    \node at (0,-1.2) {$S_3 \xrightarrow{b} S_1$};
    \node at (0.4,0.7) {$S_{1}[1]$};

    \draw[\edgecol] (1,1) -- (-1,1) -- (-1,-1) -- (1,-1) -- (1,1);
    \draw[\edgecol] (1,1) -- (-1,-1);

    \node(c) at (0,0) {};
    \node(r) at (1,0) {};
    \node(t) at (0,1) {};
    \node(l) at (-1,0) {};
    \node(b) at (0,-1) {};

    \draw[->,\arrcol] (r) -- (c) node [midway,above] {$d^{\ast}$};
    \draw[->,\arrcol] (c) -- (b) node [midway,below left] {$\left(\begin{smallmatrix}0 \\ \id_{1}\end{smallmatrix}\right)$};
    \draw[->,\arrcol] (b) -- (r) node [midway,below right] {$\left(\begin{smallmatrix}0 & d\end{smallmatrix}\right)$};

    \draw[->,\arrcol] (l) -- (c) node [midway,below] {$a^{\ast}$};
    \draw[->,\arrcol] (c) -- (t) node [midway,below left] {$\left(\begin{smallmatrix}0 \\ \id_{1}\end{smallmatrix}\right)$};
    \draw[->,\arrcol] (t) -- (l) node [midway,above left] {$\left(\begin{smallmatrix}0 & a\end{smallmatrix}\right)$};
        
    \end{scope}
        
    \end{tikzpicture}
    \]
    \caption{Mutating at a standard arc}
    \label{fig:enhance_equiv:quad_mut}
\end{figure}
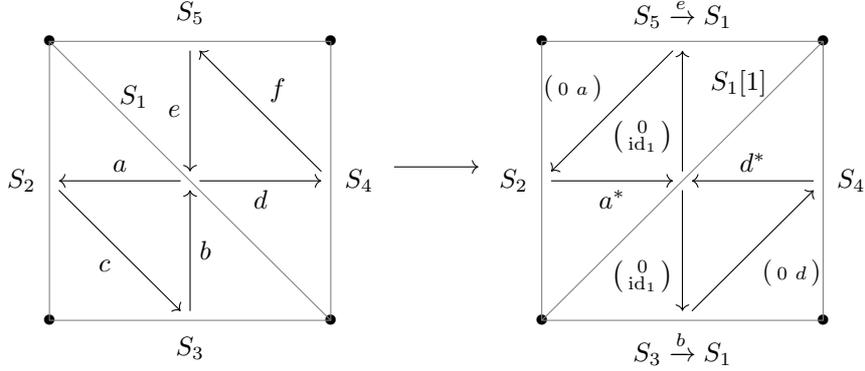

\begin{lemma}\label{lem:enhance_equiv:monogon}
If $\alpha_1$ is a monogon arc, then there is a quasi-equivalence of 3-Calabi--Yau $A_\infty$-categories $\simpinf{\mathcal{H}_{\infty}^{S_{1}, \sharp}} \simeq \akd{Q'}{W'}$, with an isomorphism of orientation data $\odchoice \cong \odcha{Q', W'}$ on $\qastn{Q', W'}$.
\end{lemma}
\begin{proof}
We take the same approach as Lemma~\ref{lem:enhance_equiv:standard}.
If the arc $\alpha_1$ is a monogon arc, then we are in the situation on the left of Figure~\ref{fig:enhance_equiv:simp_pole_mut}, and we label the simple objects corresponding to the arcs and the degree~$1$ morphisms between them as shown.
Here we have assumed that $S_2 \neq S_3$; we leave the case where this fails to hold as an exercise.

We apply Proposition~\ref{prop:simple_tilt_chq}.
Note that the fact that $S_3$ is exchanged for $S_3 \xrightarrow{b} S_1 \xrightarrow{a} S_1$ follows from \cite[Example~A.4]{chq}.
It can also be verified directly that the minimal left $\extclos{S_1[1]}$-approximation of $S_3$ is indeed $\left(\begin{smallmatrix}b \\ 0\end{smallmatrix}\right) \colon S_3 \to (S_1 \xrightarrow{a} S_1)$, since all degree~$1$ morphisms from $S_3$ to $S_1$ and $S_1 \xrightarrow{b} S_1$ factor through this.
The new simple objects after the forwards tilt at $S_1$ are hence as shown on the right in Figure~\ref{fig:enhance_equiv:simp_pole_mut}.
Outside the region of the triangulation $T$ shown, the simple objects remain the same.

We claim that the degree~$1$ morphisms in the minimal model of this subcategory are those shown.
We first show that $\left(\begin{smallmatrix}0 \\ 0 \\ \id_1 \end{smallmatrix}\right)$ spans all degree~$1$ morphisms in the minimal model from $S_1[1]$ to $S_3 \xrightarrow{b} S_1 \xrightarrow{a} S_1$.
Indeed,
\begin{align*}
    m_1^{\two}\left(\left(\begin{smallmatrix}0 \\ \id_1 \\ 0\end{smallmatrix}\right)\right) = \left(\begin{smallmatrix}0 \\ m_1(\id_1) \\ 0\end{smallmatrix}\right) - \left(\begin{smallmatrix}0 \\ m_2(a, \id_1) \\ 0\end{smallmatrix}\right) = \left(\begin{smallmatrix}0 \\ -a \\ 0\end{smallmatrix}\right).
\end{align*}
Hence $\left(\begin{smallmatrix}0 \\ 0 \\ \id_1\end{smallmatrix}\right)$ is the only morphism which is a cocycle, up to scalar.
A similar calculation establishes that $\left(\begin{smallmatrix}0 & 0 & c\end{smallmatrix}\right)$ spans the degree~$1$ cocycles from $S_3 \xrightarrow{b} S_1 \xrightarrow{a} S_1$ to $S_2$.
These morphisms moreover give cohomology classes, since there are no degree~$0$ morphisms between the relevant objects.

Showing that the $A_\infty$-operations and 3-Calabi--Yau pairing behave as expected is similar to Lemma~\ref{lem:enhance_equiv:standard}.
Hence, we indeed have a quasi-equivalence $\simpinf{\mathcal{H}_{\infty}^{S_{1}, \sharp}} \simeq \akd{Q'}{W'}$, since the degree~$1$ morphisms outside the region are unaffected.

We now prove the statement on orientation data.
We denote $S'_3 = (S_3 \xrightarrow{b} S_1 \xrightarrow{a} S_1)$.
Again, we have by Proposition~\ref{prop:od_ug} that it suffices to show that
\begin{align*}
e^{\leqslant 1}(S'_3, S'_3) + h^{\leqslant 1}(S'_3, S'_3) 
\end{align*}
is even.
As before, we have that $e^{\leqslant 1}(S'_3, S'_3)$ is $1$ if $S'_3$ is spherical and $2$ if $S'_3$ is non-spherical.
Note that $S'_3$ is spherical if and only if $S_3$ is.
Furthermore, because we are assuming $S_3 \neq S_2$, we have that $\Hom_{\mathcal{D}_{\infty}}^{1}(S_1, S_3) = 0$.
We then have
\begin{align*}
    h^{0}((S_3 \xrightarrow{b} S_1 \xrightarrow{a} S_1), (S_3 \xrightarrow{b} S_1 \xrightarrow{a} S_1)) &= h^{0}(S_3, S_3) + h^{0}(S_1^{\oplus 2}, S_1^{\oplus 2}) = 1 + 4 \\
    h^{1}((S_3 \xrightarrow{b} S_1 \xrightarrow{a} S_1), (S_3 \xrightarrow{b} S_1 \xrightarrow{a} S_1)) &= h^{1}(S_3, S_3) + h^{1}(S_3, S_1^{\oplus 2}) + h^{1}(S_1^{\oplus 2}, S_1^{\oplus 2}) \\
    &= h^{1}(S_3, S_3) + 2 + 4
\end{align*}
We then have that $h^{1}(S_3, S_3)$ is $0$ if $S_3$ is spherical and is $1$ is $S_3$ is non-spherical, so we obtain an even number in the same way as the previous case, since $S'_{3}$ is spherical if and only if $S_{3}$ is.
In the case where $S_3 = S_2$, we have an additional contribution from $h^{1}(S_1^{\oplus 2}, S_3)$, which does not change the parity.
\end{proof}

\begin{figure}
    \[
    \begin{tikzpicture}[scale=2]

    \begin{scope}[shift={(-1.75,0)}]
    
   	\path[use as bounding box] (-1.5,1.2) rectangle (1.5,-1.2);

    \draw[\edgecol] (1,0) arc (0:360:1);

    \node at (90:1) {$\bullet$};
    \node at (270:1) {$\bullet$};
    \node at (0,0) {$\bm{\otimes}$};

    \node at (90:0.25) {$S_1$};
    \node at (60:1.2) {$S_2$};
    \node at (120:1.2) {$S_3$};

    \draw[\edgecol] (270:1) .. controls (25:2) and (155:2) .. (270:1);

    \node(tl) at (120:1) {};
    \node(tr) at (60:1) {};
    \node(tll) at (130:1) {};
    \node(trr) at (50:1) {};
    \node(t) at (90:0.4) {};
    \node(bl) at (195:0.5) {};
    \node(br) at (345:0.5) {};
    \node(b) at (270:0.6) {};

    \draw[->,\arrcol] (tr) to [bend left] (tl);
    \draw[->,\arrcol] (t) -- (trr);
    \draw[->,\arrcol] (tll) -- (t);
    \draw[->,\arrcol] (bl) to [bend right] (br);
    \node at (270:0.4) {$a$};
    \node at (130:0.6) {$b$};
    \node at (50:0.6) {$c$};
    \node at (90:0.85) {$d$};
    
    \end{scope}

    \begin{scope}

    \draw[->] (-0.3,0.1) -- (0.3,0.1);
        
    \end{scope}

    \begin{scope}[shift={(1.75,0)}]
    
   	\path[use as bounding box] (-1.5,1.2) rectangle (1.5,-1.2);

    \draw[\edgecol] (1,0) arc (0:360:1);

    \node at (90:1) {$\bullet$};
    \node at (270:1) {$\bullet$};
    \node at (0,0) {$\bm{\otimes}$};

    \node at (270:0.25) {$S_1[1]$};
    \node at (240:1.2) {$S_3 \xrightarrow{b} S_1 \xrightarrow{a} S_1$};
    \node at (300:1.2) {$S_2$};

    \draw[\edgecol] (90:1) .. controls (205:2) and (335:2) .. (90:1);

    \node(tl) at (300:1) {};
    \node(tr) at (240:1) {};
    \node(tll) at (310:1) {};
    \node(trr) at (230:1) {};
    \node(t) at (270:0.4) {};
    \node(bl) at (165:0.5) {};
    \node(br) at (15:0.5) {};
    \node(b) at (90:0.6) {};

    \draw[->,\arrcol] (tr) to [bend left] (tl);
    \draw[->,\arrcol] (t) -- (trr);
    \draw[->,\arrcol] (tll) -- (t);
    \draw[->,\arrcol] (br) to [bend right] (bl);
    \node at (90:0.4) {$a$};
    \node at (310:0.6) {$c^{\ast}$};
    \node at (220:0.7) {$\left(\begin{smallmatrix}0 \\ 0 \\ \id_1\end{smallmatrix}\right)$};
    \node at (270:0.85) {$\left(\begin{smallmatrix}0 & 0 & c\end{smallmatrix}\right)$};
        
    \end{scope}
        
    \end{tikzpicture}
    \]
    \caption{Mutating at a monogon arc}
    \label{fig:enhance_equiv:simp_pole_mut}
\end{figure}
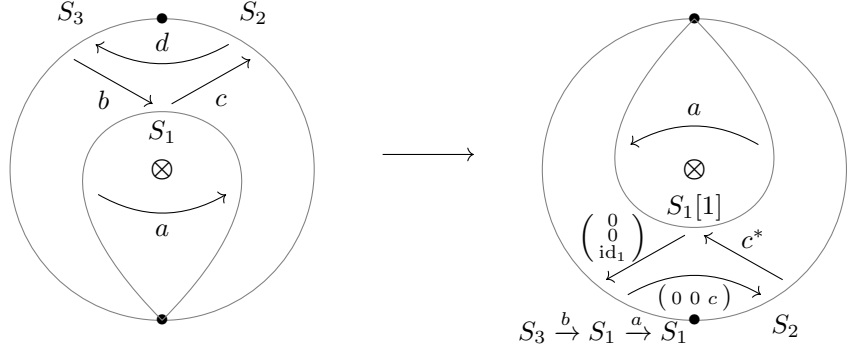

\begin{lemma}\label{lem:enhance_equiv:two_cyc}
If $\alpha_1$ is an almost encircling arc, then there is a quasi-equivalence of 3-Calabi--Yau $A_\infty$-categories $\simpinf{\mathcal{H}_{\infty}^{S_{1}, \sharp}} \simeq \akd{Q'}{W'}$, with an isomorphism of orientation data $\odchoice \cong \odcha{Q', W'}$ on $\qastn{Q', W'}$.
\end{lemma}
\begin{proof}
If $\alpha_1$ is an almost encircling arc, then it is either the encircling edge of a self-folded triangle, or one of two arcs incident to a puncture.
We assume that we are in the latter of these cases, so that we are in the situation on the left of Figure~\ref{fig:enhance_equiv:quad_mut}, and we follow the same approach as usual.

By Proposition~\ref{prop:simple_tilt_chq}, the new simple objects after the forwards tilt at $S_1$ are as shown on the right in Figure~\ref{fig:enhance_equiv:quad_mut}.
Outside the region of the triangulation $T$ shown, the simple objects remain the same.
We claim that the degree~$1$ morphisms in the minimal model of this subcategory are those shown.
The only morphism which is different from those of the previous cases is \[
 \left(\begin{smallmatrix}0 & d \\ 0 & 0\end{smallmatrix}\right) \colon (S_2 \xrightarrow{a} S_1) \to (S_2 \xrightarrow{a} S_1).
\]
We show that this generates $\Ext^{1}_{\mathcal{D}_\infty}(S_2 \xrightarrow{a} S_1, S_2 \xrightarrow{a} S_1)$.
Indeed, it is a cocycle, since
\begin{align*}
    m_1^{\two}\left( \left(\begin{smallmatrix}0 & d \\ 0 & 0\end{smallmatrix}\right) \right) = \left(\begin{smallmatrix}0 & m_1(d) \\ 0 & 0\end{smallmatrix}\right) + \left(\begin{smallmatrix}0 & m_2(d, a) \\ 0 & 0\end{smallmatrix}\right) - \left(\begin{smallmatrix}0 & m_2(a, d) \\ 0 & 0\end{smallmatrix}\right) +  \left(\begin{smallmatrix}0 & m_3(a, d, a) \\ 0 & 0\end{smallmatrix}\right) = 0.
\end{align*}
Moreover, it is not a coboundary, since
\begin{align*}
     m_1^{\two}\left( \left(\begin{smallmatrix}\lambda\id_1 & 0 \\ 0 & \lambda'\id_2\end{smallmatrix}\right) \right) =  \left(\begin{smallmatrix}\lambda m_1(\id_1) & 0 \\ 0 & \lambda' m_1(\id_2)\end{smallmatrix}\right) - \left(\begin{smallmatrix}0 & 0 \\ \lambda m_2(a, \id_1) & 0\end{smallmatrix}\right) + \left(\begin{smallmatrix}0 & 0 \\ \lambda' m_2(\id_1, a) & 0\end{smallmatrix}\right) = \left(\begin{smallmatrix}0 & 0 \\ (\lambda' - \lambda)a  & 0\end{smallmatrix}\right).
\end{align*}
This also shows that there are no other cohomology classes of degree~$1$ endomorphisms of $S_2 \xrightarrow{a} S_1$, up to scalar.

Verifying that the $A_\infty$-operations behave as they should works as in Lemma~\ref{lem:enhance_equiv:standard}.
This shows that we indeed have a quasi-equivalence $\simpinf{\mathcal{H}_{\infty}^{S_{1}, \sharp}} \simeq \akd{Q'}{W'}$, since the degree~$1$ morphisms outside the region shown are unaffected.

The backwards tilt at the encircling edge of a self-folded triangle is then the reverse of this, and hence also gives the desired equivalence.
Since the backwards tilt behaves similarly to the forwards tilt, the encircling edge case also holds for the forwards tilt.

We now prove the statement on orientation data, denoting $S'_2 := (S_2 \xrightarrow{a} S_1)$.
As ever, we follow Proposition~\ref{prop:od_ug}.%
We have that $e^{\leqslant 1}(S'_2, S'_2) = 2$, since $S'_2$ is non-spherical.
Then
\begin{align*}
    h^{0}((S_2 \xrightarrow{a} S_1), (S_2 \xrightarrow{a} S_1)) &= h^{0}(S_2, S_2) + h^{0}(S_1, S_1) = 2,  \\
    h^{1}((S_2 \xrightarrow{a} S_1), (S_2 \xrightarrow{a} S_1)) &= h^{1}(S_1, S_2) + h^{1}(S_2, S_1) = 2.
\end{align*}
Hence, we obtain an even number overall.
Note here that the orientation data is preserved even though $S_2$ is spherical whilst $S'_2$ is not.
Checking that the orientation data is even over $S_3 \xrightarrow{f} S_1$ is routine.
\end{proof}

\begin{figure}
    \[
    \begin{tikzpicture}[scale=2]

    \begin{scope}[shift={(-1.75,0)}]
    
   	\path[use as bounding box] (-1.5,1.2) rectangle (1.5,-1.2);

    \draw[\edgecol] (1,0) arc (0:360:1);

    \node at (90:1) {$\bullet$};
    \node at (270:1) {$\bullet$};
    \node at (0,0) {$\bullet$};

    \node at (80:0.7) {$S_1$};
    \node at (260:0.7) {$S_2$};
    \node at (180:1.2) {$S_4$};
    \node at (0:1.2) {$S_3$};

    \draw[\edgecol] (90:1) -- (0,0);
    \draw[\edgecol] (0,0) -- (270:1);

    \node(t) at (90:0.5) {};
    \node(tt) at (90:0.6) {};
    \node(b) at (270:0.5) {};
    \node(bb) at (270:0.6) {};
    \node(r) at (0:1) {};
    \node(l) at (180:1) {};

    \draw[->,\arrcol] (t) to [bend left] (b);
    \draw[->,\arrcol] (r) -- (tt) node [midway,above right] {$f$};
    \draw[->,\arrcol] (bb) -- (r) node [midway,below right] {$e$};
    \node at (0:0.3) {$d$};

    \draw[->,\arrcol] (b) to [bend left] (t);
    \draw[->,\arrcol] (l) -- (bb) node [midway,below left] {$c$};
    \draw[->,\arrcol] (tt) -- (l) node [midway,above left] {$b$};
    \node at (180:0.3) {$a$};
    
    \end{scope}

    \begin{scope}

    \draw[->] (-0.3,0.1) -- (0.3,0.1);
        
    \end{scope}

    \begin{scope}[shift={(1.75,0)}]
    
   	\path[use as bounding box] (-1.5,1.2) rectangle (1.5,-1.2);

    \draw[\edgecol] (1,0) arc (0:360:1);

    \node at (90:1) {$\bullet$};
    \node at (270:1) {$\bullet$};
    \node at (0,0) {$\bullet$};

    \node at (90:0.25) {$S_1[1]$};
    \node at (270:0.7) {$S_2 \xrightarrow{a} S_1$};
    \node at (120:1.2) {$S_4$};
    \node at (60:1.25) {$S_3 \xrightarrow{f} S_1$};
    
    \draw[\edgecol] (270:1) .. controls (22.5:2.25) and (157.5:2.25) .. (270:1);
    \draw[\edgecol] (0,0) -- (270:1);

    \node(tl) at (120:1) {};
    \node(tr) at (60:1) {};
    \node(tll) at (130:1) {};
    \node(trr) at (50:1) {};
    \node(t) at (90:0.4) {};
    \node(bl) at (180:0.6) {};
    \node(br) at (0:0.6) {};
    \node(b) at (270:0.6) {};

    \draw[<-,\arrcol] (280:0.125) arc (-80:260:0.125);
    \draw[->,\arrcol] (b) -- (br);
    \draw[->,\arrcol] (bl) -- (b);
    \node at (10:0.35) {$\left(\begin{smallmatrix}0 & d \\ 0 & 0\end{smallmatrix}\right)$};
    \node at (220:0.65) {$\left(\begin{smallmatrix}0 \\ \id_1\end{smallmatrix}\right)$};
    \node at (320:0.65) {$\left(\begin{smallmatrix}d^{\ast} & 0\end{smallmatrix}\right)$};
   
    \draw[->,\arrcol] (tr) to [bend left] (tl);
    \draw[->,\arrcol] (t) -- (trr);
    \draw[->,\arrcol] (tll) -- (t);
    \node at (130:0.6) {$b^{\ast}$};
    \node at (40:0.7) {$\left(\begin{smallmatrix}0 \\ \id_1\end{smallmatrix}\right)$};
    \node at (90:0.85) {$\left(\begin{smallmatrix}0 & b\end{smallmatrix}\right)$};

    \end{scope}
        
    \end{tikzpicture}
    \]
    \caption{Flipping at one of two arcs incident to a puncture}
    \label{fig:enhance_equiv:two_cyc}
\end{figure}
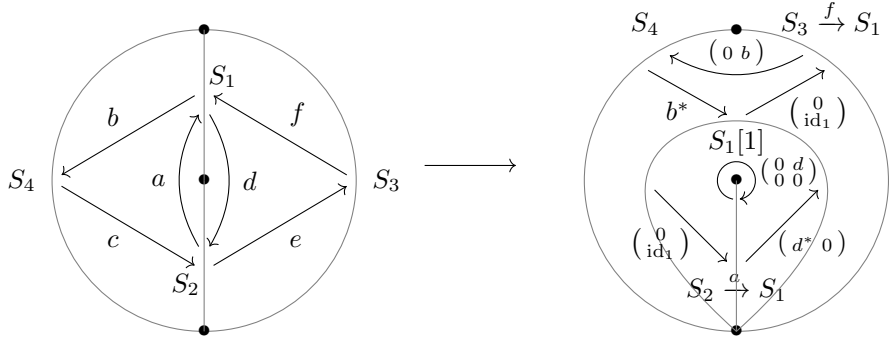

\begin{lemma}\label{lem:enhance_equiv:self_fold}
If $\alpha_1$ is the self-folded edge of a self-folded triangle, then there is a quasi-equivalence of 3-Calabi--Yau $A_\infty$-categories $\simpinf{\mathcal{H}_{\infty}^{S_{1}, \sharp}} \simeq \akd{Q'}{W'}$, with an isomorphism of orientation data $\odchoice \cong \odcha{Q', W'}$ on $\qastn{Q', W'}$.
\end{lemma}
\begin{proof}
Suppose that the arc $\alpha_1$ is the self-folded edge of a self-folded triangle.
If we suppose that the encircling edge is not also the boundary of a monogon, then we are in the situation on the left of Figure~\ref{fig:enhance_equiv:pop}.
The case where the encircling edge does form the boundary of a monogon is left as an exercise.
We proceed as always, and label objects and morphisms as shown.

After forwards tilting at $S_1$, we are at the situation on the right.
Verifying that this indeed gives us $\simpinf{\mathcal{H}_{\infty}^{S_{1}, \sharp}} \simeq \akd{Q'}{W'}$ is not substantially different from the other cases.
Note that the objects outside the self-folded triangle are unchanged, and the morphisms are essentially unchanged.

Showing that orientation data is preserved can be done as before; we leave the details to the reader.
\end{proof}

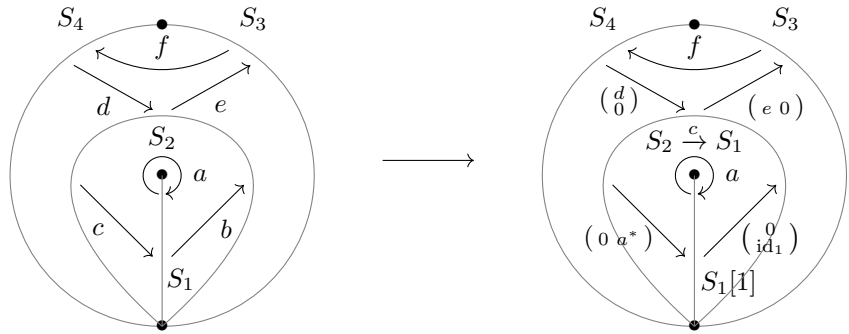
\begin{figure}
    \[
    \begin{tikzpicture}[scale=2]

    \begin{scope}[shift={(-1.75,0)}]
    
   	\path[use as bounding box] (-1.5,1.2) rectangle (1.5,-1.2);

    \draw[\edgecol] (1,0) arc (0:360:1);

    \node at (90:1) {$\bullet$};
    \node at (270:1) {$\bullet$};
    \node at (0,0) {$\bullet$};

    \node at (90:0.25) {$S_2$};
    \node at (280:0.7) {$S_1$};
    \node at (120:1.2) {$S_4$};
    \node at (60:1.2) {$S_3$};
    
    \draw[\edgecol] (270:1) .. controls (22.5:2.25) and (157.5:2.25) .. (270:1);
    \draw[\edgecol] (0,0) -- (270:1);

    \node(tl) at (120:1) {};
    \node(tr) at (60:1) {};
    \node(tll) at (130:1) {};
    \node(trr) at (50:1) {};
    \node(t) at (90:0.4) {};
    \node(bl) at (180:0.6) {};
    \node(br) at (0:0.6) {};
    \node(b) at (270:0.6) {};

    \draw[<-,\arrcol] (280:0.125) arc (-80:260:0.125);
    \draw[->,\arrcol] (b) -- (br);
    \draw[->,\arrcol] (bl) -- (b);
    \node at (0:0.25) {$a$};
    \node at (220:0.55) {$c$};
    \node at (320:0.55) {$b$};

    \draw[->,\arrcol] (tr) to [bend left] (tl);
    \draw[->,\arrcol] (t) -- (trr);
    \draw[->,\arrcol] (tll) -- (t);
    \node at (130:0.6) {$d$};
    \node at (50:0.6) {$e$};
    \node at (90:0.85) {$f$};
        
    \end{scope}

    \begin{scope}

    \draw[->] (-0.3,0.1) -- (0.3,0.1);
        
    \end{scope}

    \begin{scope}[shift={(1.75,0)}]
    
   	\path[use as bounding box] (-1.5,1.2) rectangle (1.5,-1.2);

    \draw[\edgecol] (1,0) arc (0:360:1);

    \node at (90:1) {$\bullet$};
    \node at (270:1) {$\bullet$};
    \node at (0,0) {$\bullet$};

    \node at (90:0.25) {$S_2 \xrightarrow{c} S_1$};
    \node at (287.5:0.75) {$S_1[1]$};
    \node at (120:1.2) {$S_4$};
    \node at (60:1.2) {$S_3$};
    
    \draw[\edgecol] (270:1) .. controls (22.5:2.25) and (157.5:2.25) .. (270:1);
    \draw[\edgecol] (0,0) -- (270:1);

    \node(tl) at (120:1) {};
    \node(tr) at (60:1) {};
    \node(tll) at (130:1) {};
    \node(trr) at (50:1) {};
    \node(t) at (90:0.4) {};
    \node(bl) at (180:0.6) {};
    \node(br) at (0:0.6) {};
    \node(b) at (270:0.6) {};

    \draw[<-,\arrcol] (280:0.125) arc (-80:260:0.125);
    \draw[->,\arrcol] (b) -- (br);
    \draw[->,\arrcol] (bl) -- (b);
    \node at (0:0.25) {$a$};
    \node at (220:0.65) {$\left(\begin{smallmatrix}0 & a^{\ast}\end{smallmatrix}\right)$};
    \node at (320:0.65) {$\left(\begin{smallmatrix}0 \\ \id_1\end{smallmatrix}\right)$};

    \draw[->,\arrcol] (tr) to [bend left] (tl);
    \draw[->,\arrcol] (t) -- (trr);
    \draw[->,\arrcol] (tll) -- (t);
    \node at (135:0.7) {$\left(\begin{smallmatrix}d \\ 0\end{smallmatrix}\right)$};
    \node at (40:0.7) {$\left(\begin{smallmatrix}e & 0\end{smallmatrix}\right)$};
    \node at (90:0.85) {$f$};
        
    \end{scope}
        
    \end{tikzpicture}
    \]
    \caption{Tilting at the self-folded edge of a self-folded triangle}
    \label{fig:enhance_equiv:pop}
\end{figure}

\begin{proof}[{Proof of Theorem~\ref{thm:enhance_equiv}}]
By Lemmas~\ref{lem:enhance_equiv:standard}, \ref{lem:enhance_equiv:monogon}, \ref{lem:enhance_equiv:two_cyc}, and~\ref{lem:enhance_equiv:self_fold}, we have that $\simpinf{\mathcal{H}_{\infty}^{S_{1}, \sharp}} \simeq \akd{Q'}{W'}$.
Hence, we firstly have quasi-equivalences $\mathcal{H}_\infty(Q',W') \simeq \twz{\simpinf{\mathcal{H}_{\infty}^{S_{1}, \sharp}}} \simeq \mathcal{H}_{\infty}^{S_{1}, \sharp}$.
Secondly, we have a quasi-equivalence $\mathcal{D}_\infty(Q',W') \simeq \tw{\simpinf{\mathcal{H}_{\infty}^{S_{1}, \sharp}}}$.
As can be seen from backwards tilting at $S_1[1]$, we have a quasi-equivalence $\tw{\simpinf{\mathcal{H}_{\infty}^{S_{1}, \sharp}}} \simeq \mathcal{D}_\infty(Q,W)$, and hence a quasi-equivalence $\mathcal{D}_\infty(Q,W) \simeq \mathcal{D}_\infty(Q',W')$.
The lemmas also establish the desired isomorphism of orientation data on $\qastn{Q', W'}$ by Proposition~\ref{prop:od_ug}.
\end{proof}

\section{Classifying configurations of finite-length trajectories}\label{sect:prelim}

In this section, we will classify all of the different situations in which finite-length trajectories appear for generic infinite GMN differentials, as laid out in Table~\ref{table}.
We will then study rotations of these differentials with finite-length trajectories to nearby trajectories which are saddle-free, and so have no finite-length trajectories.
This will be important later for computing the categories associated to the finite-length trajectories, since construction of the category associated to the surface uses saddle-free differentials.

\subsection{Finite-length trajectories of generic infinite GMN differentials}

Let $\varphi$ be an infinite GMN differential on a Riemann surface $X$.
This infinite GMN differential $\varphi$ is \emph{generic} if for any $\hhc{\gamma}_1, \hhc{\gamma}_2 \in \hhspan{\varphi}$ there is an implication \[\scph{\hhc{\gamma}_1}  = \scph{\hhc{\gamma}_2} \quad \Longrightarrow \quad \mathbb{Q}\{\hhc{\gamma}_1\} = \mathbb{Q}\{\hhc{\gamma}_2\} \subset \hhspan{\varphi} \otimes_{\mathbb{Z}} \mathbb{Q},\] that is, $\hhc{\gamma}_1$ and $\hhc{\gamma}_2$ are \emph{$\mathbb{Q}$-proportional}.
Here we use braces to indicate linear spans.
We are interested in generic differentials because they give well-defined DT invariants, since the pairing from Definition~\ref{def:dt} will vanish.
\emph{Generic} stability conditions are defined analogously.

The following result can be seen as a synthesis of \cite[Lemma~5.1]{bs}, in which there are no simple poles, and \cite[Proposition~5.4]{haiden}, in which there are no infinite critical points.
Our proof combines arguments from each of these.

\begin{proposition}\label{prop:fin_len_traj}
Suppose that $\varphi$ is a generic infinite GMN differential on a Riemann surface~$X$.
The union of all finite-length trajectories of $\varphi$ consists of one of the following.
\begin{enumerate}[label=\textup{(}\arabic*\textup{)}]
\item A single saddle trajectory of type~I.\label{op:fin_len_traj:single_i}
\item A single saddle trajectory of type~II.\label{op:fin_len_traj:single_ii}
\item A degenerate ring domain, with one boundary given by a double pole and the other by either\label{op:fin_len_traj:drd}
	\begin{enumerate}[label=\textup{(}\alph*\textup{)}]
	\item a closed saddle trajectory,\label{op:fin_len_traj:drd:single}
	\item two type~I saddle trajectories, which on their other side form the boundary of a torus whose interior is a spiral domain, or\label{op:fin_len_traj:drd:torus}
        \item a type~III saddle trajectory.\label{op:fin_len_traj:drd:iii}
    \end{enumerate}
\item A non-degenerate ring domain, with boundaries given by either\label{op:fin_len_traj:nrd}
	\begin{enumerate}[label=\textup{(}\alph*\textup{)}]
	\item a closed saddle trajectory for each boundary,\label{op:fin_len_traj:nrd:single}
	\item a closed saddle trajectory for one boundary, and for the other boundary two type~I saddle trajectories, which on their other side form the boundary of a torus whose interior is a spiral domain, or\label{op:fin_len_traj:nrd:torus}
        \item a closed saddle trajectory for one boundary, and a type~III saddle trajectory for the other boundary.\label{op:fin_len_traj:nrd:iii}
	\end{enumerate}
\end{enumerate}
In cases \ref{op:fin_len_traj:drd}\ref{op:fin_len_traj:drd:torus} and \ref{op:fin_len_traj:nrd}\ref{op:fin_len_traj:nrd:torus} where we have two type~I saddle trajectories bounding a torus, the two corresponding hat-homology classes are equal.
Moreover, these are the only two cases where spiral domains appear.
\end{proposition}

We illustrate the different configurations from Proposition~\ref{prop:fin_len_traj} in Figure~\ref{fig:configs}. We refer to ring domains of type~\ref{op:fin_len_traj:drd:single} as \emph{standard} ring domains, type~\ref{op:fin_len_traj:drd:torus} as \emph{toral} ring domains, and \ref{op:fin_len_traj:drd:iii} as \emph{type~III} ring domains.
Note that a standard, toral, or type~III degenerate ring domain may be obtained from a non-degenerate ring domain of the corresponding type by replacing the boundary which is a closed saddle trajectory by a double pole.
Note finally that toral and type~III degenerate ring domains form the whole surface and comprise only one polar type each, respectively the torus with a double pole and the sphere with a double pole and two simple poles.

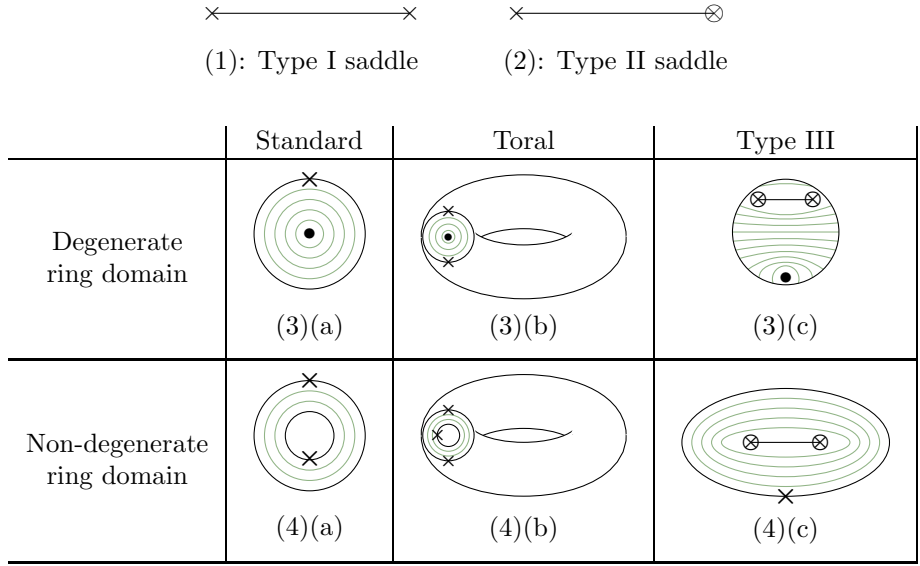
\begin{figure}
\[
    \begin{tikzpicture}

        \begin{scope}[shift={(-2,4.4)},scale=1.3]

            \node at (-1,0) {$\times$};
            \node at (1,0) {$\times$};

            \draw[\saddletraj] (-1,0) -- (1,0);
            
        \end{scope}

        \begin{scope}[shift={(2,4.4)},scale=1.3]

            \node at (-1,0) {$\times$};
            \node at (1,0) {$\otimes$};

            \draw[\saddletraj] (-1,0) -- (1,0);
            
        \end{scope}

        \node at (-2,3.75) {\ref{op:fin_len_traj:single_i}: Type~I saddle};
        \node at (2,3.75) {\ref{op:fin_len_traj:single_ii}: Type~II saddle};

    \node at (0,0) {
    \begin{tabular}{c|c|c|c|}
           & Standard & Toral & Type III \\
    \hline
       \parbox{2cm}{\centering Degenerate ring domain}  & 
       \begin{tabular}{c}
            \begin{tikzpicture}[scale=1.3]
            \node at (0,0) {$\bullet$};
            \draw[\closedtraj] (0,0) circle (4pt);
            \draw[\closedtraj] (0,0) circle (7pt);
            \draw[\closedtraj] (0,0) circle (10pt);
            \draw[\closedtraj] (0,0) circle (13pt);
            \draw[\saddletraj] (0,0) circle (16pt);
            \node at (0,0.56) {$\bm{\times}$};

            \node at (0,-0.95) {\ref{op:fin_len_traj:drd}\ref{op:fin_len_traj:drd:single}};
            
            \end{tikzpicture}
       \end{tabular}
       & \begin{tabular}{c}
            \begin{tikzpicture}

                \begin{scope}[yscale=cos(62),scale=1.1]
                \node at (0,1.75) {};
                \draw[double distance=7mm] (0:0.9) arc (0:181:0.9);
                \draw[double distance=7mm] (180:0.9) arc (180:361:0.9);    
                \end{scope}

                \begin{scope}[scale=0.6,shift={(-1.65,0)},font=\tiny]
                \filldraw[white] (0,0) circle (16pt);
                \node at (0,0) {$\bullet$};
                \draw[\closedtraj] (0,0) circle (4pt);
                \draw[\closedtraj] (0,0) circle (8pt);
                \draw[\closedtraj] (0,0) circle (12pt);
                \draw[\saddletraj] (0,0) circle (16pt);
                \node at (0,0.56) {$\bm{\times}$};
                \node at (0,-0.56) {$\bm{\times}$};
                \end{scope}

                \begin{scope}[scale=0.6,shift={(1.65,0)},font=\tiny]
                \filldraw[white] (0,0) circle (16pt);
                \end{scope}

                \node at (0,-1.2) {\ref{op:fin_len_traj:drd}\ref{op:fin_len_traj:drd:torus}};
           
            \end{tikzpicture}
       \end{tabular}
       & \begin{tabular}{c}
            \begin{tikzpicture}[scale=1.4]
            \node at (0,0.6) {};
            \coordinate(sp1) at (130:0.4);
            \node at (sp1) {\scriptsize $\bm{\otimes}$};
            \coordinate(sp2) at (50:0.4);
            \node at (sp2) {\scriptsize $\bm{\otimes}$};
            \coordinate(dp) at (270:0.435);
            \node at (dp) {$\bullet$};
            \draw[\saddletraj] (sp1) -- (sp2);
            \draw[\closedtraj] (120:0.5) to [out=10,in=170] (60:0.5);
            \draw[\closedtraj] (150:0.5) to [out=-20,in=200] (30:0.5);
            \draw[\closedtraj] (160:0.5) to [out=-10,in=190] (20:0.5);
            \draw[\closedtraj] (170:0.5) to [out=-5,in=185] (10:0.5);
            \draw[\closedtraj] (180:0.5) -- (0:0.5);
            \draw[\closedtraj] (190:0.5) to [out=5,in=175] (-10:0.5);
            \draw[\closedtraj] (200:0.5) to [out=10,in=170] (-20:0.5);
            \draw[\closedtraj] (210:0.5) to [out=15,in=165] (-30:0.5);
            \draw[\closedtraj] (225:0.5) to [out=35,in=145] (-45:0.5);
            \draw[\closedtraj] (240:0.5) to [out=90,in=90] (-60:0.5);
            \draw[\closedtraj] ($(-90:0.445)+(-20:0.125)$) arc (-20:200:0.125);
            \draw (0,0) circle (0.5cm);

            \node at (0,-0.9) {\ref{op:fin_len_traj:drd}\ref{op:fin_len_traj:drd:iii}};
                    
            \end{tikzpicture}
       \end{tabular} \\
       \hline
       \parbox{2.5cm}{\centering Non-degenerate ring domain}
        & \begin{tabular}{c}
            \begin{tikzpicture}[scale=1.3]

            \draw[\saddletraj] (0,0) circle (7pt);
            \draw[\closedtraj] (0,0) circle (10pt);
            \draw[\closedtraj] (0,0) circle (13pt);
            \draw[\saddletraj] (0,0) circle (16pt);
            \node at (0,0.56) {$\bm{\times}$};
            \node at (0,-0.236) {$\bm{\times}$};

            \node at (0,-0.95) {\ref{op:fin_len_traj:nrd}\ref{op:fin_len_traj:nrd:single}};
           
            \end{tikzpicture}
        \end{tabular}
       & \begin{tabular}{c}
           \begin{tikzpicture}

                \begin{scope}[yscale=cos(63),scale=1.1]
                \node at (0,1.75) {};
                \draw[double distance=7mm] (0:0.9) arc (0:181:0.9);
                \draw[double distance=7mm] (180:0.9) arc (180:361:0.9);    
                \end{scope}

                \begin{scope}[scale=0.6,shift={(-1.65,0)},font=\tiny]
                \filldraw[white] (0,0) circle (16pt);
                \draw[\saddletraj] (0,0) circle (7pt);
                \draw[\closedtraj] (0,0) circle (10pt);
                \draw[\closedtraj] (0,0) circle (13pt);
                \draw[\saddletraj] (0,0) circle (16pt);
                \node at (0,0.56) {$\bm{\times}$};
                \node at (0,-0.56) {$\bm{\times}$};
                \node at (-0.236,0) {$\bm{\times}$};
                \end{scope}

                \begin{scope}[scale=0.6,shift={(1.65,0)},font=\tiny]
                \filldraw[white] (0,0) circle (16pt);
                \end{scope}

                \node at (0,-1.25) {\ref{op:fin_len_traj:nrd}\ref{op:fin_len_traj:nrd:torus}};
           
            \end{tikzpicture}
       \end{tabular}
       & \begin{tabular}{c}
            \begin{tikzpicture}[scale=1.3]
                \draw[\saddletraj] (-0.35,0) -- (0.35,0);
                \node at (-0.35,0) {\scriptsize $\bm{\otimes}$};
                \node at (0.35,0) {\scriptsize $\bm{\otimes}$};
                \draw[\closedtraj] (0,0) ellipse (0.65cm and 0.15cm);
                \draw[\closedtraj] (0,0) ellipse (0.75cm and 0.25cm);
                \draw[\closedtraj] (0,0) ellipse (0.85cm and 0.35cm);
                \draw[\closedtraj] (0,0) ellipse (0.95cm and 0.45cm);
                \draw[\saddletraj] (0,0) ellipse (1.05cm and 0.55cm);
                \node at (0,-0.55) {$\bm{\times}$};
                \node at (0,0.75) {};

                \node at (0,-0.9) {\ref{op:fin_len_traj:nrd}\ref{op:fin_len_traj:nrd:iii}};
            \end{tikzpicture}
       \end{tabular} \\
       \hline
    \end{tabular}
    
    };
        
    \end{tikzpicture}
\]
    \caption{Configurations from Proposition~\ref{prop:fin_len_traj}}
    \label{fig:configs}
\end{figure}

We split parts of the proof of Proposition~\ref{prop:fin_len_traj} into the following lemmas.
We will frequently make use of the intersection and Lefschetz pairings from Section~\ref{sect:qd:pairings}.

\begin{lemma}\label{lem:gen_sit:co_conn}
Suppose that $\varphi$ is a generic infinite GMN differential on a Riemann surface~$X$.
If $\gamma_1$ and $\gamma_2$ are two distinct saddle trajectories such that $\gamma_1$ is closed and $\gamma_2$ is not, then it is not possible for $\gamma_1 \cup \gamma_2$ to be connected.
\end{lemma}
\begin{proof}
If we assume that $\gamma_1 \cup \gamma_2$ is connected, then one endpoint of $\gamma_2$ must be the zero incident to~$\gamma_1$.
There are now two cases, depending upon whether the other endpoint of $\gamma_2$ is a zero or a simple pole.

We first consider the case where the other endpoint of $\gamma_2$ is a simple pole~$s$.
It follows from the trajectory structures that there is then a ring domain $R$ with a boundary given by the sequence $\gamma_1\gamma_2\gamma_2$.
If the ring domain $R$ is non-degenerate, then one can choose a path $\alpha$ from the simple pole $s$ incident to $\gamma_2$ to a finite critical point in the other boundary of~$R$.
The ring domain $R$ cannot be the same as the ring domain on the other side of $\gamma_1$, since the hat-homology classes corresponding to $\gamma_1$ and $\gamma_1\gamma_2\gamma_2$ are not the same.
Hence, lifting $\alpha$ to its pre-image $\hhc{\alpha}$ on the spectral cover, we obtain that $\intp{\hhc{\gamma}_2}{\hhc{\alpha}} = \pm 1$ but $\intp{\hhc{\gamma}_1}{\hhc{\alpha}} = 0$, contradicting the $\mathbb{Q}$-proportionality of $\gamma_1$ and~$\gamma_2$.

Hence, suppose that the ring domain $R$ is degenerate with double pole $p$.
Then one may take a path $\alpha$ from $p$ to itself which loops around the simple pole~$s$, and so intersects~$\gamma_2$, but which does not intersect~$\gamma_1$.
If we let $\hhc{\alpha}$ be one of the two pre-images of $\alpha$ on the spectral cover, then we have the Lefschetz pairings $\lef{\hhc{\alpha}}{\hhc{\gamma}_1} = 0$ and $\lef{\hhc{\alpha}}{\hhc{\gamma}_2} = \pm 1$.
This then contradicts the $\mathbb{Q}$-proportionality of $\hhc{\gamma}_1$ and~$\hhc{\gamma}_2$.
Hence the other endpoint of $\gamma_2$ cannot be a simple pole.

We now consider the case where the other endpoint of $\gamma_2$ is a zero $z$.
We first suppose that $\gamma_1$ separates $X$ into two connected components.
Consider the connected component $Y$ containing the interior of~$\gamma_2$.
If $Y$ contains a finite critical point $r \neq z$, then one can connect $r$ to $z$ via a path $\alpha$ and obtain a contradiction via the intersection pairing.
If $Y$ contains an infinite critical point $p$, then one can take a path $\beta$ which loops around the zero $z$ and returns to $p$, and so obtain a contradiction via the Lefschetz pairing.
Hence, $Y$ contains no critical points besides~$z$.
If we double $Y$ along a closed trajectory in the ring domain incident to $\gamma_1$, then we obtain a quadratic differential with four zeros, coming from the two zeros incident to~$\gamma_2$.
Here, by doubling we mean cutting along the closed trajectory, taking the connected component containing $Y$ and gluing on its mirror image along the closed trajectory, see \cite[Section~4.4]{strebel} or \cite[proof of Lemma~5.1]{bs}.
By~\eqref{eq:rr}, the resulting surface must be genus~2, which gives that $Y$ is a torus.
Using an argument from \cite[proof of Lemma~5.1]{bs}, in this situation one can replace $X \setminus Y$ with a degenerate ring domain, obtaining a quadratic differential on the torus as shown in \cite[right-hand part of Figure~20]{bs}, but this contradicts \cite[Example~4.10]{bs}.

We now suppose that $\gamma_1$ does not separate $X$ into two connected components.
There is a ring domain $R$ incident to~$\gamma_1$.
Moreover, $R$ must be non-degenerate, since otherwise $\gamma_1$ does separate~$X$.
On the other boundary of $R$, there must be a finite critical point~$r$.
As $\gamma_1$ does not separate $X$, one can find a path $\alpha$ from $r$ to $z$ which avoids $\gamma_1$, recalling that $z$ is the zero which gives the other endpoint of~$\gamma_2$.
This then gives a contradiction via the intersection pairing.

In every case we have obtained a contradiction, which establishes the lemma.
\end{proof}

\begin{lemma}\label{lem:gen_sit}
Suppose that $\varphi$ is a generic infinite GMN differential on a Riemann surface~$X$ and let $\gamma_1$ and $\gamma_2$ be two distinct saddle trajectories. Then either:
\begin{enumerate}[label=\textup{(}\alph*\textup{)}]
\item Both $\gamma_1$ and $\gamma_2$ are closed, and $\varphi$ has a standard ring domain with $\gamma_1$ as one boundary and $\gamma_2$ as the other boundary. \label{lem:gen_sit:cc}
\item Both $\gamma_1$ and $\gamma_2$ are not closed, and $\varphi$ has a toral ring domain with $\gamma_1$ and $\gamma_2$ the two type~I saddle trajectories. \label{lem:gen_sit:oo}
\item $\gamma_1$ is closed and $\gamma_2$ is not, and $\varphi$ has a non-degenerate ring domain, with $\gamma_1$ forming one boundary, which is either\label{lem:gen_sit:co_nconn}
    \begin{enumerate}[label=\textup{(}\roman*\textup{)}]
    \item a toral ring domain, with $\gamma_2$ one of the two type~I saddle trajectories, or\label{op:gen_sit:co_nconn:toral}
    \item a type~III ring domain, with $\gamma_{2}$ the type~III saddle trajectory.\label{op:gen_sit:co_nconn:iii}
    \end{enumerate}
\end{enumerate}
\end{lemma}
\begin{proof}
Cases \ref{lem:gen_sit:cc} and \ref{lem:gen_sit:oo} follow from similar arguments to \cite[Lemma~5.1, Case~(1) and Case~(2)]{bs} respectively.
In Case \ref{lem:gen_sit:oo}, instead of showing that the torus contains no zeros, one must show that it contains no finite critical points, but this is established in the same way.

Suppose then that $\gamma_1$ is closed and $\gamma_2$ is not. By Lemma~\ref{lem:gen_sit:co_conn}, we have that $\gamma_1 \cup \gamma_2$ is not connected.
If $\gamma_1$ does not separate $X$ into two connected components, then there is a path from an infinite critical point to an infinite critical point which only meets $\gamma_1$ once and is disjoint from $\gamma_2$, which contradicts $\mathbb{Q}$-proportionality via the Lefschetz pairing.
The same argument applies if $\gamma_{1}$ separates $X$ into subsurfaces which each contain infinite critical points.
Hence, we may assume that $\gamma_{1}$ bounds a subsurface $Y$ containing no infinite critical points.

If $\gamma_{2}$ lies inside $X \setminus Y$, then we may take a path from an infinite critical point $p$ which loops around one endpoint of $\gamma_{2}$ and then returns to~$p$, whilst being entirely disjoint from $\gamma_{1}$.
Using the Lefschetz pairing argument, we have that $\hhc{\gamma}_{1}$ and $\hhc{\gamma}_{2}$ would not be $\mathbb{Q}$-proportional in this case.
We hence conclude that $\gamma_{2}$ lies inside the subsurface $Y$ bound by $\gamma_{1}$.

Suppose for contradiction that $Y$ contains a point $r$ which is a finite critical point of $\varphi$ not lying on either $\gamma_{1}$ or $\gamma_{2}$.
In this case, we may take a closed path $\alpha$ which starts at an infinite critical point of~$\varphi$, crosses $\gamma_{1}$, loops around $r$, and then returns parallel to itself to the pole it started from.
Again, the Lefschetz pairing argument gives a contradiction to $\mathbb{Q}$-proportionality here.

We conclude that $Y$ contains no infinite critical points, but contains $\gamma_2$, on which all of the finite critical points of $Y$ lie.
We now observe that the curve $\gamma_{1}$ is the boundary of a ring domain~$R$, being a closed saddle trajectory.
We must either have that $R$ lies inside $Y$ or inside $X \setminus Y$.

We first consider the case where $R$ lies inside $Y$, in which case $R$ must be non-degenerate, since $Y$ does not contain any infinite critical points.
Since all finite critical points of $Y$ lie either on $\gamma_1$ or~$\gamma_2$, we have that $\gamma_2$ must form part of the other boundary of~$R$.
If $\gamma_2$ is a type~I saddle trajectory, then doubling the surface along a closed trajectory in~$R$, we obtain a surface with a quadratic differential with four zeros and no poles, which must be a genus two surface by~\eqref{eq:rr}.
We conclude that the other side of $\gamma_2$ from $R$ must be a torus with no critical points, as in the statement of the lemma.
The fact that this torus must contain a spiral domain follows from \cite[comment after Lemma~5.1]{bs}.
This then gives case~\ref{op:gen_sit:co_nconn:toral} in the statement.
If $\gamma_2$ is a type~II saddle trajectory, then $\gamma_2$ must be incident to some other saddle trajectories to form the boundary of~$R$, but this contradicts Lemma~\ref{lem:gen_sit:co_conn} if these saddle trajectories consist of a single closed saddle trajectory and case~\ref{lem:gen_sit:oo} if they consist of non-closed saddle trajectories.
The case where $\gamma_2$ is a type~III saddle trajectory is case~\ref{op:gen_sit:co_nconn:iii} in the statement.

We now consider the case where $R$ lies inside $X \setminus Y$.
If we double along a closed trajectory of $R$ we obtain the following cases, since $\gamma_1$ and $\gamma_2$ are incident to all of the finite critical points in $Y$.
If $\gamma_2$ is a type~I saddle trajectory, then we obtain a surface with a quadratic differential with six zeros, coming from the two zeros on $\gamma_{2}$ and the one zero on~$\gamma_{1}$.
If $\gamma_2$ is a type~II saddle trajectory, then we obtain a surface with a quadratic differential with four zeros and two simple poles.
And, finally, if $\gamma_2$ is a type~III saddle trajectory, then we obtain a surface with a quadratic differential with two zeros and four simple poles.
However, these all contradict the restriction~\eqref{eq:rr}.
\end{proof}

We can now prove Proposition~\ref{prop:fin_len_traj}.

\begin{proof}[Proof of Proposition~\ref{prop:fin_len_traj}]
We first suppose that $\varphi$ has only one saddle trajectory~$\gamma$.
Note that if $\varphi$ has any finite-length trajectories then it must have at least one saddle trajectory, since the boundary of a ring domain must be formed of saddle trajectories.
If $\gamma$ is not closed, then it is of type~I, II, or~III.
If $\gamma$ is of type~I or~II, then we have case~\ref{op:fin_len_traj:single_i}, or case~\ref{op:fin_len_traj:single_ii}, respectively, noting that a single saddle trajectory cannot form the boundary of a ring domain.
If $\gamma$ is a type~III saddle trajectory, then it follows from the trajectory structures that it must be surrounded by a ring domain.
This ring domain must be degenerate, since a non-degenerate ring domain would have saddle trajectories in its other boundary.
This then gives us case~\ref{op:fin_len_traj:drd}\ref{op:fin_len_traj:drd:iii}.

If $\varphi$ has a single saddle trajectory $\gamma$ which is closed, then this must form the boundary of a ring domain.
A single closed saddle trajectory cannot form both boundaries of a non-degenerate ring domain, due to the third trajectory emanating from the zero, so the ring domain must be degenerate, which gives us case~\ref{op:fin_len_traj:drd}\ref{op:fin_len_traj:drd:single}.

Hence, we will now assume that $\varphi$ has at least two saddle trajectories $\gamma_{1}$ and~$\gamma_{2}$.
We split into cases according to how many of $\gamma_{1}$ and $\gamma_{2}$ are closed and apply the preceding lemmas.

If $\gamma_1$ and $\gamma_2$ are both closed, then we are in the situation of Lemma~\ref{lem:gen_sit}\ref{lem:gen_sit:cc}, which gives us a non-degenerate ring domain with boundaries $\gamma_{1}$ and $\gamma_{2}$, as in case~\ref{op:fin_len_traj:nrd}\ref{op:fin_len_traj:nrd:single}.
We then cannot have that $\varphi$ has another saddle trajectory~$\gamma_{3}$.
If such a saddle trajectory $\gamma_{3}$ were closed, one could apply Lemma~\ref{lem:gen_sit}\ref{lem:gen_sit:cc} to conclude that $\gamma_{3}$ has to bound a non-degenerate ring domain with each of $\gamma_{1}$ and $\gamma_{2}$, which is impossible.
If such a saddle trajectory $\gamma_{3}$ were not closed, then one could apply Lemma~\ref{lem:gen_sit:co_conn} and Lemma~\ref{lem:gen_sit}\ref{lem:gen_sit:co_nconn} to conclude that both $\gamma_{1}$ and $\gamma_{2}$ have to bound non-degenerate ring domains with $\gamma_{3}$ in the other boundary, but this is similarly impossible.
We moreover conclude that $\varphi$ cannot have any more finite-length trajectories, since this would require having another saddle trajectory.

If $\gamma_1$ is closed whilst $\gamma_2$ is not, then Lemma~\ref{lem:gen_sit:co_conn} gives us that $\gamma_{1} \cup \gamma_{2}$ cannot be connected.
By Lemma~\ref{lem:gen_sit}\ref{lem:gen_sit:co_nconn}, we obtain that there is a non-degenerate ring domain with $\gamma_1$ forming one boundary, and with the other boundary either case~\ref{op:fin_len_traj:nrd}\ref{op:fin_len_traj:nrd:torus} with $\gamma_2$ one of the two type~I saddle trajectories adjoining the torus, or case~\ref{op:fin_len_traj:nrd}\ref{op:fin_len_traj:nrd:iii} with $\gamma_2$ the type~III saddle trajectory.
Just as in the case where both $\gamma_{1}$ and $\gamma_{2}$ are closed, one can argue that there cannot be any other finite-length trajectories outside the ring domain and its boundaries.

If neither $\gamma_1$ nor $\gamma_2$ is closed, then by Lemma~\ref{lem:gen_sit}\ref{lem:gen_sit:oo}, $\gamma_{1} \cup \gamma_{2}$ is the boundary of a toral ring domain.
If this ring domain is degenerate, then we are in case~\ref{op:fin_len_traj:drd}\ref{op:fin_len_traj:drd:torus}.
If the ring domain is non-degenerate, then we must consider the options for the other boundary.
If the other boundary is a closed saddle trajectory, then we are in case~\ref{op:fin_len_traj:nrd}\ref{op:fin_len_traj:nrd:torus}.
The other boundary cannot be a type~III saddle trajectory, since then $\varphi$ would have no infinite critical points.
Lemma~\ref{lem:gen_sit:co_conn} gives us that the other boundary cannot be formed of a union of both closed trajectories and open trajectories.
Applying Lemma~\ref{lem:gen_sit}\ref{lem:gen_sit:oo} again gives us that if the other boundary is formed of several non-closed saddle trajectories, then the other boundary must also consist of two type~I saddle trajectories adjoining a torus with a spiral domain.
However, $\varphi$ would then again have no infinite critical points.
We conclude that the only possibilities are case~\ref{op:fin_len_traj:drd}\ref{op:fin_len_traj:drd:torus} and case~\ref{op:fin_len_traj:nrd}\ref{op:fin_len_traj:nrd:torus}.
As in the other cases, these ring domains and their boundaries must comprise the union of all finite-length trajectories.
The claim that the hat-homology classes of the type~I saddle trajectories $\gamma_{1}$ and $\gamma_{2}$ in \ref{op:fin_len_traj:drd}\ref{op:fin_len_traj:drd:torus} and \ref{op:fin_len_traj:nrd}\ref{op:fin_len_traj:nrd:torus} must be equal follows just as in \cite[Lemma~5.1]{bs}.

Finally, we show that cases \ref{op:fin_len_traj:drd}\ref{op:fin_len_traj:drd:torus} and \ref{op:fin_len_traj:nrd}\ref{op:fin_len_traj:nrd:torus} are the only two cases in which spiral domains appear.
The boundary of a spiral domain must be a union of saddle trajectories, as per Section~\ref{sect:qd:domains}.
A saddle trajectory of type~I or II cannot form the entire boundary itself, so spiral domains cannot appear in cases~\ref{op:fin_len_traj:single_i} or~\ref{op:fin_len_traj:single_ii}.
The only case remaining to consider is whether a closed saddle trajectory can form the boundary of a spiral domain.
A closed saddle trajectory is always incident to a ring domain, which must lie on the other side of the saddle trajectory to the spiral domain.
Lemma~\ref{lem:gen_sit}\ref{lem:gen_sit:cc} and Lemma~\ref{lem:gen_sit}\ref{lem:gen_sit:co_nconn} give us that the spiral domain cannot have any other boundaries besides the closed saddle trajectory.
Hence, doubling the surface along a closed trajectory in the ring domain would give a surface with a quadratic differential with two zeros and no poles, which contradicts \eqref{eq:rr}.
\end{proof}

\subsection{Effect of rotations}

It will be important to understand the effect of rotation on the configuration of trajectories.
First, we need to know that separating trajectories persist to nearby rotations of the differential.
This is given by the following proposition, which is a special case of \cite[Proposition~4.8]{bs}.

\begin{proposition}[{\cite[Proposition~4.8]{bs}}]\label{prop:sep_traj_persist}
Suppose that $\gamma_{0} \colon [0, \infty) \to \nocrit{X}$ is a separating trajectory for an infinite GMN differential $\varphi_{0}$, which starts at a finite critical point $r \in X$ and limits to an infinite critical point $p \in X$.
Then there exists $\varepsilon' > 0$ and a family of curves $\gamma_{\varepsilon} \colon [0, \infty) \to \nocrit{X}$ for $\varepsilon \in (-\varepsilon', \varepsilon')$ varying continuously with $\varepsilon$ such that $\gamma_{\varepsilon}$ is a separating trajectory for $\varphi^{(\varepsilon)}$ which starts at $r$ and limits to~$p$.
\end{proposition}

As for the effect of rotation on saddle trajectories, there are two different types of cases which appear.
Either the saddle trajectories occur at a phase which is a limit of phases which also contain saddle trajectories, or they occur at a phase such that all small rotations are saddle-free.
It will be useful to record this fact.
In the case where there are no simple poles, Proposition~\ref{prop:limit_walls}\ref{op:no_lim} follows from the results of \cite[Section~5.3 and 5.4]{bs}.

\begin{proposition}\label{prop:limit_walls}
Let $\varphi$ be a generic infinite GMN differential on a Riemann surface~$X$.
Suppose that $\varphi$ has a saddle trajectory~$\gamma$.
There are then two cases.
\begin{enumerate}[label=\textup{(}\arabic*\textup{)}]
    \item There exists a sequence $(\varepsilon_k)_{k \in \mathbb{Z}_{\geqslant 0}}$ such that $\varepsilon_k \to 0$ as $k \to \infty$ and such that $\varphi^{(\varepsilon_{k})}$ has a saddle trajectory for all~$k$.\label{op:lim}
    \item There exists $\varepsilon' > 0$ such that for all $\varepsilon' > \varepsilon > 0$, we have that $\varphi^{(\pm\varepsilon)}$ is saddle-free.\label{op:no_lim}
\end{enumerate}
Case~\ref{op:lim} occurs if and only if $\gamma$ lies in the boundary of a non-degenerate ring domain or spiral domain of $\varphi$, otherwise case~\ref{op:no_lim} occurs. In case~\ref{op:no_lim}, $\gamma$ is a standard saddle connection in $\varphi^{(\pm \varepsilon)}$.
\end{proposition}
\begin{proof}
Between them the two cases are clearly exhaustive, and by \cite[Proposition~5.10]{tahar} we have that case~\ref{op:lim} occurs if and only if $\gamma$ lies in the boundary of a non-degenerate ring or spiral domain of~$\varphi$.

For the last statement, let $r_1$ and $r_2$ denote the finite critical points at the ends of~$\gamma$, which are not necessarily distinct.
After rotating~$\varphi$, there are separating trajectories from $r_1$ and $r_2$ which collide to form $\gamma$ in~$\varphi$.
Since all smaller rotations are saddle-free, these separating trajectories do not pass through any finite critical points if one rotates back to~$\varphi$.
Therefore $\gamma$ is now a saddle connection contained in the domain with these separating trajectories in the boundary.
The only possibility is thus that $\gamma$ is a standard saddle connection crossing a horizontal strip.
\end{proof}

\section{Donaldson--Thomas invariants for finite-length trajectories}\label{sect:main}

In this section, we derive the DT invariants for the different types of finite-length trajectories by considering the different configurations from Proposition~\ref{prop:fin_len_traj}.
This is a two-step process, where the first step is to obtain a quiver with potential and stability condition describing the category of semistable objects of the same phase as the finite-length trajectories.
One then obtains the DT generating series for the resulting quiver with potential and stability condition.
By Definition~\ref{def:dt}, the values of the DT invariants follow from expressing this series as a particular product of quantum dilogarithms, such as they are expressed in Table~\ref{table}.

There are three different sorts of cases.
In the first and simplest sort, one can use Proposition~\ref{prop:limit_walls} to rotate to obtain the relevant finite-length connection as a standard saddle connection.
This gives the category of semistable objects as representations concentrated at a single vertex of the quiver of the rotated differential.
The cases where a single saddle trajectory appears are of this sort.
In the second and third cases, we are in the limiting case of Proposition~\ref{prop:limit_walls}.
In these, we cannot obtain the relevant finite-length connection as a standard saddle connection, but we can still make a sufficiently small rotation to determine the quiver.
In the second sort of case, we describe the category of semistable objects in terms of a subquiver of this quiver. %
The non-degenerate ring domains which are not toral are of this sort.
Finally, in the third sort of case, we describe the category of semistable objects in terms of a quiver which is not a subquiver of that of the rotated differential. %
The toral ring domains are of this sort; the detail of the computations of these categories is carried out in Section~\ref{sect:spiral}.

\subsection{Type~I and II saddle trajectories}

We begin by considering the cases of type~I and type~II saddle trajectories.

\subsubsection{Type I case}\label{sect:main:i} 

We first consider case~\ref{op:fin_len_traj:single_i} from Proposition~\ref{prop:fin_len_traj}, given by a single type~I saddle trajectory.
This case is straightforward, and already covered by \cite{bs}, but it is instructive to consider it nonetheless.

\begin{proposition}\label{prop:i}
Suppose that $\varphi$ is a generic infinite GMN differential on a Riemann surface $X$ with a single type~I saddle trajectory $\gamma$.
Then the category $\ssc{\sigma_{\varphi}}{1}$ of $\sigma_{\varphi}$-semistable representations of phase~$1$ is given by representations of the quiver 
\[\bullet.\]
Hence, the DT generating series of the category $\ssc{\sigma_{\varphi}}{1}$ is \[
\qdl{t^{\hhc{\gamma}}},
\]
and the refined DT invariants are $\dtref{\hhc{\gamma}} = 1$.
\end{proposition}
\begin{proof}
We have by Proposition~\ref{prop:limit_walls} that there is some $\varepsilon > 0$ such that $\varphi^{(\varepsilon)}$ is saddle-free and has $\gamma$ as a standard saddle connection in a horizontal strip~$H$.
Note that, due to the direction we have rotated, $\gamma$ gives the same hat-homology class for both $\varphi$ and $\varphi^{(\varepsilon)}$, rather than differing by a sign; see the construction in Section~\ref{sect:qd:ssc}.
Since $\gamma$ is a type~I saddle trajectory, $H$ has a distinct zero on each of its boundaries at the ends of~$\gamma$.
By the isomorphism between $\hhspan{\varphi^{(\varepsilon)}}$ and $\mathbb{Z}^{Q_{0}(T_{\varphi^{(\varepsilon)}})}$ from Construction~\ref{const:bs}, the class $\hhc{\gamma}$ corresponds to $\dimu S_i$, where $i$ is the vertex corresponding to $H$.
By the construction of $(Q(T_{\varphi^{(\varepsilon)}}), W(T_{\varphi^{(\varepsilon)}}))$ in Section~\ref{sect:3cy_stab:qwp_from_triang}, $i$ has no loops at it.
Hence, $\ssc{\sigma_{\varphi^{(\varepsilon)}}}{\scph{\hhc{\gamma}}}$ is given by representations of the quiver $Q(T_{\varphi^{(\varepsilon)}})$ concentrated at~$i$.
We then have $\genser{\bullet, 0} = \qdl{t}$; in particular, see Proposition~\ref{prop:haiden}.
\end{proof}

\subsubsection{Type~II case}\label{sect:main:ii}

We now consider case~\ref{op:fin_len_traj:single_ii} of Proposition~\ref{prop:fin_len_traj}, given by type~II saddle trajectories, which follows similarly to the previous.

\begin{proposition}\label{prop:ii}
Suppose that $\varphi$ is a generic infinite GMN differential on a Riemann surface $X$ with a single type~II saddle trajectory $\gamma$.
The category $\ssc{\sigma_{\varphi}}{1}$ is given by representations of the quiver \[\begin{tikzcd}\bullet \ar[loop left,"a"]\end{tikzcd}\] with potential $W = a^3$.
Hence, the DT generating series of $\ssc{\sigma_{\varphi}}{1}$ is \[
\qdl{t^{\hhc{\gamma}}}^2,
\]
and the refined DT invariants are $\dtref{\hhc{\gamma}} = 2$.
\end{proposition}
\begin{proof}
As above, we have for some $\varepsilon > 0$ that $\varphi^{(\varepsilon)}$ is saddle-free and has $\gamma$ as a standard saddle connection in a horizontal strip~$H$.
Since $\gamma$ is a type~II saddle trajectory, $H$ has a zero on one of its boundaries and a simple pole on the other.
Hence, the edge corresponding to $H$ in the triangulation $T_{\varphi^{(\varepsilon)}}$ is the edge of the monogon surrounding the orbifold point corresponding to the simple pole.
The class $\hhc{\gamma}$ thus corresponds to $\dimu S_i$ for $i$ a vertex with a loop $a$, with the term $a^3$ in the potential.
One then uses Proposition~\ref{prop:haiden}.%
\end{proof}

\subsection{Degenerate ring domains}\label{sect:main:drd}

We now treat the three different types of degenerate ring domain from Proposition~\ref{prop:fin_len_traj} in turn. In each case, we denote this ring domain by $R$.

\subsubsection{Standard case}\label{sect:main:drd:i}

We first consider case~\ref{op:fin_len_traj:drd}\ref{op:fin_len_traj:drd:single} from Proposition~\ref{prop:fin_len_traj}, where the boundary of $R$ is a single closed saddle trajectory.
As the name suggests, this is the typical case of a degenerate ring domain, since the other cases \ref{op:fin_len_traj:drd:torus} and \ref{op:fin_len_traj:drd:iii} only comprise one polar type each.

\begin{proposition}\label{prop:drd_i}
Suppose that $\varphi$ is a generic infinite GMN differential on a Riemann surface $X$ with a standard degenerate ring domain with the boundary closed saddle denoted $\gamma$.
The category $\ssc{\sigma_{\varphi}}{1}$ is given by nilpotent representations of the quiver \[\begin{tikzcd}\bullet \ar[loop left]\end{tikzcd}\] with zero potential.
Hence, the DT generating series of $\ssc{\sigma_{\varphi}}{1}$ is \[
\qdl{-q^{-1/2}t^{\hhc{\gamma}}}^{-1},
\]
and the refined DT invariants are $\dtref{\hhc{\gamma}} = q^{-1/2}$.
\end{proposition}
\begin{proof}
Again, we have for some $\varepsilon > 0$ that $\varphi^{(\varepsilon)}$ is saddle-free and has $\gamma$ as a standard saddle connection in a horizontal strip~$H$.
Since $\gamma$ is closed, $H$ is \emph{degenerate}, meaning that it is folded so that the same zero occurs on each of its boundaries; see \cite[Figure~12]{bs} for an illustration.
Hence, two of the three domains incident to the zero are actually both $H$; in terms of the associated triangulation, this means $H$ corresponds to the self-folded edge of a self-folded triangle.
Noting Figure~\ref{fig:three_triangles}, the class $\hhc{\gamma}$ corresponds to a vertex of the quiver which has a loop whose powers do not lie in the potential, and the generating series thus follows from Proposition~\ref{prop:haiden}.%
\end{proof}

\subsubsection{Toral case}\label{sect:main:drd:ii}

We now come to case~\ref{op:fin_len_traj:drd}\ref{op:fin_len_traj:drd:torus} from Proposition~\ref{prop:fin_len_traj}, the toral degenerate ring domain.
These are precisely the quadratic differentials of polar type $\{-2\}$, which do not fall under \cite[Theorem~5.4]{chq}.
However, for saddle-free differentials $\varphi$ of this polar type, one can still construct the corresponding stability condition~$\sigma_{\varphi}$, as discussed in Section~\ref{sect:3cy_stab:bs}.

Hence, instead of choosing a generic infinite GMN differential with the relevant configuration of finite-length trajectories as we do in other sections, here we choose a saddle-free differential whose rotation has a toral degenerate ring domain.
The detail of the calculation of the corresponding category of semistable objects is carried in out in Section~\ref{sect:spiral:drd}.

\begin{proposition}\label{prop:drd_ii}
Suppose that $\varphi$ is a saddle-free generic infinite GMN differential on the torus $X$ such that for some $0 < \varepsilon < 1$, $\varphi^{(\varepsilon)}$ has a toral degenerate ring domain, with bounding type~I saddles of class~$\hhc{\gamma}$.
We have that $\ssce{\sigma_{\varphi}}{1 - \varepsilon}$ is quasi-equivalent to $\mathcal{H}_\infty(Q',0)$ for the quiver $Q'$ \[\begin{tikzcd}\overset{1}{\bullet} \ar[r,shift left] & \overset{2}{\bullet}. \ar[l,shift left]\end{tikzcd}\]
There is moreover an isomorphism between the canonical choices of orientation data on the stacks of objects in each of these categories.

It follows that the DT generating series of the category $\ssc{\sigma_{\varphi}}{1 - \varepsilon}$ is \[
\qdl{t^{\hhc{\gamma}}}^2\qdl{-q^{-1/2}t^{2\hhc{\gamma}}}^{-1},
\]
so that the refined DT invariants are $\dtref{\hhc{\gamma}} = 2$ and $\dtref{2\hhc{\gamma}} = q^{-1/2}$.
\end{proposition}
\begin{proof}
Note that because the finite-length trajectories we are interested in have phase~$1$ for $\varphi^{(\varepsilon)}$, the corresponding finite-length connections have phase $1 - \varepsilon$ for $\varphi$.
The description of the category of semistable objects along with the isomorphism of orientation data follows from Proposition~\ref{prop:drd_ii:2_vertex}.
In terms of the quiver~$Q'$, the DT generating series is \[\qdl{t^{(1, 0)}}\qdl{t^{(0, 1)}}\qdl{-q^{-1/2}t^{(1, 1)}}^{-1},\] by Proposition~\ref{prop:haiden}.
It follows from Propositions~\ref{prop:drd_ii:3_vertex} and \ref{prop:drd_ii:2_vertex} that the classes $(1, 0)$ and $(0, 1)$ both correspond to the class $\hhc{\gamma}$, so we obtain the generating series $\qdl{t^{\hhc{\gamma}}}^2\qdl{-q^{-1/2}t^{2\hhc{\gamma}}}^{-1}$.
\end{proof}

\subsubsection{Type~III case}\label{sect:main:drd:iii}

We now consider case~\ref{op:fin_len_traj:drd}\ref{op:fin_len_traj:drd:iii} from Proposition~\ref{prop:fin_len_traj}, where the degenerate ring domain has a type~III saddle trajectory in its boundary.
This can only occur in polar type $\{-1,-1,-2\}$, and corresponds precisely to the Legendre differential on $X=\mathbb{P}^1_{\mathbb{C}}$ in  the terminology of \cite[Section~4.3.3]{ik}.

\begin{proposition}\label{prop:drd_iii}
Suppose that $\varphi$ is a generic infinite GMN differential on a Riemann surface $X$ with a type~III degenerate ring domain, with the type~III saddle trajectory denoted~$\gamma$.
The category $\ssc{\sigma_{\varphi}}{1}$ is given by nilpotent representations of the quiver \[\begin{tikzcd} \bullet \ar[loop left,"a"] \ar[loop right,"b"]\end{tikzcd}\] with potential $W = a^3 + b^3$.
Hence, the DT generating series of $\ssc{\sigma_{\varphi}}{1}$ is \[
\qdl{t^{\hhc{\gamma}}}^4\qdl{-q^{-1/2}t^{2\hhc{\gamma}}}^{-1},
\]
and the refined DT invariants are $\dtref{\hhc{\gamma}} = 4$ and $\dtref{2\hhc{\gamma}} = q^{-1/2}$.
\end{proposition}
\begin{proof}
As in previous proofs, we have for some $\varepsilon > 0$ that $\varphi^{(\varepsilon)}$ is saddle-free and has $\gamma$ as a standard saddle connection in a horizontal strip~$H$.
Since $\gamma$ is a type~III saddle trajectory, $H$ has one of the simple poles on each side.
Note here that, since only one trajectory emanates from a simple pole, $H$ folds up to give the whole surface~$X$; recall, in any case, from Figure~\ref{fig:configs} that $X$ must be a sphere.
Hence, the associated triangulation $T_{\varphi^{(\varepsilon)}}$ consists of a single arc from the puncture to itself, which on each side encloses a monogon containing an orbifold point, from which the quiver with potential follows. The sole vertex corresponds to the horizontal strip, and hence to the class~$\hhc{\gamma}$.
The generating series then follows from Proposition~\ref{prop:haiden}.
\end{proof}

\begin{remark}
In each of these cases where a degenerate ring domain appears, one could obtain the refined DT invariant $q^{1/2}$ for the class of a closed trajectory in the ring domain by taking the uncompleted algebra, by Proposition~\ref{prop:haiden}.
This would then also give a numerical DT invariant of $-1$, agreeing with \cite[Definition~3.12]{ik}.
On the one hand, using the completed algebras and so obtaining the invariant $q^{-1/2}$ is preferable from a theoretical perspective.
It is convenient to have finitely many simple objects.
This is the case to which the results of \cite{chq} apply, since these results concern the subcategory generated by the simple objects corresponding to vertices of the quiver.
Furthermore, it is not clear how to construct orientation data for modules over the uncompleted path algebra.
On the other hand, one disadvantage of using the completed path algebra is that in the toral case, there are no stable objects with the same class as the ring domain.
If we instead considered the uncompleted path algebra, there would be such stable objects.
Using the uncompleted path algebra also accords better with homological mirror symmetry, whereby the mirror of the degenerate ring domain should be $\mathbb{A}_{\mathbb{C}}^{1}$ \cite[Section~4.5]{bocklandt}.
The virtual motive of $\mathbb{A}_{\mathbb{C}}^{1}$, in the sense of \cite[Definition~2.14]{bbs}, is $\mathbb{L}^{1/2}$, which corresponds to the DT invariant $q^{1/2}$.
Both $q^{1/2}$ and $q^{-1/2}$ are consistent with the result of \cite{ko}.
\end{remark}

\subsection{Non-degenerate ring domains}\label{sect:main:nrd}

We consider the three different types of non-degenerate ring domain from Proposition~\ref{prop:fin_len_traj} in turn.
As per Proposition~\ref{prop:limit_walls}, non-degenerate ring domains are less straightforward than the other (non-toral) cases, since we cannot make a small rotation to obtain the saddle trajectory as a standard saddle connection (which made it easy to identify the relevant category of semistable objects).  In each case, we denote the ring domain by $R$.

\begin{remark}
In \cite{bs}, the categories of semistable objects for the cases when $R$ has only zeros in its boundary is obtained using a technique known as `ring-shrinking'.
We take a different approach to identifying this category by taking a rotation of the quadratic differential small enough that that we can still determine the trajectory structure.
There would be additional technicalities in applying the ring-shrinking technique to the case in which $R$ surrounds a type~III saddle trajectory, due to the fact that such $R$ is then not `strongly non-degenerate' in the sense of \cite[Section~3.4]{bs}.
We could apply the ring-shrinking argument in the other cases, but we give an alternative argument, showing that the same category can be obtained by different means.
\end{remark}

\subsubsection{Standard case}\label{sect:main:nrd:i}

We begin by considering case~\ref{op:fin_len_traj:nrd}\ref{op:fin_len_traj:nrd:single} from Proposition~\ref{prop:fin_len_traj}, where each of the two boundaries of the non-degenerate ring domain is a closed saddle trajectory.
This is a case where more than one saddle trajectory appears, and which can appear in polar type~$\{-2\}$.
We make the restriction that the polar type is not $\{-2\}$, but we could instead proceed in the same way as in Section~\ref{sect:main:drd}.
Let $\gamma$ be a closed trajectory lying in the ring domain and $\hhc{\gamma} \in \hh{\varphi}$ the associated hat-homology class.

\begin{proposition}\label{prop:nrd_i}
Suppose that $\varphi$ is a generic infinite GMN differential on a Riemann surface~$X$, whose polar type is not $\{-2\}$, with a standard non-degenerate ring domain of class $\hhc{\gamma}$. The category $\ssc{\sigma_{\varphi}}{1}$ is given by semistable representations of the Kronecker quiver \[
    \begin{tikzcd}
        \overset{1}{\bullet} \ar[r,shift left] \ar[r,shift right] & \overset{2}{\bullet}
    \end{tikzcd}
    \] with zero potential which have phase $\scph{S_1 \oplus S_2}$ under a stability condition such that $\scph{S_2} < \scph{S_1}$.
Hence, the DT generating series of $\ssc{\sigma_{\varphi}}{1}$ is \[
\qdl{-q^{1/2}t^{\hhc{\gamma}}}^{-1}\qdl{-q^{-1/2}t^{\hhc{\gamma}}}^{-1},
\]
and the refined DT invariants are $\dtref{\hhc{\gamma}} = q^{1/2} + q^{-1/2}$.
\end{proposition}
\begin{proof}
Aside from the non-degenerate ring domain $R$, the remainder of $X\setminus\mathcal{W}_0(\varphi)$ consists of horizontal strips and half-planes, there being no spiral domains by Proposition~\ref{prop:fin_len_traj}.
Hence, the domains $D_1$, $D_2$ adjacent to $R$ are either horizontal strips or half-planes.
An example is shown in Figure~\ref{fig:nrd_i:before}, where $D_1$ and $D_2$ are each horizontal strips with zeros on both boundaries.

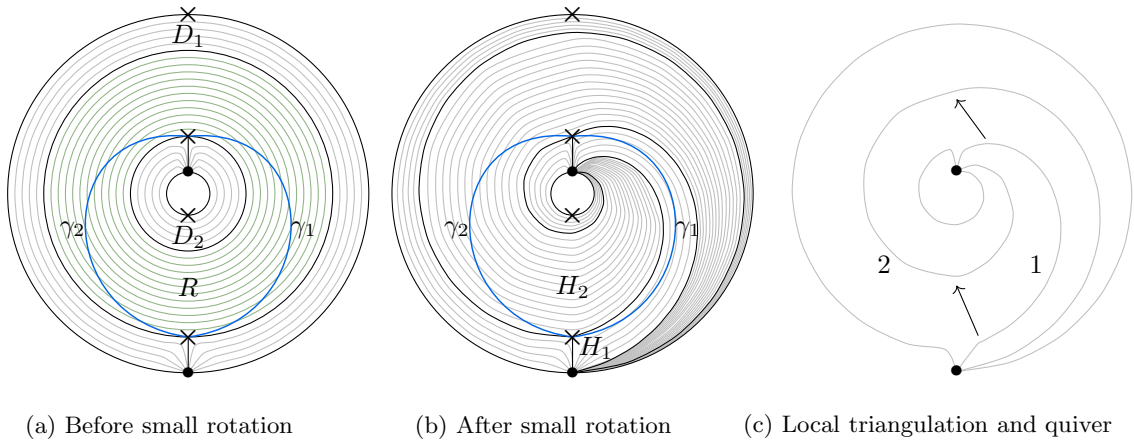
\begin{figure}[ht]
 \begin{subfigure}[t]{0.31\textwidth}
 \[
	\begin{tikzpicture}[scale=0.95]
	
	\draw[\septraj] (0,0) circle (0.3cm);
\draw[\saddletraj] (0,0) circle (0.8cm);
\draw[\saddletraj] (0,0) circle (2cm);
\draw[\septraj] (0,0) circle (2.5cm);

\draw[\generictraj] plot [smooth] coordinates {(0,0.3) (110:0.4) (120:0.4) (130:0.4) (140:0.4) (150:0.4) (160:0.4) (170:0.4) (180:0.4) (190:0.4) (200:0.4) (210:0.4) (220:0.4) (230:0.4) (240:0.4) (250:0.4) (260:0.4) (270:0.4) (280:0.4) (290:0.4) (300:0.4) (310:0.4) (320:0.4) (330:0.4) (340:0.4) (350:0.4) (360:0.4) (10:0.4) (20:0.4) (30:0.4) (40:0.4) (50:0.4) (60:0.4) (70:0.4) (0,0.3)};
\draw[\generictraj] plot [smooth] coordinates {(0,0.3) (105:0.475) (110:0.5) (120:0.5) (130:0.5) (140:0.5) (150:0.5) (160:0.5) (170:0.5) (180:0.5) (190:0.5) (200:0.5) (210:0.5) (220:0.5) (230:0.5) (240:0.5) (250:0.5) (260:0.5) (270:0.5) (280:0.5) (290:0.5) (300:0.5) (310:0.5) (320:0.5) (330:0.5) (340:0.5) (350:0.5) (360:0.5) (10:0.5) (20:0.5) (30:0.5) (40:0.5) (50:0.5) (60:0.5) (70:0.5) (75:0.475) (0,0.3)};
\draw[\generictraj] plot [smooth] coordinates {(0,0.3) (100:0.55) (110:0.6) (120:0.6) (130:0.6) (140:0.6) (150:0.6) (160:0.6) (170:0.6) (180:0.6) (190:0.6) (200:0.6) (210:0.6) (220:0.6) (230:0.6) (240:0.6) (250:0.6) (260:0.6) (270:0.6) (280:0.6) (290:0.6) (300:0.6) (310:0.6) (320:0.6) (330:0.6) (340:0.6) (350:0.6) (360:0.6) (10:0.6) (20:0.6) (30:0.6) (40:0.6) (50:0.6) (60:0.6) (70:0.6) (80:0.55) (0,0.3)};
\draw[\generictraj] plot [smooth] coordinates {(0,0.3) (95:0.625) (110:0.7) (120:0.7) (130:0.7) (140:0.7) (150:0.7) (160:0.7) (170:0.7) (180:0.7) (190:0.7) (200:0.7) (210:0.7) (220:0.7) (230:0.7) (240:0.7) (250:0.7) (260:0.7) (270:0.7) (280:0.7) (290:0.7) (300:0.7) (310:0.7) (320:0.7) (330:0.7) (340:0.7) (350:0.7) (360:0.7) (10:0.7) (20:0.7) (30:0.7) (40:0.7) (50:0.7) (60:0.7) (70:0.7) (85:0.625) (0,0.3)};

\draw[\generictraj] plot [smooth] coordinates {(0,-2.5) (272.5:2.125) (280:2.1) (290:2.1) (300:2.1) (310:2.1) (320:2.1) (330:2.1) (340:2.1) (350:2.1) (360:2.1) (10:2.1) (20:2.1) (30:2.1) (40:2.1) (50:2.1) (60:2.1) (70:2.1) (80:2.1) (90:2.1) (100:2.1) (110:2.1) (120:2.1) (130:2.1) (140:2.1) (150:2.1) (160:2.1) (170:2.1) (180:2.1) (190:2.1) (200:2.1) (210:2.1) (220:2.1) (230:2.1) (240:2.1) (250:2.1) (260:2.1) (267.5:2.125) (0,-2.5)};
\draw[\generictraj] plot [smooth] coordinates {(0,-2.5) (275:2.25) (280:2.2) (290:2.2) (300:2.2) (310:2.2) (320:2.2) (330:2.2) (340:2.2) (350:2.2) (360:2.2) (10:2.2) (20:2.2) (30:2.2) (40:2.2) (50:2.2) (60:2.2) (70:2.2) (80:2.2) (90:2.2) (100:2.2) (110:2.2) (120:2.2) (130:2.2) (140:2.2) (150:2.2) (160:2.2) (170:2.2) (180:2.2) (190:2.2) (200:2.2) (210:2.2) (220:2.2) (230:2.2) (240:2.2) (250:2.2) (260:2.2) (265:2.25) (0,-2.5)};
\draw[\generictraj] plot [smooth] coordinates {(0,-2.5) (277.5:2.325) (280:2.3) (290:2.3) (300:2.3) (310:2.3) (320:2.3) (330:2.3) (340:2.3) (350:2.3) (360:2.3) (10:2.3) (20:2.3) (30:2.3) (40:2.3) (50:2.3) (60:2.3) (70:2.3) (80:2.3) (90:2.3) (100:2.3) (110:2.3) (120:2.3) (130:2.3) (140:2.3) (150:2.3) (160:2.3) (170:2.3) (180:2.3) (190:2.3) (200:2.3) (210:2.3) (220:2.3) (230:2.3) (240:2.3) (250:2.3) (260:2.3) (262.5:2.325) (0,-2.5)};
\draw[\generictraj] plot [smooth] coordinates {(0,-2.5) (280:2.4) (290:2.4) (300:2.4) (310:2.4) (320:2.4) (330:2.4) (340:2.4) (350:2.4) (360:2.4) (10:2.4) (20:2.4) (30:2.4) (40:2.4) (50:2.4) (60:2.4) (70:2.4) (80:2.4) (90:2.4) (100:2.4) (110:2.4) (120:2.4) (130:2.4) (140:2.4) (150:2.4) (160:2.4) (170:2.4) (180:2.4) (190:2.4) (200:2.4) (210:2.4) (220:2.4) (230:2.4) (240:2.4) (250:2.4) (260:2.4) (0,-2.5)};

\draw[\closedtraj] (0,0) circle (0.9cm);
\draw[\closedtraj] (0,0) circle (1cm);
\draw[\closedtraj] (0,0) circle (1.1cm);
\draw[\closedtraj] (0,0) circle (1.2cm);
\draw[\closedtraj] (0,0) circle (1.3cm);
\draw[\closedtraj] (0,0) circle (1.4cm);
\draw[\closedtraj] (0,0) circle (1.5cm);
\draw[\closedtraj] (0,0) circle (1.6cm);
\draw[\closedtraj] (0,0) circle (1.7cm);
\draw[\closedtraj] (0,0) circle (1.8cm);
\draw[\closedtraj] (0,0) circle (1.9cm);

\draw[\septraj] (0,-2.5) -- (0,-2);
\draw[\septraj] (0,0.8) -- (0,0.3);

\draw[\conn, line width=0.55pt] (0,0.8) .. controls (-1.9,1) and (-1.9,-1.75)  .. (0,-2);
\draw[\conn, line width=0.55pt] (0,0.8) .. controls (1.9,1) and (1.9,-1.75)  .. (0,-2);

\node at (0,2.5) {$\bm{\times}$};
\node at (0,-2.5) {$\bullet$};
\node at (0,-2) {$\bm{\times}$};
\node at (0,0.8) {$\bm{\times}$};
\node at (0,0.3) {$\bullet$};
\node at (0,-0.3) {$\bm{\times}$};

\node at (0,2.2) {$D_1$};
\node at (0,-0.6) {$D_2$};

\node at (1.6,-0.5) {$\gamma_{1}$};
\node at (-1.6,-0.5) {$\gamma_{2}$};

\node at (0,-1.3) {$R$};

\path[use as bounding box] (-3,-2.75) rectangle (3,2.75);

	\end{tikzpicture} 
 \]
 \caption{Before small rotation}\label{fig:nrd_i:before}
 \end{subfigure}
 \begin{subfigure}[t]{0.31\textwidth}
 \[
	\begin{tikzpicture}[scale=0.95]

\draw[\generictraj] plot [smooth] coordinates {(90:0.3) (82:0.4) (79.4:0.5) (68.8:0.6) (58.2:0.7) (47.6:0.8) (37.1:0.9) (26.5:1) (15.9:1.1) (5.3:1.2) (354:1.3) (344:1.4) (333.5:1.5) (322.9:1.6) (312.4:1.7) (301.7:1.8) (291.2:1.9) (280.6:2) (275.3:2.075) (270:2.5)};
\draw[\generictraj] plot [smooth] coordinates {(90:0.3) (84:0.5) (79.4:0.6) (68.8:0.7) (58.2:0.8) (47.6:0.9) (37.1:1) (26.5:1.1) (15.9:1.2) (5.3:1.3) (354:1.4) (344:1.5) (333.5:1.6) (322.9:1.7) (312.4:1.8) (301.7:1.9) (291.2:2) (280.6:2.1) (277.4:2.15) (270:2.5)};
\draw[\generictraj] plot [smooth] coordinates {(90:0.3) (86:0.6) (79.4:0.7) (68.8:0.8) (58.2:0.9) (47.6:1) (37.1:1.1) (26.5:1.2) (15.9:1.3) (5.3:1.4) (354:1.5) (344:1.6) (333.5:1.7) (322.9:1.8) (312.4:1.9) (301.7:2) (291.2:2.1) (280.6:2.2) (270:2.5)};
\draw[\generictraj] plot [smooth] coordinates {(90:0.3) (88:0.7) (79.4:0.8) (68.8:0.9) (58.2:1) (47.6:1.1) (37.1:1.2) (26.5:1.3) (15.9:1.4) (5.3:1.5) (354:1.6) (344:1.7) (333.5:1.8) (322.9:1.9) (312.4:2) (301.7:2.1) (291.2:2.2) (280.6:2.3) (270:2.5)};

\draw[\generictraj] plot [smooth] coordinates {(90:0.3) (54:0.37) (18:0.44) (342:0.51) (322.9:0.55) (306:0.58) (291.2:0.61) (270:0.65) (234:0.7) (198:0.75) (162:0.8) (126:0.85) (90:0.9) (72:1.06) (54:1.24) (36:1.38) (18:1.54) (0:1.70) (342:1.86) (324:2.02) (306:2.18) (288:2.34) (270:2.5)};
\draw[\generictraj] plot [smooth] coordinates {(90:0.3) (54:0.39) (18:0.48) (342:0.57) (322.9:0.62) (306:0.66) (291.2:0.70) (270:0.76) (234:0.81) (198:0.86) (162:0.91) (126:0.96) (90:1) (72:1.16) (54:1.32) (36:1.46) (18:1.6) (0:1.75) (342:1.9) (324:2.05) (306:2.2) (288:2.35) (270:2.5)};
\draw[\generictraj] plot [smooth] coordinates {(90:0.3) (54:0.41) (18:0.52) (342:0.64) (322.9:0.7) (306:0.74) (291.2:0.79) (270:0.86) (234:0.91) (198:0.96) (162:1.01) (126:1.06) (90:1.11) (72:1.25)  (54:1.40) (36:1.53) (18:1.67) (0:1.81) (342:1.94) (324:2.08) (306:2.22) (288:2.36) (270:2.5)};
\draw[\generictraj] plot [smooth] coordinates {(90:0.3) (54:0.43) (18:0.57) (342:0.70) (322.9:0.77) (306:0.83) (291.2:0.89) (270:0.96) (234:1.01) (198:1.06) (162:1.11) (126:1.16) (90:1.21) (72:1.34)(54:1.49) (36:1.6) (18:1.73) (0:1.86) (342:1.99) (324:2.11) (306:2.24) (288:2.37) (270:2.5)};
\draw[\generictraj] plot [smooth] coordinates {(90:0.3) (54:0.45) (18:0.61) (342:0.76) (322.9:0.84) (306:0.91) (291.2:0.98) (270:1.07) (234:1.12) (198:1.17) (162:1.22) (126:1.27) (90:1.32) (72:1.44) (54:1.57) (36:1.67) (18:1.79) (0:1.91) (342:2.03) (324:2.14) (306:2.26) (288:2.38) (270:2.5)};
\draw[\generictraj] plot [smooth] coordinates {(90:0.3) (54:0.47) (18:0.65) (342:0.82) (322.9:0.92) (306:0.99) (291.2:1.07) (270:1.17) (234:1.22) (198:1.27) (162:1.32) (126:1.37) (90:1.42) (72:1.53) (54:1.65) (36:1.75) (18:1.85) (0:1.96) (342:2.07) (324:2.18) (306:2.28) (288:2.39) (270:2.5)};
\draw[\generictraj] plot [smooth] coordinates {(90:0.3) (54:0.49) (18:0.69) (342:0.88) (322.9:0.99) (306:1.07) (291.2:1.16) (270:1.28) (234:1.33) (198:1.38) (162:1.43) (126:1.48) (90:1.53) (72:1.62) (54:1.73) (36:1.82) (18:1.92) (0:2.01) (342:2.11) (324:2.21) (306:2.3) (288:2.4) (270:2.5)};
\draw[\generictraj] plot [smooth] coordinates {(90:0.3) (54:0.52) (18:0.73) (342:0.95) (322.9:1.06) (306:1.15) (291.2:1.25) (270:1.38) (234:1.43) (198:1.48) (162:1.53) (126:1.58) (90:1.63) (72:1.72) (54:1.81) (36:1.89) (18:1.98) (0:2.06) (342:2.15) (324:2.24) (306:2.33) (288:2.41) (270:2.5)};
\draw[\generictraj] plot [smooth] coordinates {(90:0.3) (54:0.54) (18:0.77) (342:1) (322.9:1.13) (306:1.23) (291.2:1.34) (270:1.48) (234:1.53) (198:1.58) (162:1.63) (126:1.68) (90:1.73) (72:1.81) (54:1.89) (36:1.96) (18:2.04) (0:2.12) (342:2.19) (324:2.27) (306:2.35) (288:2.42) (270:2.5)};
\draw[\generictraj] plot [smooth] coordinates {(90:0.3) (54:0.56) (18:0.81) (342:1.07) (322.9:1.21) (306:1.31) (291.2:1.43) (270:1.59) (234:1.64) (198:1.69) (162:1.74) (126:1.79) (90:1.84) (72:1.9) (54:1.97) (36:2.04) (18:2.1) (0:2.17) (342:2.23) (324:2.3) (306:2.37) (288:2.43) (270:2.5)};
\draw[\generictraj] plot [smooth] coordinates {(90:0.3) (54:0.58) (18:0.86) (342:1.13) (322.9:1.28) (306:1.40) (291.2:1.53) (270:1.69) (234:1.74) (198:1.79) (162:1.84) (126:1.89) (90:1.94) (72:2) (54:2.06) (36:2.11) (18:2.16) (0:2.22) (342:2.28) (324:2.33) (306:2.39) (288:2.44) (270:2.5)};
\draw[\generictraj] plot [smooth] coordinates {(90:0.3) (54:0.6) (18:0.9) (342:1.20) (322.9:1.35) (306:1.48) (291.2:1.62) (270:1.79) (234:1.84) (198:1.89) (162:1.94) (126:2) (90:2.04) (72:2.09) (54:2.14) (36:2.18) (18:2.23) (0:2.27) (342:2.32) (324:2.36) (306:2.41) (288:2.45) (270:2.5)};
\draw[\generictraj] plot [smooth] coordinates {(90:0.3) (54:0.62) (18:0.94) (342:1.26) (322.9:1.43) (306:1.56) (291.2:1.71) (270:1.9) (234:1.95) (198:2) (162:2.05) (126:2.1) (90:2.15) (72:2.18) (54:2.22) (36:2.25) (18:2.29) (0:2.32) (342:2.36) (324:2.39) (306:2.43) (288:2.46) (270:2.5)};

\draw[\generictraj] plot [smooth] coordinates {(270:2.5) (268:2.1) (252:2.12) (234:2.14) (216:2.16) (198:2.18) (180:2.2) (162:2.22) (144:2.24) (126:2.26) (108:2.28) (90:2.3) (72:2.32) (54:2.34) (36:2.36) (18:2.38) (0:2.4) (342:2.42) (324:2.44) (306:2.46) (288:2.48) (270:2.5)};
\draw[\generictraj] plot [smooth] coordinates {(270:2.5) (266:2.2) (252:2.22) (234:2.23) (216:2.245) (198:2.26) (180:2.275) (162:2.29) (144:2.3) (126:2.32) (108:2.34) (90:2.35) (72:2.37) (54:2.38) (36:2.4) (18:2.41) (0:2.425) (342:2.44) (324:2.455) (306:2.47) (288:2.485) (270:2.5)};
\draw[\generictraj] plot [smooth] coordinates {(270:2.5) (264:2.3) (252:2.31) (234:2.32) (216:2.33) (198:2.34) (180:2.35) (162:2.36) (144:2.37) (126:2.38) (108:2.39) (90:2.4) (72:2.41) (54:2.42) (36:2.43) (18:2.44) (0:2.45) (342:2.46) (324:2.47) (306:2.48) (288:2.49) (270:2.5)};
\draw[\generictraj] plot [smooth] coordinates {(270:2.5) (262:2.4) (252:2.4)(234:2.41) (216:2.415) (198:2.42) (180:2.425) (162:2.43) (144:2.43) (126:2.44) (108:2.45) (90:2.45) (72:2.45) (54:2.46) (36:2.47) (18:2.47) (0:2.475) (342:2.48) (324:2.485) (306:2.49) (288:2.495) (270:2.5)};

\draw[\generictraj] plot [smooth] coordinates {(90:0.3) (94:0.7) (126:0.66) (162:0.62) (198:0.58) (234:0.54) (270:0.5) (306:0.46) (342:0.42) (18:0.38) (54:0.34) (90:0.3)};
\draw[\generictraj] plot [smooth] coordinates {(90:0.3) (98:0.6) (126:0.57) (162:0.54) (198:0.51) (234:0.48) (270:0.45) (306:0.42) (342:0.39) (18:0.36) (54:0.33) (90:0.3)};
\draw[\generictraj] plot [smooth] coordinates {(90:0.3) (102:0.5) (126:0.48) (162:0.46) (198:0.44) (234:0.42) (270:0.4) (306:0.38) (342:0.36) (18:0.34) (54:0.32) (90:0.3)};
\draw[\generictraj] plot [smooth] coordinates {(90:0.3) (106:0.4) (126:0.39) (162:0.38) (198:0.37) (234:0.36) (270:0.35) (306:0.34) (342:0.33) (18:0.32) (54:0.31) (90:0.3)};

\draw[\septraj] (0,0) circle (0.3cm);
\draw[\septraj] (0,0) circle (2.5cm);

\draw[\septraj] plot [smooth] coordinates {(90:0.8) (126:0.75) (162:0.7) (198:0.65) (234:0.6) (270:0.55) (306:0.5) (342:0.45) (18:0.4) (54:0.35) (90:0.3)}; %
\draw[\septraj] plot [smooth] coordinates {(90:0.8) (79.4:0.9) (68.8:1) (58.2:1.1) (47.6:1.2) (37.1:1.3) (26.5:1.4) (15.9:1.5) (5.3:1.6) (354.7:1.7) (344.1:1.8) (333.5:1.9) (322.9:2) (312.4:2.1) (301.7:2.2) (291.2:2.3) (280.6:2.4)  (270:2.5)}; %
\draw[\septraj] plot [smooth] coordinates {(90:0.3) (79.4:0.4) (68.8:0.5) (58.2:0.6) (47.6:0.7) (37.1:0.8) (26.5:0.9) (15.9:1) (5.3:1.1) (354.7:1.2) (344.1:1.3) (333.5:1.4) (322.9:1.5) (312.4:1.6) (301.7:1.7) (291.2:1.8) (280.6:1.9)  (270:2)}; %
\draw[\septraj] plot [smooth] coordinates {(270:2) (252:2.025) (234:2.05) (216:2.075) (198:2.1) (180:2.125) (162:2.15) (144:2.175) (126:2.2) (108:2.225) (90:2.25) (72:2.275) (54:2.3) (36:2.325) (18:2.35) (0:2.375) (342:2.4) (324:2.425) (306:2.45) (288:2.475) (270:2.5)}; %

\draw[\septraj] (0,-2.5) -- (0,-2);
\draw[\septraj] (0,0.8) -- (0,0.3);

\draw[\conn, line width=0.55pt] (0,0.8) .. controls (-1.9,1) and (-1.9,-1.75)  .. (0,-2);
\draw[\conn, line width=0.55pt] (0,0.8) .. controls (1.9,1) and (1.9,-1.75)  .. (0,-2);

\node at (0,2.5) {$\bm{\times}$};
\node at (0,-2.5) {$\bullet$};
\node at (0,-2) {$\bm{\times}$};
\node at (0,0.8) {$\bm{\times}$};
\node at (0,0.3) {$\bullet$};
\node at (0,-0.3) {$\bm{\times}$};

\node at (0,-1.3) {$H_2$};
\node at (0.325,-2.15) {$H_1$};

\node at (1.6,-0.5) {$\gamma_{1}$};
\node at (-1.6,-0.5) {$\gamma_{2}$};

\path[use as bounding box] (-3,-2.75) rectangle (3,2.75);
	
	\end{tikzpicture} 
 \]
 \caption{After small rotation}\label{fig:nrd_i:after}
 \end{subfigure}
 \begin{subfigure}[t]{0.31\textwidth}
 \[
	\begin{tikzpicture}[scale=0.95]

\draw[\generictraj] plot [smooth] coordinates {(90:0.3) (84:0.5) (79.4:0.6) (68.8:0.7) (58.2:0.8) (47.6:0.9) (37.1:1) (26.5:1.1) (15.9:1.2) (5.3:1.3) (354:1.4) (344:1.5) (333.5:1.6) (322.9:1.7) (312.4:1.8) (301.7:1.9) (291.2:2) (280.6:2.1) (277.4:2.15) (270:2.5)};
\draw[\generictraj] plot [smooth] coordinates {(90:0.3) (54:0.47) (18:0.65) (342:0.82) (322.9:0.92) (306:0.99) (291.2:1.07) (270:1.17) (234:1.22) (198:1.27) (162:1.32) (126:1.37) (90:1.42) (72:1.53) (54:1.65) (36:1.75) (18:1.85) (0:1.96) (342:2.07) (324:2.18) (306:2.28) (288:2.39) (270:2.5)};
\draw[\generictraj] plot [smooth] coordinates {(270:2.5) (266:2.2) (252:2.22) (234:2.23) (216:2.245) (198:2.26) (180:2.275) (162:2.29) (144:2.3) (126:2.32) (108:2.34) (90:2.35) (72:2.37) (54:2.38) (36:2.4) (18:2.41) (0:2.425) (342:2.44) (324:2.455) (306:2.47) (288:2.485) (270:2.5)};
\draw[\generictraj] plot [smooth] coordinates {(90:0.3) (98:0.6) (126:0.57) (162:0.54) (198:0.51) (234:0.48) (270:0.45) (306:0.42) (342:0.39) (18:0.36) (54:0.33) (90:0.3)};
\node at (0,-2.5) {$\bullet$};
\node at (0,0.3) {$\bullet$};

\draw[->] (0.3,-2) -- (0,-1.3);
\draw[->] (0.4,0.75) -- (0,1.3);

\node at (1.1,-1) {$1$};
\node at (-1,-1) {$2$};

\path[use as bounding box] (-3,-2.75) rectangle (3,2.75);
	
	\end{tikzpicture} 
 \]
 \caption{Local triangulation and quiver}\label{fig:nrd_i:triang}
 \end{subfigure}
\caption{Trajectory structure before and after a small rotation for a standard non-degenerate ring domain.}\label{fig:nrd_i_rotation}
\end{figure}

By Proposition~\ref{prop:sep_traj_persist}, we can rotate by some $\varepsilon>0$ so that $\varphi^{(\varepsilon)}$ preserves the separating trajectories on the boundaries of $D_1$ and $D_2$.
The region of the associated surface $\mbso$ between the arcs or boundary components corresponding to their generic trajectories is an annulus with one marked point on each boundary.
All triangulations of this surface give rise to the same quiver, and are given by a pair of arcs between each of the marked points, forming two triangles.
An example of such a triangulation can be seen by taking one generic trajectory from each strip in Figure~\ref{fig:nrd_i:after}, as shown in Figure~\ref{fig:nrd_i:triang}.

The quiver corresponding to the triangulation of this annulus is then the Kronecker quiver, since the two arcs are anti-clockwise adjacent to each other in both triangles.
We label the vertices corresponding to $H_1$ and $H_2$ as $1$ and $2$, such that the pair of arrows goes from vertex $1$ to vertex~$2$, with $\gamma_{1}$ and $\gamma_{2}$ the corresponding standard saddle connections.
We have that $\gamma_{2}$ is in the horizontal strip anti-clockwise adjacent from $\gamma_{1}$ around each zero, and that points on these saddle connections are linked by the geodesics given by the closed connections in $R$, as can be seen in in Figure~\ref{fig:nrd_i:before}.
We therefore have that $\qdph{\gamma_{2}} < \qdph{\gamma_{1}}$ by Lemma~\ref{lem:geodesics}. 
This gives the desired stability condition.

It is not hard to see from Figure~\ref{fig:nrd_i:before} that the hat-homology class of the closed connections in $R$ is equal to $\hhc{\gamma}_{1} + \hhc{\gamma}_2$.
Note that this remains true if the standard saddle connections wind around $R$ more times.

One can derive that the DT generating series of this category is $\qdl{-q^{1/2}t}^{-1}\qdl{-q^{-1/2}t}^{-1}$ from \cite[Theorem~1.1]{mr}, as is already well-known.
This is done, for example, in \cite[Theorem~2.4]{reineke_wqdi}, although note that the conventions there differ from ours.
\end{proof}

\subsubsection{Toral case}\label{sect:main:nrd:ii}

We now treat case~\ref{op:fin_len_traj:nrd}\ref{op:fin_len_traj:nrd:torus} in Proposition~\ref{prop:fin_len_traj}, the toral non-degenerate ring domain.

\begin{proposition}\label{prop:nrd_ii}
Suppose that $\varphi$ is a generic infinite GMN differential on a Riemann surface $X$ with a toral non-degenerate ring domain, and let the two type~I saddle trajectories be of class $\hhc{\gamma}$.
Then $\ssce{\sigma_{\varphi}}{1}$ is quasi-equivalent to the full $A_\infty$-subcategory of $\mathcal{H}_\infty(Q',0)$, where $Q'$ is the quiver 
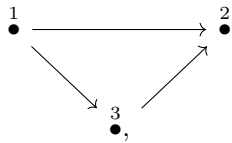
\[
    \begin{tikzcd}
        \overset{1}{\bullet} \ar[rr] \ar[dr] && \overset{2}{\bullet} \\
        & \overset{3}{\bullet}, \ar[ur] &
    \end{tikzcd}
    \] 
consisting of objects in $\mathcal{H}(Q',0)$ which are semistable of phase $\scph{S_1 \oplus S_2} = \scph{S_3}$ under a stability condition where $\scph{S_2} < \scph{S_1}$.
Moreover, there is an isomorphism between the canonical choices of orientation data on the stacks of objects in each of these categories.

It follows that the DT generating series of the category $\ssc{\sigma_{\varphi}}{1}$ is \[
\qdl{t^{\hhc{\gamma}}}^{2}\qdl{-q^{1/2}t^{2\hhc{\gamma}}}^{-1}\qdl{-q^{-1/2}t^{2\hhc{\gamma}}}^{-1},
\]
and so the refined DT invariants are $\dtref{\hhc{\gamma}} = 2$ and $\dtref{2\hhc{\gamma}} = q^{1/2} + q^{-1/2}$.
\end{proposition}
\begin{proof}
We derive the relevant category of semistable objects and isomorphism of orientation data in Propositions~\ref{prop:nrd_ii:4_vertex} and \ref{prop:nrd_ii:3_vertex}.
We have by \cite[Lemma~11.9(J2)]{bs} that the stable objects in this subcategory of $\modules \jac{Q', 0}$ consist of the simple representation with dimension vector $(0, 0, 1)$, the indecomposable representation of dimension vector $(1, 1, 0)$ and a $\mathbb{P}^1$ family of representations of dimension vector $(1, 1, 1)$.
Using \cite[Theorem~1.1]{mr}, we obtain that the DT generating series of this category is \[\qdl{t^{(1, 1, 0)}}\qdl{t^{(0, 0, 1)}}\qdl{-q^{1/2}t^{(1, 1, 1)}}^{-1}\qdl{-q^{-1/2}t^{(1, 1, 1)}}^{-1}.\]
Here the factors $\qdl{t^{(1, 1, 0)}}$ and $\qdl{t^{(0, 0, 1)}}$ correspond to the single stable objects and the factor $\qdl{-q^{1/2}t^{(1, 1, 1)}}^{-1}\qdl{-q^{-1/2}t^{(1, 1, 1)}}^{-1}$ corresponds to the $\mathbb{P}^1$ family.
Each of the classes $(1, 1, 0)$ and $(0, 0, 1)$ corresponds to the class $\hhc{\gamma} \in \hhspan{\varphi}$ by Propositions~\ref{prop:nrd_ii:4_vertex} and \ref{prop:nrd_ii:3_vertex}.
We obtain that the DT generating series of  $\ssc{\sigma_{\varphi}}{1}$ is $\qdl{t^{\hhc{\gamma}}}^{2}\qdl{-q^{1/2}t^{2\hhc{\gamma}}}^{-1}\qdl{-q^{-1/2}t^{2\hhc{\gamma}}}^{-1}$ .%
\end{proof}

\subsubsection{Type~III case}\label{sect:main:nrd:iii}

We now consider case~\ref{op:fin_len_traj:nrd}\ref{op:fin_len_traj:nrd:iii} in Proposition~\ref{prop:fin_len_traj}
We call the quiver with potential appearing below the \emph{barbell quiver with potential}, and postpone computing its DT generating series to Section~\ref{sect:db}.

\begin{proposition}\label{prop:nrd_iii}
Suppose that $\varphi$ is a generic infinite GMN differential on a Riemann surface $X$ with a type~III non-degenerate ring domain, with $\gamma$ the type~III saddle trajectory. The category $\ssc{\sigma_{\varphi}}{1}$ is given by semistable representations of the quiver $Q^{\db}$\[
    \begin{tikzcd}
    \overset{1}{\bullet} \ar[loop left,"a"] \ar[r] & \overset{2}{\bullet} \ar[loop right,"b"]
    \end{tikzcd}
    \] with potential $W^{\db}:=a^3 + b^3$ of phase $\scph{S_1 \oplus S_2}$ under a stability condition such that $\scph{S_2} < \scph{S_1}$.
The DT generating series of the category $\ssc{\sigma_{\varphi}}{1}$ is \[
\qdl{t^{\hhc{\gamma}}}^4\qdl{-q^{1/2}t^{2\hhc{\gamma}}}^{-1}\qdl{-q^{-1/2}t^{2\hhc{\gamma}}}^{-1},
\]
and the refined DT invariants are $\dtref{\hhc{\gamma}} = 4$ and $\dtref{2\hhc{\gamma}} = q^{1/2} + q^{-1/2}$.
\end{proposition}
\begin{proof}
Let $s_1$ and $s_2$ be the simple poles connected by the type~III saddle trajectory, with $R$ the non-degenerate ring domain.
We know from Proposition~\ref{prop:fin_len_traj} that the only finite-length trajectories of $\varphi$ lie in $R$ and its boundaries, and that there are no spiral domains.
Hence, the rest of $X\setminus\mathcal{W}_0(\varphi) $ consists of horizontal strips and half-planes.
Let $H$ be the horizontal strip or half-plane which is incident to the zero $z$ on the other boundary of $R$.
In the case where $H$ is a horizontal strip with a zero on its other boundary, we are in the situation of Figure~\ref{fig:nrd_iii:before}.

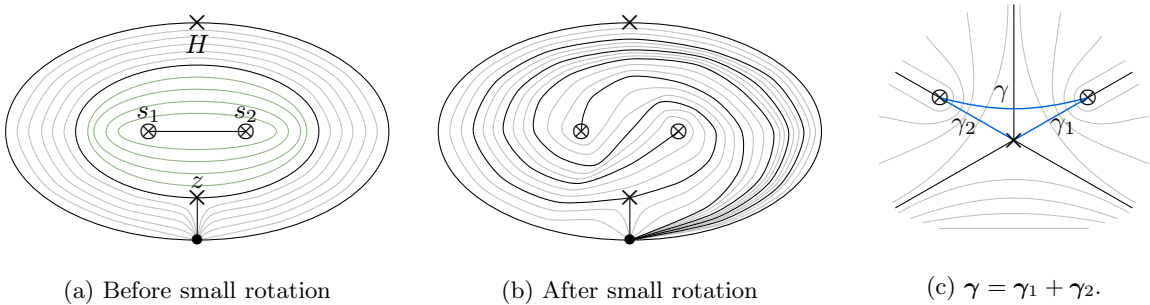
\begin{figure}[h]
        \begin{subfigure}[t]{0.35\textwidth}
    \[
    \begin{tikzpicture}[scale=0.8]
        
                \draw[\saddletraj] (-0.8,0) -- (0.8,0);
                \draw[\closedtraj] (0,0) ellipse (1.3cm and 0.3cm);
                \draw[\closedtraj] (0,0) ellipse (1.5cm and 0.5cm);
                \draw[\closedtraj] (0,0) ellipse (1.7cm and 0.7cm);
                \draw[\closedtraj] (0,0) ellipse (1.8cm and 0.9cm);
                \draw[\saddletraj] (0,0) ellipse (2cm and 1.1cm);

                \draw[\generictraj] plot [smooth] coordinates {(270:1.8) (225:2.05) (200:2.65) (180:2.96) (160:2.65) (135:2.07) (90:1.68) (45:2.07) (20:2.65) (0:2.96) (340:2.65) (315:2.07) (270:1.8)};

                \draw[\generictraj] plot [smooth] coordinates {(270:1.8) (260:1.65) (225:1.93) (200:2.48) (180:2.77) (160:2.48) (135:1.93) (90:1.57) (45:1.93) (20:2.48) (0:2.77) (340:2.48) (315:1.93) (280:1.65) (270:1.8)};
                
                \draw[\generictraj] plot [smooth] coordinates {(270:1.8) (260:1.525) (225:1.79) (200:2.3) (180:2.56) (160:2.3) (135:1.79) (90:1.45) (45:1.79) (20:2.3) (0:2.56) (340:2.3) (315:1.79) (280:1.525) (270:1.8)};

                \draw[\generictraj] plot [smooth] coordinates {(270:1.8) (260:1.4) (225:1.65) (200:2.13) (180:2.38) (160:2.13) (135:1.65) (90:1.33) (45:1.65) (20:2.13) (0:2.38) (340:2.13) (315:1.65) (280:1.4) (270:1.8)};

                \draw[\generictraj] plot [smooth] coordinates {(270:1.8) (262:1.3) (225:1.5) (200:1.95) (180:2.19) (160:1.95) (135:1.5) (90:1.22) (45:1.5) (20:1.95) (0:2.19) (340:1.95) (315:1.5) (278:1.3) (270:1.8)};

                \node at (-0.8,0) {\scriptsize $\bm{\otimes}$};
                \node at (0.8,0) {\scriptsize $\bm{\otimes}$};
                \node at (-0.8,0) [above] {$s_1$};
                \node at (0.8,0) [above] {$s_2$};

                \node at (0,-1.1) {$\bm{\times}$};
                \node at (0,-1.1) [above] {$z$};

                \node at (0,-1.8) {$\bullet$};
                \node at (0,1.8) {$\bm{\times}$};

                \draw[\septraj] (0,-1.1) -- (0,-1.8);
                \draw[\septraj] (0,0) ellipse (3.15cm and 1.8cm);

                \node at (0,1.45) {$H$};
        
    \end{tikzpicture}
    \]
    \caption{Before small rotation}\label{fig:nrd_iii:before}
    \end{subfigure}
    \begin{subfigure}[t]{0.35\textwidth}
    \[
    \begin{tikzpicture}[scale=0.8]

                \draw[\generictraj] plot [smooth] coordinates {(270:1.8) (280:1.33) ($0.67*(315:1.18)+0.33*(315:1.78)$) ($0.67*(340:1.39)+0.33*(340:1.91)$) ($0.67*(0:1.41)+0.33*(0:1.84)$) ($0.67*(20:1.11)+0.33*(20:1.51)$) ($0.65*(45:0.68)+0.35*(45:1.15)$) ($0.67*(90:0)+0.33*(90:0.91)$) ($0.67*(90:0)+0.33*(-0.8,0)$) ($0.67*(225:0.78)+0.33*(-0.8,0)$) ($0.67*(200:1.24)+0.33*(-0.8,0)$) ($0.7*(180:1.575)+0.3*(-0.8,0)$) ($0.7*(160:1.56)+0.3*(-0.8,0)$) ($0.7*(135:1.35)+0.3*(135:0.84)$) ($0.67*(90:1.2)+0.33*(90:0.91)$) ($0.67*(45:1.56)+0.33*(45:1.15)$) ($0.67*(20:1.96)+0.33*(20:1.51)$) ($0.67*(0:2.36)+0.33*(0:1.84)$) ($0.67*(340:2.28)+0.33*(340:1.91)$) ($0.67*(315:1.93)+0.33*(315:1.78)$) (270:1.8)};
                \draw[\generictraj] plot [smooth] coordinates {(270:1.8) (290:1.5) ($0.33*(315:1.18)+0.67*(315:1.78)$) ($0.33*(340:1.39)+0.67*(340:1.91)$) ($0.33*(0:1.41)+0.67*(0:1.84)$) ($0.33*(20:1.11)+0.67*(20:1.51)$) ($0.3*(45:0.68)+0.7*(45:1.15)$) ($0.3*(90:0)+0.7*(90:0.91)$) ($0.4*(90:0)+0.6*(-0.8,0)$) ($0.4*(225:0.78)+0.6*(-0.8,0)$) ($0.4*(200:1.24)+0.6*(-0.8,0)$) ($0.4*(180:1.575)+0.6*(-0.8,0)$) ($0.4*(160:1.56)+0.6*(-0.8,0)$)       ($0.33*(135:1.35)+0.67*(135:0.84)$) ($0.33*(90:1.2)+0.67*(90:0.91)$) ($0.33*(45:1.56)+0.67*(45:1.15)$) ($0.33*(20:1.96)+0.67*(20:1.51)$) ($0.33*(0:2.36)+0.67*(0:1.84)$) ($0.33*(340:2.28)+0.67*(340:1.91)$) ($0.33*(315:1.93)+0.67*(315:1.78)$) (270:1.8)};

                \draw[\generictraj] plot [smooth] coordinates {(270:1.8) (250:1.57) (225:2.03) (200:2.63) (180:2.96) (160:2.67) (135:2.1) (90:1.71) (45:2.1) (20:2.72) (0:3.05) (340:2.76) (315:2.17) (270:1.8)};
                \draw[\generictraj] plot [smooth] coordinates {(270:1.8) (260:1.33) (225:1.84) (200:2.44) (180:2.77) (160:2.52) (135:1.99) (90:1.62) (45:2) (20:2.62) (0:2.96) (340:2.67) (315:2.14) (270:1.8)};

                \draw[\generictraj] plot [smooth] coordinates {(270:1.8) (315:2.05) (340:2.51) (0:2.69) (20:2.33) (45:1.78) (90:1.41) (135:1.71) (160:2.09) (180:2.245) (200:1.91) (225:1.37) (270:0.73) ($0.7*(0.8,0)+0.33*(315:1.18)$) ($0.65*(0.8,0)+0.35*(340:1.39)$) ($0.6*(0.8,0)+0.4*(0:1.41)$) ($0.5*(0.8,0)+0.5*(20:1.11)$) ($0.5*(0.8,0)+0.5*(45:0.68)$) ($0.55*(0.8,0)+0.55*(90:0)$) ($0.7*(270:0.55)+0.3*(90:0)$) ($0.75*(225:1.22)+0.25*(225:0.78)$) ($0.7*(200:1.74)+0.3*(200:1.24)$) ($0.67*(180:2.08)+0.33*(180:1.575)$) ($0.67*(160:1.96)+0.33*(160:1.56)$) ($0.67*(135:1.62)+0.33*(135:1.35)$) ($0.67*(90:1.35)+0.33*(90:1.2)$) ($0.67*(45:1.73)+0.33*(45:1.56)$) ($0.67*(20:2.24)+0.33*(20:1.96)$) ($0.7*(0:2.61)+0.3*(0:2.36)$) ($0.7*(340:2.45)+0.3*(340:2.28)$) ($0.7*(315:2.02)+0.3*(315:1.93)$) (270:1.8)};
                \draw[\generictraj] plot [smooth] coordinates {(270:1.8) (315:2.07) (340:2.56) (0:2.78) (20:2.42) (45:1.84) (90:1.47) (135:1.79) (160:2.23) (180:2.41) (200:2.07) (225:1.51) (270:0.92) ($0.33*(0.8,0)+0.67*(315:1.18)$) ($0.3*(0.8,0)+0.7*(340:1.39)$) ($0.3*(0.8,0)+0.7*(0:1.41)$) ($0.2*(0.8,0)+0.8*(20:1.11)$) ($0.3*(0.8,0)+0.7*(45:0.68)$) ($0.3*(0.8,0)+0.7*(90:0)$) ($0.4*(270:0.55)+0.6*(90:0)$) ($0.45*(225:1.22)+0.55*(225:0.78)$) ($0.4*(200:1.74)+0.6*(200:1.24)$) ($0.33*(180:2.08)+0.67*(180:1.575)$) ($0.33*(160:1.96)+0.67*(160:1.56)$) ($0.33*(135:1.62)+0.67*(135:1.35)$) ($0.33*(90:1.35)+0.67*(90:1.2)$) ($0.33*(45:1.73)+0.67*(45:1.56)$) ($0.33*(20:2.24)+0.67*(20:1.96)$) ($0.4*(0:2.61)+0.6*(0:2.36)$) ($0.4*(340:2.45)+0.6*(340:2.28)$) ($0.4*(315:2.02)+0.6*(315:1.93)$) (270:1.8)};

                \draw[\septraj] plot [smooth] coordinates {(270:1.1) (225:1.66) (200:2.24) (180:2.575) (160:2.36) (135:1.88) (90:1.53) (45:1.89) (20:2.51) (0:2.86) (340:2.62) (315:2.1) (270:1.8)}; %
                
                \draw[\septraj] plot [smooth] coordinates {(270:1.1) (315:1.18) (340:1.39) (0:1.41) (20:1.11) (45:0.68) (90:0) (225:0.78) (200:1.24) (180:1.575) (160:1.56) (135:1.35) (90:1.2) (45:1.56) (20:1.96) (0:2.36) (340:2.28) (315:1.93) (270:1.8)}; %

                \draw[\septraj] plot [smooth] coordinates {(180:0.8) (160:0.81) (135:0.84) (90:0.91) (45:1.15) (20:1.51) (0:1.84) (340:1.91) (315:1.78) (270:1.8)}; %

                \draw[\septraj] plot [smooth] coordinates {(0:0.8) (270:0.55) (225:1.22) (200:1.74) (180:2.08) (160:1.96) (135:1.62) (90:1.35) (45:1.73) (20:2.24) (0:2.61) (340:2.45) (315:2.02) (270:1.8)}; %

                \node at (-0.8,0) {\scriptsize $\bm{\otimes}$};
                \node at (0.8,0) {\scriptsize $\bm{\otimes}$};
                \node at (0,-1.1) {$\bm{\times}$};

                \draw (0,-1.1) -- (0,-1.8);
                \node at (0,-1.8) {$\bullet$};
                \node at (0,1.8) {$\bm{\times}$};
                \draw (0,0) ellipse (3.15cm and 1.8cm);
        
    \end{tikzpicture}
    \]
    \caption{After small rotation}\label{fig:nrd_iii:after}
    \end{subfigure}
    \begin{subfigure}[t]{0.27\textwidth}
           \[
    \begin{tikzpicture}[scale=0.9]
        \draw[\septraj] (0:0) -- (210:2);
        \draw[\septraj] (0:0) -- (90:2);
        \draw[\septraj] (0:0) -- (330:2);

        \draw[\generictraj] (87.5:2) to [out=270,in=150] (332.5:2);
        \draw[\generictraj] (70:2) .. controls (30:0.5) .. (350:2);
        \draw[\generictraj] ($(30:1.25)+(170:0.35)$) arc (170:250:0.35);
        \draw[\generictraj] ($(30:1.25)+(170:0.35)$) to [out=80,in=232] (50:2);
        \draw[\generictraj] ($(30:1.25)+(250:0.35)$) to [out=-20,in=188] (10:2);
        \draw[\generictraj] ($(30:1.25)+(120:0.2)$) arc (120:300:0.2);
        \draw[\generictraj] ($(30:1.25)+(120:0.2)$) -- ($(30:2)+(120:0.2)$);
        \draw[\generictraj] ($(30:1.25)+(300:0.2)$) -- ($(30:2)+(300:0.2)$);

        \draw[\generictraj] (207.5:2) to [out=30,in=270] (92.5:2);
        \draw[\generictraj] (190:2) .. controls (150:0.55) .. (110:2);
        \draw[\generictraj] ($(150:1.25)+(290:0.35)$) arc (290:370:0.35);
        \draw[\generictraj] ($(150:1.25)+(290:0.35)$) to [out=200,in=352] (170:2);
        \draw[\generictraj] ($(150:1.25)+(10:0.35)$) to [out=100,in=308] (130:2);
        \draw[\generictraj] ($(150:1.25)+(240:0.2)$) arc (240:420:0.2);
        \draw[\generictraj] ($(150:1.25)+(240:0.2)$) -- ($(150:2)+(240:0.2)$);
        \draw[\generictraj] ($(150:1.25)+(60:0.2)$) -- ($(150:2)+(60:0.2)$);

        \draw[\generictraj] (327.5:2) to [out=150,in=30] (212.5:2);
        \draw[\generictraj] (322.5:1.9) to [out=160,in=20] (217.5:1.9);
        \draw[\generictraj] (317.5:1.8) to [out=170,in=10] (222.5:1.8);
        \draw[\generictraj] (310:1.7) to [out=180,in=0] (230:1.7);

        \draw[\septraj] (30:1.25) -- (30:2);
        \draw[\septraj] (150:1.25) -- (150:2);

        \draw[\conn, line width=0.55pt] (30:1.25) -- (0,0);
        \draw[\conn, line width=0.55pt] (150:1.25) -- (0,0);
        \draw[\conn, line width=0.55pt] (30:1.25) to [out=195,in=-15] (150:1.25);

        \node at (30:1.25) {\scriptsize $\bm{\otimes}$};
        \node at (150:1.25) {\scriptsize $\bm{\otimes}$};
        \node at (0,0) {$\bm{\times}$};

        \node at (165:0.75) {$\gamma_2$};
        \node at (15:0.75) {$\gamma_1$};
        \node at (105:0.7) {$\gamma$};

        \node at (270:1.5) {};
    \end{tikzpicture}
    \]
    \caption{$\hhc{\gamma}=\hhc{\gamma}_1+\hhc{\gamma}_2$.}
    \label{fig:iii_vs_ii}
    \end{subfigure}
    \caption{Trajectory structure before and after a small rotation of a type~III saddle trajectory surrounded by a non-degenerate ring domain}\label{fig:nrd_iii}
\end{figure}

As for the standard case, the rotation $\varphi^{(\varepsilon)}$ preserves the separating trajectories on the boundary of~$H$.
The part of the associated surface $\mbso$ enclosed by the arc or boundary component corresponding to generic trajectories of $H$ is a disc with one marked point on its boundary and two orbifold points in its interior.
All triangulations of this surface give the same quiver: in each case we have that each orbifold point is enclosed in a monogon, with the remainder of the surface forming a triangle.
As for the standard case, the arcs of the triangulation enclosing the orbifold points may wind around each other many times, but this does not affect the associated quiver.
An example of a possible trajectory structure is shown in Figure~\ref{fig:nrd_iii:after}.

Since the monogons containing the orbifold points border the same triangle in the triangulation, in $\varphi^{(\varepsilon)}$ we have that $s_1$ and $s_2$ are in adjacent horizontal strips around the zero~$z$.
It follows from considering the pre-image of Figure~\ref{fig:iii_vs_ii} on the spectral cover that $\hhc{\gamma}=\hhc{\gamma}_1+\hhc{\gamma}_2$, since the central charge of all of these classes must have positive imaginary part. This remains true no matter how many times the two saddle connections wind around each other.

Without loss of generality, we can assume that everything is labelled such that the horizontal strip $H_2$ containing $\gamma_2$ is anti-clockwise-adjacent to the horizontal strip $H_1$ containing~$\gamma_1$, and let the corresponding vertices of the quiver be labelled $1$ and~$2$.
Since the arcs of the triangulation corresponding to these vertices enclose monogons containing orbifold points, they have respective loops $a$ and $b$ such that $a^3 + b^3$ is a summand of the potential.
Moreover, since these two monogons only border one triangle together, there is a single arrow from $1$ to $2$ in $Q(T)$.
Hence, we obtain that the full subquiver of $Q(T_{\varphi^{(\varepsilon)}})$ at vertices $1$ and $2$ is exactly $Q^{\db}$
with potential $W^{\db}$. The category $\ssc{\sigma_{\varphi}}{1}$ then consists of representations of this quiver with dimension vector $(d, d)$ under some stability condition, since the class of the type~III saddle corresponds to dimension vector~$(1, 1)$.

To determine the stability condition, note that $\gamma$ is a geodesic between $s_1$ and $s_2$, so that by Lemma~\ref{lem:geodesics} we must have $\qdph{\gamma_2} < \qdph{\gamma_1}$. Thus $\scph{S_2} < \scph{S_1}$ for the stability condition.

The generating series for this category is then given by Proposition~\ref{prop:db_gen_series}.
\end{proof}

The results of this section are summarised in Table~\ref{table}.
This illustrates all of the different possible configurations of finite-length trajectories for a generic infinite GMN differential, along with the quiver with potential, stability condition, and phase of the corresponding category, along with its DT generating series.
For aesthetic reasons, in Table~\ref{table} we write $t$ instead of~$t^{\hhc{\gamma}}$, where $\hhc{\gamma}$ is the smallest hat-homology class associated to the finite-length trajectories appearing.

\begin{proof}[{Proof of Theorem~\ref{thm:intro}}]
Let $\hhc{\gamma} \in \fgg \setminus \{0\}$ be a class of phase~$\vartheta$.
In the case where $\varphi^{(\vartheta)}$ is saddle-free, then it follows from Construction~\ref{const:bs} that the category of semistable objects of $\sigma_{\varphi^{(\vartheta)}}$ of phase~1 is zero, so $\dtref{\hhc{\gamma}} = 0$.
This then holds because $\varphi$ has no finite-length trajectories.
If $\varphi^{(\vartheta)}$ is not saddle-free, then, by genericity, we have one of the configurations from Proposition~\ref{prop:fin_len_traj}.
Genericity also shows that the subgroup of $\fgg$ given by classes of phase~1 of $\sigma_{\varphi^{(\vartheta)}}$ is one-dimensional, and so $\hhc{\gamma}$ is $\mathbb{Q}$-proportional to the classes of the finite-length trajectories that appear.
The formula from Theorem~\ref{thm:intro} then holds for each of the cases we have considered in this section.
\end{proof}

\section{Semistable objects for toral ring domains}\label{sect:spiral}

In this section, we describe the categories of semistable objects corresponding to the toral ring domains.

\subsection{String and band modules}

In order to describe these categories of semistable objects, we will use the notions of string and band modules, which we now describe.

\subsubsection{Construction of string and band modules}

Let $(Q, I)$ be a quiver with relations.
For every arrow $a \colon i \to j$ in $Q_{1}$, we write $a^{-1} \colon j \to i$ for its formal inverse. 
We write $Q_{1}^{-1}$ for the set of formal inverses of arrows of $Q_{1}$. 
We refer to elements of $Q_{1}$ as \emph{direct arrows}, and elements of $Q_{1}^{-1}$ as \emph{inverse arrows}. 
A \emph{walk} is a sequence $a_{1}^{\pm 1} \dots a_{r}^{\pm 1}$ of elements of $Q_{1} \cup Q_{1}^{-1}$ such that $t(a_{i}^{\pm 1}) = s(a_{i + 1}^{\pm 1})$ and $a_{i + 1}^{\pm 1} \neq a_{i}^{\mp 1}$ for every $i \in \{1, \dots, r - 1\}$. 
Given a walk $w = a_{1}^{\pm 1} \dots a_{r}^{\pm 1}$, we define $w^{-1} = a_{r}^{\mp 1} \dots a_{1}^{\mp 1}$.

A \emph{string} for $(Q, I)$ is a walk $w$ in $Q$ such that neither $w$, $w^{-1}$ nor any subword of either is a summand of an element of~$I$.
Given a string $w = a_{1}^{\pm 1} \dots a_{r}^{\pm 1}$, there are $r + 1$ instances where $w$ traverses a vertex, counting with multiplicity.
If we label these instances $0, 1, \dots, r$, and let $i \in Q_0$, then we label the instances of vertex $i$ by $k_1^i, k_2^i, \dots, k_{d_i}^i \in \{0, 1, \dots, r\}$, where $d_i$ is the number of times $w$ traverses the vertex~$i$.
That is, let $\{k_1^i, k_2^i, \dots, k_{d_i}^i\} =$
\[
 \{k \in \{0\} \st s(a_{k}^{\pm 1}) = i\} \cup \{k \in \{1, \dots r - 1\} \st t(a_{k}^{\pm 1}) = s(a_{k + 1}^{\pm 1}) = i\} \cup \{k \in \{r\} \st t(a_{k}^{\pm 1}) = i\},   
\]
 where $k_1^i < k_2^i < \dots < k_{d_i}^i$.
We define the quiver representation $M(w) := (V_i, f_a)_{i \in Q_0, a \in Q_1}$ with $V_i \cong \mathbb{C}^{d_i}$ and, for $a \colon i \to j \in Q_1$, $f_a \colon V_i \to V_j$ the linear map with components \[
(f_a)_{l, l'} = \left\{
    \begin{array}{ll}
        \mathrm{id}_{\mathbb{C}} & \text{ if } k_{l'}^{j} = k_{l}^{i} + 1 \text{ and } a_{k_{l'}^{j}}^{\pm 1} = a_{k_{l'}^{j}} \text{ is direct}, \\
        \mathrm{id}_{\mathbb{C}} & \text{ if } k_{l'}^{j} = k_{l}^{i} - 1 \text{ and } a_{k_{l}^{i}}^{\pm 1} = a_{k_{l}^{i}}^{-1} \text{ is inverse}, \\
        0 & \text{ otherwise.}
    \end{array}
\right.
\]
One can check that $M(w) \cong M(w^{-1})$.
A \emph{string module} is a $\mathbb{C}Q/I$-module which is isomorphic to $M(w)$ for a string~$w$.

A \emph{cyclic} string is a string which begins and ends at the same vertex.
A \emph{band} $v$ for $(Q, I)$ is defined to be a cyclic string such that every power $v^{n}$ is a string, but $v$ itself is not the proper power of some string~$w$.
A band $v = a_{1}^{\pm 1}a_{2}^{\pm 1} \dots a_{r}^{\pm 1}$ can be represented by several different words, such as $a_{2}^{\pm 1} \dots a_{r}^{\pm 1}a_{1}^{\pm 1}$. We refer to these words as \emph{rotations} of the band.
Given a band $v = a_{1}^{\pm 1}a_{2}^{\pm 1} \dots a_{r}^{\pm 1}$, we label the instances where a vertex $i \in Q_0$ is traversed in a similar way to before.
That is, we write 
\[\{k_1^i, k_2^i, \dots, k_{d_i}^i\} = \{k \in \{1, 2, \dots r\} \st t(a_{k}^{\pm 1}) = s(a_{k + 1}^{\pm 1}) = i\},\]
where $k_1^i < k_2^i < \dots < k_{d_i}^i$, with $k + 1$ taken modulo $r$.
Given $\lambda \in \mathbb{C}^{\ast}$ and $m \in \mathbb{Z}_{>0}$, we define the quiver representation $M(v, \lambda, m) := (V_i, f_a)_{i \in Q_0, a \in Q_1}$ with $V_i \cong (\mathbb{C}^{m})^{d_i}$.
For $a \colon i \to j \in Q_1$, we then have that $f_a \colon V_i \to V_j$ is the linear map with $m \times m$ block components \[
(f_a)_{l, l'} = \left\{
    \begin{array}{lll}
        \mathrm{id}_{\mathbb{C}^m} & \text{ if } k_{l'}^{j} = k_{l}^{i} + 1 \neq r &\text{ and } a_{k_{l'}^{j}}^{\pm 1} = a_{k_{l'}^{j}} \text{ is direct}, \\
        \mathrm{id}_{\mathbb{C}^m} & \text{ if } k_{l'}^{j} = k_{l}^{i} - 1 \neq r - 1 &\text{ and } a_{k_{l}^{i}}^{\pm 1} = a_{k_{l}^{i}}^{-1} \text{ is inverse}, \\
        \jord{\lambda}{m} & \text{ if } k_{l'}^{j} = k_{l}^{i} + 1 = r &\text{ and } a_{k_{l'}^{j}}^{\pm 1} = a_{k_{l'}^{j}} \text{ is direct}, \\
        \jord{\lambda}{m} & \text{ if } k_{l'}^{j} = k_{l}^{i} - 1 = r - 1 &\text{ and } a_{k_{l}^{i}}^{\pm 1} = a_{k_{l}^{i}}^{-1} \text{ is inverse}, \\
        0 & \text{ otherwise.}
    \end{array}
\right.
\]
where $\jord{\lambda}{m}$ is the $m \times m$ Jordan block with eigenvalue~$\lambda$.
A \emph{band module} is a $\mathbb{C}Q/I$-module which is isomorphic to $M(v, \lambda, m)$ for some band $v$, $\lambda \in \mathbb{C}$, and $m \in \mathbb{Z}_{>0}$.

\subsubsection{Gentle and locally gentle algebras}

We recall the notions of gentle and locally gentle algebras.

\begin{definition}[\cite{as,bh}]\label{def:gentle}
A quiver with relations $(Q, I)$ is \emph{locally gentle} if
 \begin{enumerate}
     \item for each vertex $i \in Q_0$, there are at most two arrows $a \in Q_1$ with $s(a) = i$ and at most two arrows $b \in Q_1$ with $t(b) = i$;\label{op:gentle:head_tail}
     \item the ideal $I$ is generated by paths of length exactly two;\label{op:gentle:length_two}
     \item for any arrow $b \in Q_1$, there is at most one path $ab$ of length two with $ab \notin I$ and at most one path $bc$ of length two with $bc \notin I$;\label{op:gentle:string}
     \item for any arrow $b \in Q_1$, there is at most one path $ab$ of length two with $ab \in I$ and at most one path $bc$ of length two with $bc \in I$.\label{op:gentle:not_string}
 \end{enumerate}
 One then calls $\mathbb{C}Q/I$ a \emph{locally gentle algebra}. If $\mathbb{C}Q/I$ is finite-dimensional, it is called a \emph{gentle algebra}.
\end{definition}

For locally gentle algebras, and hence also for gentle algebras, it is known that all indecomposable representations are either string modules or band modules.

\begin{theorem}[{\cite[Theorem~1.2]{cb}, \cite[Section~3]{br}, \cite{ww}}]\label{thm:gentle_string_band}
If $\mathbb{C}Q/I$ is a locally gentle algebra, and $M$ is a finite-dimensional indecomposable $\mathbb{C}Q/I$-module, then $M$ is either a string module or a band module.
\end{theorem}

If $\mathbb{C}Q/I$ is a locally gentle algebra with $\widehat{\mathbb{C}Q}/\widehat{I}$ the corresponding completed algebra, then, as in Section~\ref{sect:3cy_stab:qwp:quiver_reps}, we have that modules over $\widehat{\mathbb{C}Q}/\widehat{I}$ are given by nilpotent modules over $\mathbb{C}Q/I$.
Hence, indecomposable modules over $\widehat{\mathbb{C}Q}/\widehat{I}$ must also be either string or band modules.
We refer to $\widehat{\mathbb{C}Q}/\widehat{I}$ as a \emph{completed locally gentle algebra}.
If $\mathbb{C}Q/I$ is finite-dimensional, and so gentle, then $\widehat{\mathbb{C}Q}/\widehat{I} \cong \mathbb{C}Q/I$, so there is no distinction to make.

\subsubsection{Extensions between string and band modules}\label{sect:spiral:string_band:ext}

We now describe some criteria for the vanishing of extension spaces between string and band modules.
Extensions between string and band modules are explicitly described in \cite{cps} following earlier work of \cite{schroer} and \cite{zhang}.

We first define the following notions.
Let $w = a_{1}^{\pm 1} \dots a_{i}^{\pm 1} \dots a_{j}^{\pm 1} \dots a_{r}^{\pm 1}$ be a string with substring $u = a_{i}^{\pm 1} \dots a_{j}^{\pm 1}$. 
We say that $u$ is a \emph{submodule substring} if $a_{i - 1}^{\pm 1} = a_{i - 1}$ is a direct arrow, or $i = 1$, and $a_{j + 1}^{\pm 1} = a_{j + 1}^{-1}$ is an inverse arrow, or $j = r$.
We say that $u$ is a \emph{factor substring} if $a_{i - 1}^{\pm 1} = a_{i - 1}^{-1}$ is an inverse arrow, or $i = 1$, and $a_{j + 1}^{\pm 1} = a_{j + 1}$ is a direct arrow, or $j = r$.
In the former case $M(u)$ is indeed a submodule of $M(w)$, and in the latter case it is a factor module.

Given two strings $w_1$ and $w_2$, if there is an inverse arrow $a^{-1}$ such that $w_{1}a^{-1}w_{2}$ is a string, then there is a short exact sequence
\[0 \to M(w_1) \to M(w_{1}a^{-1}w_2) \to M(w_{2}) \to 0,\]
which we call an \emph{arrow extension}.
On the other hand, if we have $w_1 = w_{1, L}uw_{1, R}$ with $u$ a factor substring and $w_2 = w_{2, L}uw_{2, R}$ with $u$ a submodule substring, then there is a short exact sequence \[
0 \to M(w_1) \to M(w_{1, L}uw_{2, R}) \oplus M(w_{2, L}uw_{1, R}) \to M(w_2) \to 0,
\]
which we call an \emph{overlap extension}.
Here $w_{1, L}$, $w_{2, L}$, $w_{1, R}$, and $w_{2, R}$ are possibly empty.

\begin{theorem}[{\cite[Theorem~A]{cps}}]\label{thm:cps_a}
For $\mathbb{C}Q/I$ a gentle algebra with $w_{1}$ and $w_{2}$ two strings, the arrow and overlap extensions form a basis of $\Ext_{\mathbb{C}Q/I}^{1}(M(w_2), M(w_1))$.
\end{theorem}

Given a band $v$, we write $\dinf{v}$ for the infinite string given by repeating $v$ in both directions.
Let $M(v, \lambda, m)$ be a band module given by $v$ and $M(w)$ be a string module given by a string $w$.

\begin{theorem}[{\cite[Theorem~B]{cps}, \cite[Theorem~C]{cps}}]\label{thm:cps_bc}
Suppose that $\mathbb{C}Q/I$ is a gentle algebra with $w$ a string and $v$ a band.
If there is no factor substring of $\dinf{v}$ which is a submodule substring of $w$, then $\Ext_{\mathbb{C}Q/I}^{1}(M(w), M(v, \lambda, m)) = 0$.

Similarly, if there is no submodule substring of $\dinf{v}$ which is a factor substring of $w$, then $\Ext_{\mathbb{C}Q/I}^{1}(M(v, \lambda, m), M(w)) = 0$.
\end{theorem}

\subsection{Toral degenerate ring domains}\label{sect:spiral:drd}

We now consider the case from Section~\ref{sect:main:drd:ii}.
Recall that in this case the Riemann surface $X$ must be the torus, and the quadratic differential $\varphi$ has a double pole and two zeros as its critical points.
We begin by explicitly determining the category and stability condition.
Recall that, as explained in Section~\ref{sect:main:drd:ii}, in this case we must pick a quadratic differential $\varphi$ which has a toral degenerate ring domain in a rotation $\varphi^{(\varepsilon)}$, rather than having one itself.

\begin{proposition}\label{prop:drd_ii:3_vertex}
Suppose that $\varphi$ is a generic infinite GMN differential on a Riemann surface $X$ such that $\varphi^{(\varepsilon)}$ has a toral degenerate ring domain. The category $\ssc{\sigma_{\varphi}}{1 - \varepsilon}$ of $\sigma_\varphi$-semistable representations of phase $1 - \varepsilon$ is given by semistable representations of dimension vector $(d, d, d)$ of the quiver $Q$ 
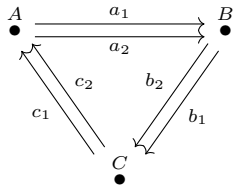
\[
\begin{tikzcd}
    \overset{A}{\bullet} \ar[rr,shift left,"a_1"] \ar[rr,shift right,"a_2",swap] && \overset{B}{\bullet} \ar[ddl,shift left,"b_1"] \ar[ddl,shift right,"b_2",swap] \\ \\
    & \overset{C}{\bullet} \ar[uul,shift left,"c_1"] \ar[uul,shift right,"c_2",swap] &
\end{tikzcd}
\] with potential $W = a_1 b_1 c_1 + a_2 b_2 c_2$ under a generic stability condition such that\[\scph{S_{\pi(C)}} < \scph{S_{\pi(A)} \oplus S_{\pi(B)} \oplus S_{\pi(C)}} < \scph{S_{\pi(A)}}\] where $\pi \colon \{A, B, C\} \to \{A, B, C\}$ is an even permutation, or equivalently, $\pi(A) < \pi(B) < \pi(C)$ is the usual cyclic ordering in the quiver.
Moreover, if we let the two type~I saddle trajectories in the boundary of the ring domain be $\gamma_1$ and $\gamma_2$, so that $\hhc{\gamma}_1 = \hhc{\gamma}_2 =: \hhc{\gamma}$, then we have $\hhc{\gamma}$ corresponds to the dimension vector $(1, 1, 1)$.
\end{proposition}
\begin{proof}
We first determine the quiver with potential.
The marked bordered surface $\mbso$ associated to $\varphi$ is the once-punctured torus.
All triangulations of the once-punctured torus are combinatorially equivalent.
The triangulations are as shown in Figure~\ref{fig:torus_triang}, and so give the quiver with potential in the statement of the proposition.
We therefore label the horizontal strips $H_A$, $H_B$, and~$H_C$, corresponding to the vertices shown, and let $\gamma_{A}$, $\gamma_{B}$, and $\gamma_{C}$ be the corresponding (type~I) standard saddle connections.

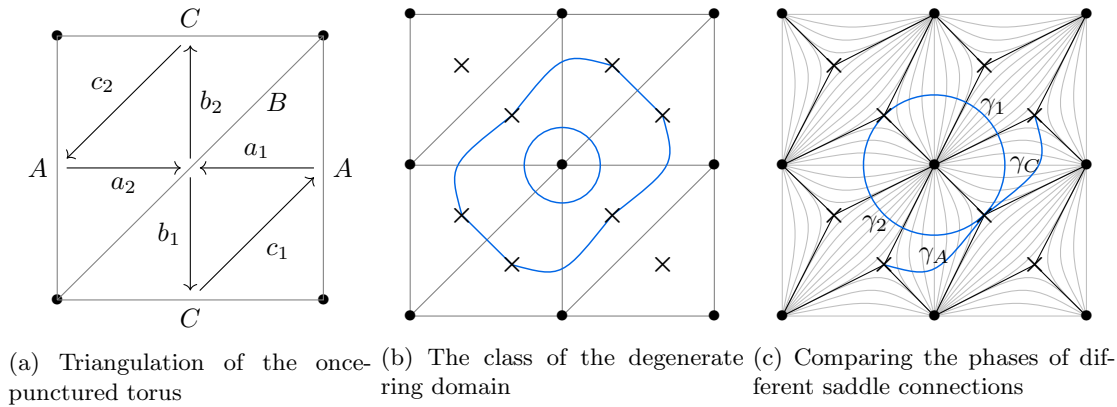
\begin{figure}[h]
\begin{subfigure}[t]{0.30\textwidth}
\[
    \begin{tikzpicture}[scale=3.5]

        \node at (0,0) {$\bullet$};
        \node at (1,0) {$\bullet$};
        \node at (0,1) {$\bullet$};
        \node at (1,1) {$\bullet$};

        \draw[\edgecol] (0,0) -- (1,0) -- (1,1) -- (0,1) -- (0,0);
        \draw[\edgecol] (0,0) -- (1,1);

        \node (l) at (0,0.5) {};
        \node (b) at (0.5,0) {};
        \node (r) at (1,0.5) {};
        \node (t) at (0.5,1) {};
        \node (m) at (0.5,0.5) {};

        \draw[\arrcol,->] (r) -- (m) node [midway,above,black] {$a_1$};
        \draw[\arrcol,->] (m) -- (b) node [midway,left,black] {$b_1$};
        \draw[\arrcol,->] (b) -- (r) node [midway,below right,black] {$c_1$};

        \draw[\arrcol,->] (l) -- (m) node [midway,below,black] {$a_2$};
        \draw[\arrcol,->] (m) -- (t) node [midway,right,black] {$b_2$};
        \draw[\arrcol,->] (t) -- (l) node [midway,above left,black] {$c_2$};

        \node at (l) [left] {$A$};
        \node at (r) [right] {$A$};
        \node at (t) [above] {$C$};
        \node at (b) [below] {$C$};
        \node at (0.75,0.75) [right] {$B$};  
     
    \end{tikzpicture}    
\]
    \subcaption{Triangulation of the once-punctured torus}\label{fig:torus_triang}
\end{subfigure}
\begin{subfigure}[t]{0.30\textwidth}
     \[
    \begin{tikzpicture}[scale=2.0]

    \draw[\edgecol] (-1,-1) -- (1,-1) -- (1,1) -- (-1,1) -- (-1,-1);
    \draw[\edgecol] (-1,0) -- (1,0);
    \draw[\edgecol] (0,-1) -- (0,1);
    \draw[\edgecol] (-1,-1) -- (1,1);
    \draw[\edgecol] (0,-1) -- (1,0);
    \draw[\edgecol] (-1,0) -- (0,1);

    \node at (-1,-1) {$\bullet$};
    \node at (1,-1) {$\bullet$};
    \node at (-1,1) {$\bullet$};
    \node at (1,1) {$\bullet$};
    \node at (-1,0) {$\bullet$};
    \node at (1,0) {$\bullet$};
    \node at (0,1) {$\bullet$};
    \node at (0,-1) {$\bullet$};
    \node at (0,0) {$\bullet$};

    \coordinate (ltrb) at ($0.33*(-1,0) + 0.33*(0,1) + 0.33*(0,0)$);
    \coordinate (ltlt) at ($0.33*(-1,0) + 0.33*(0,1) + 0.33*(-1,1)$);

    \coordinate (rtrb) at ($0.33*(0,0) + 0.33*(1,1) + 0.33*(1,0)$);
    \coordinate (rtlt) at ($0.33*(0,0) + 0.33*(1,1) + 0.33*(0,1)$);

    \coordinate (lbrb) at ($0.33*(-1,-1) + 0.33*(0,0) + 0.33*(0,-1)$);
    \coordinate (lblt) at ($0.33*(-1,-1) + 0.33*(0,0) + 0.33*(-1,0)$);

    \coordinate (rbrb) at ($0.33*(0,-1) + 0.33*(1,0) + 0.33*(1,-1)$);
    \coordinate (rblt) at ($0.33*(0,-1) + 0.33*(1,0) + 0.33*(0,0)$);

    \draw[\conn, line width=0.55pt] (lblt) .. controls (-0.75,0) .. (ltrb);
    \draw[\conn, line width=0.55pt] (ltrb) .. controls (0,0.75) .. (rtlt);
    \draw[\conn, line width=0.55pt] (rtlt) -- (rtrb);
    \draw[\conn, line width=0.55pt] (rtrb) .. controls (0.75,0) .. (rblt);
    \draw[\conn, line width=0.55pt] (rblt)  .. controls (0,-0.75) .. (lbrb);
    \draw[\conn, line width=0.55pt] (lbrb) -- (lblt);

    \draw[\conn, line width=0.55pt] (0,0) circle (0.25cm);
    
    \node at (ltrb) {$\bm{\times}$};
    \node at (ltlt) {$\bm{\times}$};

    \node at (rtrb) {$\bm{\times}$};
    \node at (rtlt) {$\bm{\times}$};

    \node at (lbrb) {$\bm{\times}$};
    \node at (lblt) {$\bm{\times}$};

    \node at (rbrb) {$\bm{\times}$};
    \node at (rblt) {$\bm{\times}$};    
        
    \end{tikzpicture}
    \]
    \subcaption{The class of the degenerate ring domain}\label{fig:torus_drd}
\end{subfigure}
\begin{subfigure}[t]{0.30\textwidth}
    \[
    \begin{tikzpicture}[scale=2.0]

    \draw[\generictraj] (-1,-1) -- (1,-1) -- (1,1) -- (-1,1) -- (-1,-1);
    \draw[\generictraj] (-1,0) -- (1,0);
    \draw[\generictraj] (0,-1) -- (0,1);
    \draw[\generictraj] (-1,-1) -- (1,1);
    \draw[\generictraj] (0,-1) -- (1,0);
    \draw[\generictraj] (-1,0) -- (0,1);

    \coordinate (ltrb) at ($0.33*(-1,0) + 0.33*(0,1) + 0.33*(0,0)$);
    \coordinate (ltlt) at ($0.33*(-1,0) + 0.33*(0,1) + 0.33*(-1,1)$);

    \coordinate (rtrb) at ($0.33*(0,0) + 0.33*(1,1) + 0.33*(1,0)$);
    \coordinate (rtlt) at ($0.33*(0,0) + 0.33*(1,1) + 0.33*(0,1)$);

    \coordinate (lbrb) at ($0.33*(-1,-1) + 0.33*(0,0) + 0.33*(0,-1)$);
    \coordinate (lblt) at ($0.33*(-1,-1) + 0.33*(0,0) + 0.33*(-1,0)$);

    \coordinate (rbrb) at ($0.33*(0,-1) + 0.33*(1,0) + 0.33*(1,-1)$);
    \coordinate (rblt) at ($0.33*(0,-1) + 0.33*(1,0) + 0.33*(0,0)$);

    \draw[\generictraj] (0,0) .. controls ($1*(ltrb)+0*(0,0.5)$) .. (0,1);
    \draw[\generictraj] (0,0) .. controls ($0.67*(ltrb)+0.33*(0,0.5)$) .. (0,1);
    \draw[\generictraj] (0,0) .. controls ($0.33*(ltrb)+0.67*(0,0.5)$) .. (0,1);

    \draw[\generictraj] (-1,0) .. controls ($1*(ltrb)+0*(-0.5,0.5)$) .. (0,1);
    \draw[\generictraj] (-1,0) .. controls ($0.67*(ltrb)+0.33*(-0.5,0.5)$) .. (0,1);
    \draw[\generictraj] (-1,0) .. controls ($0.33*(ltrb)+0.67*(-0.5,0.5)$) .. (0,1);

    \draw[\generictraj] (-1,0) .. controls ($1*(ltrb)+0*(-0.5,0)$) .. (0,0);
    \draw[\generictraj] (-1,0) .. controls ($0.67*(ltrb)+0.33*(-0.5,0)$) .. (0,0);
    \draw[\generictraj] (-1,0) .. controls ($0.33*(ltrb)+0.67*(-0.5,0)$) .. (0,0);

    \draw[\septraj] (ltrb) -- (0,1);
    \draw[\septraj] (ltrb) -- (-1,0);
    \draw[\septraj] (ltrb) -- (0,0);

    \draw[\generictraj] (1,0) .. controls ($1*(rtrb)+0*(1,0.5)$) .. (1,1);
    \draw[\generictraj] (1,0) .. controls ($0.67*(rtrb)+0.33*(1,0.5)$) .. (1,1);
    \draw[\generictraj] (1,0) .. controls ($0.33*(rtrb)+0.67*(1,0.5)$) .. (1,1);

    \draw[\generictraj] (0,0) .. controls ($1*(rtrb)+0*(0.5,0.5)$) .. (1,1);
    \draw[\generictraj] (0,0) .. controls ($0.67*(rtrb)+0.33*(0.5,0.5)$) .. (1,1);
    \draw[\generictraj] (0,0) .. controls ($0.33*(rtrb)+0.67*(0.5,0.5)$) .. (1,1);

    \draw[\generictraj] (0,0) .. controls ($1*(rtrb)+0*(0.5,0)$) .. (1,0);
    \draw[\generictraj] (0,0) .. controls ($0.67*(rtrb)+0.33*(0.5,0)$) .. (1,0);
    \draw[\generictraj] (0,0) .. controls ($0.33*(rtrb)+0.67*(0.5,0)$) .. (1,0);

    \draw[\septraj] (rtrb) -- (1,1);
    \draw[\septraj] (rtrb) -- (0,0);
    \draw[\septraj] (rtrb) -- (1,0);

    \draw[\generictraj] (0,-1) .. controls ($1*(lbrb)+0*(0,-0.5)$) .. (0,0);
    \draw[\generictraj] (0,-1) .. controls ($0.67*(lbrb)+0.33*(0,-0.5)$) .. (0,0);
    \draw[\generictraj] (0,-1) .. controls ($0.33*(lbrb)+0.67*(0,-0.5)$) .. (0,0);

    \draw[\generictraj] (-1,-1) .. controls ($1*(lbrb)+0*(-0.5,-0.5)$) .. (0,0);
    \draw[\generictraj] (-1,-1) .. controls ($0.67*(lbrb)+0.33*(-0.5,-0.5)$) .. (0,0);
    \draw[\generictraj] (-1,-1) .. controls ($0.33*(lbrb)+0.67*(-0.5,-0.5)$) .. (0,0);

    \draw[\generictraj] (-1,-1) .. controls ($1*(lbrb)+0*(-0.5,-1)$) .. (0,-1);
    \draw[\generictraj] (-1,-1) .. controls ($0.67*(lbrb)+0.33*(-0.5,-1)$) .. (0,-1);
    \draw[\generictraj] (-1,-1) .. controls ($0.33*(lbrb)+0.67*(-0.5,-1)$) .. (0,-1);

    \draw[\septraj] (lbrb) -- (0,0);
    \draw[\septraj] (lbrb) -- (-1,-1);
    \draw[\septraj] (lbrb) -- (0,-1);

    \draw[\generictraj] (1,-1) .. controls ($1*(rbrb)+0*(1,-0.5)$) .. (1,0);
    \draw[\generictraj] (1,-1) .. controls ($0.67*(rbrb)+0.33*(1,-0.5)$) .. (1,0);
    \draw[\generictraj] (1,-1) .. controls ($0.33*(rbrb)+0.67*(1,-0.5)$) .. (1,0);

    \draw[\generictraj] (0,-1) .. controls ($1*(rbrb)+0*(0.5,-0.5)$) .. (1,0);
    \draw[\generictraj] (0,-1) .. controls ($0.67*(rbrb)+0.33*(0.5,-0.5)$) .. (1,0);
    \draw[\generictraj] (0,-1) .. controls ($0.33*(rbrb)+0.67*(0.5,-0.5)$) .. (1,0);

    \draw[\generictraj] (0,-1) .. controls ($1*(rbrb)+0*(0.5,-1)$) .. (1,-1);
    \draw[\generictraj] (0,-1) .. controls ($0.67*(rbrb)+0.33*(0.5,-1)$) .. (1,-1);
    \draw[\generictraj] (0,-1) .. controls ($0.33*(rbrb)+0.67*(0.5,-1)$) .. (1,-1);

    \draw[\septraj] (rbrb) -- (1,0);
    \draw[\septraj] (rbrb) -- (0,-1);
    \draw[\septraj] (rbrb) -- (1,-1);

    \draw[\generictraj] (-1,1) .. controls ($1*(ltlt)+0*(-1,0.5)$) .. (-1,0);
    \draw[\generictraj] (-1,1) .. controls ($0.67*(ltlt)+0.33*(-1,0.5)$) .. (-1,0);
    \draw[\generictraj] (-1,1) .. controls ($0.33*(ltlt)+0.67*(-1,0.5)$) .. (-1,0);

    \draw[\generictraj] (0,1) .. controls ($1*(ltlt)+0*(-0.5,0.5)$) .. (-1,0);
    \draw[\generictraj] (0,1) .. controls ($0.67*(ltlt)+0.33*(-0.5,0.5)$) .. (-1,0);
    \draw[\generictraj] (0,1) .. controls ($0.33*(ltlt)+0.67*(-0.5,0.5)$) .. (-1,0);

    \draw[\generictraj] (0,1) .. controls ($1*(ltlt)+0*(-0.5,1)$) .. (-1,1);
    \draw[\generictraj] (0,1) .. controls ($0.67*(ltlt)+0.33*(-0.5,1)$) .. (-1,1);
    \draw[\generictraj] (0,1) .. controls ($0.33*(ltlt)+0.67*(-0.5,1)$) .. (-1,1);

    \draw[\septraj] (ltlt) -- (-1,0);
    \draw[\septraj] (ltlt) -- (0,1);
    \draw[\septraj] (ltlt) -- (-1,1);

    \draw[\generictraj] (0,1) .. controls ($1*(rtlt)+0*(0,0.5)$) .. (0,0);
    \draw[\generictraj] (0,1) .. controls ($0.67*(rtlt)+0.33*(0,0.5)$) .. (0,0);
    \draw[\generictraj] (0,1) .. controls ($0.33*(rtlt)+0.67*(0,0.5)$) .. (0,0);

    \draw[\generictraj] (1,1) .. controls ($1*(rtlt)+0*(0.5,0.5)$) .. (0,0);
    \draw[\generictraj] (1,1) .. controls ($0.67*(rtlt)+0.33*(0.5,0.5)$) .. (0,0);
    \draw[\generictraj] (1,1) .. controls ($0.33*(rtlt)+0.67*(0.5,0.5)$) .. (0,0);

    \draw[\generictraj] (1,1) .. controls ($1*(rtlt)+0*(0.5,1)$) .. (0,1);
    \draw[\generictraj] (1,1) .. controls ($0.67*(rtlt)+0.33*(0.5,1)$) .. (0,1);
    \draw[\generictraj] (1,1) .. controls ($0.33*(rtlt)+0.67*(0.5,1)$) .. (0,1);

    \draw[\septraj] (rtlt) -- (0,0);
    \draw[\septraj] (rtlt) -- (1,1);
    \draw[\septraj] (rtlt) -- (0,1);

    \draw[\generictraj] (-1,0) .. controls ($1*(lblt)+0*(-1,-0.5)$) .. (-1,-1);
    \draw[\generictraj] (-1,0) .. controls ($0.67*(lblt)+0.33*(-1,-0.5)$) .. (-1,-1);
    \draw[\generictraj] (-1,0) .. controls ($0.33*(lblt)+0.67*(-1,-0.5)$) .. (-1,-1);

    \draw[\generictraj] (0,0) .. controls ($1*(lblt)+0*(-0.5,-0.5)$) .. (-1,-1);
    \draw[\generictraj] (0,0) .. controls ($0.67*(lblt)+0.33*(-0.5,-0.5)$) .. (-1,-1);
    \draw[\generictraj] (0,0) .. controls ($0.33*(lblt)+0.67*(-0.5,-0.5)$) .. (-1,-1);

    \draw[\generictraj] (0,0) .. controls ($1*(lblt)+0*(-0.5,0)$) .. (-1,0);
    \draw[\generictraj] (0,0) .. controls ($0.67*(lblt)+0.33*(-0.5,0)$) .. (-1,0);
    \draw[\generictraj] (0,0) .. controls ($0.33*(lblt)+0.67*(-0.5,0)$) .. (-1,0);

    \draw[\septraj] (lblt) -- (-1,-1);
    \draw[\septraj] (lblt) -- (0,0);
    \draw[\septraj] (lblt) -- (-1,0);

    \draw[\generictraj] (0,0) .. controls ($1*(rblt)+0*(0,-0.5)$) .. (0,-1);
    \draw[\generictraj] (0,0) .. controls ($0.67*(rblt)+0.33*(0,-0.5)$) .. (0,-1);
    \draw[\generictraj] (0,0) .. controls ($0.33*(rblt)+0.67*(0,-0.5)$) .. (0,-1);

    \draw[\generictraj] (1,0) .. controls ($1*(rblt)+0*(0.5,-0.5)$) .. (0,-1);
    \draw[\generictraj] (1,0) .. controls ($0.67*(rblt)+0.33*(0.5,-0.5)$) .. (0,-1);
    \draw[\generictraj] (1,0) .. controls ($0.33*(rblt)+0.67*(0.5,-0.5)$) .. (0,-1);

    \draw[\generictraj] (1,0) .. controls ($1*(rblt)+0*(0.5,0)$) .. (0,0);
    \draw[\generictraj] (1,0) .. controls ($0.67*(rblt)+0.33*(0.5,0)$) .. (0,0);
    \draw[\generictraj] (1,0) .. controls ($0.33*(rblt)+0.67*(0.5,0)$) .. (0,0);

    \draw[\septraj] (rblt) -- (0,-1);
    \draw[\septraj] (rblt) -- (1,0);
    \draw[\septraj] (rblt) -- (0,0);

    \draw[\conn, line width=0.55pt] (ltrb) arc (135:-45:0.465);
    \draw[\conn, line width=0.55pt] (ltrb) arc (135:315:0.465);

    \draw[\conn, line width=0.55pt] (rblt) .. controls (0.75,0) .. (rtrb);
    \draw[\conn, line width=0.55pt] (rblt) .. controls (0,-0.75) .. (lbrb);

    \node at (0,-0.6) {$\gamma_A$};
    \node at (0.6,0) {$\gamma_C$};
    \node at (45:0.55) {$\gamma_{1}$};
    \node at (225:0.55) {$\gamma_{2}$};

    \node at (ltrb) {$\bm{\times}$};
    \node at (ltlt) {$\bm{\times}$};

    \node at (rtrb) {$\bm{\times}$};
    \node at (rtlt) {$\bm{\times}$};

    \node at (lbrb) {$\bm{\times}$};
    \node at (lblt) {$\bm{\times}$};

    \node at (rbrb) {$\bm{\times}$};
    \node at (rblt) {$\bm{\times}$};

    \node at (-1,-1) {$\bullet$};
    \node at (1,-1) {$\bullet$};
    \node at (-1,1) {$\bullet$};
    \node at (1,1) {$\bullet$};
    \node at (-1,0) {$\bullet$};
    \node at (1,0) {$\bullet$};
    \node at (0,1) {$\bullet$};
    \node at (0,-1) {$\bullet$};
    \node at (0,0) {$\bullet$};      
        
    \end{tikzpicture}
    \]
    \subcaption{Comparing the phases of different saddle connections}    \label{fig:torus_drd_saddles}
\end{subfigure}
\caption{Triangulation and finite-length connections, illustrated using four copies of the fundamental domain of the torus}\label{fig:toral_drd}       
\end{figure}

Now we determine which dimension vectors correspond to the degenerate ring domain and the type~I saddle trajectories in its boundaries.
Considering the pre-images on the spectral cover, it can be seen from Figure~\ref{fig:torus_drd} that the hat-homology class associated to a closed connection in the degenerate ring domain is equal to the sum of six standard saddle classes, up to sign.
Moreover, the signs must all be positive due to the definition of the orientation of the hat-homology classes in Section~\ref{sect:qd:ssc}: the evaluation of the 1-form on the spectral cover must have positive imaginary part for each of the standard saddle classes, as well as the class of the closed connection in the ring domain.
Amongst these six saddle connections, each of the standard saddle connections corresponding to the vertices of the quiver is represented twice, so we obtain that $2\hhc{\gamma} = (2, 2, 2)$ and $\hhc{\gamma} = (1, 1, 1)$.

Finally, we determine the stability condition.
We know that the stability condition must be generic, since $\varphi$ was a generic quadratic differential.
We pick a zero and assume that $\gamma_1$ first crosses $H_C$ as it emanates from this zero, which implies that $\gamma_2$ first crosses $H_A$.
That is, we assume that we are in the situtation shown in Figure~\ref{fig:torus_drd_saddles}.
The cases where $\gamma_1$ first crosses a horizontal strip with a different label are related to this one by an even permutation of the letters $\{A, B, C\}$.

It can be seen from $H_C$ that we must have that $\qdph{\gamma_C} < \qdph{\gamma_{1}}$, since the phase of $\gamma_C$ must be less than the phase of $\gamma_{1}$ if $\gamma_{1}$ is to pass $\gamma_C$ on the left, as illustrated in Figure~\ref{fig:torus_drd_saddles}.
Hence we have $\scph{S_C} = \qdph{\gamma_C} < \qdph{\gamma_{1}} = \scph{S_A \oplus S_B \oplus S_C}$.
Similarly, by considering $H_A$, one obtains that $\scph{S_A \oplus S_B \oplus S_C} < \scph{S_A}$.
\end{proof}

We denote the category of semistable quiver representations from Proposition~\ref{prop:drd_ii:3_vertex} by $\mathcal{P}$, and denote its 3-Calabi--Yau $A_\infty$-enhancement by $\mathcal{P}_{\infty}$, which consists of twisted complexes in $\mathcal{H}_\infty(Q,W)$ which give semistable objects in~$\mathcal{P}$.
We now describe the stable representations in~$\mathcal{P}$.

\begin{lemma}\label{lem:drd_ii:stables}
There are precisely two stable representations in $\mathcal{P}$, namely $M(a_{1}b_{2})$ and $M(a_{2}b_{1})$.
\end{lemma}
\begin{proof}
It can be verified using Definition~\ref{def:gentle} that the Jacobian algebra $\jac{Q, W}$ is a completed locally gentle algebra.
Hence, we know that all of its indecomposable representations are string and band modules by Theorem~\ref{thm:gentle_string_band}.
If we then let $M$ be a stable module of phase $\scph{S_{A} \oplus S_{B} \oplus S_{C}}$, we have that $M$ is either a string module or a band module.
Hence, let $w$ be the string giving rise to $M$.

Suppose first that $w$ does not contain either $c_{1}^{\pm 1}$ or $c_{2}^{\pm 1}$.
Since the stability condition is generic, we know that $\dimu M = (d, d, d)$.
If $w$ does not contain $c_1^{\pm 1}$ or~$c_2^{\pm 1}$, then the only way of reaching the vertices $A$ and $C$ is by going via the vertex $B$.
Hence, the only way that $\dimu M$ can be a multiple of $(1, 1, 1)$ is if $w^{\pm 1} = a_{1}b_{2}$ or $w^{\pm 1} = a_{2}b_{1}$.

Suppose now that $w$ contains $c_{1}^{\pm 1}$ as a letter, the case where $w$ contains $c_{2}^{\pm 1}$ being similar.
By inverting $w$, we may assume that we have $c_{1}$.
We know that $c_{1}$ cannot be the last letter of $w$ with $M = M(w)$ a string module, since then $S_{A}$ is a submodule of $M$, which contradicts stability.
Similarly, $c_{1}$ cannot be succeeded by $c_{2}^{-1}$ in $w$, since then $S_{A}$ is again a submodule.
It follows that $c_{1}$ must be succeeded by $a_{2}$.
We cannot then have $c_{1}a_{2}a_{1}^{-1}$, since $a_{2}$ is then a submodule substring with $\dimu M(a_{2}) = (1, 1, 0)$, which contradicts stability; for the same reason, the string cannot end $c_{1}a_{2}$.
Hence, we must have $c_{1}a_{2}b_{1}$.
If we have $c_{1}a_{2}b_{1}b_{2}^{-1}$, then $a_{2}b_{1}$ is a submodule substring with $\dimu M(a_{2}b_{1}) = (1, 1, 1)$, which contradicts stability.
We must therefore have $c_{1}a_{2}b_{1}c_{2}$.

One can repeat the same analysis with $c_{2}$, so we conclude that $w$ must be a string of the form $\dots c_{1}a_{2}b_{1}c_{2}a_{1}b_{2}c_{1}a_{2} \dots$.
Consequently, $M$ cannot be a band module, since such a band module $M(w, \lambda, m)$ would not be nilpotent.
Since $w$ cannot end or begin with $c_{1}$ or $c_{2}$, as argued in the previous paragraph, we must have that $w = a_{1}b_{2}c_{1} \dots$ or $w = a_{2}b_{1}c_{2} \dots$.
This then gives that either $M(a_{1}b_{2})$ or $M(a_{2}b_{1})$ is a factor module, which again contradicts stability.
We conclude that $M(a_{1}b_{2})$ and $M(a_{2}b_{1})$ are the only stable representations of phase $\scph{S_{A} \oplus S_{B} \oplus S_{C}}$, as desired.
\end{proof}

We can now describe $\mathcal{P}_\infty$ in the more tractable form given in Proposition~\ref{prop:drd_ii}, which allowed us to compute the associated DT invariants.

\begin{proposition}\label{prop:drd_ii:2_vertex}
The $A_\infty$-category $\mathcal{P}_\infty$ is quasi-equivalent to $\mathcal{H}_\infty(Q',0)$, where $Q'$ is the quiver
\[\begin{tikzcd}\overset{1}{\bullet} \ar[r,shift left] & \overset{2}{\bullet}. \ar[l,shift left]\end{tikzcd}\]
Under this equivalence, $S_{1}$ and $S_{2}$ both correspond to modules of dimension vector $(1, 1, 1)$.
Moreover, there is an isomorphism between the canonical choices of orientation data on the stacks of objects in each of these categories.
\end{proposition}
\begin{proof}
We assume, following the proof of Proposition~\ref{prop:drd_ii:3_vertex}, that the stability condition $\sigma_\varphi$ is such that $\scph{S_C} < \scph{S_A \oplus S_B \oplus S_C}  < \scph{S_A}$, since the other cases are similar.
Hence, we may use Lemma~\ref{lem:drd_ii:stables} to conclude that the only stable representations of the quiver with potential $(Q, W)$ are the string modules $M_{1} := M(a_{1}b_{2})$ and $M_{2} := M(a_{2}b_{1})$.
Note that this fact remains true even if we are given more information about the phases of other objects.

We consider the full $A_\infty$-subcategory of $\mathcal{P}_\infty$ consisting of $M_1$ and $M_2$.
We have that \[\Ext_{\mathcal{P}_\infty}^{1}(M_{1}, M_{2}) = \Ext_{\mathcal{P}_\infty}^{1}(M_{2}, M_{1}) = \mathbb{C},\] being generated respectively by arrow extensions given by $c_{2}$ and~$c_{1}$.
The degree~$0$ morphisms are just the identities.
The morphisms of degree~$2$ and $3$ are paired with those of degree~$0$ and $1$ the 3-Calabi--Yau property.
Moreover, one can check from the definition of the $A_\infty$-operations in $\mathcal{D}_\infty(Q,W)$ that all $A_\infty$-operations are zero on the degree~$1$ morphisms.
The objects $M_1$ and $M_2$ generate the category of semistable objects, in the sense that every semistable module possesses a Jordan--H\"older filtration with stable modules as factors.
Hence, we have that $\mathcal{P}_\infty$ is quasi-equivalent to $\mathcal{H}_\infty(Q',0)$.

We now verify that the orientation data $\odchoice$ for $(Q, W)$ coincides with that of $\odcha{Q'}$ for the representations $M_{1}$ and~$M_{2}$, which correspond to the nilpotent simple representations of $Q'$.
We compute the parity of the orientation data by computing \[
e^{\leqslant 1}(M_1, M_1) + h^{\leqslant 1}(M_1, M_1),
\]
following Proposition~\ref{prop:od_ug}\ref{op:od_ug:calc}.
By Proposition~\ref{prop:od_ug}\ref{op:od_ug:simples->even}, we must show that this is even to show that we have an isomorphism between $\odchoice$ and $\odcha{Q'}$. 
We have that $e^{\leqslant 1}(M_1, M_1) = 1$, and
\begin{align*}
   h^{0}(M_1, M_1) &= h^{0}(S_A \oplus S_B \oplus S_C, S_A \oplus S_B \oplus S_C) = 3, \\
    h^{1}(M_1, M_1) &= h^{1}(S_A \oplus S_B \oplus S_C, S_A \oplus S_B \oplus S_C) \\
    &= h^{1}(S_C, S_A) +h^{1}(S_A, S_B) + h^{1}(S_B, S_C) = 6.
\end{align*}
Adding these up, we obtain $1 + 3 + 6 = 10$, which is even. The computation for $M_2$ is identical.
\end{proof}

\subsection{Toral non-degenerate ring domains}\label{sect:spiral:nrd}

We consider the case from Section~\ref{sect:main:nrd:ii}, where we have a toral non-degenerate ring domain.
Note that this case cannot have polar type $\{-2\}$.

\begin{proposition}\label{prop:nrd_ii:4_vertex}
Suppose that we have a generic infinite GMN differential $\varphi$ on a Riemann surface $X$ with a toral non-degenerate ring domain. The category $\ssc{\sigma_{\varphi}}{1}$ is given by semistable representations of dimension vector $(d, d, d, d)$ of the quiver $Q$ \[
\begin{tikzcd}
    A \ar[rr,"b_1"] \ar[ddrr,"b_2",near end,swap] && B \ar[ddll,"c_1",near start,swap] \\\\
    D \ar[uu,shift left,"a_1"] \ar[uu,shift right,"a_2",swap] && C \ar[ll,"c_2"] \ar[uu,"a_3",swap]
\end{tikzcd}
\] with potential $W = a_1 b_1 c_1 + a_2 b_2 c_2$ under a stability condition such that
\begin{align*}
    \scph{S_B} < &\scph{S_A \oplus S_B \oplus S_C \oplus S_D} < \scph{S_C}, \\
    \scph{S_D} < &\scph{S_A \oplus S_B \oplus S_C \oplus S_D} < \scph{S_A}, \\
    \scph{S_A \oplus S_B} < &\scph{S_A \oplus S_B \oplus S_C \oplus S_D} < \scph{S_C \oplus S_D}.
\end{align*}
Denoting the hat-homology class of the two type~I saddles in the boundary of the ring domain by $\hhc{\gamma}$, we have that $\hhc{\gamma}$ corresponds to the dimension vector $(1, 1, 1, 1)$.
\end{proposition}
\begin{proof}
Aside from the non-degenerate ring domain $R$ and the adjacent spiral domain $S$, the rest of $X\setminus \mathcal{W}_0(\varphi)$ consists of horizontal strips and half-planes by Proposition~\ref{prop:fin_len_traj}.
Hence, the domain $H$ adjacent to the other side of $R$ from $S$ is either a horizontal strip or a half-plane, analogous to the other non-degenerate ring domain cases in Section~\ref{sect:main:nrd}.

By Proposition~\ref{prop:sep_traj_persist}, we can rotate $\varphi$ to a differential $\varphi^{(\varepsilon)}$ which preserves the separating trajectories on the other boundary of~$H$, as well as the separating trajectory which emerges from the zero which has its own boundary of~$R$.
The part of the corresponding marked bordered surface $\mbso$ cut out by the arc or boundary component corresponding to~$H$ is a torus with one boundary component with a marked point on it.
All triangulations of the torus with one boundary component with one marked point are combinatorially equivalent, and are given as shown in Figure~\ref{fig:nrd_ii_triang}, which then gives the quiver with potential in the statement. 

We illustrate the associated trajectory structures of $\varphi^{(\varepsilon)}$ in Figure~\ref{fig:nrd_ii:traj}, with the connections forming the boundary of the ring domain drawn in blue, and $S$ the region in the centre.
Note that the zero shared by $H$ and $R$ must also be the zero whose separating trajectories enclose the boundary component of the torus.
Indeed, before the rotation, this zero has one trajectory going into the infinite critical point; rotation preserves this separating trajectory, whilst one of the other trajectories emanating from the zero emanates around the other side of the boundary component into the infinite critical point.
This is similar to the situation in Section~\ref{sect:main:nrd:i}.
(Note that Figure~\ref{fig:nrd_ii:traj} should only be viewed as schematic, since the trajectories may spiral into the infinite critical point, rather than go in straight as shown.)

It follows from an argument similar to that of Proposition~\ref{prop:drd_ii:3_vertex} that the hat-homology class of each of the type~I saddle trajectories in the boundary of $R$ is $(1, 1, 1, 1)$ and that the class of a closed trajectory in the ring domain is $(2, 2, 2, 2)$.

The statements on the relative phases of certain modules under the stability condition then also follow from an argument similar to that of Proposition~\ref{prop:drd_ii:3_vertex}.
Indeed, consider the standard saddle connections $\gamma_{B}$ and $\gamma_{C}$ crossing the respective horizontal strips corresponding to $B$ and~$C$.
Comparing the phases of these saddle connections with the closed saddle connection which forms one boundary of $R$, one sees in Figure~\ref{fig:nrd_ii:traj} that $\scph{S_B} < \scph{S_A \oplus S_B \oplus S_C \oplus S_D} < \scph{S_C}$ from considering the standard saddle connections $\gamma_B$ and $\gamma_C$ corresponding to $B$ and $C$ respectively, which are incident to the zero which lies on its own boundary of the non-degenerate ring domain.

Similarly, we compare the phases of the standard saddle connections $\gamma_{A}$ and $\gamma_{D}$ crossing the horizontal strips corresponding to $A$ and~$D$ with the other boundary of the ring domain.
Consulting Figure~\ref{fig:nrd_ii:traj}, we obtain that $\scph{S_D} < \scph{S_A \oplus S_B \oplus S_C \oplus S_D} < \scph{S_A}$.

Finally, note that the ring domain contains many saddle connections from the zero on one boundary to the zeros on the other.
In particular, there is a saddle connection whose hat-homology class is $\hhc{\gamma}_{A} + \hhc{\gamma}_{B}$.
In Figure~\ref{fig:nrd_ii:traj}, this saddle connection is illustrated in red.
By comparing the phase of this saddle connection to the phase of the boundary saddle connection of the ring domain emanating from the bottom left zero, we obtain that $\scph{S_A \oplus S_B} < \scph{S_A \oplus S_B \oplus S_C \oplus S_D}$. Using a similar argument with a saddle connection going to the other zero on the boundary adjoining the spiral domain, we obtain that $\scph{S_A \oplus S_B} < \scph{S_A \oplus S_B \oplus S_C \oplus S_D} < \scph{S_C \oplus S_D}$.
\end{proof}

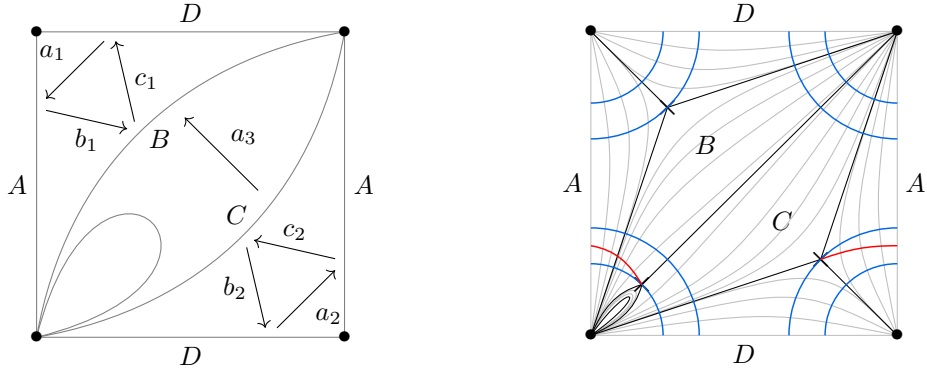
\begin{figure}
\begin{subfigure}[t]{0.45\textwidth}
\[
    \begin{tikzpicture}[scale=4.05]

        \draw[\edgecol] (0,0) -- (0,1) -- (1,1) -- (1,0) -- (0,0);
        \draw[\edgecol] (0,0) .. controls (0.2,0.8) and (0.8,0.2) .. (0,0);

        \draw[\edgecol] (0,0) to [out=10,in=260] (1,1);
        \draw[\edgecol] (0,0) to [out=80,in=190] (1,1);

        \node at (0,0.5) [left] {$A$};
        \node at (1,0.5) [right] {$A$};
        \node at (0.4,0.65) {$B$};
        \node at (0.65,0.4) {$C$};
        \node at (0.5,0) [below] {$D$};
        \node at (0.5,1) [above] {$D$};

        \node (AL) at (0,0.75) {};
        \node (DA) at (0.25,1) {};
        \node (B1) at (0.325,0.675) {};

        \draw[\arrcol,->] (DA) -- (AL) node [midway,above left] {\textcolor{black}{$a_1$}};
        \draw[\arrcol,->] (AL) -- (B1)  node [midway,below] {\textcolor{black}{$b_1$}};
        \draw[\arrcol,->] (B1) -- (DA)  node [midway,right] {\textcolor{black}{$c_1$}};

        \node (AR) at (0.75,0) {};
        \node (DB) at (1,0.25) {};
        \node (C2) at (0.675,0.325) {};

        \draw[\arrcol,->] (AR) -- (DB) node [midway,below right] {\textcolor{black}{$a_2$}};
        \draw[\arrcol,->] (C2) -- (AR) node [midway,left] {\textcolor{black}{$b_2$}};
        \draw[\arrcol,->] (DB) -- (C2) node [midway,above] {\textcolor{black}{$c_2$}};

        \node (BM) at (0.45,0.75) {};
        \node (CM) at (0.75,0.45) {};

        \draw[\arrcol,->] (CM) -- (BM) node [midway,above right] {\textcolor{black}{$a_3$}};

        \node at (0,0) {$\bullet$};
        \node at (1,0) {$\bullet$};
        \node at (0,1) {$\bullet$};
        \node at (1,1) {$\bullet$};
        
    \end{tikzpicture}
    \]
    \subcaption{Triangulation of the torus with one boundary component with one marked point}\label{fig:nrd_ii_triang}
    \end{subfigure}
    \begin{subfigure}[t]{0.45\textwidth}
    \[
\begin{tikzpicture}[scale=0.95]

\draw[\generictraj] (225:3) .. controls ($0.95*(135:1.5)+0.05*(225:2)$) .. (45:3);
\draw[\generictraj] (225:3) .. controls ($0.75*(135:1.5)+0.25*(225:2)$) .. (45:3);
\draw[\generictraj] (225:3) .. controls ($0.5*(135:1.5)+0.5*(225:2)$) .. (45:3);
\draw[\generictraj] (225:3) .. controls (212.5:2) and (220:1) .. (45:3);

\draw[\generictraj] (225:3) .. controls ($0.95*(315:1.5)+0.05*(225:2)$) .. (45:3);
\draw[\generictraj] (225:3) .. controls ($0.75*(315:1.5)+0.25*(225:2)$) .. (45:3);
\draw[\generictraj] (225:3) .. controls ($0.5*(315:1.5)+0.5*(225:2)$) .. (45:3);
\draw[\generictraj] (225:3) .. controls (237.5:2) and (230:1) .. (45:3);

\draw[\generictraj] (225:3) -- (135:3);
\draw[\generictraj] (225:3) .. controls ($0.33*(135:1.5)+0.335*(135:3)+0.335*(225:3)$) .. (135:3);
\draw[\generictraj] (225:3) .. controls ($0.67*(135:1.5)+0.165*(135:3)+0.165*(225:3)$) .. (135:3);
\draw[\generictraj] (225:3) .. controls (135:1.5) .. (135:3);

\draw[\generictraj] (135:3) -- (45:3);
\draw[\generictraj] (135:3) .. controls ($0.33*(135:1.5)+0.335*(135:3)+0.335*(45:3)$) .. (45:3);
\draw[\generictraj] (135:3) .. controls ($0.67*(135:1.5)+0.165*(45:3)+0.165*(135:3)$) .. (45:3);
\draw[\generictraj] (135:3) .. controls (135:1.5) .. (45:3);

\draw[\generictraj] (315:3) -- (45:3);
\draw[\generictraj] (315:3) .. controls ($0.33*(315:1.5)+0.335*(315:3)+0.335*(45:3)$) .. (45:3);
\draw[\generictraj] (315:3) .. controls ($0.67*(315:1.5)+0.165*(45:3)+0.165*(315:3)$) .. (45:3);
\draw[\generictraj] (315:3) .. controls (315:1.5) .. (45:3);

\draw[\generictraj] (315:3) -- (225:3);
\draw[\generictraj] (315:3) .. controls ($0.33*(315:1.5)+0.335*(315:3)+0.335*(225:3)$) .. (225:3);
\draw[\generictraj] (315:3) .. controls ($0.67*(315:1.5)+0.165*(225:3)+0.165*(315:3)$) .. (225:3);
\draw[\generictraj] (315:3) .. controls (315:1.5) .. (225:3);

\draw[\generictraj] (225:3) .. controls (216:1.95) and (234:1.95) .. (225:3);
\draw[\generictraj] (225:3) .. controls (214:1.875) and (236:1.875) .. (225:3);

\draw[\septraj] (135:1.5) -- (135:3);
\draw[\septraj] (135:1.5) -- (45:3);
\draw[\septraj] (135:1.5) -- (225:3);
\node at (135:1.5) {$\bm{\times}$};

\draw[\septraj] (315:1.5) -- (315:3);
\draw[\septraj] (315:1.5) -- (45:3);
\draw[\septraj] (315:1.5) -- (225:3);
\node at (315:1.5) {$\bm{\times}$};

\draw[\septraj] (225:3) .. controls (220:2) and (230:2) .. (225:3);

\draw[\septraj] (225:2) to [out=255,in=25] (225:3);
\draw[\septraj] (225:2) to [out=195,in=65] (225:3);
\draw[\septraj] (225:2) -- (45:3);
\node at (225:2) {$\bm{\times}$};

\node at (45:3) {$\bullet$};
\node at (135:3) {$\bullet$};
\node at (225:3) {$\bullet$};
\node at (315:3) {$\bullet$};

\draw[\conn, line width=0.55pt] (225:2) arc (45:0:1);
\draw[\conn, line width=0.55pt] (225:2) arc (45:90:1);
\draw[\conn, line width=0.55pt] (135:2) arc (315:360:1);
\draw[\conn, line width=0.55pt] (135:2) arc (315:270:1);
\draw[\conn, line width=0.55pt] (45:2) arc (225:270:1);
\draw[\conn, line width=0.55pt] (45:2) arc (225:180:1);
\draw[\conn, line width=0.55pt] (315:2) arc (135:90:1);
\draw[\conn, line width=0.55pt] (315:2) arc (135:180:1);

\draw[\conn, line width=0.55pt] (225:1.5) arc (45:0:1.5);
\draw[\conn, line width=0.55pt] (225:1.5) arc (45:90:1.5);
\draw[\conn, line width=0.55pt] (135:1.5) arc (315:360:1.5);
\draw[\conn, line width=0.55pt] (135:1.5) arc (315:270:1.5);
\draw[\conn, line width=0.55pt] (45:1.5) arc (225:270:1.5);
\draw[\conn, line width=0.55pt] (45:1.5) arc (225:180:1.5);
\draw[\conn, line width=0.55pt] (315:1.5) arc (135:90:1.5);
\draw[\conn, line width=0.55pt] (315:1.5) arc (135:180:1.5);

\draw[\otherconn, line width=0.55pt] (225:2) to [out=120,in=-10] (-2.12,-0.87);
\draw[\otherconn, line width=0.55pt] (2.12,-0.87) to [out=180,in=20] (315:1.5);

\node at (180:2.375) {$A$};
\node at (0:2.375) {$A$};
\node at (135:0.75) {$B$};
\node at (315:0.75) {$C$};
\node at (90:2.375) {$D$};
\node at (270:2.375) {$D$};

\end{tikzpicture}
\]
\subcaption{Trajectory structure of $\varphi^{(\varepsilon)}$, with connections giving boundaries of ring domain indicated}\label{fig:nrd_ii:traj}
    \end{subfigure}
    \caption{Triangulation and trajectory structure for the torus with one boundary component containing one marked point}
\end{figure}

We now describe the stable representations under a non-generic stability condition which is a limiting case of the stability condition from the previous proposition.
Considering this nearby non-generic stability condition will be useful in later describing the category of semistable objects.

\begin{lemma}\label{lem:nrd_ii:stables}
Under a stability condition such that for any module $M$, we have $\scph{M} = \scph{S_A \oplus S_B \oplus S_C \oplus S_D}$ if and only if $\dimu M \in \mathbb{Z}\{(1, 1, 0, 0), (0, 0, 1, 1)\} \subset \mathbb{Z}^{Q_{0}}$, and for which
\begin{align*}
    \scph{S_B} < &\scph{S_A \oplus S_B \oplus S_C \oplus S_D} < \scph{S_C}, \\
    \scph{S_D} < &\scph{S_A \oplus S_B \oplus S_C \oplus S_D} < \scph{S_A}, \\
    \scph{S_A \oplus S_B} = &\scph{S_A \oplus S_B \oplus S_C \oplus S_D} = \scph{S_C \oplus S_D},
\end{align*}
the quiver with potential $(Q,W)$
has three stable representations of phase $\scph{S_{A} \oplus S_{B} \oplus S_{C} \oplus S_{D}}$, namely the string modules $M(b_{1})$, $M(c_{2})$, and $M(b_{2}a_{3}c_{1})$.
\end{lemma}
\begin{proof}
The Jacobian algebra $\jac{Q, W}$ is a gentle algebra, being in fact finite-dimensional, and hence its indecomposable representations are all string or band modules.
The only indecomposable modules of dimension vector $(1, 1, 0, 0)$ and $(0, 0, 1, 1)$ respectively are $M(b_{1})$ and $M(c_{2})$, which are both stable.
Other stable modules in the category of semistable objects must therefore have a dimension vector which is at least one in every coordinate, there being no indecomposable modules of dimension vector $(d, d, 0, 0)$ or $(0, 0, d, d)$ for $d > 1$.

Let $M$ be a stable module with $\dimu M \in \mathbb{Z}\{(1, 1, 0, 0), (0, 0, 1, 1)\} \setminus \{(1, 1, 0, 0), (0, 0, 1, 1)\}$.
We claim that $M$ must come from a string or band containing~$a_{3}^{\pm 1}$.
Indeed, the only other way for a string $w$ to cover all four vertices of the quiver is if it contains either $a_{1}a_{2}^{-1}$, $a_{1}^{-1}a_{2}$ or their inverses.
But if $w$ contains $a_{1}a_{2}^{-1}$ then it has $S_{A}$ as a submodule and if it contains $a_{1}^{-1}a_{2}$ then it has $S_{D}$ as a factor module, both of which contradict stability.

Hence, if we have a string $w$ which gives rise to~$M$, then $w$ must contain~$a_{3}^{\pm 1}$.
We now split into two cases depending on whether $w$ contains one instance of $a_3^{\pm 1}$, or more than one.
We first consider the case where $w$ contains only one instance of~$a_3^{\pm 1}$.
There are no bands that cover all of the vertices of the quiver and contain only one instance of $a_{3}^{\pm 1}$, due to the relations in the Jacobian ideal $\widehat{J}(W)$.
We can then make the following restrictions on the letters that $w$ may begin and end with, due to the fact that $M$ must be stable.
\begin{itemize}
\item $w$ cannot begin or end with $b_{1}^{\pm 1}$, since then $M(b_{1})$ is either a submodule or a factor module.
\item $w$ cannot begin or end with $c_{2}^{\pm 1}$, since then $M(c_{2})$ is either a submodule or a factor module.
\item $w$ cannot begin with $a_{1}$ or $a_{2}$, or end with $a_{1}^{-1}$ or $a_{2}^{-1}$, since then $S_{D}$ is a factor module.
\item $w$ cannot begin with $a_{1}^{-1}$ or $a_{2}^{-1}$, or end with $a_{1}$ or $a_{2}$, since then $S_{A}$ is a submodule.
\item $w$ cannot begin with $b_{2}^{-1}$ or end with $b_{2}$, since then $S_{C}$ is a submodule.
\item $w$ cannot begin with $c_{1}$ or end with $c_{1}^{-1}$, since then $S_{B}$ is a factor module.
\end{itemize}
We conclude that $w$ may begin with $a_{3}^{\pm 1}$, $b_{2}$, or $c_{1}^{-1}$, and end with $a_{3}^{\pm 1}$, $b_{2}^{-1}$, or~$c_{1}$.

If $w$ contains only one instance of $a_{3}^{\pm 1}$, the only strings which cover all of the vertices are \[
a_{3}c_{1}a_{2},\: a_{3}b^{-1}_{1}a_{2}^{-1},\: a_{1}b_{2}a_{3},\: a_{1}^{-1}c_{2}^{-1}a_{3},\: b_{2}a_{3}c_{1}, \text{ and } c_{2}^{-1}a_{3}b_{1}^{-1},
\]
up to inverting the string.
By the restrictions on letters that $w$ may begin or end with, we note that the only possibility is $w = b_{2}a_{3}c_{1}$, which gives our final stable module $M(b_{2}a_{3}c_{1})$.

We now consider the case where $w$ contains more than one instance of $a_3^{\pm 1}$, claiming that this cannot result in a stable module.
Note that, due to the relations in the Jacobian ideal $\widehat{J}(W)$, we have that the only possible successors of $b_{2}$ and $b_{1}$ are $a_{3}$ and~$a_{3}^{-1}$, respectively, whilst the only possible predecessors of $c_{1}$ and $c_{2}$ are $a_{3}$ and~$a_{3}^{-1}$.
It follows that between instances of $a_{3}$ and $a_{3}^{-1}$ in~$w$, there are only four possibilities, up to inverting the string, namely \[
a_{3}c_{1}c_{2}^{-1}a_{3},\: a_{3}c_{1}a_{2}b_{1}a_{3}^{-1},\: a_{3}b_{1}^{-1}b_{2}a_{3}, \text{ and } a_{3}^{-1}c_{2}a_{1}b_{2}a_{3}.
\]
Indeed, recall that $a_{1}a_{2}^{-1}$, $a_{1}^{-1}a_{2}$ or their inverses cannot occur as substrings.
Furthermore, by the restrictions on the first and last letter of $w$, in the case that $w$ is a string: if $w$ does not begin with~$a_{3}^{\pm 1}$, then its second letter must be~$a_{3}^{\pm 1}$, and likewise if $w$ does not end with~$a_{3}^{\pm 1}$.

We claim that, in fact, since we must have $\dimu M \in \mathbb{Z}\{(1, 1, 0, 0), (0, 0, 1, 1)\}$, $w$ cannot contain the substrings $a_{3}c_{1}c_{2}^{-1}a_{3}$ and $a_{3}b_{1}^{-1}b_{2}a_{3}$.
Indeed, if $M$ is a band module, then $w$ cannot contain the substrings $a_{3}c_{1}c_{2}^{-1}a_{3}$ and $a_{3}b_{1}^{-1}b_{2}a_{3}$, since the former does not cover the vertex $A$ and the latter does not cover the vertex~$D$.
The band $w$ must end where it started, and every time $a_{3}^{\pm 1}$ is traversed, both of the vertices $B$ and $C$ are covered.
Hence, if $w$ contains either of $a_{3}c_{1}c_{2}^{-1}a_{3}$ or $a_{3}b_{1}^{-1}b_{2}a_{3}$, the band module $M = M(w, \lambda, m)$ cannot have $\dimu M \in \mathbb{Z}\{(1, 1, 0, 0), (0, 0, 1, 1)\}$, since the dimension at vertex $A$ or $D$ would have to be strictly less than that at $B$ and $C$.

Hence, in order to show that the substrings $a_{3}c_{1}c_{2}^{-1}a_{3}$ and $a_{3}b_{1}^{-1}b_{2}a_{3}$ cannot occur, we may assume that $M$ is a string module.
If we let $n$ be the number of occurrences of $a_{3}^{\pm 1}$ in $w$, then we have that $w$ covers $B$ and $C$ each $n$ times.
Between the first and final instances of $a_{3}^{\pm 1}$, $A$ can therefore only be covered at most $n - 1$ times.
On the other hand, $D$ can only be covered at most $n - 2$ times, since $w$ contains $b_{1}^{-1}b_{2}$, which does not cover~$D$.
Then we know that before and after the first and final instances of $a_{3}^{\pm 1}$, we can only have $b_{2}^{\pm 1}$ or $c_{1}^{\pm 1}$, each of which can only cover $A$ or $D$ one more time.
Hence, we cannot have $\dimu M(w) \in \mathbb{Z}\{(1, 1, 0, 0), (0, 0, 1, 1)\}$.
A similar argument shows that $w$ cannot contain the substring~$c_{1}c_{2}^{-1}$.

Hence, we know that $w$ is formed out of the strings $a_{3}c_{1}a_{2}b_{1}a_{3}^{-1}$ and $a_{3}^{-1}c_{2}a_{1}b_{2}a_{3}$, or their inverses, possibly with $b_{2}^{\pm 1}$ or $c_{1}^{\pm 1}$ added at the ends.
However, if $w$ then contains $b_{1}^{\pm 1}$, we have that $M(b_{1})$ is a submodule of $M$, whether $M$ is a string module or a band module, which contradicts stability of~$M$.
Similarly, if $w$ contains $c_{2}^{\pm 1}$, then $M(c_{2})$ is a factor module of~$M$, which likewise contradicts stability.
We conclude that the only stable modules are those listed.
\end{proof}

As in Section~\ref{sect:spiral:drd}, we denote the category of semistable representations from Proposition~\ref{prop:nrd_ii:4_vertex} by $\mathcal{P}^{\mathrm{gen}}$, with $\mathcal{P}^{\mathrm{gen}}_\infty$ its 3-Calabi--Yau $A_\infty$-enhancement.
We furthermore denote by $\mathcal{P}^{\mathrm{ng}}$ the category of semistable representations from Lemma~\ref{lem:nrd_ii:stables}, with $\mathcal{P}^{\mathrm{ng}}_{\infty}$ its enhancement.
As before, and as used in Proposition~\ref{prop:nrd_ii}, we can now derive a simpler description of~$\mathcal{P}^{\mathrm{gen}}_{\infty}$.

\begin{proposition}\label{prop:nrd_ii:3_vertex}
Let $Q'$ denote the quiver 
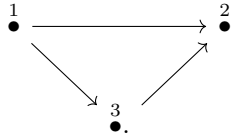
\[
    \begin{tikzcd}
        \overset{1}{\bullet} \ar[rr] \ar[dr] && \overset{2}{\bullet} \\
        & \overset{3}{\bullet}. \ar[ur] &
    \end{tikzcd}
    \]
The $A_\infty$-category $\mathcal{P}^{\mathrm{gen}}_\infty$ is quasi-equivalent to the subcategory of $\mathcal{H}_\infty(Q',0)$ consisting of objects which are semistable of phase $\vartheta = \scph{S_1 \oplus S_2}$ in $\mathcal{H}(Q',0)$, under a stability condition where we have $\scph{S_2} < \scph{S_1}$ and $\scph{S_1 \oplus S_2} = \scph{S_3}$.
Under this equivalence, representations of $Q'$ of dimension vector $(1, 1, 0)$ and $(0, 0, 1)$ both correspond to representations of $Q$ of dimension vector $(1, 1, 1, 1)$.
There is an isomorphism between the canonical choices of orientation data on the stacks of objects in each of these categories.
\end{proposition}
\begin{proof}
We first describe~$\mathcal{P}^{\mathrm{ng}}_\infty$.
Indeed, we know from Lemma~\ref{lem:nrd_ii:stables} that $\mathcal{P}^{\mathrm{ng}}$ contains only three stable modules, namely $M_{1} := M(c_{2})$, $M_{2} := M(b_{1})$, and $M_{3} := M(b_{2}a_{3}c_{1})$.
Note that this remains true even if we are given more information about the phases of other objects.

Consider the full $A_\infty$-subcategory of $\mathcal{P}^{\mathrm{ng}}_\infty$ consisting of $M_1$, $M_2$, and $M_3$.
We claim that this is quasi-equivalent to $\akd{Q'}{0}$ where the vertex $i$ of $Q'$ corresponds to~$M_i$.
The only non-zero extension spaces are $\Ext_{\jac{Q, W}}^{1}(M_{1}, M_{2}) \cong \Ext_{\jac{Q, W}}^{1}(M_{1}, M_{3}) \cong \Ext_{\jac{Q, W}}^{1}(M_{3}, M_{2}) \cong \mathbb{C}$ by Theorem~\ref{thm:cps_a}.
The only degree~$0$ morphisms are the identity morphisms, and then the fact that the degree~$2$ and $3$ morphisms match with those of $\akd{Q'}{0}$ follows from the 3-Calabi--Yau property.
One can check using the definition of the $A_\infty$-operations on $\mathcal{D}_\infty(Q,W)$ that $m_2$ is zero on $\Ext_{\mathcal{D}_\infty(Q,W)}^{1}(M_{1}, M_{2}) \otimes \Ext_{\mathcal{D}_\infty(Q,W)}^{1}(M_{3}, M_{2})$.
Since the category of semistable representations consists of the extension closure of the stable objects $M_1$, $M_2$, and $M_3$, we have that $\mathcal{P}^{\mathrm{ng}}_\infty$ is therefore quasi-equivalent to $\mathcal{H}_\infty(Q',0)$.

The category $\mathcal{P}^{\mathrm{gen}}$ consists of the semistable objects of the category $\mathcal{P}^{\mathrm{ng}}$ of phase $\scph{S_A \oplus S_B \oplus S_C \oplus S_D}$ under the stability condition $\scph{S_A \oplus S_B} < \scph{S_A \oplus S_B \oplus S_C \oplus S_D} < \scph{S_C \oplus S_D}$. Since $\dimu M_1 = (0, 0, 1, 1)$, $\dimu M_2 = (1, 1, 0, 0)$, whilst $\dimu M_3 = (1, 1, 1, 1)$, we have that $\mathcal{P}^{\mathrm{gen}}_\infty$ is quasi-equivalent to the $A_\infty$-subcategory of $\mathcal{H}_\infty(Q',0)$ which are semistable of phase $\vartheta$ under a stability condition such that $\scph{S_{2}} < \scph{S_{1}}$ and $\vartheta = \scph{S_{1} \oplus S_{2}} = \scph{S_{3}}$.

We only show that the orientation data over $M_{3}$ coincides with $\odcha{Q'}$, the cases of $M_{1}$ and $M_{2}$ being left as exercises.
We compute the parity of the orientation data by computing \[
e^{\leqslant 1}(M_3, M_3) + h^{\leqslant 1}(M_3, M_3),
\]
following Proposition~\ref{prop:od_ug}\ref{op:od_ug:calc}.
By Proposition~\ref{prop:od_ug}\ref{op:od_ug:simples->even}, we must show that this is even to show that we have an isomorphism between $\odchoice$ and $\odcha{Q'}$. We have that $e^{\leqslant 1}(M_3, M_3) = 1$, and moreover that
\begin{align*}
    h^0(M_3, M_3) &= h^{0}(S_A \oplus S_B \oplus S_C \oplus S_{D}, S_A \oplus S_B \oplus S_C \oplus S_{D}) = 4
\end{align*}    
and
\begin{align*}
    h^{1}(M_3, M_3)  &= h^1(S_A \oplus S_B \oplus S_C \oplus S_{D}, S_A \oplus S_B \oplus S_C \oplus S_{D}) \\
    &= h^{1}(S_D, S_A) + h^{1}(S_A, S_B) + h^{1}(S_C, S_B) + h^{1}(S_A, S_C) + h^{1}(S_B, S_C) + h^{1}(S_C, S_D) \\
    &= 2 + 5 \times 1.
\end{align*}
Adding these up, we obtain $1 + 4 + 7 = 12$, which is even.
\end{proof}

\section{Donaldson--Thomas invariants for the barbell quiver}\label{sect:db}

In this section we compute the  generating series for the barbell quiver with potential, which arises when considering a type~III non-degenerate ring domain. Recall that this quiver is \[Q^{\db} = 
\begin{tikzcd}
    \overset{1}{\bullet} \ar[loop left,"a"] \ar[r,"c"] & \overset{2}{\bullet} \ar[loop right,"b"]
\end{tikzcd}
\] with potential $W^{\db} = a^3 + b^3$.
Of the quivers with potential and stability condition that arise in Table~\ref{table}, this is the only one where both the potential and the stability condition are non-trivial.

We tackle this case as follows.
We divide the category of semistable representations of dimension vector $(d, d)$ into two categories which have no non-split extensions between them.
If we let $\vartheta = \scph{d, d}$, this produces a factorisation in the motivic Hall algebra $\kstaffqwdb$ of the element $[\sdbpst \hookrightarrow \dbst]$ into the elements given by the stacks of representations in each of these categories.
Consequently, by applying the integration map, the generating series for this category of semistable representations factorises into the generating series for the individual categories.
(Note that all modules over the Jacobian algebra $\jacu{\qwdb}$ are nilpotent, so we do not need to distinguish between $\jac{\qwdb}$ and $\jacu{\qwdb}$ or between $\dbst$ and $\dbstn$.)
We compute the generating series of the first category by comparing it with the category of representations of the quiver \[Q^{\kb} := \begin{tikzcd}\bullet \ar[loop left,"a"] \ar[loop right,"b"]\end{tikzcd},\] with potential $W^{\kb} := a^3 + b^3$.
We then compute the  generating series for the second category by showing that the $A_\infty$-enhancement is quasi-equivalent to degree 0 twisted complexes for the one-loop quiver.

\subsection{Dividing into two subcategories}

First we define the two categories into which we are factorising.
Let $\kbcat$ be the subcategory of $\modules \jac{\qwdb}$ consisting of quiver representations with dimension vector $(d, d)$ such that the linear map $f_c \colon V_1 \to V_2$ is an isomorphism.
Furthermore, let $\olcat = \additive\{M(acb(c^{-1}acb)^n) \st n \geqslant 0\}$.
Here `$\additive$' denotes the closure under direct sums and direct summands.

\begin{lemma}\label{lem:q2_reps}
Given a stability condition on $\modules \jac{Q^{\db}, W^{\db}}$ where $\scph{S_{2}} < \scph{S_{1}}$, the category of semistable modules of dimension vector $(d, d)$ is $\kbcat \oplus \olcat$, the closure of $\kbcat \cup \olcat$ under direct sums.
\end{lemma}
\begin{proof}
First note that any representation in $\kbcat$ must be semistable.
Indeed, if a module $M \in \kbcat$ given as a quiver representation $M = (V_1, V_2, f_a, f_b, f_c)$ has a submodule of dimension vector $(d_1, d_2)$, then we must have $d_1 \leqslant d_2$, since $f_c$ has to send $\mathbb{C}^{d_1}$ injectively into $\mathbb{C}^{d_2}$, as $f_c \colon V_1 \to V_2$ is an isomorphism.
We therefore have $\scph{d_1, d_2} \leqslant \scph{d, d}$, and so $M$ is semistable.

Let $M = (V_1, V_2, f_a, f_b, f_c)$ now be a semistable representation with $f_c$ not an isomorphism.
It suffices to show that if $M$ is indecomposable, then $M \cong M(acb(c^{-1}acb)^n)$ for some $n \geqslant 0$.
We know that such a representation $M$ is a string module or a band module, since these comprise all indecomposable representations of $\jac{\qwdb}$ by Theorem~\ref{thm:gentle_string_band}.
Since $f_c \colon V_1 \to V_2$ is not isomorphism, the kernel of $f_c$ must be non-zero.
By the definition of string and band modules, it can be seen that, in order for the kernel of $f_c$ to be non-zero, there must be at least one occurrence of the vertex $1$ in the string $w$ corresponding to $M$ which it is not incident to $c$ or~$c^{-1}$.
It follows that $M$ is a string module $M = M(w)$ and that the initial or final letter of $w$ is~$a^{\pm 1}$.
By replacing $w$ with $w^{-1}$ if necessary, we can assume that the initial letter of $w$ is~$a^{\pm 1}$.

Due to the rules defining strings, $w$ must be of the form $a^{\pm 1}cb^{\pm 1}c^{-1}a^{\pm 1} \dots$.
The first letter of $w$ cannot then be $a^{-1}$, since then the simple module $S_{1}$ is a submodule, contradicting semistability.
Hence, the first letter of $w$ is~$a$, and so we must have an initial segment $acb^{\pm 1}$.
If we have~$b^{-1}$, then $ac$ forms a submodule subword with $\dimu M(ac) = (2, 1)$, so we again do not have semistability.
Hence, the initial segment of $w$ is $acbc^{-1}a^{\pm 1}$, and we can again argue that we must in fact have $acbc^{-1}a$, on pain of losing semistability.
Since the dimension vector must be of the form $(d, d)$, by continuing in this way, we conclude that $w = acb(c^{-1}acb)^n$ for $n \geqslant 0$, which establishes the claim.
\end{proof}

It will be useful to describe the indecomposables of $\kbcat$ explicitly as string and band modules.

\begin{lemma}\label{lem:first_cat_strings_bands}
The indecomposable modules in $\kbcat$ consist of all band modules and any string module $M(w)$ such that $w = c^{\pm 1}uc^{\pm 1}$ for some string~$u$.
\end{lemma}
\begin{proof}
Suppose that $M$ is an indecomposable module in $\mathcal{C}_{1}$, and so a string module or a band module.
We have that, in terms of quiver representations, the map $f_{c}$ is an isomorphism if in the string or band the vertex $1$ is always incident to either $c$ or $c^{-1}$, up to possible rotations of the band.

Consider first the case where $M = M(v, \lambda, m)$ is a band module.
Since bands are required to be cyclic, and the vertex $1$ can only be incident to either $c^{\pm 1}$ and $a^{\pm 1}$, we must have that the vertex $1$ must be incident to both $c^{\pm 1}$ and~$a^{\pm 1}$.
Hence, $f_c$ is an isomorphism in~$M$.
Moreover, it is clear that a band must cover each of $1$ and $2$ equally many times, and so $M(v, \lambda, m)$ must have have dimension vector $(d, d)$ for some $d$.

Hence, in order for $f_c$ to be an isomorphism in a string module $M = M(w)$ with $\dimu M = (d, d)$, it is necessary and sufficient that $w$ does not begin or end with $a^{\pm 1}$.
Consequently, $w$ cannot begin or end with $b^{\pm 1}$ either, since then the other end would have to be $a^{\pm 1}$ in order for the dimension vector to be $(d, d)$.
We conclude that $w = c^{\pm 1}uc^{\pm 1}$, as desired.
\end{proof}

We can then apply this lemma to show the following.

\begin{lemma}\label{lem:db:trivial_ext}
Given $M_1 \in \kbcat$ and~$M_2 \in \olcat$, we have \[\Ext_{\jac{\qwdb}}^{1}(M_1, M_2) = \Ext_{\jac{\qwdb}}^{1}(M_2, M_1) = 0.\]
\end{lemma}
\begin{proof}
It suffices to show the claim for indecomposable modules.
For this we use the description of extensions between string and band modules from \cite{cps} as described in Section~\ref{sect:spiral:string_band:ext}.

Consider the strings of~$\olcat$.
Since the only inverse letter of $acb(c^{-1}acb)^n$ is~$c^{-1}$, it must precede all factor substrings which are not at the beginning of the string.
All factor substrings must therefore begin with $a$.
However, $a$ must always be preceded by $c^{-1}$ for any string of $(Q^{\db}, J(W^{\db}))$, unless it is at the start of the string.
Hence, no factor substring of $acb(c^{-1}acb)^n$ can be a submodule substring of $\dinf{v}$ for $v$ a band of~$\olcat$.
Since $a$ cannot start a string of $\kbcat$ by Lemma~\ref{lem:first_cat_strings_bands}, we also have that no factor substring of $acb(c^{-1}acb)^n$ can be a submodule substring of a string of~$\olcat$.
One similarly shows that no submodule substring of $acb(c^{-1}acb)^n$ can be a factor substring of a string or band of $\olcat$ by observing that such a substring of $acb(c^{-1}acb)^n$ must be succeeded by~$c^{-1}$.

It follows by Theorem~\ref{thm:cps_bc} from the above that there can be no extensions between the band modules of $\kbcat$ and the modules of $\olcat$ and that there can be no overlap extensions between string modules of $\kbcat$ and those of~$\olcat$.
By Theorem~\ref{thm:cps_a}, all that remains to show is that there can be no arrow extensions between the string modules of $\kbcat$ and $\olcat$.
This then follows from the fact that the strings of $\kbcat$ begin and end $c^{\pm 1}$, while the strings of $\olcat$ begin and end $a$ and $b$.
There is no arrow which can be placed between these to give an arrow extension.
\end{proof}

\subsection{Calculation of generating series}
\subsubsection{Generating series of first category}

The variety of representations in $\kbcat$ of dimension vector $(d, d)$ is given by the image of the injection
\begin{align*}
    \qaasch{\qwkb}{d} \times \gla{d} &\hookrightarrow \qaasch{\qwdb}{(d, d)} \\
    ((f_a, f_b), g) &\mapsto (f_a, g, f_{b}),
\end{align*}
where $f_a$ and $f_b$ are the linear maps at the loops $a$ and $b$ of $Q^{\kb}$, and $(f_a, g, f_b)$ give the linear maps assigned to $a$, $c$, and $b$ respectively for $Q^{\db}$.
This is well-defined since $f_a$ and $f_b$ clearly satisfy the relations of the Jacobian ideal $J(W^{\kb})$ if and only if they satisfy the relations of the Jacobian ideal $J(W^{\db})$.
Hence, the image of the map indeed consists of the representations for which the map assigned to $c$ is an isomorphism, namely the representations lying in $\kbcat$.
We denote the image of this map by $\dbcatadsch$, the stack-theoretic quotient of this by $\gla{(d, d)}$ by $\dbcatadst$, and define $\dbcatast$ as $\coprod_{d = 0}^{\infty}\dbcatadst$.
One can then define the generating series for the category $\kbcat$ as \[
\genser{\kbcat} = \tsp(\intks([\dbcatast \hookrightarrow \dbst])).
\]
We compute this generating series.

\begin{proposition}\label{prop:a1_gen_series}
    The  generating series for $\kbcat$ is \[\genser{\kbcat} = \qdl{t^{(1,1)}}^4\qdl{-q^{1/2}t^{(2,2)}}^{-1}.\]
\end{proposition}
\begin{proof}
We have that
\begin{align*}
\dbcatadst = (\qaasch{\qwkb}{d} \times \gla{d})/\gla{(d, d)} \cong \qaasch{\qwkb}{d}/\gla{d} = \qaast{\qwkb}{d}.
\end{align*}
Since the trace function $\tr{W^{\db}}$ coincides with the trace function $\tr{W^{\kb}}$ as the potentials are the same, we therefore obtain that the  generating series of $\kbcat$ is
\begin{align*}
    \tsp(\intks([\dbcatast \hookrightarrow \dbst])) &= \tsp(\intbbs([\dbcatast \hookrightarrow \dbst])) \\
    &= \tsp(\intbbs([\qast{\qwkb} \to \qast{\qwkb}])) \\
    &= \genser{\qwkb} = \qdl{t^{(1,1)}}^4\qdl{-q^{1/2}t^{(2,2)}}^{-1},
\end{align*}
with the final step holding by Proposition~\ref{prop:haiden}.
\end{proof}

Note that here we considered modules over $\jacu{\qwkb}$ rather than modules over $\jac{\qwkb}$, which is why we needed to use the integration map $\intbbs$ rather than~$\intks$.

\subsubsection{Generating series of second category}

We now compute the  generating series for $\olcat$.

Non-zero modules in $\olcat$ have dimension vector $(2d, 2d)$ for $d \geqslant 1$ and there is one indecomposable module up to isomorphism for each $d$, namely $M(acb(c^{-1}acb)^{d - 1})$.
The variety $\dbcatbdsch$ of representations in $\olcat$ of dimension vector $(2d, 2d)$ is given by the union of the $\gla{(2d, 2d)}$-orbit of this indecomposable representation with the $\gla{(2d, 2d)}$-orbits of the decomposable representations in $\olcat$ of dimension vector~$(2d, 2d)$.
We denote the stack-theoretic quotient of this by $\gla{(d, d)}$ by $\dbcatbdst$, and define $\dbcatast$ as $\coprod_{d = 0}^{\infty}\dbcatbdst$.
One can then define the generating series for the category $\olcat$ as we did for~$\kbcat$.

The 3-Calabi--Yau $A_\infty$-enhancement of $\olcat$ is quasi-equivalent to a more familiar $A_{\infty}$-category.

\begin{proposition}\label{prop:second_cat=one_loop}
The 3-Calabi--Yau $A_\infty$-enhancement of $\olcat$ is quasi-equivalent to $\mathcal{H}_\infty(Q^{{\circ}},0)$, where $Q^{{\circ}} = \begin{tikzcd}
    \bullet \ar[loop right]
\end{tikzcd}$.
Moreover, the canonical orientation data on the respective stacks of objects in each of these categories coincides.
\end{proposition}
\begin{proof}
The category $\olcat$ has a unique simple object, namely $M(acb)$.
In the minimal model of the $A_\infty$-enhancement this has a unique degree~$1$ endomorphism given by~$c^\ast$, corresponding to an arrow extension of $M(acb)$ with itself.
However, by the description of $\akddb$, we have that all of the $A_\infty$-operations $m_n$ in $\akddb$ on $c^\ast$ with itself are zero, since $c$ does not occur in the potential~$W^{\db}$.
Hence, the 3-Calabi--Yau $A_\infty$-category which consists of the single object $M(acb)$ is quasi-equivalent to $\akd{Q'}{0}$, and so the 3-Calabi--Yau $A_\infty$-enhancement of $\olcat$ is quasi-equivalent to $\mathcal{H}_{\infty}(Q^{{\circ}}, 0)$.

We now show that the orientation data on $\olcat$ coincides with the canonical choice of orientation data for~$Q^{{\circ}}$.
By Proposition~\ref{prop:od_ug}, it suffices to show that \[
e^{\leqslant 1}(M(acb), M(acb)) + h^{\leqslant 1}(M(acb), M(acb))
\]
is even. 
We have that $e^{0}(M(acb), M(acb)) = 1$ since $M(acb)$ is stable, and that $e^{1}(M(acb), M(acb)) = 1$ since this space is generated by the arrow extension given by~$c$.
We furthermore calculate that
\begin{align*}
h^{0}(M(acb), M(acb)) &=  h^{0}(S_1^{\oplus 2} \oplus S_2^{\oplus 2}, S_1^{\oplus 2} \oplus S_2^{\oplus 2}) = 8,
\end{align*}
and 
\begin{align*}
h^{1}(M(acb), M(acb)) &= h^{1}(S_1^{\oplus 2} \oplus S_2^{\oplus 2}, S_1^{\oplus 2} \oplus S_2^{\oplus 2}) = 12.
\end{align*}
So that \[e^{\leqslant 1}(M(acb), M(acb)) + h^{\leqslant 1}(M(acb), M(acb)) = 22,\] which is even.
\end{proof}

\begin{corollary}\label{cor:a2_gen_series}
The  generating series of $\olcat$ is \[\qdl{-q^{-1/2}t^{(2,2)}}^{-1}.\]
\end{corollary}
\begin{proof}
By Proposition~\ref{prop:second_cat=one_loop}, we have
\begin{align*}
    \tsp(\intks([\dbcatbst \hookrightarrow \dbst])) &= \tsp(\intks([\qastn{Q^{{\circ}}} \rightarrow \qastn{Q^{{\circ}}}]))|_{t \mapsto t^{(2, 2)}} \\
                    &= \qdl{-q^{-1/2}t}^{-1}|_{t \mapsto t^{(2, 2)}} \\
                    &= \qdl{-q^{-1/2}t^{(2,2)}}^{-1},
\end{align*}
with the penultimate step following from Proposition~\ref{prop:haiden}.
\end{proof}

Putting the two generating series from $\kbcat$ and $\olcat$ together, we finally obtain the generating series for $\sdbpst$, recalling that $\vartheta = \scph{1, 1}$.

\begin{proposition}\label{prop:db_gen_series}
The quiver 
with potential $(Q^{\db},W^{\db})$ has generating series \[\qdl{t^{(1,1)}}^4\qdl{-q^{1/2}t^{(2,2)}}^{-1}\qdl{-q^{-1/2}t^{(2,2)}}^{-1}\] for representations of dimension vector $(d,d)$ given a stability condition where $\scph{S_2} < \scph{S_1}$.
\end{proposition}
\begin{proof}
Since $\olcat$ and $\kbcat$ together contain all indecomposable semistable representations of phase $\vartheta$ of the quiver with potential, we obtain that every semistable representation of $(\qwdb)$ of phase $\vartheta$ is an extension --- in fact, a direct sum --- of a representation of $\kbcat$ with a representation from $\olcat$.
Moreover, since by Lemma~\ref{lem:db:trivial_ext} we have no non-split extensions between $\kbcat$ and $\olcat$, every semistable representation of $(\qwdb)$ of phase $\vartheta$ can be expressed \textit{uniquely} as an extension of a representation from $\kbcat$ and a representation from $\olcat$.
Hence, in the motivic Hall algebra $\kstaffqwdb$ we have a factorisation \[
[\sdbpst \to \dbst] = [\dbcatast \to \dbst] \star [\dbcatbst \to \dbst].
\]
Applying the integration map $\intks$, noting that every representation of $(\qwdb)$ is nilpotent and recalling that this is an algebra homomorphism, we obtain that
\begin{align*}
\tsp(\intks([\sdbpst \to \dbst])) &= \tsp(\intks([\dbcatast \to \dbst])) \\
&\quad \times \tsp(\intks([\dbcatbst \to \dbst])) \\
&= \qdl{t^{(1,1)}}^4\qdl{-q^{1/2}t^{(2,2)}}^{-1}\qdl{-q^{-1/2}t^{(2,2)}}^{-1},
\end{align*}
by Proposition~\ref{prop:a1_gen_series} and Corollary~\ref{cor:a2_gen_series}.
\end{proof}

Note the advantage of using the Kontsevich--Soibelman theory in computing the generating series here.
It allows us to consider the category $\olcat$ intrinsically and compare it with nilpotent representations of the one-loop quiver with potential.
The alternative is computing the pullback of the motivic vanishing cycle of $\trst{W^{\db}_{(2d, 2d)}}$ on $\dbdst$ to $\dbcatbdst$.

Because $Q^{\db}$ is not symmetric, we obtain the following quantum dilogarithm identity using the Kontsevich--Soibelman wall-crossing formula. 
A numerical version of this identity appeared in \cite[(2.24)]{gmn_cmp}.%

\begin{corollary}
In the ring $\qringdb$, we have the identity
\begin{align*}
    \qdl{t^{(0,1)}}^2\qdl{t^{(1,0)}}^2 &= \qdl{t^{(1,0)}}^2\qdl{t^{(2,1)}}^2 \dots \qdl{t^{(1,1)}}^4\qdl{-q^{1/2}t^{(2,2)}}^{-1}\qdl{-q^{-1/2}t^{(2,2)}}^{-1} \dots \\ 
    &\quad \dots \qdl{t^{(1,2)}}^2\qdl{t^{(0,1)}}^2.
\end{align*}
\end{corollary}
\begin{proof}
We have by the Kontsevich--Soibelman wall-crossing formula \cite[p.27]{ks_stability}, \cite[Proposition~6.23]{dm_hdo} that \[
\genser{\qwdb} = \prod_{\substack{\vartheta \in (0, 1] \\ \text{decreasing}}} \genserp{\vartheta}{\qwdb}
\]
for any stability condition.
For stability conditions where $\scph{S_{1}} < \scph{S_{2}}$, the only stable modules are the simple modules, and so the product on the right is $\qdl{t^{(0,1)}}^2\qdl{t^{(1,0)}}^2$.
For stability conditions where $\scph{S_1} > \scph{S_2}$, we computed $\genserp{\vartheta}{\qwdb}$ where $\vartheta = \scph{1, 1}$ in Corollary~\ref{prop:db_gen_series}.
One can show in a similar way to Lemma~\ref{lem:q2_reps} that the only other stable modules for such a stability condition are of dimension vectors $(d, d + 1)$ and $(d + 1, d)$ for $d \geqslant 0$.
These stable modules are respectively submodules and factor modules of the family of string modules $M(acb(c^{-1}acb)^{n - 1})$.
One can then show that for $\vartheta = \scph{d, d + 1}$, we have $\genserp{\vartheta}{\qwdb} = \qdl{t^{(d, d + 1)}}^{2}$ by describing the category of semistables in terms of the quiver $Q = \begin{tikzcd}\bullet \ar[loop left,"a"]\end{tikzcd}$ with potential $W = a^3$.
The same applies for $(d + 1, d)$, which gives that the product on the right for these stability conditions is \[
\qdl{t^{(1,0)}}^2\qdl{t^{(2,1)}}^2 \dots \qdl{t^{(1,1)}}^4\qdl{-q^{1/2}t^{(2,2)}}^{-1}\qdl{-q^{-1/2}t^{(2,2)}}^{-1} \dots \qdl{t^{(1,2)}}^2\qdl{t^{(0,1)}}^2.
\]
Applying the wall-crossing formula, we have that these two products are both equal to $\genser{\qwdb}$, which gives the identity.
\end{proof}

One can alternatively derive the generating series of the categories of semistable objects of dimension vectors $(nd, n(d + 1))$ and $(n(d + 1), nd)$ using the quadratic differentials perspective.
Indeed, the barbell quiver arises from a quadratic differential on the sphere with a triple pole and two simple poles.
The quadratic differentials corresponding to stability conditions with $\scph{S_2} < \scph{S_1}$ have a type~III non-degenerate ring domain in some rotation.
The stable objects of dimension vector $(d, d + 1)$ and $(d + 1, d)$ correspond to the type~II saddle connections lying inside the ring domain, wrapping around the type~III saddle trajectory.
We then know that the  generating series corresponding to a type~II saddle trajectory $\gamma$ is $\qdl{t^{\hhc{\gamma}}}^2$ by Proposition~\ref{prop:ii}.

\printbibliography

\end{document}